\ifdefined\directlua
\DocumentMetadata{
tagging=on,
pdfstandard=ua-2,
lang=en,
tagging-setup = {math/setup={mathml-AF,mathml-SE}}
}
\tagpdfsetup{math/alt/use}
\ExplSyntaxOn
\int_set:Nn \l__tag_loglevel_int { -1 }
\msg_redirect_name:nnn { tag } { role-parent-child-forbidden } { none }
\msg_redirect_name:nnn { tag } { role-struct-parent-child-forbidden } { none }
\msg_redirect_name:nnn { tag } { role-parent-child-unresolved } { none }
\ExplSyntaxOff
\fi

\documentclass[12pt]{amsart}

\usepackage{stackengine}

\usepackage{amssymb}
\usepackage{amscd}
\usepackage[noadjust]{cite}
\usepackage{booktabs}
\usepackage{url}
\usepackage{hyphenat}
\usepackage{mathtools}
\usepackage[T1]{fontenc}
\usepackage{enumitem}
\usepackage{mathrsfs}
\usepackage{xr-hyper}
\usepackage{xcolor}

\definecolor{darkblue}{rgb}{0,0,0.4}
\usepackage[lmargin=1in, rmargin=1in, tmargin=1in, bmargin=1in]{geometry}
  \usepackage[bookmarks=true, bookmarksopen=true,%
    bookmarksdepth=3,bookmarksopenlevel=2,%
    colorlinks=true,%
    linkcolor=darkblue,%
    citecolor=darkblue,%
    filecolor=darkblue,%
    menucolor=darkblue,%
    urlcolor=darkblue]{hyperref}
  \hypersetup{pdftitle={Real bordered Floer homology}}
  \hypersetup{pdfauthor={Robert Lipshitz and Peter S. Ozsváth}}
\usepackage{color}
\usepackage{tikz}
\usetikzlibrary{matrix,arrows,cd,shapes.geometric,patterns}
\ifdefined\directlua\else
  \tikzset{alt/.code={\relax}}
  \pgfkeys{/tikzcd/alt/.code={\relax}}
\fi

\newcommand{\st}{^\text{st}}
\def\mathcenter#1{%
  \vcenter{\hbox{$#1$}}%
}

\newcommand{\RR}{\mathbb R}
\newcommand{\CC}{\mathbb C}
\newcommand{\ZZ}{\mathbb Z}

\newcommand{\FF}{\mathbb F}

\newcommand{\Image}{\mathrm{Im}}

\newcommand{\wt}{\widetilde}
\newcommand{\ul}{\underline}

\newcommand{\co}{\nobreak\mskip2mu\mathpunct{}\nonscript
  \mkern-\thinmuskip{:}\penalty300\mskip6muplus1mu\relax}

\newcommand{\OneHalf}{{\textstyle\frac{1}{2}}}
\newcommand{\OneQuarter}{{\textstyle\frac{1}{4}}}

\newcommand{\bdy}{\partial}
\newcommand{\lra}{\longrightarrow}

\newcommand{\lbracket}{[}
\newcommand{\rbracket}{]}

\newcommand{\spinc}{\mathfrak s}

\DeclareMathOperator{\Hom}{Hom}

\DeclareMathOperator{\spin}{spin}
\newcommand{\SpinC}{\spin^c}

\DeclareMathOperator{\Pin}{Pin}

\DeclareMathOperator{\ind}{ind}

\DeclareMathOperator{\ev}{ev}
\DeclareMathOperator{\gr}{gr}

\theoremstyle{plain}

\numberwithin{equation}{section}
\newtheorem{theorem}[equation]{Theorem}

\newtheorem{proposition}[equation]{Proposition}
\newtheorem{lemma}[equation]{Lemma}
\newtheorem{corollary}[equation]{Corollary}

\newtheorem{definition}[equation]{Definition}

\theoremstyle{definition}

\theoremstyle{remark}
\newtheorem{example}[equation]{Example}
\newtheorem{remark}[equation]{Remark}

\newcommand{\HF}{\mathit{HF}}
\newcommand{\HFa}{\widehat {\HF}}
\newcommand{\SFH}{\mathit{SFH}}

\newcommand{\CFa}{\widehat {\mathit{CF}}}

\newcommand{\HFKa}{\widehat{\mathit{HFK}}}

\newcommand{\x}{\mathbf x}
\newcommand{\y}{\mathbf y}
\newcommand{\z}{\mathbf z}
\newcommand{\w}{\mathbf w}

\newcommand\HH{\mathit{HH}}

\newcommand\Hochschild\HH

\newcommand{\Alg}{\mathcal{A}}

\newcommand{\Idem}{\mathcal{I}}
\newcommand{\alphas}{{\boldsymbol{\alpha}}}
\newcommand{\betas}{{\boldsymbol{\beta}}}

\newcommand{\rhos}{{\boldsymbol{\rho}}}

\newcommand{\cM}{\mathcal{M}}

\newcommand{\DD}{\textit{DD}}

\newcommand{\AAm}{\textit{AA}} 
\newcommand{\CFD}{\mathit{CFD}}
\newcommand{\CFDD}{\mathit{CFDD}}
\newcommand{\CFA}{\mathit{CFA}}

\newcommand{\CFDA}{\mathit{CFDA}}
\newcommand{\CFDAa}{\widehat{\CFDA}}
\newcommand{\CFAA}{\mathit{CFAA}}
\newcommand{\CFAAa}{\widehat{\CFAA}}
\newcommand{\CFDa}{\widehat{\CFD}}
\newcommand{\CFDRa}{\widehat{\mathit{CFDR}}}
\newcommand{\CFARa}{\widehat{\mathit{CFAR}}}

\newcommand{\CFK}{\mathit{CFK}}
\newcommand{\CFKa}{\widehat{\CFK}}
\newcommand{\CFKm}{\CFK^-}

\newcommand{\CFDDa}{\widehat{\CFDD}}

\newcommand{\CFAa}{\widehat{\CFA}}

\newcommand{\Source}{{S^{\mspace{1mu}\triangleright}}}

\newcommand{\cZ}{\mathcal{Z}}
\newcommand{\PtdMatchCirc}{\cZ}
\newcommand{\PMC}{\PtdMatchCirc}

\newcommand{\CircPts}{{\mathbf{a}}}
\newcommand{\bbpt}{\mathbf{z}}

\newcommand\Id{\mathbb{I}}

\newcommand\DT{\boxtimes}
\newcommand\Gen{\mathfrak{S}}

\newcommand{\Field}{\FF_2}

\DeclareMathOperator{\nbd}{nbd}

\newcommand{\dbar}{\bar{\partial}}
\newcommand{\Heegaard}{\mathcal{H}}
\newcommand{\HD}{\Heegaard}

\renewcommand{\th}{^\text{th}}
\renewcommand{\st}{^\text{st}}
\newcommand{\bigGroup}{G'}

\newcommand{\grb}{\gr'}

\DeclareMathOperator{\Mor}{Mor}

\newcommand{\AZ}{\mathsf{AZ}}
\newcommand{\AZbar}{\overline{\AZ}}

\newcommand{\op}{\mathrm{op}}

\newcommand{\Chord}{\mathrm{Chord}}

\newcommand{\ol}[1]{\overline{#1}{}}
\newcommand{\HB}{\mathsf{H}}

\newcommand{\sgHB}{\HB_{\mathit{sg}}}
\newcommand{\Wh}{\mathit{Wh}}

\makeatletter
\newcommand\honestalg[3]{\bigl\lbracket
\begin{smallmatrix} #1\@ifempty{#3}{}{&#3} \\ #2 \end{smallmatrix}
\bigr\rbracket}

\makeatother
\newcommand{\lab}[1]{$\scriptstyle #1$}

\newcommand{\lsup}[2]{{}^{#1}\mskip-.6\thinmuskip#2}

\newcommand{\kotimes}[1]{\otimes}

\newcommand\goesto\mapsto

\newcommand\dotimes\times 

\newcommand{\HFR}{\mathit{HFR}}
\newcommand{\HFRa}{\widehat{\mathit{HFR}}}
\newcommand{\CFR}{\mathit{CFR}}
\newcommand{\CFRa}{\widehat{\CFR}}

\definecolor{darkgreen}{rgb}{0,0.4,0} 
\definecolor{darkbrown}{rgb}{.48,0.33,.24}  

\begin{document}
\title[Real bordered Floer homology]{Real bordered Floer homology}
\author[Lipshitz]{Robert Lipshitz}
\thanks{RL was supported by NSF grant DMS-2505715, a Simons Foundation travel grant, and the Fernholz Foundation and Princeton University.}
\address{Department of Mathematics, University of Oregon\\
  Eugene, OR 97403}
\email{lipshitz@uoregon.edu}

\author[Ozsv\'ath]{Peter Ozsv\'ath}
\thanks{PSO was supported by NSF grants DMS-1708284 and DMS-2104536.}
\address {Department of Mathematics, Princeton University\\ New
  Jersey, 08544}
\email {petero@math.princeton.edu}

\begin{abstract}
  Fix a 3-manifold $Y$ with boundary $F\amalg F$ and an orientation-preserving involution $\tau\co Y\to Y$ exchanging the boundary components, with nonempty fixed set. To an appropriate kind of Heegaard diagram for $Y$, we describe how to associate a module over the bordered Heegaard Floer algebra $\Alg(F)$. These modules satisfy a gluing, or pairing, theorem, and extend Guth-Manolescu's real Heegaard Floer homology $\HFRa$. Using these modules, we give a practical algorithm to compute $\HFRa(Y,\tau)$ for real 3-manifolds $(Y,\tau)$ with connected fixed set.
\end{abstract}

\maketitle 

\tableofcontents


\section{Introduction}

Many recent applications of low-dimensional Floer theory have relied on invariants incorporating the setting's symmetries. Notably, these include the internal symmetries of the invariants, particularly the conjugation symmetry on $\SpinC$-structures considered in 
$\Pin(2)$-equivariant monopole Floer homology~\cite{Manolescu16:triangulation,Lin18:Pin2} and
involutive Heegaard Floer homology~\cite{HM17:involutive} (with applications, for instance, to the Triangulation Conjecture~\cite{Manolescu16:triangulation} and the structure of the homology cobordism group~\cite{DHST23:homol-cob-summand}); and external symmetries of the manifolds, such as the $\ZZ/2$-action exchanging the sheets in branched double covers. Combining these symmetries has allowed these invariants to resolve problems that seemed previously inaccessible, such as proving that cables of the figure-8 knot are not slice~\cite{DKM24:cable-of-8} or giving exotic surfaces that do not dissolve after a single stabilization~\cite{Guth:exotic-stab} (see also~\cite{Kang:exotic-stab,LinMukherjee25:exotic-surfs,KPT:two-stabs}, among others).

Li introduced another way of studying $\ZZ/2$-symmetries in Seiberg-Witten theory, for a three-manifold equipped with an involution, by composing the involution of the manifold with $\SpinC$-conjugation to define real monopole Floer homology~\cite{Li:real-monopole}. Inspired by his work, Guth-Manolescu defined real Heegaard Floer homology groups $\HFR^\circ$, equivariant versions of the various variants of Heegaard Floer homology~\cite{GM:real-HF}. (See~\cite{GM:real-HF} for a more thorough review of the literature. In particular, 4-dimensional constructions that these Floer theories intend to extend have had particularly striking applications, such as~\cite{Miyazawa:exotic-RP2}.)

Real Heegaard Floer homology groups are invariants of a 3-manifold $Y$ with an orientation-preserving, smooth involution $\tau\co Y\to Y$ with nonempty (hence 1-dimensional) fixed set.
Unlike involutive Heegaard Floer homology, real Heegaard Floer homology is defined via a strict symmetry of a Heegaard diagram, which
exchanges the $\alpha$- and $\beta$-curves and reverses the orientation of the Heegaard surface.

Although its construction is concrete, in terms of a specific diagram, real Heegaard Floer homology groups remain somewhat mysterious, and have been computed in relatively few cases. The main goal of this paper is to start to ameliorate that, by giving a practical algorithm to compute $\HFRa(Y,\tau)$, at least when the fixed set of $\tau$ is connected. Our strategy is to give an extension of $\HFRa$ to 3-manifolds with boundary, as a variant of bordered Heegaard Floer homology, and then extend the algorithm for computing bordered Heegaard Floer homology from our earlier work (with D.~Thurston)~\cite{LOT4}. We have implemented this algorithm, as part of Bohua Zhan's bordered Floer homology python software package~\cite{Zhan:code}. 

(To date, both real Heegaard Floer homology and bordered Floer homology have only been constructed with $\FF_2$-coefficients, so all Floer complexes in this paper are over $\FF_2$.)

A real bordered 3-manifold is an arced bordered $3$-manifold $Y$ with two boundary components,
together with an involution $\tau\co Y\to Y$ exchanging the two boundary components while respecting their parametrizations
by a pointed matched circle $\PMC$.
(See Definition~\ref{def:real-bord-mfld}.) In the interests of brevity, this paper develops only the aspects of the bordered invariants of such $(Y,\tau)$ needed to compute $\HFRa(Y,\tau)$. We summarize the structure of our invariants, as developed in this paper, in the following: 
\begin{theorem}
  Given a real bordered 3-manifold $(Y,\tau)$ with boundary $-F(\PMC)\amalg -F(\PMC)$, represented by a real bordered Heegaard diagram $(\HD,\tau)$, there is an associated type $D$ structure $\CFDRa(Y,\tau)$ over the (usual) bordered algebra $\Alg(\PMC)$. Moreover, 
  for any bordered $3$-manifold $Y'$ with boundary $-F(\PMC')\amalg F(\PMC)$, we have
  \[
  \CFDRa(Y'\cup_{F(\PMC)}Y\cup_{F(\PMC)}Y',\wt{\tau})\simeq \CFDAa(Y')\DT\CFDRa(Y),  
  \]
  where $\wt{\tau}$ is the evident extension of $\tau$, exchanging the two copies of $Y'$.
  Finally, if $Y$ is closed (or the boundary consists of two copies of $S^2$), then $\CFDRa(Y,\tau)\simeq \CFRa(Y,\tau)$, Guth-Manolescu's real Heegaard Floer complex.
\end{theorem}
This is proved below as Theorems~\ref{thm:CFDR-defined} and~\ref{thm:gen-pairing}. (We also give a more direct proof of the pairing theorem for nice diagrams, Theorem~\ref{thm:nice-pairing}.) As some readers may notice, the form of real bordered Floer homology is pleasantly simpler than the form of involutive bordered Floer homology~\cite{HL19:bord-involutive}; this leads also to a more efficient computer implementation. (A special case of the type $A$ invariant, Proposition~\ref{prop:CFAR-AZ}, also makes the computer implementation more efficient.)

To compute $\HFRa(Y,\tau)$, we start from a real Heegaard splitting $Y=H\cup_\Sigma \tau(H)$. There is a particularly attractive real bordered Heegaard diagram for $\nbd(\Sigma)$, coming from the Auroux-Zarev diagram $\AZ(\PMC)$~\cite{Auroux10:ICM,Zarev:JoinGlue,LOTHomPair} associated to a symmetric pointed matched circle $(\PMC,\tau)$, and the associated module $\CFDRa(\AZ(\PMC),\tau)$ has a simple description. (See, for instance, Corollary~\ref{cor:AZ-simplest-model}, which suggests that perhaps analogous constructions would make sense in other Fukaya categories or in purely algebraic settings.) Tensoring this with the bordered invariant $\CFAa(H)$ (with appropriate framing), which we computed in earlier work~\cite{LOT4}, gives $\CFRa(Y,\tau)$.

Although we develop mainly the aspects of real bordered Floer homology needed for this algorithm, the construction has other computational and conceptual applications as well. The existence of real bordered Floer homology implies a surgery exact triangle for $\HFRa$. Identifications between the real bordered modules for real tori and certain bordered-sutured invariants imply that if a real 3-manifold $(Y,\tau)$ admits a separating symmetric torus containing the fixed set, then $\HFRa(Y,\tau)$ agrees with a familiar sutured Floer homology group. In a similar direction, if $(Y,\tau)$ admits a separating symmetric surface with orientable quotient, then $\HFRa(Y,\tau)$ agrees with a real knot Floer homology group. We also obtain a proof of a structural conjecture of Hendricks's, about the first differential in her spectral sequence to real Heegaard Floer homology~\cite{Hendricks:HFR-loc}, at least if the fixed set is connected. Finally, we observe that real bordered Floer homology can be used to compute the real Floer homology of branched covers of even winding number satellites, and illustrate this with the Whitehead pattern and $(2,1)$ cabling pattern. This gives, for instance:
\begin{theorem}\label{thm:Whitehead-nontriv}
  Given a knot $K\subset S^3$, let $\Wh(K)$ denote the (positive or negative) Whitehead double of $K$. 
  If $K$ is a nontrivial alternating knot, then there is a strict inequality 
  \[
  \dim\HFRa(\Sigma(\Wh(K)),\tau)<\dim\HFa(\Sigma(\Wh(K))).
  \] 
\end{theorem}
\noindent
In other words, Hendricks's spectral sequence is nontrivial in these cases.

This paper is organized as follows. We describe the kinds of Heegaard diagrams used in the construction, as well as states and domains connecting them, in Section~\ref{sec:HD}. The analysis of the moduli spaces of holomorphic curves is in Section~\ref{sec:moduli}, starting with the definition of moduli spaces and then establishing their Fredholm theory (transversality, expected dimensions) and codimension-1 boundary. Section~\ref{sec:CFDR} defines the bordered type $D$ structure associated to a real bordered 3-manifold; the gradings for these modules are deferred to Section~\ref{sec:gradings}. In Section~\ref{sec:nice-and-AZ} we study nice real bordered Heegaard diagrams and their invariants in general, and then the Auroux-Zarev diagrams specifically. Section~\ref{sec:pairing} proves the pairing theorem for nice real bordered Heegaard diagrams, Theorem~\ref{thm:nice-pairing}.

Applications begin in Section~\ref{sec:compute-HFRa}, with the algorithm to compute $\HFRa$. Other applications are in Section~\ref{sec:applications}, to Hendricks's spectral sequence, a surgery exact triangle, and the real Heegaard Floer homology of certain classes of manifolds, like branched covers of satellite knots. We conclude, in Section~\ref{sec:not-here}, with a discussion of results not in this paper---topics we omit for brevity, but which would be part of a thorough development of real bordered Floer theory.

\subsection*{Acknowledgments} We thank Fraser Binns, Gary Guth, Kristen Hendricks, Adam Levine, Semon Rezchi\-kov, and Yong\-han Xiao for helpful conversations and corrections. This research was conducted while RL was visiting Princeton University, supported by the Fernholz Foundation, and he is grateful for their support.

The first author used Anthropic's Claude chatbot for late-stage proofreading and to help with drawing Figure~\ref{fig:tori-domain-types}. The second author used xfig.

\section{Real bordered Heegaard diagrams}\label{sec:HD}
In Section~\ref{sec:real-bord-mflds}, we spell out the class of 3-dimensional cobordisms with real involutions that we will work with, the kinds of bordered Heegaard diagrams we use for them, the gluing operation we consider, and some examples. In Section~\ref{sec:states-and-domains}, we explain the topological framework for bordered Floer homology of these diagrams, discussing states (generators) and domains connecting them.

\subsection{Real bordered 3-manifolds and their Heegaard diagrams}\label{sec:real-bord-mflds}

Recall that a \emph{pointed matched circle} is a tuple
$\PMC=(Z,\CircPts,M,z)$, where $Z$ is a circle,
$\CircPts\subset Z$ consists of $4k$ points,
$M\co\CircPts\to\CircPts$ is a fixed point-free involution (the
\emph{matching}), and $z\in Z\setminus \CircPts$ is a basepoint~\cite[Definition 3.16]{LOT1}. We require that performing surgery on $Z$ along $(\CircPts,M)=\coprod_{2k}S^0$ results in a single circle. When
considering bordered Heegaard diagrams, the points $\CircPts$ will be
the endpoints of the $\alpha$-arcs. A \emph{$\beta$-pointed matched
circle} consists of exactly the same data as a pointed matched circle,
but we think of $\CircPts$ as the endpoints of the $\beta$-arcs~\cite[Section 3.1]{LOTHomPair}. In
particular, given a pointed matched circle $\PMC$ there is a
corresponding $\beta$-pointed matched circle $\PMC^\beta$.

A pointed matched circle specifies a surface $F(\PMC)$ by thickening $Z$ to $Z\times [1,2]$ and attaching 2-dimensional 1-handles to $\CircPts\times\{1\}$ as specified by $M$. A $\beta$-pointed matched circle also specifies a surface, but now the surface is obtained from $Z\times[1,2]$ by attaching 2-dimensional 1-handles to $\CircPts\times\{2\}$. So, there is a canonical orientation-reversing diffeomorphism $K_{\alpha,\beta}\co F(\PMC)\to F(\PMC^\beta)$, or equivalently an orientation-preserving diffeomorphism $\ol{K}_{\alpha,\beta}\co F(\PMC)\to -F(\PMC^\beta)$. 
A pointed matched circle $\PMC$ also specifies a bordered algebra $\Alg(\PMC)$~\cite[Chapter 3]{LOT1}. By convention, $\Alg(\PMC^\beta)$ is the opposite algebra to $\Alg(\PMC)$~\cite[Section 3.2]{LOTHomPair}.

A \emph{real 3-manifold} is a closed 3-manifold with an orientation-preserving, smooth involution with nonempty fixed set. We will consider the following extension to 3-manifolds with parameterized boundary:

\begin{definition}\label{def:real-bord-mfld}
  A \emph{real bordered 3-manifold} consists of a compact, connected, oriented $3$-manifold $Y$ with boundary; a diffeomorphism 
  \[
  \phi_L\amalg \phi_R \co F(\PMC)\amalg F(\PMC)\to \bdy Y,
  \]
  for some pointed matched circle $\PMC$; a framed arc $\gamma$ in $Y$ connecting the two boundary components; and
  an orientation-preserving, smooth involution $\tau\co Y\to Y$. These data are required to satisfy:
  \begin{itemize}
    \item The involution $\tau$ commutes with the parametrization of the boundary, in the sense that
    \begin{equation}\label{eq:real-bord-bdy-compat-V1}
      \tau\circ \phi_L=\phi_R.
    \end{equation}
    \item The fixed set of $\tau$ is nonempty (hence also $1$-dimensional and closed).
    \item The involution $\tau$ preserves $\gamma$, $\tau(\gamma)=\gamma$, and if we view the framing as a normal vector field $\nu$ along $\gamma$, then $d\tau\circ\nu=-\nu$.
  \end{itemize}
\end{definition}

It is equivalent, but more natural later, to think of $\phi_R$ as a map $F(-\PMC^\beta)\to \bdy Y$, from the surface associated to $\PMC^\beta$ with its orientation reversed. Then Formula~\eqref{eq:real-bord-bdy-compat-V1} becomes
\begin{equation}\label{eq:real-bord-bdy-compat-V2}
  \tau\circ \phi_L = \phi_R\circ \ol{K}_{\alpha,\beta}.
\end{equation}

We will often suppress $\phi_L\amalg\phi_R$ and $\gamma$, and write a real bordered 3-manifold just as $(Y,\tau)$.

\begin{figure}
  \centering
  \includegraphics[alt={Real surfaces}, scale=.75]{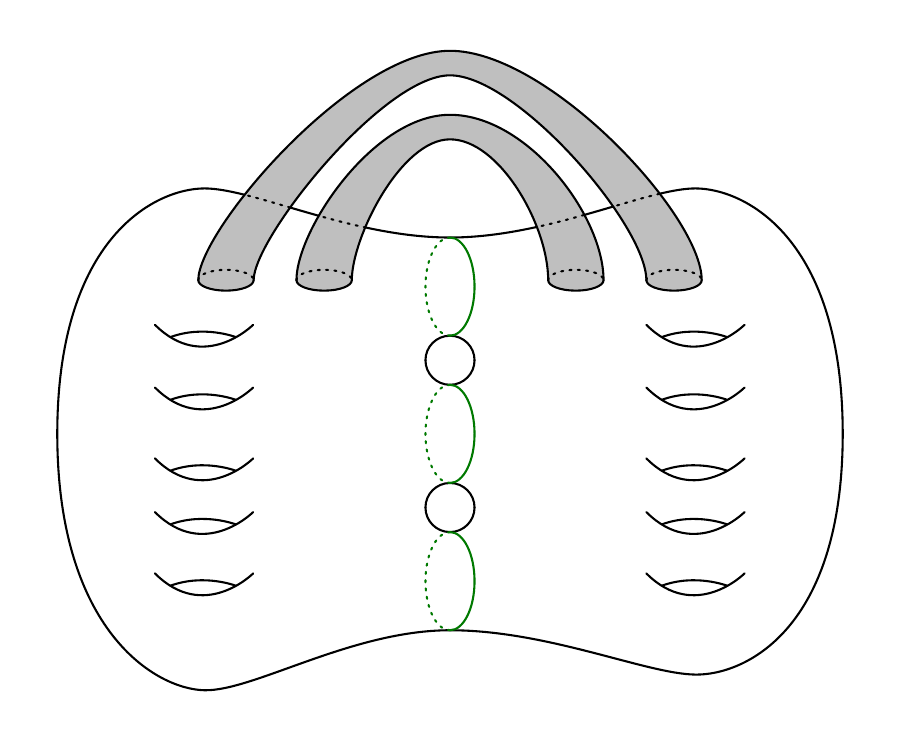}
  \caption{\textbf{Real surfaces}. The involution on the surface is given by reflection across the \textcolor{darkgreen}{vertical circles in the middle}, except that the involution on the thin, shaded cylinders is the free involution with quotient a M\"obius band. This surface has $|C|=3$, genus $14$, and nonorientable quotient; the quotient is orientable if and only if there are no shaded cylinders.}
  \label{fig:real-surfaces}
\end{figure}

\begin{example}\label{eg:real-surf-times-I}
  For us, a \emph{real surface} is a closed, connected, oriented surface $\Sigma$ together with an orientation-reversing involution $\tau\co \Sigma\to\Sigma$ with nonempty fixed set. Real surfaces $(\Sigma,\tau)$ and $(\Sigma',\tau')$ are diffeomorphic if there is a diffeomorphism $\psi\co \Sigma\to \Sigma'$ so that $\tau'\circ \psi=\psi\circ\tau$. Up to diffeomorphism, a real surface is determined by:
  \begin{itemize}
    \item The genus $g$ of $\Sigma$,
    \item The number of components $|C|$ of the fixed set $C$ of $\tau$, and
    \item Whether $\Sigma/\tau$ is orientable or nonorientable.
  \end{itemize}
  Further, the only conditions these data must satisfy are that $1\leq |C|\leq g+1$; if $\Sigma/\tau$ is nonorientable then $|C|\leq g$; and if $\Sigma/\tau$ is orientable, then $|C|\equiv g+1 \pmod{2}$.
  (See, for example,~\cite{Liu20:moduli} and the references there.) Figure~\ref{fig:real-surfaces} indicates how to realize any such tuple of data.

  Because we will use it later, note that there is a basis for the homology of $\Sigma$ so that the action of $\tau$ is block-diagonal, with $|C|-1$ blocks of the form
  \[
  \begin{bmatrix}
    -1 & 0\\
    0 & 1
  \end{bmatrix}
  \]
  and the remaining blocks of the forms
  \[
  \begin{bmatrix}
    0 & 0 & 1 & 0\\
    0 & 0 & 0 & 1\\
    1 & 0 & 0 & 0\\
    0 & 1 & 0 & 0
  \end{bmatrix}
  \qquad\qquad\text{or}\qquad\qquad
  \begin{bmatrix}
    0 & 1\\
    1 & 0
  \end{bmatrix}.
  \]
  (One can trade the former for two copies of the latter, but the former seem a little more natural for orientable quotients.)

  Fix a real surface $(\Sigma,\tau)$, a point $z$ on the fixed set of $\tau$, a vector $v\in T_z\Sigma$, and an orientation-preserving diffeomorphism $\phi_L\co -F(\PMC)\to \Sigma$ for some pointed matched circle $\PMC$. Then $\bigl([0,1]\times\Sigma,\phi_L\amalg\tau\circ\phi_L,(t,p)\mapsto(1-t,\tau(p)),[0,1]\times\{z\},[0,1]\times v\bigr)$, which we call a \emph{real thick surface}, is a prototypical real bordered $3$-manifold. We call $(\Sigma,\tau)$ the \emph{central real surface} of the real thick surface.

  The orientations are, perhaps, a little confusing. The boundary orientation identifies $\{0\}\times\Sigma=-\Sigma$, so $\phi_L$ is an orientation-preserving map $F(\PMC)\to \{0\}\times \Sigma$. On the other hand, $\{1\}\times\Sigma=\Sigma$, $\phi_L$ is an orientation-reversing map from $F(\PMC)$ to $\Sigma$, and $\tau$ is an orientation-reversing map from $\Sigma$ to $\Sigma$, so $\tau\circ\phi_L$ is an orientation-preserving map from $F(\PMC)$ to $\{1\}\times\Sigma$.
\end{example}

\begin{definition}
  A \emph{real bordered Heegaard diagram} is an $\alpha$-$\beta$-bordered Heegaard diagram $\HD=(\Sigma,\alphas,\betas,\bbpt)$ (in the sense of~\cite[Section 3.3]{LOTHomPair}) together with an orientation-reversing involution $\tau\co\Sigma\to\Sigma$ such that
  \begin{itemize}
    \item $\tau$ exchanges the two boundary components of $\Sigma$,
    \item $\tau(\alphas)=\betas$, and
    \item $\tau(\bbpt)=\bbpt$.
  \end{itemize}
\end{definition}
In particular, the boundary of a real bordered Heegaard diagram is an $\alpha$-pointed matched circle $\PMC=\bdy_L\HD$ and the $\beta$-pointed matched circle $\bdy_R\HD=-\PMC^\beta$ obtained by reversing the orientation of $Z$ and viewing the points as $\beta$-endpoints rather than $\alpha$-endpoints.
Some examples of real bordered Heegaard diagrams are shown in Figures~\ref{fig:T2-diag-two-fixed-circles},~\ref{fig:AZ} and~\ref{fig:AZ-bar}.

An $\alpha$-$\beta$-bordered Heegaard diagram $\HD$ specifies an arced bordered 3-manifold $Y(\HD)$ by capping off the boundary components of $\Sigma$ to obtain $\ol{\Sigma}$, thickening the result to $\ol{\Sigma}\times[1,2]$ and attaching 2-handles along the $\alpha$-circles in $\Sigma\times\{1\}$ and the $\beta$-circles in $\Sigma\times\{2\}$. The $\alpha$-arcs (respectively $\beta$-arcs) induce a parametrization of the part of the boundary corresponding to $\Sigma\times\{1\}$ (respectively $\Sigma\times\{2\}$).
(Compare~\cite[Construction 3.17]{LOTHomPair}.)
If we start with a real bordered Heegaard diagram, then the corresponding 3-manifold inherits an involution $\tau$ induced from the involution $\tau(p,t)=(\tau(p),3-t)$ on $\Sigma\times[1,2]$. The fixed sets of $\tau\co \Sigma\to \Sigma$ and $\tau\co Y(\HD)\to Y(\HD)$ are identified.

\begin{figure}
  \centering
  \includegraphics[alt={A real bordered Heegaard diagram}]{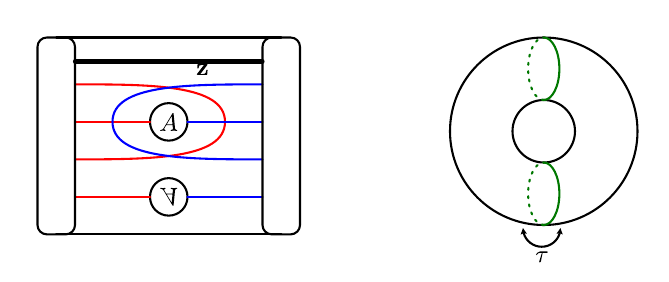}
  \caption{\textbf{A real bordered Heegaard diagram.} This diagram (on the left) represents $[0,1]\times T^2$ with the involution induced by the reflection of $T^2$ with two fixed circles (as indicated on the right).}
  \label{fig:T2-diag-two-fixed-circles}
\end{figure}

\begin{example}\label{eg:T2-two-fixed-circles-HD}
  Figure~\ref{fig:T2-diag-two-fixed-circles} shows a real bordered Heegaard diagram for $([0,1]\times T^2,\tau)$ where $\tau$ is induced from the involution $\tau\co T^2\to T^2$ as in Example~\ref{eg:real-surf-times-I} given by reflection across two circles (with orientable quotient).

  Figures~\ref{fig:AZ} and~\ref{fig:AZ-bar} both show bordered Heegaard diagrams for $[0,1]\times T^2$ with connected fixed set and nonorientable quotient, as well as bordered Heegaard diagrams for $[0,1]\times\Sigma_2$ with connected fixed set and orientable quotient. Changing the pointed matched circle from the split one to the antipodal one gives Heegaard diagrams for $[0,1]\times\Sigma_2$ with connected fixed set and nonorientable quotient. See Section~\ref{sec:AZ-diagrams} for more discussion of these diagrams.
\end{example}

We do not glue pairs of real bordered 3-manifolds together. Rather, given a real bordered 3-manifold $Y=(Y,\phi_L\amalg\phi_R,\gamma,\tau)$, where $\phi_L\co F(\PMC)\to \bdy Y$, and an (ordinary) arced bordered 3-manifold $Y'=(Y',\phi'_L,\phi'_R,\gamma')$ with $\phi'_R\co -F(\PMC)\to \bdy Y'$, we can form a new real bordered 3-manifold
\[
\bigl(Y'\cup_{F(\PMC)}Y\cup_{F(\PMC)}Y',\phi'_L\amalg \phi'_L,\gamma'\cup\gamma\cup\gamma',\wt{\tau}\bigr),
\]
where $\wt{\tau}|_{Y}=\tau$ and $\wt{\tau}|_{Y'}$ is the identity map exchanging the two copies of $Y'$. (Note that, on both sides, we are gluing $\bdy_R Y'$ to $Y$.) That is, we think of $Y$ as a cobordism $-F(\PMC)\amalg -F(\PMC) \to\emptyset$, and compose with $Y'\amalg Y'$.

Similarly, given a real bordered Heegaard diagram $(\HD,\tau)=(\Sigma,\alphas,\betas,\bbpt,\tau)$ and an (ordinary) arced bordered Heegaard diagram $\HD'=(\Sigma',\alphas',\betas',\bbpt')$ with $\bdy_L\HD=\PMC=-\bdy_R\HD'$, we can form a new real bordered Heegaard diagram
\[
(\HD'\cup_{\PMC}\HD\cup_{\PMC}(-\HD')^\beta,\wt{\tau})=\bigl(\Sigma'\cup\Sigma\cup(-\Sigma'), \alphas'\cup\alphas\cup\betas', \betas'\cup\betas\cup\alphas',\bbpt'\cup\bbpt\cup\bbpt',\wt{\tau} \bigr),
\] 
where $(-\HD')^\beta$ is the $\beta$-$\beta$-bordered Heegaard diagram obtained by reversing the orientation of $\Sigma$ and relabeling the $\alpha$-curves in $\HD$ as $\beta$-curves and the $\beta$-circles as $\alpha$-circles. The map $\wt{\tau}$ is given by $\tau$ on $\Sigma$ and the (orientation-reversing) identity map $\Sigma'\to -\Sigma'$.

\begin{lemma}
  If $(\HD,\tau)$ is a real bordered Heegaard diagram representing a real bordered 3-manifold $(Y(\HD),\tau)$ with boundary $F(\PMC)\amalg F(\PMC)$ and $\HD'$ is an arced bordered Heegaard diagram representing an arced bordered 3-manifold $Y(\HD')$, then $(\HD'\cup_{\PMC}\HD\cup_{\PMC}(-\HD')^\beta,\wt{\tau})$ represents $\bigl(Y'\cup_{F(\PMC)}Y\cup_{F(\PMC)}Y',\wt{\tau}\bigr)$.
\end{lemma}
\begin{proof}
  This is immediate from the definitions.
\end{proof}

\subsection{States, domains, and \texorpdfstring{$\SpinC$-}{spin-c }decomposition}\label{sec:states-and-domains}

Recall that a \emph{generator} or \emph{state} in a bordered Heegaard diagram $\HD=(\Sigma,\alphas,\betas,z)$ (with one boundary component), where $\alphas=\{\alpha_1^c,\dots,\alpha_{g-k}^c,\alpha_1^a,\dots,\alpha_{2k}^a\}$ and $\betas=\{\beta_1,\dots,\beta_g\}$, is a $g$-tuple $\x=\{x_i\}\subset\alphas\cap\betas$ so that one point of $\x$ lies on each $\beta$-circle, one point on each $\alpha$-circle, and at most one point on each $\alpha$-arc. In particular, exactly half of the $\alpha$-arcs are occupied by points in $\x$. (In terms of $\SpinC$-structures, this is the condition that $\langle c_1(\spinc),[\bdy Y]\rangle=0$.) We denote the set of states for $\HD$ by $\Gen(\HD)$.

For $\alpha$-$\beta$-bordered Heegaard diagrams, one can consider states where any number of $\alpha$-arcs are occupied, but most turn out to be irrelevant for Heegaard Floer homology of closed 3-manifolds. In a real $\alpha$-$\beta$ bordered Heegaard diagram $(\HD,\tau)$, $\alphas=\{\alpha_1^c,\dots,\alpha_{g-k}^c,\alpha_1^a,\dots,\alpha_{2k}^a\}$ and $\betas=\{\beta_1^c,\dots,\beta_{g-k}^c,\beta_1^a,\dots,\beta_{2k}^a\}$ contain the same number of arcs. If we glue an $\alpha$-bordered Heegaard diagram $\HD'$ (with connected boundary) to the $\alpha$-boundary of $\HD$, then in a state for $\HD'$, half of the $\alpha$-arcs are occupied. Hence, only states for $\HD$ in which $k$ $\alpha$-arcs (that is, half of them) and $k$ $\beta$-arcs are occupied contribute to the (real) Heegaard Floer homology of the closed diagram.

So, by a \emph{real generator} or \emph{real state} for $(\HD,\tau)$ we mean a $g$-tuple of points $\x=\{x_i\}\subset\alphas\cap\betas$ so that:
\begin{itemize}
  \item Every $\alpha$-circle and every $\beta$-circle contain a point in $\x$.
  \item Exactly half of the $\alpha$-arcs and half of the $\beta$-arcs contain a point in $\x$ (and each curve contains at most one point in $\x$).
  \item The point $\x$ is fixed by $\tau$, $\tau(\x)=\x$.
\end{itemize} 
Let $\Gen_R(\HD,\tau)$ denote the set of real states for $(\HD,\tau)$.

Given $\x,\y\in\Gen(\HD)$ in a real bordered Heegaard diagram $(\HD,\tau)$, a \emph{domain} from $\x$ to $\y$ is a linear combination $B$ of components of $\Sigma\setminus(\alphas\cup\betas)$ so that 
\[
\bdy [(\bdy B)\cap (\betas\cup\bdy_R\Sigma)]
=-\bdy [(\bdy B)\cap (\alphas\cup\bdy_L\Sigma)]=\y-\x
\]
and the region containing the arc $\bbpt$ has coefficient $0$ in $B$. As usual, we denote the set of domains from $\x$ to $\y$ by $\pi_2(\x,\y)$.
If $B\in\pi_2(\x,\y)$, then $\tau_*(B)\in\pi_2(\tau(\y),\tau(\x))$. (Note that, since $\tau$ is orientation-reversing on $\Sigma$, the coefficients of $\tau_*(B)$ are obtained by negating the coefficients of $B$ and applying $\tau$ to them.) See Figure~\ref{fig:genus-1-domains}.

Given real states $\x,\y\in\Gen_R(\HD,\tau)$, a \emph{real domain} from $\x$ to $\y$ is a domain $B\in\pi_2(\x,\y)$ so that $B=-\tau_*(B)$. Let $\pi_2^R(\x,\y)$ denote the set of real domains from $\x$ to $\y$.

A domain (or real domain) $B$ is \emph{provincial} if the coefficients of the regions adjacent to $\bdy\Sigma$ are all $0$.

A \emph{real periodic domain} is an element of $\pi_2^R(\x,\x)$; the real periodic domains form a subgroup of $\pi_2(\x,\x)\cong H_2(Y,\bdy Y)$. Similarly, the provincial real periodic domains $\pi_2^{\bdy,R}$ form a subgroup of the provincial periodic domains $\pi_2^\bdy(\x,\x)\cong H_2(Y)$. As in the closed case~\cite[Lemma 3.19]{GM:real-HF}, these groups have simple topological interpretations:
\begin{lemma}
  Fix a real bordered Heegaard diagram $(\HD,\tau)$ for a real bordered 3-manifold $(Y,\phi_L\amalg\phi_R,\gamma,\tau)$. There are canonical isomorphisms
  \begin{align*}
    \pi_2^R(\x,\x) &\cong \ker\bigl((1+\tau_*)\co H_2(Y,\bdy Y)\to H_2(Y,\bdy Y)\bigr)\\
    \pi_2^{\bdy,R}(\x,\x) &\cong \ker\bigl((1+\tau_*)\co H_2(Y\setminus\gamma)\to H_2(Y\setminus\gamma)\bigr)
  \end{align*}
\end{lemma}
\begin{proof}
  Observe that the identification between periodic domains and $H_2(Y,\bdy Y)$ intertwines the action of $\tau_*$ on $\pi_2(\x,\x)$ (from the involution of the Heegaard surface) and the action of $\tau_*$ on $H_2(Y,\bdy Y)$ (from the involution of the 3-manifold). By definition, $\pi_2^R(\x,\x)$ is the kernel of $1+\tau_*$ on $\pi_2(\x,\x)$, and hence the corresponding statement also holds for $H_2(Y,\bdy Y)$. The proof for $\pi_2^{\bdy,R}$ is similar.
\end{proof}

Following Guth-Manolescu, we will sometimes write
\[
H_2(Y,\bdy Y)^{-\tau_*} = \ker\bigl((1+\tau_*)\co H_2(Y,\bdy Y)\to H_2(Y,\bdy Y)\bigr)
\]
since this is the fixed set for the action of $-\tau_*$; and similarly for other homology groups.

Recall that a bordered Heegaard diagram is \emph{admissible} (respectively \emph{provincially admissible}) if every periodic domain (respectively provincial periodic domain) has both positive and negative coefficients. Following Guth-Manolescu~\cite[Section 3.8]{GM:real-HF}:
\begin{definition}
  A real bordered Heegaard diagram $(\HD,\tau)$ is \emph{admissible} (respectively \emph{provincially admissible}) if $\HD$ is admissible (respectively provincially admissible) as an ordinary bordered Heegaard diagram.
\end{definition}

\begin{figure}
  \centering
  \includegraphics[alt={Domains and the obstruction zeta}]{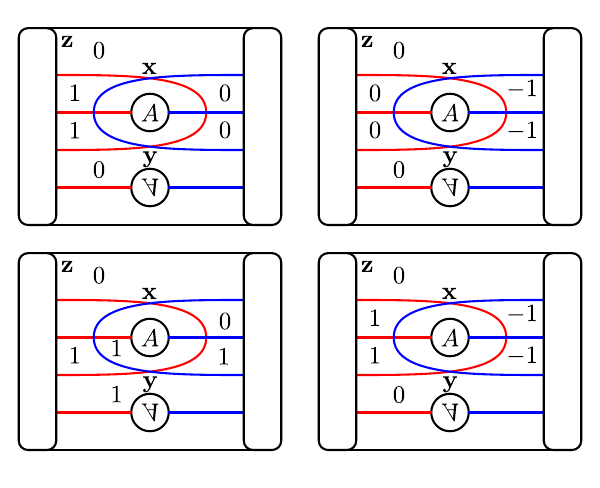}
  \caption{\textbf{Domains and the obstruction $\zeta$.} Top left: a domain $B\in \pi_2(\x,\y)$. Top right: the domain $\tau_*(B)$. Bottom left: a state for the space of real periodic domains in this diagram; as an element of $\pi_2(\y,\y)$, in fact, this domain has a holomorphic representative. Bottom right: the obstruction $\zeta(\x,\y)$ is the periodic domain $B+\tau_*(B)$. Note that this is not a real periodic domain: it is fixed by $\tau_*$, not $-\tau_*$.}
  \label{fig:genus-1-domains}
\end{figure}

There is a familiar obstruction 
\[
\epsilon(\x,\y)\in H_1(Y(\HD))/H_1(\bdy Y(\HD))=H_1\bigl(\Sigma\setminus\nbd(\bbpt),\bdy(\Sigma\setminus\nbd(\bbpt))\bigr)/[H_1(\alphas)+H_1(\betas)]
\] 
to the existence of a domain from $\x$ to $\y$, obtained by connecting $\x$ to $\y$ by a 1-chain $b$ in $\betas\cup\bdy_R\Sigma$ and a 1-chain $a$ in $\alphas\cup\bdy_L\Sigma$ and considering the difference $b-a$. As Guth-Manolescu observe~\cite[Section 3.4]{GM:real-HF}, in the real case there are two refinements of this construction:
\begin{enumerate}
\item We may take $a=\tau_*(b)$, so $\epsilon(\x,\y)$ lies in the kernel of $(1+\tau_*)\co H_1(Y)/H_1(\bdy Y)\to H_1(Y)/H_1(\bdy Y)$. (In fact, they argue that, in the closed case, this lies in a further subgroup.)
\item If $\epsilon(\x,\y)=0$, then there is a domain $B\in\pi_2(\x,\y)$. The element $B+\tau_*(B)$ is a periodic domain ($\tau$-invariant, not real), so we obtain an element
\begin{align*}
\zeta(\x,\y)&=B+\tau_*(B)\\
&\in \ker\bigl((1-\tau_*)\co H_2(Y,\bdy Y)\to H_2(Y,\bdy Y)\bigr)/\Image\bigl((1+\tau_*)\co H_2(Y,\bdy Y)\to H_2(Y,\bdy Y)\bigr).
\end{align*}
See Figure~\ref{fig:genus-1-domains}.
\end{enumerate}
Generalizing the closed case~\cite[Section 3.4]{GM:real-HF}, we have:
\begin{lemma}
  The element $\zeta(\x,\y)$ is independent of the choice of $B$.
  Moreover, $\pi_2^R(\x,\y)\neq\emptyset$ if and only if
  $\epsilon(\x,\y)=0$ and $\zeta(\x,\y)=0$.
\end{lemma}
\begin{proof}
  For the first statement, any other choice of domain can be written as $B+P$ where $P$ is a periodic domain. Then
  \[
  (B+P)+\tau_*(B+P)=[B+\tau_*(B)]+[P+\tau_*(P)],
  \]
  and hence differs from $B+\tau_*(B)$ by an element of the image of $1+\tau_*$.

  If $\pi_2^R(\x,\y)\neq\emptyset$, then $\pi_2(\x,\y)\neq\emptyset$, so $\epsilon(\x,\y)=0$ and, by using an element of $\pi_2^R(\x,\y)$ to define $\zeta(\x,\y)$, $\zeta(\x,\y)=0$. Conversely, suppose $\epsilon(\x,\y)=0$ and $\zeta(\x,\y)=0$. Write $B+\tau_*(B)=(1+\tau_*)(P)$ for some periodic domain $P$. Then $(1+\tau_*)(B-P)=0$, so $B-P\in\pi_2^R(\x,\y)$, as desired.
\end{proof}

Call states $\x$ and $\y$ \emph{real $\SpinC$-equivalent} if $\pi_2^R(\x,\y)\neq\emptyset$.

\begin{example}\label{eg:Sigma-times-I-zeta}
  Fix a real surface $(\Sigma,\tau)$ and an orientation-preserving diffeomorphism $\phi_L\co -F(\PMC)\to\Sigma$, and consider the real bordered manifold $[0,1]\times\Sigma$ from Example~\ref{eg:real-surf-times-I}. Fix a real bordered Heegaard diagram $(\HD,\tau)$ representing $([0,1]\times\Sigma,\tau)$. We have $H_1([0,1]\times\Sigma)/H_1(\{0,1\}\times\Sigma)=0$, so the obstruction class $\epsilon(\x,\y)$ vanishes identically and every pair of states can be connected by a domain. (Recall that we are only considering the central $\SpinC$-structure---the one with $c_1(\spinc)=0$.) The long exact sequence for the pair gives
  \[
  0\to H_2(\Sigma\times[0,1],\Sigma\times\{0,1\})\to H_1(\Sigma\times\{0,1\})\to H_1(\Sigma\times[0,1]),
  \]
  i.e.,
  \[
  \begin{tikzcd}[ampersand replacement=\&, alt={Exact sequence in homology}]
    0\arrow[r] \& H_2(\Sigma\times[0,1],\Sigma\times\{0,1\})\arrow[r] \& H_1(\Sigma)\oplus H_1(\Sigma)\arrow[r,"{\begin{bsmallmatrix}1 & 1\end{bsmallmatrix}}"] \& H_1(\Sigma).
  \end{tikzcd}
  \]
  Thus, $H_2(\Sigma\times[0,1],\Sigma\times\{0,1\})=\{(x,-x)\mid x\in H_1(\Sigma)\}$, with the action given by $\tau_*(x,-x)=(-\tau_*(x),\tau_*(x))$. So, the space of real periodic domains is
  \[
  \ker\bigl((1+\tau_*)\co H_2(Y,\bdy Y)\to H_2(Y,\bdy Y)\bigr)\cong\ker\bigl((1-\tau_*)\co H_1(\Sigma)\to H_1(\Sigma)\bigr).
  \]
  From the form of $\tau_*$ from Example~\ref{eg:real-surf-times-I}, this is $\ZZ^{g}$. That is, $\pi_2^R(\x,\x)\cong \ZZ^g$.

  Finally, the obstruction $\zeta$ lies in $\ker(1-\tau_*)/\Image(1+\tau_*)$. Again, from Example~\ref{eg:real-surf-times-I}, this is $(\ZZ/2\ZZ)^{|C|-1}$. For example, for the diagrams $\AZ(\PMC)$ and $\AZbar(\PMC)$ discussed in Section~\ref{sec:AZ-diagrams}, $\zeta$ lies in a trivial group, hence vanishes identically. For the diagram in Figure~\ref{fig:genus-1-domains}, $\zeta(\x,\y)\in\ZZ/2\ZZ$ (and, for the states marked in that diagram, is the nontrivial element).
\end{example}

\section{Moduli spaces}\label{sec:moduli}

\subsection{Definitions of the moduli spaces}\label{sec:def-of-moduli}
We briefly recapitulate the definitions of the bordered moduli spaces, but now imposing the symmetry relevant to the real case. The material in this section is a straightforward synthesis of the original bordered moduli spaces~\cite{LOT1} and Guth-Manolescu's real Floer moduli spaces, though we give somewhat complementary details to the ones given in their paper~\cite{GM:real-HF}.

Fix a real bordered Heegaard diagram $(\HD,\tau)=(\Sigma,\alphas,\betas,\bbpt,\tau)$. There is an induced map 
$$\ul{\tau}\co \Sigma\times[0,1]\times\RR\to\Sigma\times[0,1]\times\RR$$ 
given by
\[
\ul{\tau}(p,s,t)=(\tau(p),1-s,t).
\]
We consider almost complex structures $J$ on $\Sigma\times[0,1]\times\RR$ satisfying the usual conditions from bordered Floer theory~\cite[Definition 5.1]{LOT1} and the additional condition that $\ul{\tau}$ is $J$-anti-holomorphic, i.e.,
\[
J\circ d\ul{\tau}=-d\ul{\tau}\circ J.
\]
We call such almost complex structures $J$ \emph{symmetric almost complex structures}.

Write $\bdy_L\HD=\PMC$, so $\bdy_R\HD=-\PMC^\beta$.
Fix a symmetric almost complex structure $J$. We consider two kinds of moduli spaces of $J$-holomorphic maps:
\begin{itemize}
\item Moduli spaces $\cM^B_R(\x,\y;\Source;P)$, where $\Source$ is a given decorated Riemann surface (see~\cite[Definition 5.2]{LOT1}) and $P$ is an ordered partition of the $e$ punctures of $\Source$. This is the subspace of the bordered moduli space $\cM^B(\x,\y;\Source;P)$~\cite[Definition 5.5]{LOT1} given by holomorphic maps $u\co S\to \Sigma\times[0,1]\times\RR$ such that there exists an anti-holomorphic involution $\eta\co S\to S$ with $\ul{\tau}\circ u=u\circ\eta$. (Since $u$ is nonconstant on every component of $S$~\cite[Condition (M-5)]{LOT1}, if $\eta$ exists, then it is unique.) 

For real curves, punctures necessarily come in pairs, of one puncture mapped to $\bdy_L\HD$ and one mapped to $\bdy_R\HD$. If we are imposing no (other) height constraints, we could view the real moduli space as a subspace of $\cM^B(\x,\y;\Source;P)$ where $P$ is discrete, or where each part in $P$ has a pair of punctures exchanged by $\eta$. We will take the former view, and in general view $P$ as a pair of partitions, one of the punctures mapped to $\bdy_L\HD$ and the other of punctures mapped to $\bdy_R\HD$. (This is similar to the construction of bordered bimodules~\cite{LOT2}, where we do not impose constraints between punctures on different boundary components.)
In particular, for the construction of $\CFDRa$, we are only interested in the case that $P$ is the discrete partition. (The proof of the structure equation---specifically, Case~\ref{item:collision} of Proposition~\ref{prop:codim-1}---involves partitions $P$ where one part has two elements. If we were to construct $\CFARa$, more general partitions would be needed.)
\item Moduli spaces $\cM^B_R(\x,\y;\vec{\rho})$ of embedded holomorphic curves. These are subspaces of the moduli spaces $\cM^B(\x,\y;\vec{\rhos})$~\cite[Definition 5.68]{LOT1}. Here, $\vec{\rho}=(\rho_1,\dots,\rho_n)$ is a sequence of chords in $\PMC$ and $\vec{\rhos}=(\{\rho_1,\tau(\rho_1)\},\dots,\{\rho_n,\tau(\rho_n)\})$. The subspace $\cM^B_R(\x,\y;\vec{\rho})$ consists of curves $u$ so that the image of $u$ is preserved by $\ul{\tau}$ (or, equivalently, there is an anti-holomorphic automorphism of the source that $u$ intertwines with $\ul{\tau}$). More consistent with viewing $P$ as discrete in the previous point, we can view $\cM^B_R(\x,\y;\vec{\rho})$ as a subspace of $\cM^B(\x,\y;\vec{\rho};\vec{\rho})$ from the construction of bimodules~\cite[Section 6.3]{LOT2}. (We will not, in fact, make use of either embedding below.)
\end{itemize}

Given a decorated surface $\Source$, we can also consider maps $\phi\co S\to \Sigma\times[0,1]\times\RR$ respecting the decorations, but which are merely smooth, not $J$-holomorphic. By a \emph{real smooth map} from $\Source$ we mean such a smooth map $\phi$ together with an orientation-reversing, smooth involution $\eta\co S\to S$ so that $\phi\circ\eta=\ul{\tau}\circ\phi$. Call a symmetric almost complex structure $J$ \emph{generic} if the moduli spaces $\cM^B_R(\x,\y;\Source;P)$ are transversely cut out inside the space of all real smooth maps respecting the decorations and the partition $P$, for all choices of decorated Riemann surface $\Source$ and all ordered partitions $P$.

Of course, we can also talk about real maps of other regularity, and in particular real versions of the weighted Sobolev spaces $W^{1,p}_\delta$ used in~\cite[Section 5.2]{LOT1}. Fix a smooth surface $S$ (with boundary and punctures) and a smooth involution $\eta\co S\to S$. We have a space $\mathcal{B}^R$, with objects triples $(j,J,u)$ of a symmetric almost complex structure $j$ on $(S,\eta)$, a symmetric almost complex structure $J$ on $(\Sigma\times[0,1]\times\RR,\tau)$ (as above), and a real $W^{1,p}_\delta$-map $u\co S\to \Sigma\times[0,1]\times\RR$. Dropping the equivariance conditions gives a map $\mathcal{B}^R\to \mathcal{B}$. Over $\mathcal{B}^R$, we have a bundle $\mathcal{E}^R$ of $(0,1)$-forms $\alpha$ on $S$ valued in $u^*T(\Sigma\times[0,1]\times\RR)$, satisfying 
\begin{equation}\label{eq:1-form-invariant}
  \alpha=d\tau\circ \alpha\circ d\eta.
\end{equation}
There is a projection from $\mathcal{B}^R$ to the space of (real) almost complex structures on $\Sigma\times[0,1]\times\RR$; let $\mathcal{B}_J^R$ denote the fiber over $J$ (i.e., the space of tuples $(S,\eta,j,u)$). Similarly, if we hold $j$ and $J$ fixed, we have a subspace we denote $\mathcal{B}_{J,j}^R$. The $\dbar$-operator is a section $\mathcal{B}\to \mathcal{E}$, and its linearization is a bundle map $T\mathcal{B}^R\to \mathcal{E}^R$. More precisely, this linearization is defined canonically when $\dbar(u)=0$. At other points in $\mathcal{B}^R$, it depends on a choice of projection $T\mathcal{E}^R\to \mathcal{E}^R$. We can use $J$ and the standard symplectic form to induce a metric and then such a projection, which gives the formula
\begin{equation}\label{eq:MS-linearization}
  (D_u\dbar)(\xi) = \OneHalf\left(\nabla\xi + J\nabla\xi\circ j\right)-\OneHalf J(\nabla_\xi J)\bdy_J(u)
\end{equation}
\cite[Proposition 3.1.1]{MS04:HolomorphicCurvesSymplecticTopology}, where $\nabla$ is the Levi-Civita connection, and $$\bdy_J(u)=\OneHalf\left(du-J\circ du\circ j\right).$$

The space $\mathcal{E}^R$ is a subspace of the space $\mathcal{E}|_{\mathcal{B^R}}$ where we drop Condition~\ref{eq:1-form-invariant}, and $\mathcal{E}|_{\mathcal{B^R}}$ inherits an involution $\alpha\mapsto d\tau\circ\alpha\circ d\eta$. Similarly, when $u$ is a real curve, the space $T_u\mathcal{B}^R_{J,j}$ inherits an involution $\xi\mapsto d\tau\circ \xi\circ d\eta$. Denote both of these involutions by $\ul{\tau}$.

\begin{lemma}
  For any real map $u$, the operator $D_u\dbar\co T_u\mathcal{B}\to \mathcal{E}$ from Equation~\eqref{eq:MS-linearization} satisfies $D_u\dbar=\ul{\tau}\circ D_u\dbar\circ\ul{\tau}$. In particular, $D_u\dbar$ sends $T\mathcal{B}^R_{J,j}$ to $\mathcal{E}^R$. Moreover, for any real $u$ and symmetric $J$, this induced map $D_u\dbar\co T_u\mathcal{B}^R_{J,j}\to \mathcal{E}^R_{J,j,u}$ is Fredholm.
\end{lemma}
\begin{proof}
  We can check the first statement directly, or observe that this linearization depends only on the Riemannian metrics induced by the symplectic forms and almost complex structures on $S$ and $\Sigma\times[0,1]\times\RR$. Since
  \[
  \omega(v,\ol{J}w)=\omega(J^2v,J^2\ol{J}w)=\omega(-v,Jw)=-\omega(v,Jw),
  \]
  $(\omega,J)$ and $(-\omega,\ol{J})$ induce the same Riemannian metrics, hence the linearization is necessarily $\ul{\tau}$-equivariant.

  The fact that $D_u\dbar$ sends $T_u\mathcal{B}^R_{J,j}$ to $\mathcal{E}^R_{J,j,u}$ follows: $T_u\mathcal{B}^R_{J,j}$ and $\mathcal{E}^R_{J,j,u}$ are the $1$-eigenspaces of $\ul{\tau}$ on $T_u\mathcal{B}_{J,j}$ and $\mathcal{E}_{J,j,u}$, respectively, and since $D_u\dbar$ intertwines the operators $\ul{\tau}$, it respects the $\ul{\tau}$ eigenspaces.

  Finally, to see this induced map is Fredholm, the kernel of $D_u\dbar\co T_u\mathcal{B}^R_{J,j}\to \mathcal{E}^R_{J,j,u}$ is a subspace of the kernel of $D_u\dbar \co T_u\mathcal{B}_{J,j}\to \mathcal{E}_{J,j,u}$, and since the latter map is Fredholm, the kernel is finite-dimensional. For the cokernel, since $T_u\mathcal{B}_{J,j}$ is the sum of the $-1$ and $1$-eigenspaces of $\ul{\tau}$, $D_u\dbar$ preserves these eigenspaces, and the cokernel of $D_u\dbar$ is finite-dimensional, the induced map between $1$-eigenspaces must also have finite-dimensional cokernel.
\end{proof}

\begin{proposition}\label{prop:transversality}
  For any real bordered Heegaard diagram $\HD$, there exist generic symmetric almost complex structures.
\end{proposition}
\begin{proof}
  This is a combination of the arguments in the bordered case~\cite[Proposition 5.6]{LOT1} and the closed real case~\cite[Proposition 2.2]{GM:real-HF}. As usual, the main step is to prove that for any $(j,J,u)\in \mathcal{B}^R$ with $\dbar_{J,j}(u)=0$, the linearized map $D_u\dbar\co T_{J,j,u}\mathcal{B}^R\to \mathcal{E}^R_{J,j,u}$ is surjective. The proof proceeds one connected component of $S$ at a time. For trivial strips, the statement is clear. So, let $S_0$ be a connected component which is not a trivial strip.  
  Choose a point $(p,s)\in\Sigma\times[0,1]$ so that $u^{-1}((p,s)\times\RR)$ consists of a single point on $S_0$, and that point is a regular point for $\pi_\Sigma\circ u$ and $\pi_D\circ u$ (cf.~\cite[Lemma 3.3]{Lipshitz06:CylindricalHF}). As in the usual proof of transversality for cylindrical Heegaard Floer homology~\cite[Proposition 3.7]{Lipshitz06:CylindricalHF}, suppose some nontrivial $\zeta$ in $(\mathcal{E}^R_{J,j,u})^*$ vanishes on the image of $D_u\dbar$. (In~\cite[Proposition 3.7]{Lipshitz06:CylindricalHF}, $\zeta$ is denoted $\eta$, but we are using $\eta$ above for the involution on $S$.) Since we can achieve transversality for the $\dbar$-operator via a nonequivariant perturbation, there is either:
  \begin{itemize}
  \item A vector field $\xi$ along $u$ so that $\int_S \langle \zeta, D_u\xi\rangle \neq 0$ or
  \item A perturbation $Y_s$ of $J$ so that $\int_S \langle \zeta, Y_s\circ du\circ j\rangle\neq 0$ or
  \item A perturbation $Y$ of $j$ so that $\int_S \langle \zeta, J\circ du\circ Y\rangle\neq 0$.
  \end{itemize}
  (As in~\cite[Proof of Proposition 3.7]{Lipshitz06:CylindricalHF}, we are thinking of $\zeta$ as a $(0,1)$-form on $S$ valued in $u^*TX$, so the pairing with the image of $D_u\dbar$ is the integral of the dot product.)
  Of course, these perturbations are not symmetric.

  Suppose we are in the first case. Replace $\xi$ by $\xi+d\tau\circ \xi\circ d\eta$. Then
  \begin{align*}
  \int_S \langle \zeta, D_u\xi+d\tau\circ\xi\circ d\eta\rangle &= \int_S\langle \zeta, D_u\xi\rangle + \langle \zeta,D_u d\tau\circ\xi\circ d\eta\rangle\\ &=\int_S \langle \zeta, D_u\xi\rangle + \langle d\tau\circ \zeta\circ d\eta,d\tau\circ(D_u\xi)\circ d\eta\rangle=2\int_S \langle \zeta, D_u\xi\rangle.
  \end{align*}
  This contradicts the fact that $\zeta$ annihilates the image of $D_u$. The second case is similar, replacing $Y_s$ by $Y_s+d\tau\circ Y_s\circ d\tau$, and the third is similar, but replacing $Y$ by $d\eta\circ Y\circ d\eta$. The result follows.
\end{proof}

For computations, it is typically more convenient to work with a split complex structure, and instead perturb the $\alpha$- and $\beta$-curves. Following Oh~\cite{Oh96:perturb-bdy}, Ozsv\'ath-Szab\'o~\cite[Proposition 3.9]{OS04:HolomorphicDisks} gave conditions under which such perturbations are sufficient in the closed case (see also~\cite[Lemma 3.10]{Lipshitz06:CylindricalHF}). Here is a formulation in the real bordered case:
\begin{proposition}\label{prop:bdy-injectivity}
  Fix a real bordered Heegaard diagram $\HD$, a real domain $B\in\pi_2^R(\x,\y)$, and a sequence of Reeb chords $\vec\rho$. Suppose that, for some split complex structure, every holomorphic representative $u\co S\to \Sigma\times[0,1]\times\RR$ of $B$ has the property that for each component $S_0$ of $S$ which is not a trivial strip, there is an open subset $U$ of some $\alpha$-curve in $\Sigma$ so that 
  \[
  (\pi_\Sigma\circ u)\co \bigl[(\bdy S_0)\cap (\pi_\Sigma\circ u)^{-1}(U)\bigr]\to U
  \]
  is a diffeomorphism (equivalently, bijection). Then for a small, generic perturbation of $\alphas$, and corresponding perturbation of $\betas$, the real moduli space $\cM^B_R(\x,\y;\vec\rho)$ is transversely cut out.

  In particular, if $S$ has either a single nontrivial component or two nontrivial components exchanged by $\eta$, and there is an arc in $\alphas$ so that the coefficient of $B$ on one side is $0$ and on the other side is $1$, then this hypothesis holds and, for any generic perturbation of $\alphas$, $\cM^B_R(\x,\y;\vec\rho)$ is transversely cut out.
\end{proposition}
\begin{proof}
  The proof is a simple modification of the unreal case, along the lines of Proposition~\ref{prop:transversality}, and is left to the reader.
\end{proof}
Proposition~\ref{prop:bdy-injectivity} will be used every time we perform explicit computations of moduli spaces, though we may not always mention it.

\subsection{Expected dimensions}

Recall that there are two versions of the index formula in bordered Heegaard Floer homology: a \emph{source-dependent index formula},
\[
g-\chi(S)+2e(B)+|P|
\]
\cite[Proposition 5.8]{LOT1}
that does not assume curves are embedded and rather depends on the topology of the source $S$, and an \emph{embedded index formula} 
\[
e(B)+n_\x(B)+n_\y(B)+|\vec{\rho}|+\iota(\vec{\rho})
\]
\cite[Definition 5.68]{LOT1}
for the moduli spaces of embedded curves. Here, $P$ is the partition of the $e\infty$ punctures of $S$ induced by the $\RR$-coordinate of $\Sigma\times[0,1]\times\RR$, so $|P|$ is the number of $e\infty$ punctures minus the number of height constraints. Since we are considering type $D$ structures,
$|P|$ agrees with the number of Reeb chords $|\vec{\rho}|$.
(Having two notations for the same quantity is an artifact of how the
properties of these moduli spaces were laid out.)

Let $\ind^R$ denote the expected dimension of the real moduli spaces. Guth and Manolescu show that, for closed Heegaard diagrams, this expected dimension at a real domain $B\in\pi_2^R(\x,\y)$ is given by 
\begin{align}
  \ind^R(B) = \frac{1}{2}\ind(B)+\frac{1}{4}\bigl(\sigma(\alphas,\y)-\sigma(\alphas,\x)\bigr)
\end{align}
where $\ind(B)=e(B)+n_\x(B)+n_\y(B)$ is the expected dimension of the ordinary Heegaard Floer moduli space and $\sigma(\alphas,\y)$ is the sum over the points $y_i$ in $\y$ that are fixed by the involution of a local contribution that depends on the cyclic order of $\alphas$, $\betas$, and the fixed set at $y_i$:
\[
\begin{tikzpicture}[alt={Local configuration at a fixed intersection point with sigma equals plus one}]
\begin{scope}
    \clip (-1.5,0) rectangle (1.5,1.5);
    \filldraw[black!10] (0,0) circle(1);
\end{scope}
\begin{scope}
    \clip (0,0) circle (1);
    \draw[red,thick] (-1,-1) to (1,1);
    \draw[blue,thick] (1,-1) to (-1,1);
\end{scope}
\draw[darkgreen,thick] (-1,0) to (1,0);
\node at (2,0) (fixlabel) {\textcolor{darkgreen}{Fixed set}};
\node at (-1,-.8) (alphalabel) {\textcolor{red}{$\alphas$}};
\node at (1,-.8) (betalabel) {\textcolor{blue}{$\betas$}};
\node at (0,-1.3) (sigmalabel) {$\sigma=+1$};
\end{tikzpicture}\qquad\qquad
\begin{tikzpicture}[alt={Local configuration at a fixed intersection point with sigma equals minus one}]
\begin{scope}
    \clip (-1.5,0) rectangle (1.5,1.5);
    \filldraw[black!10] (0,0) circle(1);
\end{scope}
\begin{scope}
    \clip (0,0) circle (1);
    \draw[red,thick] (1,-1) to (-1,1);
    \draw[blue,thick] (-1,-1) to (1,1);
\end{scope}
\draw[darkgreen,thick] (-1,0) to (1,0);
\node at (2,0) (fixlabel) {\textcolor{darkgreen}{Fixed set}};
\node at (1,-.8) (alphalabel) {\textcolor{red}{$\alphas$}};
\node at (-1,-.8) (betalabel) {\textcolor{blue}{$\betas$}};
\node at (0,-1.3) (sigmalabel) {$\sigma=-1$};
\end{tikzpicture}
\]
\cite[Section 4.2]{GM:real-HF}.
The analogous results in the bordered case are:
\begin{proposition}\label{prop:real-index}
  The expected dimension of the real moduli space with a fixed source $S$ and the expected dimension for embedded curves are given by
  \begin{align}
    \ind^R(B,S,P)&=\OneHalf\bigl(g-\chi(S)+2e(B)+|P|\bigr) + \OneQuarter\bigl(\sigma(\alphas,\y)-\sigma(\alphas,\x)\bigr)\label{eq:index-source}\\
    \ind^R(B,\vec{\rho})&=\OneHalf\bigl(e(B)+n_\x(B)+n_\y(B)\bigr)+|\vec{\rho}|+\iota(\vec{\rho})+\OneQuarter\bigl(\sigma(\alphas,\y)-\sigma(\alphas,\x)\bigr)\label{eq:index-embedded}
  \end{align}
  respectively. Here, $g$ is the number of $-\infty$ corners, $\vec{\rho}$ is the sequence of $\alpha$-Reeb chords (as in the notation $\cM^B_R(\x,\y;\vec{\rho})$ from Section~\ref{sec:def-of-moduli}), and $|P|=2|\vec{\rho}|$ is the number of punctures of $P$ not mapped to $\pm\infty$.
  
  In particular, the Euler characteristic of a real embedded curve is given by the same formula as in the ordinary case,
  \begin{equation}\label{eq:embedded-index}
  \chi_{\mathrm{emb}}(B,\vec{\rho})=g+e(B)-n_\x(B)-n_\y(B)-\iota(\vec{\rho}).
  \end{equation}
\end{proposition}

\begin{proof}
  The computation of the embedded Euler characteristic in our previous paper~\cite[Proposition 5.69]{LOT1} did not make any assumptions on transversality. So, since real holomorphic curves are special cases of unreal holomorphic curves, the embedded Euler characteristic is given by Formula~\eqref{eq:embedded-index}. Thus, to prove Proposition~\ref{prop:real-index}, it suffices to prove Formula~\eqref{eq:index-source}. (Recall that Formula~\eqref{eq:index-source} is a version of Rasmussen's index formula~\cite[Theorem 9.1]{Rasmussen03:Knots}, but in the real bordered setting.)

  Write $X=\Sigma\times[0,1]\times\RR$ and $L=(\alphas\times\{1\}\times\RR)\cup(\betas\times\{0\}\times\RR)$. Fix a real smooth map $u\co S\to \Sigma\times[0,1]\times\RR$ in the homology class $B$. We have a complex vector bundle $u^*TX$ over $S$, and real subbundle $u^*TL$ over $\bdy S$. The involution $\eta\co S\to S$ is covered by an involution $\tau\co (u^*TX,u^*TL)\to (u^*TX,u^*TL)$; on $u^*TX$ this is complex anti-linear.
  The real index of $D_u\dbar$ is the same as the index of the $\dbar$-operator for real sections of $(u^*TX,u^*TL)$. In particular, the only role of $u$ and $B$ is to determine this complex bundle and the real subbundles. So, given a complex 2-dimensional vector bundle $E$ over $S$, totally real subbundles $F$ over $\bdy S$, and involutions $\tau$ of these covering $\eta$, write $\ind^R(S,E,F,\tau)$ for the index of the $\dbar$-operator on real sections.

  We prove by induction on the complexity of $S$ that 
  \begin{equation}\label{eq:real-index-proof}
    \ind^R(S,E,F,\tau) = \OneHalf\ind(S,E,F) + \OneQuarter\bigl(\sigma(\alphas,\y)-\sigma(\alphas,\x)\bigr).
  \end{equation}
  If $S$ consists of a single bigon (and $\tau$ is reflection), then this is a special case of Guth-Manolescu's computation~\cite[Section 2.3]{GM:real-HF}. If $S$ consists of two components exchanged by $\tau$, then it is clear that the real index is half the ordinary one (and the $\sigma$ terms vanish). If $S$ is a 2-sphere, $\tau$ fixes a circle, and the bundle $E$ over $S$ is trivial, then $\ind(S,E,F)=4$ (the 4-dimensional space of constant maps $S^2\to \CC^2$) and $\ind^R(S,E,F)=2$ (the conjugation-invariant ones, i.e., the constant maps to $\RR^2$). The same holds if $\tau$ is the fixed point-free involution of the 2-sphere, and $E$ is trivial. In particular, Formula~\eqref{eq:real-index-proof} holds in both of these cases.

  We can reduce any $S$ to a union of surfaces of these four types by cutting along pairs of circles (respectively arcs) exchanged by $\tau$. More precisely, given a pair of disjoint circles $Z,Z'\subset S$ with $\tau(Z)=Z'$, fix a trivialization of $E$ over $Z\cup Z'$. Surgering $S$ along $Z\cup Z'$ gives a surface $S'$. Using the trivializations, $E$ induces a bundle $E'$ over $S'$. This operation increases $\ind(S,E,F)$ by $4$ and $\ind^R(S,E,F,\tau)$ by $2$: each surgery is equivalent to splicing $(S,E,F)$ with a trivial vector bundle over a 2-sphere (see, for instance, the nice explanation in~\cite[Section 1.8.2]{EGH00:IntroductionSFT}). Similarly, given a pair of disjoint arcs $A,A'$ exchanged by $\tau$, with boundary on $\bdy S$, after fixing a trivialization of $E$ along $A$ extending a trivialization of $F$ on $\bdy A$, surgering along $A\cup A'$ gives a new surface $S'$ and vector bundles $E',F'$ and $\ind^R(S',E',F',\tau')=\ind^R(S,E,F,\tau)+1$ while $\ind(S',E',F')=\ind(S,E,F)+2$. (This is splicing with a pair of disks; see, for instance,~\cite[Section 8.1]{EtnyreNgSabloff02:coherent-ors}.) In particular, both sides of Equation~\eqref{eq:real-index-proof} change in the same way under these operations. Since we can use them to reduce to the base cases considered above, the proposition follows.
\end{proof}

\subsection{Codimension-1 degenerations}

The following result about the boundary of the 2-dimensional real moduli spaces is essentially equivalent to the fact that $\CFDRa(\HD,\tau)$ satisfies the type $D$ structure relation:
\begin{proposition}\label{prop:codim-1}
  Fix a real bordered Heegaard diagram $(\HD,\tau)$ with $\bdy_L\HD=\PMC$, and a generic symmetric almost complex structure $J$.
  Given $\x,\y\in\Gen_R(\HD,\tau)$, $B\in\pi_2(\x,\y)$, and a sequence of Reeb chords $\vec{\rho}$ (compatible with $B$) so that $\ind^R(B,\vec{\rho})=2$, the sum of the following numbers is even: 
  \begin{enumerate}
  \item\label{item:2-story} The number of real 2-story holomorphic curve ends, i.e., elements of 
  \[
  \bigcup_{\substack{B_1*B_2=B\\(\vec{\rho}_1,\vec{\rho}_2)=\vec{\rho}\\\ind_R(B_i,\vec{\rho}_i)=1}}\cM^{B_1}_R(\x,\w;\vec{\rho}_1)\times\cM^{B_2}_R(\w,\y;\vec{\rho}_2).
  \]
  \item\label{item:join} The number of join curve ends, i.e., elements of
  \[
  \bigcup_{i}\bigcup_{\rho_i=\rho'\rho''}\cM^B_R(\x,\y;\rho_1,\dots,\rho_{i-1},\rho'',\rho',\rho_{i+1},\dots,\rho_n),
  \]
  where $\vec{\rho}=(\rho_1,\dots,\rho_n)$.
  \item\label{item:collision} The number of collisions of levels, i.e., elements of
  \[
  \bigcup_{i}\cM^B_R(\x,\y;\rho_1,\dots,\rho_{i-1},\rho_i\rho_{i+1},\rho_{i+2},\dots,\rho_n).
  \]
  Here, $\rho_i\rho_{i+1}$ denotes the concatenation of the chords if they are composable, and the two-element set $\{\rho_i,\rho_{i+1}\}$ otherwise.
  \end{enumerate}
\end{proposition}
\begin{proof}
  This follows from essentially the same arguments as in the unreal case~\cite[Theorem 5.61]{LOT1}, which we recall.

  The first step is to observe that the moduli spaces have compactifications in terms of holomorphic combs~\cite[Proposition 5.24]{LOT1}. Since the real moduli spaces are subspaces of the usual ones, and this statement does not rely on transversality, the compactness result for the real moduli spaces, via \emph{real holomorphic combs}, follows from the usual one.

  We then prove several gluing results~\cite[Section 5.5]{LOT1}, for 2-story buildings and certain $1$-story combs with curves at $e\infty$. The proofs (which are largely omitted) carry over without change to the real case. In particular, if one starts with a two-story building in which both stories are $\tau$-invariant, the pre-glued approximately holomorphic curves are also $\tau$-invariant, and applying the inverse function theorem to the configuration space of real maps gives real holomorphic curves. Similarly, pre-gluing a symmetric pair of $e\infty$ curves on $\bdy_L\Sigma$ and $\bdy_R\Sigma$ to a real curve, with the same gluing parameter on the two sides, again gives an approximately holomorphic real curve, and applying the inverse function theorem in the real configuration space gives a real holomorphic curve.

  The first restrictions on codimension-1 degenerations~\cite[Proposition 5.43]{LOT1} come from the index formula---the unreal analogue of Formula~\eqref{eq:index-source}. In the real case, the analogous proposition says that the boundary of the index-2 moduli spaces consists of two-story buildings; degenerations at $e\infty$ consisting of a join curve at each boundary component; degenerations of some split curves; and nodal curves where some arcs degenerate. (Shuffle curves do not appear because our moduli spaces do not have height constraints between the Reeb chords, or at least not between ones on the same boundary component.) We review the part of that argument that used the index formula; the following discussion assumes the reader is consulting that proof simultaneously.

  Since we are considering an index-2 real moduli space,
  \[
  \OneHalf\bigl(g-\chi(S)+2e(B)+|P|\bigr) + \OneQuarter\bigl(\sigma(\alphas,\y)-\sigma(\alphas,\x)\bigr)=2.
  \]
  As there, since $\chi(S')\geq \chi(S)$ and $|P'|\leq |P|$, $\ind(B,S',P')\leq \ind(B,S,P)$.
  If a holomorphic comb has more than one story, the main component of each story has index $\geq 1$. So, there are at most two stories. If there are two stories with index $1$ each, then the same argument as in the unreal case implies there are no curves at $e\infty$. Similarly, if $\ind^R(B,S',P')=2$, then there is no curve at $e\infty$.

  So, suppose we have a single story with $\ind^R(B,S',P')=1$ and some curves at $e\infty$. Curves at $e\infty$ come in pairs; write $T'=T'_L\cup T'_R$. Similarly, $|P|-|P'|$ is even, and the $\sigma$-terms in the formula are the same for $(B,S,P)$ and $(B,S',P')$. Thus, either $|P|=|P'|$ and $\chi(T')=m-2$ or $|P'|=|P|-2$ and $\chi(T')=m$. The former case is a join curve and the latter is a collision of levels (possibly including a pair of split curves).
  
  The next step in the unreal case is ruling out boundary degenerations and other nodal degenerations~\cite[Lemmas 5.47, 5.48, and 5.56]{LOT1}; those arguments are the same in the real case. Combining these restrictions with the gluing results gives the proposition, just as in the unreal case~\cite[Proof of Theorem 5.61]{LOT1}.
\end{proof}

\section{Real bordered modules}\label{sec:CFDR}

With the moduli spaces in hand, the definition of the real bordered Floer invariants is the same as the original definition of $\CFDa$, except with the real states and moduli spaces in place of the usual ones.
\begin{definition}\label{def:CFDR}
	Let $(\HD,\tau)$ be a provincially admissible real bordered Heegaard diagram, with $\alpha$-boundary $\PMC$ and a generic symmetric almost complex structure $J$. Let $\CFDRa(\HD,\tau)$ be the left type $D$ structure over $\Alg(-\PMC)$ defined as follows. As an $\FF_2$-vector space, $\CFDRa(\HD,\tau)$ has basis $\Gen_R(\HD,\tau)$. The left idempotent of a state $\x\in\Gen_R(\HD,\tau)$ is the set of $\alpha$-arcs not containing points on $\x$, viewed as a subset of the matched pairs of $-\PMC$ (which are the same as the matched pairs in $\PMC$, i.e., the $\alpha$-arcs). The differential is defined by
\[
\delta^1(\x)=\sum_{\substack{\y\in\Gen_R(\HD,\tau)\\B\in\pi_2(\x,\y)}}\sum_{\ind^R(B;\vec{\rho})=1}\bigl(\#\cM_R^B(\x,\y;\vec{\rho})\bigr)a(-\vec{\rho})\otimes\y
\]
where, if $\vec{\rho}=(\rho_1,\dots,\rho_n)$, then $a(-\vec{\rho})=a(-\rho_1)\cdots a(-\rho_n)$, and $-\rho_i$ is the chord in $-Z$ corresponding to $\rho_i\subset Z$ (the same subset of the circle, but viewed as having the opposite orientation).
\end{definition}

\begin{lemma}\label{lem:admis-finite}
  The sum defining $\delta^1$ is finite, and if $\HD$ is admissible, then $(\CFDRa(\HD,\tau),\delta^1)$ is a bounded type $D$ structure (in the sense of~\cite[Section 2.3]{LOT1}).
\end{lemma}
\begin{proof}
  The proof is the same as in the unreal case~\cite[Lemma 6.5]{LOT1}.
\end{proof}

\begin{theorem}\label{thm:CFDR-defined}
	Definition~\ref{def:CFDR} defines a type $D$ structure. That is, $\delta^1$ satisfies the type $D$ structure relation.
\end{theorem}
\begin{proof}
  Except that the moduli spaces have a different meaning, Proposition~\ref{prop:codim-1} is the same as the type $D$ case of the corresponding result in the unreal case~\cite[Theorem 5.61]{LOT1}. (The unreal case also mentions shuffle curve degenerations, but these only occur when the partition has parts consisting of more than one Reeb chord, and the type $D$ invariant does not consider such partitions.) The results on boundary monotonicity~\cite[Lemma 6.8 and 6.9]{LOT1} still hold (with the same proofs). The restrictions that collisions of levels must be composable and join curves correspond to valid splittings~\cite[Lemma 5.77]{LOT1} follow from the embedded index formula~\eqref{eq:index-embedded} exactly as in the unreal case: non-composable collisions and invalid splittings drop $|\vec\rho|+\iota(\vec\rho)$ by more than $1$. (These terms are the same as in the unreal index formula, not multiplied by $1/2$.) So, this follows from exactly the same argument as in the unreal case~\cite[Proposition 6.7]{LOT1}.
\end{proof}

While we do not prove invariance of $\CFDRa$, we do record that it is independent of the choice of almost complex structure, and hence is an invariant of the real bordered Heegaard diagram:
\begin{proposition}
  Fix a provincially admissible real bordered Heegaard diagram $\HD$ and generic symmetric almost complex structures $J,J'$. Then the type $D$ structures $\CFDRa(\HD;J)$ and $\CFDRa(\HD;J')$ computed with respect to $J$ and $J'$ are homotopy equivalent.
\end{proposition} 
\begin{proof}
  Like the proof of Theorem~\ref{thm:CFDR-defined}, this is a straightforward adaptation of the closed case~\cite[Section 6.3.1]{LOT1}, and is left to the reader.
\end{proof}

\subsection{An example}\label{sec:CFDR-example}
Consider the bordered Heegaard diagram from Figures~\ref{fig:T2-diag-two-fixed-circles} and~\ref{fig:genus-1-domains}, representing $[0,1]\times T^2$ with involution induced by the reflection on $T^2$ (with two fixed circles). We have $\pi_2(\x,\x)=\pi_2(\y,\y)=\ZZ$, with generator $P$ as shown in Figure~\ref{fig:genus-1-domains}.
Since $\epsilon(\x,\y)$ is the nontrivial element of $\ZZ/2\ZZ$ (see Example~\ref{eg:Sigma-times-I-zeta}), the only possible contributions to the differential on $\CFDRa(\HD)$ are real periodic domains. Only domains all of whose coefficients at the boundary are $0$ or $1$ contribute to the differential on $\CFDRa(\HD)$, because of the form of $\Alg(T^2)$. Thus, the only domain to consider is the domain $P$. By Proposition~\ref{prop:bdy-injectivity}, we may compute the contribution of $P$ using a split almost complex structure (and generic perturbation of $\alphas$). Viewed as an element of $\pi_2(\x,\x)$, $P$ clearly has no holomorphic representative. Viewed as an element of $\pi_2(\y,\y)$, a real holomorphic representative of $P$ corresponds to a pair of slits starting at $\y$, in some direction along the $\alpha$-arc and the corresponding direction along the $\beta$-arc, and then a degree-1 holomorphic map to $[0,1]\times\RR^2$ from the result. Thus, the result must be a disk. There is exactly one way to achieve that: for the slits to cut out the boundary. The result is an operation $\delta^1(\y)=\rho_1\rho_2\otimes \y$ (in the usual notation for the torus algebra). To summarize:
\[
\CFDRa(\HD)=\langle \x,\y\rangle \qquad\qquad \delta^1(\x)=0\qquad\qquad \delta^1(\y)=\rho_1\rho_2\otimes\y.
\]

\section{Real-nice diagrams and the real Auroux-Zarev piece}\label{sec:nice-and-AZ}
In this section, we discuss an extension of Sarkar-Wang's notion of nice diagrams~\cite{SarkarWang07:ComputingHFhat} to real bordered Heegaard diagrams. In Section~\ref{sec:real-nice-basics} we formulate the definition of real nice diagrams and show that holomorphic curve counts in such diagrams are combinatorial. (Substantially more domains contribute than in the unreal case; see Figure~\ref{fig:real-nice-curves}.) In Section~\ref{sec:AZ-diagrams}, we recall two nice diagrams that will be used in our computation of real Floer homology, the Auroux-Zarev diagrams. Their bordered Floer invariants are recalled in Section~\ref{sec:AZ-modules-review}, before we compute their real analogues in Section~\ref{sec:real-AZ}.

\subsection{Real-nice diagrams and their real bordered invariants}\label{sec:real-nice-basics}
\begin{definition}
  \label{def:real-nice}
  A real bordered Heegaard diagram $\HD$ is \emph{real-nice} if all the
  regions in $\HD$ not containing the basepoint are bigons or rectangles
  (an edge of which may be on $\bdy\Sigma$) and for all points $x$ in
  $\alphas\cap C$, $\sigma(x)$ has the same value (i.e., either
  $\sigma(x)$ is always $1$ or is always $-1$). If $\sigma(x)$ is
  identically $1$, call the diagram \emph{positive real-nice}, and if
  $\sigma(x)$ is identically $-1$, call it \emph{negative real-nice}.
\end{definition}

\begin{figure}
  \centering
  \includegraphics[alt={Rigid curves in real-nice diagrams}, scale=.97]{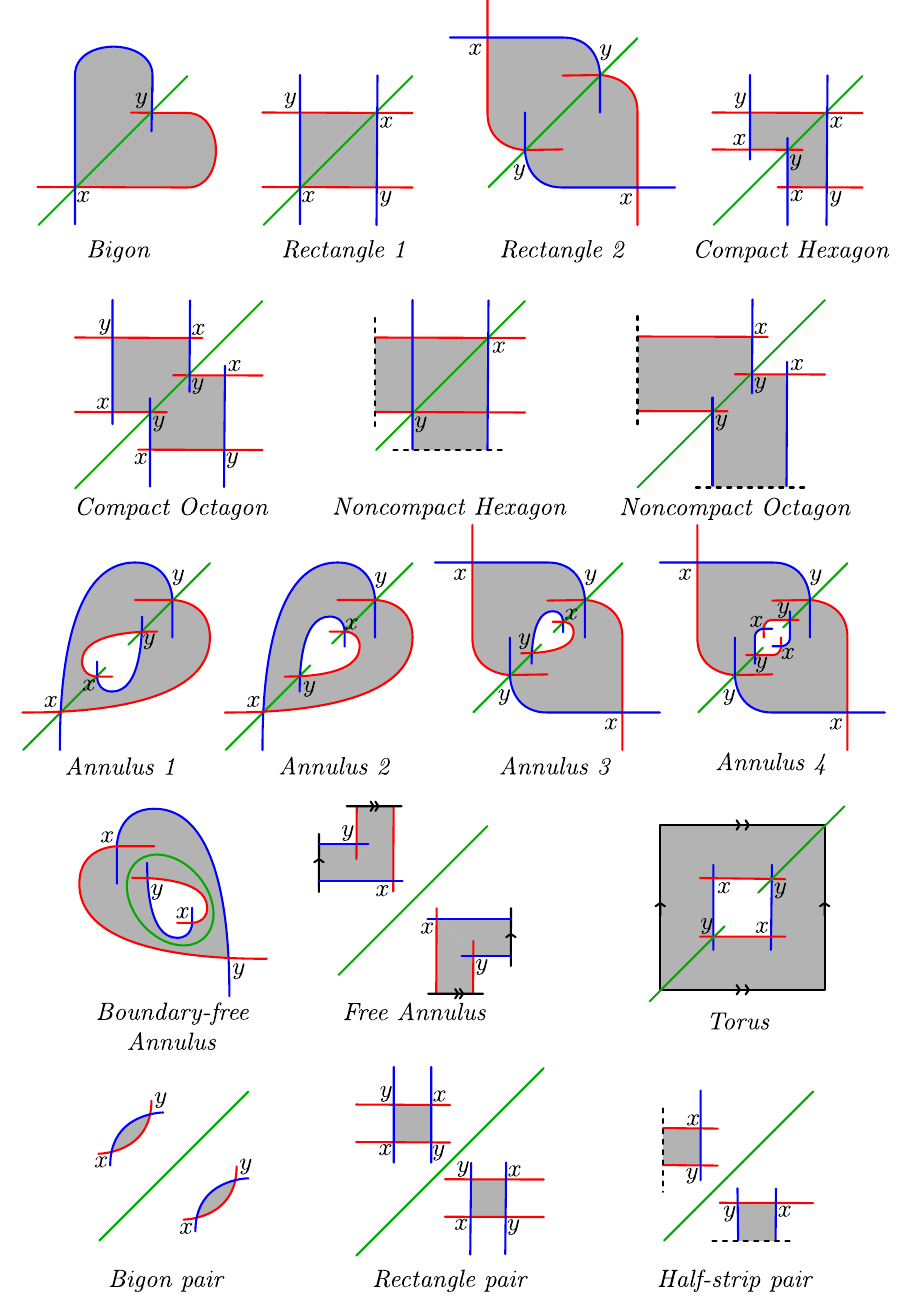}
  \caption{\textbf{Rigid curves in real-nice diagrams}. All the examples are in the $\sigma=-1$ case. For the $\sigma=1$-case, exchange the $\alpha$- and $\beta$-curves.}
  \label{fig:real-nice-curves}
\end{figure}

\begin{proposition}\label{prop:classify-domains}
  If $\HD$ is negative real-nice, then the differential on $\CFDRa(\HD)$ counts:
  \begin{enumerate}
    \item\label{item:bigons} $\tau$-invariant empty bigons;
    \item $\tau$-invariant empty rectangles, which come in two types: rectangles with two $90^\circ$ $\x$-corners on the fixed set and two $90^\circ$ $\y$-corners off the fixed set, and rectangles with two $270^\circ$ $\y$-corners on the fixed set and two $90^\circ$ $\x$-corners off the fixed set;
    \item $\tau$-invariant compact empty hexagons, which are disks with four corners (two in $\x$ and two in $\y$) off the fixed set and two corners (one $90^\circ$ $\x$-corner and one $270^\circ$ $\y$-corner) on the fixed set;
    \item $\tau$-invariant compact empty octagons, which are disks with two $270^\circ$ $\y$-corners on the fixed set and four ($\,90^\circ\!$) $\x$-corners and two ($\,90^\circ\!$) $\y$-corners off the fixed set;
    \item $\tau$-invariant noncompact empty hexagons, which are disks with one fixed $\x$-corner ($\,90^\circ\!$), one fixed $\y$-corner ($\,270^\circ\!$), and two $e\infty$-punctures; 
    \item $\tau$-invariant noncompact empty octagons, with two $\y$-corners ($\,270^\circ\!$) on the fixed set, two $\x$-corners ($\,90^\circ\!$) off the fixed set, and two $e\infty$ Reeb chords;
    \item\label{item:annuli-fixed} $\tau$-invariant compact annuli so that the fixed set intersects the boundary, which come in four types: 
    \begin{enumerate}[label=(A\arabic*)]
      \item\label{item:A1} annuli with one (fixed) $\x$-corner and one (fixed) $\y$-corner on each boundary component ($\,90^\circ$ and $270^\circ$, respectively) and so that the $\alpha$-boundary of the inner component is on the same side of the fixed set as the $\beta$-boundary of the outer component;
      \item\label{item:A2} annuli with one (fixed) $\x$-corner and one (fixed) $\y$-corner on each boundary component ($\,90^\circ$ and $270^\circ$, respectively) and so that the $\alpha$-boundary of the inner component is on the same side of the fixed set as the $\alpha$-boundary of the outer component;
      \item\label{item:A3} annuli so that one boundary component has one (fixed) $\x$- and one (fixed) $\y$-corner ($\,90^\circ$ and $270^\circ$, respectively) and the other boundary component has two fixed $\y$-corners (both $270^\circ\!$) and two non-fixed $\x$ corners (both $90^\circ\!$); and
      \item\label{item:A4} annuli so that each boundary component has two fixed $\y$-corners (both $270^\circ\!$) and two non-fixed $\x$ corners (both $90^\circ\!$);
    \end{enumerate}
    \item\label{item:A0} $\tau$-invariant boundary-free annuli, with no fixed corners, one $90^\circ$ corner and one $270^\circ$ corner (either an $\x$- or $\y$-corner) on each boundary component, and on which $\tau$ acts by reflection across a circle in the interior of the annulus;
    \item\label{item:A-free} $\tau$-invariant free annuli, so that the action of $\tau$ on the annulus is free (and exchanges the boundary components), and the annulus has one $90^\circ$ corner and one $270^\circ$ corner (either an $\x$- or $\y$-corner) on each boundary component;
    \item\label{item:tori} $\tau$-invariant punctured tori, with one boundary component with two fixed $\y$-corners and two non-fixed $\x$-corners (all $270^\circ\!$), and so that the involution on the torus has fixed set an arc and quotient a M\"obius band;
    \item\label{item:bigon-pairs} pairs of empty bigons which are exchanged by $\tau$;
    \item pairs of empty rectangles which are exchanged by $\tau$; and
    \item\label{item:halfstrip-pairs} pairs of empty half-strips which are exchanged by $\tau$.
  \end{enumerate}
  If $\HD$ is positive real-nice, the differential admits an exactly analogous description, except in each case the $\x$ and $\y$ corners are exchanged.
  In all cases, all corners of the domain (or curve) not fixed by the involution are $90^\circ$ angles The map from the holomorphic curve to $\Sigma$ is an immersion, except for boundary-free annuli, free annuli, and tori (Cases~(\ref{item:A0}),~(\ref{item:A-free}), and~(\ref{item:tori})), for which there are two boundary branch points.
\end{proposition}
\noindent
See Figure~\ref{fig:real-nice-curves} for examples of these domains. \emph{Empty} means that no points of $\x$ or $\y$ are in the interior of the domain.
Note also that the domains need not be embedded; for example, pairs of rectangles can overlap (cf.\ Figure~\ref{fig:AZ-diff-domains}). 
\begin{proof}
  We will consider the case that all points $x$ on $\alphas\cap C$ have $\sigma(x)=-1$. The case that $\sigma(x)$ is identically $1$ follows from the observation that reversing the orientation of $\Sigma$ changes the sign of $\sigma$ and the directions of holomorphic curves (from $\y$ to $\x$ instead of $\x$ to $\y$), but not the counts of holomorphic curves.

  From Formula~\eqref{eq:index-source},
  \begin{equation}\label{eq:nice-ind-1}
    \ind^R(B,S,\vec{\rho})=\frac{1}{2}(g-\chi(S)+2e(B)+|P|)-\frac{1}{4}\#(\y\cap C)+\frac{1}{4}\#(\x\cap C)=1.
  \end{equation}
  Note that $|P|=2|\vec{\rho}|$: for $|P|$, both $\alpha$- and $\beta$-Reeb chords count.

  The index is additive over components of $S$ (where $2g$ is replaced by the number of $\pm\infty$ corners on the component, and so on). For trivial strips, $\ind^R(B,S,P)=0$. So, we consider two cases: a pair of components which are exchanged by $\tau$, and a component which is preserved by $\tau$. We first prove that the domain is as in the statement of the proposition without the condition of being empty or the map to $\Sigma$ being an immersion, and then verify the emptiness and immersion conditions at the end.

  \emph{Case 1.} A pair of non-fixed components. In this case, $\x\cap C=\y\cap C=\emptyset$, and $\chi(S)\leq 2$ with equality if and only if $S$ consists of two disks. We have 
  \[
  \frac{1}{2}(g-\chi(S)+2e(B)+|P|)=1,
  \]
  so for each of the two components $S_1$ and $S_2$, $g_i-\chi(S_i)+2e(B_i)+|P_i|=1$. This is exactly the index formula in the nonequivariant case, so from the corresponding result for nonequivariant nice diagrams~\cite[Proposition 8.4]{LOT1}, $S_1$ and $S_2$ are each a bigon, rectangle, or half-strip.

  Note also that this argument precludes an index-0 pair of non-fixed components.

  \emph{Case 2.} A fixed component. Observe that the component has $|P|$ even, because chords come in pairs, exchanged by the involution; and $\chi(S)\leq 1$, with equality if and only if $S$ is a disk. Since the curve has $2g$ corners at $\pm\infty$, we can rewrite Formula~\eqref{eq:nice-ind-1} as 
  \begin{equation}\label{eq:index-formula-again}
  \mathit{nfc}/4+\#(\x\cap C)/2 + |P|/2-\chi(S)/2+e(B)=1,
  \end{equation}
  where $\mathit{nfc}$ denotes the number of non-fixed corners---the
  number of points in $\x$ not on $C$ plus the number of points in $\y$
  not on $C$. Moreover, $S$ has at least one fixed $x$-corner or at least two non-fixed $x$-corners, so the left side is at least $|P|/2+e(B)$. So, $e(B)\leq 1$, and if $e(B)\neq 0$ then $|P|=0$. We take three subcases.

  \emph{Subcase 2a.} $e(B)=1$. In this case, $\mathit{nfc}/4+\#(\x\cap C)/2-\chi(S)/2=0$. Since there is at least one $x$-corner, we must have $\chi(S)=1$ (so $S$ is a disk) and $\mathit{nfc}/2 + \#(x\cap C)=1$. Thus, the domain either has two non-fixed corners or one fixed $x$-corner. The number of $x$-corners equals the number of $y$-corners, so we either have a bigon (with one fixed $x$-corner and one fixed $y$-corner) or a rectangle of the second type (with two fixed $y$-corners and two non-fixed $x$-corners).

  \emph{Subcase 2b.} $e(B)=1/2$. This does not occur. The condition that $\sigma(x)$ is identically $-1$ implies that no elementary region in the diagram is a bigon that is fixed by $\tau$. So, an even number of the elementary regions in $B$ are bigons, and hence the Euler measure of $B$ is an integer.

  \emph{Subcase 2c.} $e(B)=0$. In this case, $\mathit{nfc}/4+\#(\x\cap C)/2 + |P|/2-\chi(S)/2=1$. By counting $x$-corners, the first two terms contribute at least $1/2$. So, if $|P|\neq 0$, then the left side is at least $3/2-\chi(S)/2$, so $\chi(S)=1$ ($S$ is a disk), $|P|=2$, and $S$ either has one fixed $x$-corner (and no non-fixed corners) or two non-fixed $x$-corners (and no non-fixed $y$-corners). The first case is a noncompact hexagon, and the second case is a noncompact octagon.

  Finally, if $|P|=0$, we could have $\chi(S)=-1$ (a pair of pants or punctured torus), $\chi(S)=0$ (an annulus), or $\chi(S)=1$. (The case $\chi(S)=-2$ does not occur because it would imply there are no $x$ corners.)

  If $\chi(S)=-1$, then Formula~\eqref{eq:index-formula-again} gives $\mathit{nfc}/4+\#(\x\cap C)/2+|P|/2=1/2$. Thus, $|P|=0$ and $\mathit{nfc}/4+\#(\x\cap C)/2=1/2$. So, the domain either has $1$ fixed $x$-corner and $1$ fixed $y$-corner, or $2$ non-fixed corners. In the latter case, since there is at least one $x$-corner per boundary component, there are $2$ non-fixed $x$-corners and $2$ fixed $y$-corners. Moreover, in both cases, there are not enough corners for three boundary components, so $S$ must be a punctured torus. Further, since the Euler measure of the domain is $0$ and the Euler measure of a torus with one $90^\circ$ and one $270^\circ$ corner is $-1$, in fact the domain has $2$ non-fixed $x$-corners and $2$ fixed $y$-corners, as in Figure~\ref{fig:real-nice-curves} (Type~(\ref{item:tori})).
  
  If $\chi(S)=0$, then $\mathit{nfc}/4+\#(\x\cap C)/2+|P|/2=1$. Thus, $|P|=0$ and $\mathit{nfc}/2+\#(\x\cap C)=2$. So, the domain either has $2$ fixed $x$-corners, one fixed $x$-corner and 2 non-fixed corners, or $4$ non-fixed corners. If the involution exchanges the two components of the boundary, then there must be two corners on each. We take some cases:
  \begin{itemize}
  \item If the fixed set of the action on $S$ is empty, then $S$ has two corners on each boundary component, $90^\circ$ and one $270^\circ$, and a free, orientation-reversing involution exchanging the boundary components. This is the free annulus (Type~(\ref{item:A-free})).
  \item If the fixed set is a circle (in the interior of $S$), then $S$ has two corners on each boundary component, one $90^\circ$ and one $270^\circ$; this is the boundary-free annulus, Type~(\ref{item:A0}).
  \item If the involution preserves each boundary component, each of the two components of $\bdy S$ must have two fixed corners, since every orientation-reversing involution of the circle has two fixed points, and an even number of non-fixed corners. So, these are the type \ref{item:A1}--\ref{item:A4} annuli.
  \end{itemize}
 
  If $e(B)=0$, $|P|=0$ and $\chi(S)=1$, then we are considering a disk with $\mathit{nfc}/4+\#(\x\cap C)/2=3/2$. As in the previous paragraph, the boundary component has two fixed corners. Also, the disk is built entirely of rectangles. The possibilities are:
  \begin{itemize}
    \item Two fixed $x$-corners and two non-fixed $y$-corners; this is a rectangle of the first type.
    \item One fixed $x$-corner, one fixed $y$-corner, and four non-fixed corners (two of each type). This is a compact hexagon.
    \item Two fixed $y$-corners and six non-fixed corners. This is a compact octagon.
  \end{itemize}

  Note also that the argument in each case precludes an index-0 fixed component.

  Combining the cases so far, we have proved:
  \begin{itemize}
    \item The curve consists of either a single, fixed component or a single pair of non-fixed components, and
    \item The component(s) have one of the forms from the proposition statement, except perhaps they could be nonempty and the map to $\Sigma$ might not be an immersion.
  \end{itemize}
  To see that the map to $\Sigma$ is an immersion except in cases~(\ref{item:A0}),~(\ref{item:A-free}), and~(\ref{item:tori}), note that if the (real) map to $\Sigma$ has branch points, then there is in fact a family of (real) maps by varying the branch points. We will see in Lemma~\ref{lem:nice-unbranched-domains-contribute} that any real disk, or real annulus of types \ref{item:A1}--\ref{item:A4}, admits a real holomorphic branched cover to the strip, so if the map to $\Sigma$ has branch points, then the moduli space is not rigid (by varying the branch point in $\Sigma$).

  To show that the domain is empty and all non-fixed corners have $90^\circ$ angles, we appeal to the embedded index formula. Specifically, by inspection, an empty domain of each of the types we have reduced to has $\ind^R(B,\vec{\rho})=1$, and a nonempty domain has index strictly larger than this.

  It remains to show that all of these domains admit unique holomorphic representatives. We state and prove this as three separate lemmas, Lemmas~\ref{lem:nice-unbranched-domains-contribute},~\ref{lem:A0-Afree-contribute}, and~\ref{lem:torus-contribute}.
\end{proof}

\begin{lemma}\label{lem:nice-unbranched-domains-contribute}
  In a real-nice diagram, for any choice of symmetric almost complex structure,
  each of the domains in Proposition~\ref{prop:classify-domains} of types~(\ref{item:bigons})--(\ref{item:annuli-fixed}) and~(\ref{item:bigon-pairs})--(\ref{item:halfstrip-pairs}) has a
  unique real holomorphic representative.
\end{lemma}
\begin{proof}
  This is familiar except for the annuli; see~\cite[Examples 4.4--4.6]{GM:real-HF}. For example, consider an octagon. We may work with a split almost complex structure when computing the moduli space; we discuss this more in the annulus case below. The source $S$ of a holomorphic representative is obtained from the domain by perhaps making a cut from each of the $270^\circ$ corners. Such cuts cannot be made $\tau$-equivariantly, so if the cuts are nontrivial, then $\pi_\Sigma\circ u\circ\ul{\tau} \neq \tau\circ\pi_\Sigma\circ u$ (because the images of the branch points are different, say). So, there are no cuts: the source $S$ is identified with the domain and $\pi_\Sigma\circ u$ is the identity map. (If the map to $\Sigma$ is not an embedding, we mean that if the domain $B=\sum n_iR_i$, then we take $n_i$ copies of $R_i$ and glue these together along shared edges.) We must show there is a unique holomorphic map to $[0,1]\times\RR$ respecting the punctures and involution. It is well-known that there is a unique holomorphic map $u_D$ to $[0,1]\times\RR$ respecting the corners. (Explicitly, identify $[0,1]\times\RR$ with the upper half plane, so $-\infty$ is identified with $0$ and $+\infty$ with $\infty$. Identify $S$ with the upper half-plane, so that the $-\infty$ punctures are at points $x_1,x_2,x_3$ and the $+\infty$ punctures are at points $y_1,y_2,y_3$. Then the map is $z\mapsto [(z-x_1)(z-x_2)(z-x_3)]/[(z-y_1)(z-y_2)(z-y_3)]$.) If this map did not respect the involution, then $\tau\circ u_D\circ \ul{\tau}$ would be another holomorphic map respecting the corners, contradicting uniqueness. (The existence of the map $u_D$ for any complex structure on $S$ is why the maps to $\Sigma$ in Proposition~\ref{prop:classify-domains} are immersions, in the case of disks; a similar comment applies to annuli, via the argument below.)
  
  So, in the rest of the proof, we will focus on the four types of annuli in item~(\ref{item:annuli-fixed}). Note first, because there are no (real) index-0 holomorphic curves in a real-nice diagram, the count is independent of the choice of almost complex structure. Also, by Proposition~\ref{prop:bdy-injectivity}, we can achieve transversality by working with a split complex structure and perturbing the $\alpha$- and $\beta$-curves ($\tau$-equivariantly).

  For all the kinds of annuli under consideration, 
  the holomorphic curve in $\Sigma\times[0,1]\times\RR$ can be described as holomorphic maps from the domain in $\Sigma$ (glued together maximally along shared edges), possibly cut along some $\alpha$- or $\beta$-curves, to $[0,1]\times\RR$. In the real case, the cuts must come in pairs, respecting the $\ZZ/2$-action. By considering the local structure at the corners, it follows that for these domains, no cuts are possible. 
  So, we want to show that for some (or all) complex structure on $\Sigma$, there is a holomorphic map from the domain shown to $[0,1]\times\RR$, sending the $\alpha$ boundary to $\{1\}\times\RR$ and the $\beta$-boundary to $\{0\}\times\RR$. The existence of such a map is independent of how the $\alpha$- and $\beta$-curves continue beyond the portions indicated in Figure~\ref{fig:real-nice-curves}.

  For type~\ref{item:A1} and~\ref{item:A2} annuli, this is immediate: double branched covers of $[0,1]\times\RR$ correspond to involutions on the annulus taking the $\alpha$-boundary to the $\alpha$-boundary. Uniformizing the domain by a standard annulus, the real involution shows that, on each component of the boundary, the $\alpha$-segment subtends an angle of $\pi$, and hence such an involution exists (and is unique). (A closely related case, using the same argument, was also described in~\cite[Example 4.7]{GM:real-HF}. See also~\cite[Lemma 9.3]{OS04:HolomorphicDisks}.)

  \begin{figure}
    \centering
    \includegraphics[alt={Diagrams to prove the existence of annuli}]{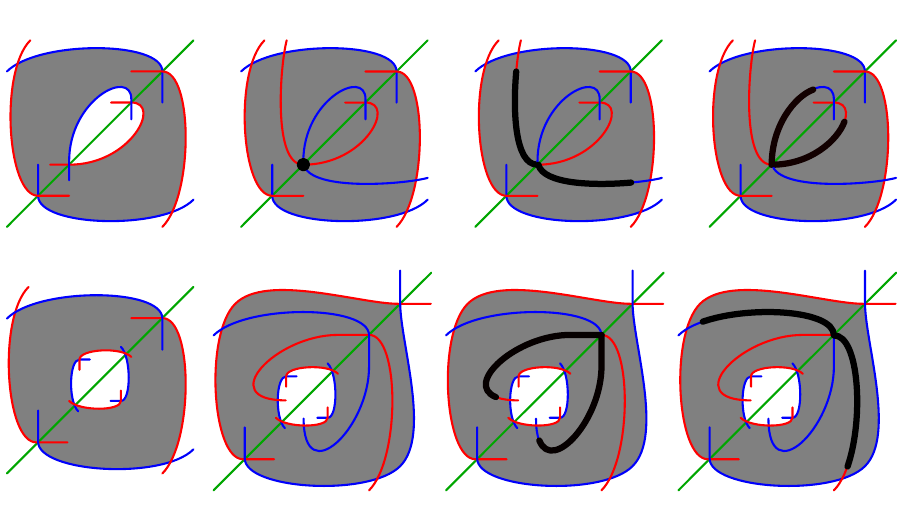}
    \caption{\textbf{Diagrams to prove the existence of annuli.} Top row: the type \ref{item:A3} annuli from Proposition~\ref{prop:classify-domains}; a particular choice of extension of two of the curves in its boundary and a larger domain in the resulting diagram; a decomposition of the result into a pair of type 2 rectangles; the other end of the moduli space, proving the existence of the desired curve. Bottom row: the type \ref{item:A4} annulus; an extension of four of the curves in its boundary and a domain in the larger diagram; a decomposition of that domain into a type \ref{item:A3} annulus and a compact hexagon; the other end of the moduli space, where the domain decomposes into a rectangle and the type \ref{item:A4} annulus, proving the existence of the type \ref{item:A4} annulus.}
    \label{fig:annuli-exist}
  \end{figure}

  We reduce the type~\ref{item:A3} and~\ref{item:A4} annuli to this case. As noted above, the count is independent of how the $\alpha$- and $\beta$-curves continue, so consider the (incomplete, but slightly less incomplete) diagrams shown in Figure~\ref{fig:annuli-exist}. Starting with the type \ref{item:A3} annuli (top row in Figure~\ref{fig:annuli-exist}), extend the diagram for the annulus by filling the inner boundary component (the one with a $90^\circ$ and a $270^\circ$ corner) with a bigon, and extending the arcs from the $270^\circ$ corner of the annulus on the inner boundary out of the annulus's outer boundary without crossing the fixed set. This domain decomposes equivariantly in two ways: as an annulus and a bigon, and as a hexagon and a rectangle. Since the pair of bigons both clearly have holomorphic representatives, gluing them gives a 1-dimensional family of annuli, corresponding to a pair of cuts, exchanged by $\tau$, starting from a corner on the inner boundary of the annulus. (The domain does not admit any other real pair of cuts, guaranteeing the moduli space has this form.) The other end of the moduli space is the annulus whose existence we are trying to prove, concatenated with a bigon. This proves annuli of this type exist.

  We handle type \ref{item:A4} similarly. Extend the domain and some of the curves as shown in Figure~\ref{fig:annuli-exist} (bottom row). The larger domain admits a decomposition as a hexagon, which clearly exists, and a type~\ref{item:A3} annulus, which we just proved exists. Again, there is a unique 1-parameter family of real pairs of cuts leading to this broken curve. At its other edge we see the desired fourth-type annulus glued to a rectangle. This proves the existence of the type \ref{item:A4} annuli.
\end{proof}

  \begin{figure}
    \centering
    \includegraphics[alt={Boundary-free annulus}]{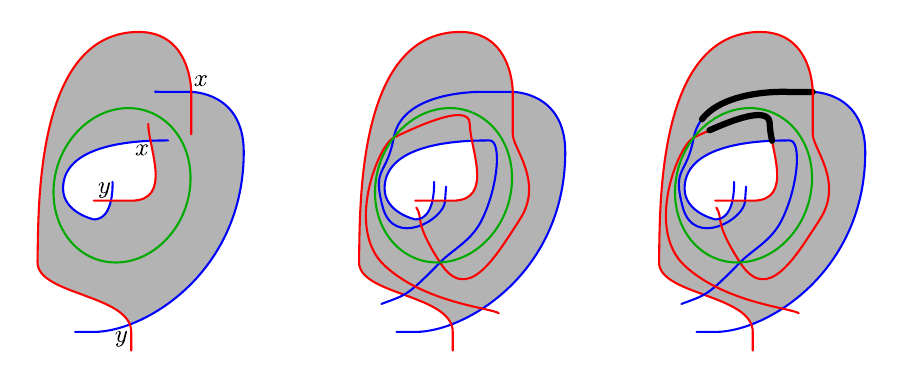}
    \caption{\textbf{Boundary-free annulus}. Left: A boundary-free annulus; we have drawn this less symmetrically than in Figure~\ref{fig:real-nice-curves}, to make it clearer that the symmetry is reflection across the green circle, not some other symmetry. Center: completing the curves locally in a way compatible with a nice diagram. Other ways of completing the diagram come from applying Dehn twists to $\alpha$- and $\beta$-curves, along circles parallel to the fixed set. Right: an equivariant pair of cuts which nearly intersect. For this pair, the inner boundary is mostly $\alpha$-boundary, while the outer boundary is mostly $\beta$-boundary.}
    \label{fig:A0}
  \end{figure}

\begin{lemma}\label{lem:A0-Afree-contribute}
  In a real-nice diagram, for any choice of generic symmetric almost complex structure or generic perturbation of the $\alpha$- and $\beta$-curves,
  each of the domains in Proposition~\ref{prop:classify-domains} of types~(\ref{item:A0}) and~(\ref{item:A-free}) has an
  odd number of real holomorphic representatives.
\end{lemma}
\begin{proof}
  First, consider boundary-free annuli (case~(\ref{item:A0})). Unlike the annuli considered in Lemma~\ref{lem:nice-unbranched-domains-contribute}, there are two possible cuts: along the $\alpha$-curve from one boundary component and the $\beta$-curve (to which it is reflected) from the other. It is easy to see that, because the diagram is nice, an $\alpha$-cut from one boundary intersects a $\beta$-cut from the other (at a point on the fixed set) before exiting the domain; see Figure~\ref{fig:A0}. As the cuts approach this point, the angle subtended by the $\alpha$-part of one boundary component approaches $2\pi$, while the $\beta$-part of the other boundary component approaches $2\pi$. Making the other pair of cuts instead, the situation is reversed. Thus, by a familiar argument (see, e.g.,~\cite[Lemma 9.3]{OS04:HolomorphicDisks}), the domain has an odd number of holomorphic representatives.

  \begin{figure}
    \centering
    \includegraphics[alt={Free annuli}]{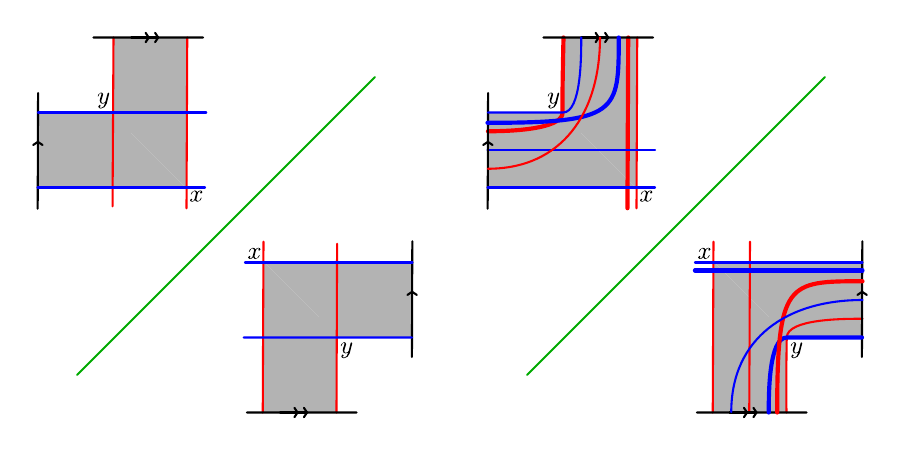}
    \caption{\textbf{Free annuli.} Left: a domain for a free annulus where the cuts exit without crossing. Right: a domain where the cuts cross before exiting. On the right, we have drawn one pair of maximal cuts thicker, to make it clear that they cross before exiting the domain.}
    \label{fig:free-annuli}
  \end{figure}

  The case of free annuli (case~(\ref{item:A-free})) is similar. Now, for each of the two $\tau$-invariant pairs of cuts, either they cross before they exit the domain, or they exit through the opposite boundaries without crossing. (See Figure~\ref{fig:free-annuli}.) Order the two boundary components of the annulus arbitrarily. In either case, one $\tau$-invariant pair limits to a curve where the first boundary component is entirely $\alpha$-boundary and the second is entirely $\beta$-boundary. The second $\tau$-invariant pair limits to a curve where the first boundary component is entirely $\beta$-boundary and the second is entirely $\alpha$-boundary. In either case, there are an odd number of lengths in between where the $\alpha$ portion of the first boundary subtends the same arc as the $\alpha$ portion of the second boundary.
\end{proof}

\begin{lemma}\label{lem:torus-contribute}
  In a real-nice diagram, for any choice of generic symmetric almost complex structure or generic perturbation of the $\alpha$- and $\beta$-curves,
  each domain $D$ in Proposition~\ref{prop:classify-domains} of type~(\ref{item:tori}) has an
  odd number of real holomorphic representatives.
\end{lemma}
\begin{proof}
  We will adopt a similar strategy to the types~\ref{item:A3} and~\ref{item:A4} annuli in the proof of Lemma~\ref{lem:nice-unbranched-domains-contribute}, after some preparatory work.

  As in the other cases, since there are no index $0$ positive domains in the diagram, the count of holomorphic representatives of $D$ is independent of the choice of almost complex structure or perturbation of the $\alpha$- and $\beta$-circles.

  Following an idea of Binns-Guth-Xiao, here is a source of such domains. Start from a closed torus, viewed as $\RR^2/\ZZ^2$; this has an orientation-reversing involution given by reflection across the line $y=x$. Let $\alphas$ be the image of two lines of slope $p/q$, and $\betas$ the image of two lines of slope $q/p$. (This is a special case of the twisted toroidal grid diagrams for lens spaces in~\cite{BGH08:lens-grid}.) These lines divide the torus into parallelograms. If there is an embedded $3\times 3$ block of parallelograms with two opposite corners on the line $y=x$, then deleting that block from the torus gives a domain of type~(\ref{item:tori}). See Figure~\ref{fig:tori-domain-types}.
  
  We claim that, in fact, all domains of type~(\ref{item:tori}) are of this form. First, since the boundary is connected, the fixed set of the involution is also connected. Next, because the domain under consideration has Euler measure $0$, it is tiled by rectangles. Label the four arcs in the boundary $a_1,b_1,a_2,b_2$, so that the $a_i$ lie on $\alpha$-circles.
  As the $\alpha$-curves $\alpha_1$, $\alpha_2$ containing $a_1$, $a_2$ continue past $b_1$ into the interior of the domain, they are parallel: there is a rectangle (made of many small rectangles) with one edge on $\alpha_1$, the opposite edge on $\alpha_2$, and an edge in between that is the segment $b_1$. Since the domain is tiled by rectangles, $\alpha_1$ and $\alpha_2$ remain parallel---cobounding a rectangle $R$---until one of them exits the domain---necessarily through the interior of $b_2$, since the rectangle already has $b_1$ on its boundary. Eventually, the other $\alpha_i$ also exits through the interior of $b_2$. Moreover, by considering how the rectangle $R$ intersects $b_2$, it is clear that the four points in $(\alpha_1\cup\alpha_2)\cap b_2$ alternate between points on $\alpha_1$ and on $\alpha_2$. See Figure~\ref{fig:extend-to-grid}.
  
  \begin{figure}
    \centering
    \begin{tabular}{cccc}
      \def\p{1}
      \def\q{3}
      \begin{tikzpicture}[scale=3, alt={Toroidal domain with slope 1/3}]
  \draw[white] (-.05,-.05) rectangle (1.05,1.05);
  \draw[thick] (0,0) rectangle (1,1);
  \begin{scope}
  \clip (0,0) rectangle (1,1);
  %
  \pgfmathsetmacro{\lastline}{\p+\q}
  \draw[darkgreen] (0,0) -- (1,1);
  \foreach \i in {0,...,\lastline} {
    \pgfmathsetmacro{\x}{-\q/\p + \i/\p}
    \pgfmathsetmacro{\z}{-\p/\q + \i/\q}
    \draw[red] (\x, 0) -- (\x + \q, \p);
    \draw[blue] (\z, 0) -- (\z + \p, \q);
    \draw[orange] (\x+1/2/\p, 0) -- (\x + \q+1/2/\p, \p);
    \draw[cyan] (\z+1/2/\q, 0) -- (\z + \p+1/2/\q, \q);
  }
  \fill[white] (0.306,0.306) -- (0.466,0.784) -- (0.944,0.944) -- (0.784,0.466) -- cycle;
  \end{scope}
  \draw[fill=white] (1/4,1/4) circle (0.025) node[below left, font=\scriptsize, xshift=2pt, yshift=2pt] {$y$};
  \draw[fill=white] (1,1) circle (0.025) node[above right, font=\scriptsize, xshift=-2pt, yshift=-2pt] {$y$};
  \fill (7/16,13/16) circle (0.025) node[above left, font=\scriptsize, xshift=2pt, yshift=-2pt] {$x$};
  \fill (13/16,7/16) circle (0.025) node[below right, font=\scriptsize, xshift=-2pt, yshift=2pt] {$x$};
  \draw[->] (0,0.46) -- (0,0.54);
  \draw[->] (1,0.46) -- (1,0.54);
  \draw[->] (0.40,0) -- (0.47,0);
  \draw[->] (0.53,0) -- (0.60,0);
  \draw[->] (0.40,1) -- (0.47,1);
  \draw[->] (0.53,1) -- (0.60,1);
      \end{tikzpicture}
      &
      \def\p{2}
      \def\q{5}
  \begin{tikzpicture}[scale=3, alt={Toroidal domain with slope 2/5}]
      \draw[white] (-.05,-.05) rectangle (1.05,1.05);
  \draw[thick] (0,0) rectangle (1,1);
  \begin{scope}
  \clip (0,0) rectangle (1,1);
  %
  \pgfmathsetmacro{\lastline}{\p+\q}
  \draw[darkgreen] (0,0) -- (1,1);
  \foreach \i in {0,...,\lastline} {
    \pgfmathsetmacro{\x}{-\q/\p + \i/\p}
    \pgfmathsetmacro{\z}{-\p/\q + \i/\q}
    \draw[red] (\x, 0) -- (\x + \q, \p);
    \draw[blue] (\z, 0) -- (\z + \p, \q);
    \draw[orange] (\x+1/2/\p, 0) -- (\x + \q+1/2/\p, \p);
    \draw[cyan] (\z+1/2/\q, 0) -- (\z + \p+1/2/\q, \q);
  }
  \fill[white] (0.371,0.371) -- (0.492,0.674) -- (0.796,0.796) -- (0.674,0.492) -- cycle;
  \end{scope}
  \draw[fill=white] (1/3,1/3) circle (0.025) node[below left, font=\scriptsize, xshift=2pt, yshift=2pt] {$y$};
  \draw[fill=white] (5/6,5/6) circle (0.025) node[above right, font=\scriptsize, xshift=-2pt, yshift=-2pt] {$y$};
  \fill (10/21,29/42) circle (0.025) node[above left, font=\scriptsize, xshift=2pt, yshift=-2pt] {$x$};
  \fill (29/42,10/21) circle (0.025) node[below right, font=\scriptsize, xshift=-2pt, yshift=2pt] {$x$};
  \draw[->] (0,0.46) -- (0,0.54);
  \draw[->] (1,0.46) -- (1,0.54);
  \draw[->] (0.40,0) -- (0.47,0);
  \draw[->] (0.53,0) -- (0.60,0);
  \draw[->] (0.40,1) -- (0.47,1);
  \draw[->] (0.53,1) -- (0.60,1);
      \end{tikzpicture}
      &
      \def\p{1}
      \def\q{3}
  \begin{tikzpicture}[scale=3, alt={Toroidal domain with slope -1/3 (reflected)}]
      \draw[white] (-.05,-.05) rectangle (1.05,1.05);
  \draw[thick] (0,0) rectangle (1,1);
  \begin{scope}
  \clip (0,0) rectangle (1,1);
  %
  \pgfmathsetmacro{\lastline}{\p+\q}
  \draw[darkgreen] (0,0) -- (1,1);
  \foreach \i in {0,...,\lastline} {
    \pgfmathsetmacro{\x}{-\q/\p + \i/\p}
    \pgfmathsetmacro{\z}{-\p/\q + \i/\q}
    \draw[red] ({1-\x}, 0) -- ({1-\x-\q}, \p);
    \draw[blue] ({1-\z}, 0) -- ({1-\z-\p}, \q);
    \draw[orange] ({1-\x-1/2/\p}, 0) -- ({1-\x-\q-1/2/\p}, \p);
    \draw[cyan] ({1-\z-1/2/\q}, 0) -- ({1-\z-\p-1/2/\q}, \q);
  }
  \fill[white] (0.2875,0.2875) -- (0.925-3/16,0.075+1/16) -- (0.7125-1/8,0.7125-1/8) -- (0.075+1/16,0.925-3/16) -- cycle;
  \end{scope}
  \draw[fill=white] (1/4,1/4) circle (0.025) node[below left, font=\scriptsize, xshift=2pt, yshift=2pt] {$y$};
  \draw[fill=white] (3/4-1/8,3/4-1/8) circle (0.025) node[above right, font=\scriptsize, xshift=-2pt, yshift=-2pt] {$y$};
  \fill (13/16,1/16) circle (0.025) node[below right, font=\scriptsize, xshift=-2pt, yshift=2pt] {$x$};
  \fill (1/16,13/16) circle (0.025) node[above left, font=\scriptsize, xshift=2pt, yshift=-2pt] {$x$};
  \draw[->] (0,0.46) -- (0,0.54);
  \draw[->] (1,0.46) -- (1,0.54);
  \draw[->] (0.40,0) -- (0.47,0);
  \draw[->] (0.53,0) -- (0.60,0);
  \draw[->] (0.40,1) -- (0.47,1);
  \draw[->] (0.53,1) -- (0.60,1);
      \end{tikzpicture}
      &
      \def\p{2}
      \def\q{5}
  \begin{tikzpicture}[scale=3, alt={Toroidal domain with slope -2/5 (reflected)}]
  \draw[white] (-.05,-.05) rectangle (1.05,1.05);
  \draw[thick] (0,0) rectangle (1,1);
  \begin{scope}
  \clip (0,0) rectangle (1,1);
  %
  \pgfmathsetmacro{\lastline}{\p+\q}
  \draw[darkgreen] (0,0) -- (1,1);
  \foreach \i in {0,...,\lastline} {
    \pgfmathsetmacro{\x}{-\q/\p + \i/\p}
    \pgfmathsetmacro{\z}{-\p/\q + \i/\q}
    \draw[red] ({1-\x}, 0) -- ({1-\x-\q}, \p);
    \draw[blue] ({1-\z}, 0) -- ({1-\z-\p}, \q);
    \draw[orange] ({1-\x-1/2/\p}, 0) -- ({1-\x-\q-1/2/\p}, \p);
    \draw[cyan] ({1-\z-1/2/\q}, 0) -- ({1-\z-\p-1/2/\q}, \q);
  }
  \fill[white] (0.373,0.373) -- (0.677,0.252) -- (0.555,0.555) -- (0.252,0.677) -- cycle;
  \end{scope}
  \draw[fill=white] (5/14,5/14) circle (0.025) node[below left, font=\scriptsize, xshift=2pt, yshift=2pt] {$y$};
  \draw[fill=white] (4/7,4/7) circle (0.025) node[above right, font=\scriptsize, xshift=-2pt, yshift=-2pt] {$y$};
  \fill (5/7,3/14) circle (0.025) node[below right, font=\scriptsize, xshift=-2pt, yshift=2pt] {$x$};
  \fill (3/14,5/7) circle (0.025) node[above left, font=\scriptsize, xshift=2pt, yshift=-2pt] {$x$};
  \draw[->] (0,0.46) -- (0,0.54);
  \draw[->] (1,0.46) -- (1,0.54);
  \draw[->] (0.40,0) -- (0.47,0);
  \draw[->] (0.53,0) -- (0.60,0);
  \draw[->] (0.40,1) -- (0.47,1);
  \draw[->] (0.53,1) -- (0.60,1);
      \end{tikzpicture}
      \\[1em]
      \def\p{1}
      \def\q{3}
  \begin{tikzpicture}[scale=3, alt={Slope 1/3 domain extended to twisted grid diagram}]
      \draw[white] (-.05,-.05) rectangle (1.05,1.05);
  \draw[thick] (0,0) rectangle (1,1);
  \begin{scope}
  \clip (0,0) rectangle (1,1);
  %
  \pgfmathsetmacro{\lastline}{\p+\q}
  \draw[darkgreen] (0,0) -- (1,1);
  \foreach \i in {0,...,\lastline} {
    \pgfmathsetmacro{\x}{-\q/\p + \i/\p}
    \pgfmathsetmacro{\z}{-\p/\q + \i/\q}
    \draw[red] (\x, 0) -- (\x + \q, \p);
    \draw[blue] (\z, 0) -- (\z + \p, \q);
    \draw[orange] (\x+1/2/\p, 0) -- (\x + \q+1/2/\p, \p);
    \draw[cyan] (\z+1/2/\q, 0) -- (\z + \p+1/2/\q, \q);
  }
  \fill[pattern=dots, pattern color=black!50] (1/4,1/4) -- (7/16,13/16) -- (1,1) -- (13/16,7/16) -- cycle;
  \end{scope}
  \draw[fill=white] (1/4,1/4) circle (0.025);
  \draw[fill=white] (1,1) circle (0.025);
  \fill (7/16,13/16) circle (0.025);
  \fill (13/16,7/16) circle (0.025);
  \draw[->] (0,0.46) -- (0,0.54);
  \draw[->] (1,0.46) -- (1,0.54);
  \draw[->] (0.40,0) -- (0.47,0);
  \draw[->] (0.53,0) -- (0.60,0);
  \draw[->] (0.40,1) -- (0.47,1);
  \draw[->] (0.53,1) -- (0.60,1);
      \end{tikzpicture}
      &
      \def\p{2}
      \def\q{5}
  \begin{tikzpicture}[scale=3, alt={Slope 2/5 domain extended to twisted grid diagram}]
      \draw[white] (-.05,-.05) rectangle (1.05,1.05);
  \draw[thick] (0,0) rectangle (1,1);
  \begin{scope}
  \clip (0,0) rectangle (1,1);
  %
  \pgfmathsetmacro{\lastline}{\p+\q}
  \draw[darkgreen] (0,0) -- (1,1);
  \foreach \i in {0,...,\lastline} {
    \pgfmathsetmacro{\x}{-\q/\p + \i/\p}
    \pgfmathsetmacro{\z}{-\p/\q + \i/\q}
    \draw[red] (\x, 0) -- (\x + \q, \p);
    \draw[blue] (\z, 0) -- (\z + \p, \q);
    \draw[orange] (\x+1/2/\p, 0) -- (\x + \q+1/2/\p, \p);
    \draw[cyan] (\z+1/2/\q, 0) -- (\z + \p+1/2/\q, \q);
  }
  \fill[pattern=dots, pattern color=black!50] (1/3,1/3) -- (10/21,29/42) -- (5/6,5/6) -- (29/42,10/21) -- cycle;
  \end{scope}
  \draw[fill=white] (1/3,1/3) circle (0.025);
  \draw[fill=white] (5/6,5/6) circle (0.025);
  \fill (10/21,29/42) circle (0.025);
  \fill (29/42,10/21) circle (0.025);
  \draw[->] (0,0.46) -- (0,0.54);
  \draw[->] (1,0.46) -- (1,0.54);
  \draw[->] (0.40,0) -- (0.47,0);
  \draw[->] (0.53,0) -- (0.60,0);
  \draw[->] (0.40,1) -- (0.47,1);
  \draw[->] (0.53,1) -- (0.60,1);
      \end{tikzpicture}
      &
      \def\p{1}
      \def\q{3}
  \begin{tikzpicture}[scale=3, alt={Slope -1/3 domain extended to twisted grid diagram}]
      \draw[white] (-.05,-.05) rectangle (1.05,1.05);
  \draw[thick] (0,0) rectangle (1,1);
  \begin{scope}
  \clip (0,0) rectangle (1,1);
  %
  \pgfmathsetmacro{\lastline}{\p+\q}
  \draw[darkgreen] (0,0) -- (1,1);
  \foreach \i in {0,...,\lastline} {
    \pgfmathsetmacro{\x}{-\q/\p + \i/\p}
    \pgfmathsetmacro{\z}{-\p/\q + \i/\q}
    \draw[red] ({1-\x}, 0) -- ({1-\x-\q}, \p);
    \draw[blue] ({1-\z}, 0) -- ({1-\z-\p}, \q);
    \draw[orange] ({1-\x-1/2/\p}, 0) -- ({1-\x-\q-1/2/\p}, \p);
    \draw[cyan] ({1-\z-1/2/\q}, 0) -- ({1-\z-\p-1/2/\q}, \q);
  }
  \fill[pattern=dots, pattern color=black!50] (1/4,1/4) -- (13/16,1/16) -- (5/8,5/8) -- (1/16,13/16) -- cycle;
  \end{scope}
  \draw[fill=white] (1/4,1/4) circle (0.025);
  \draw[fill=white] (5/8,5/8) circle (0.025);
  \fill (13/16,1/16) circle (0.025);
  \fill (1/16,13/16) circle (0.025);
  \draw[->] (0,0.46) -- (0,0.54);
  \draw[->] (1,0.46) -- (1,0.54);
  \draw[->] (0.40,0) -- (0.47,0);
  \draw[->] (0.53,0) -- (0.60,0);
  \draw[->] (0.40,1) -- (0.47,1);
  \draw[->] (0.53,1) -- (0.60,1);
      \end{tikzpicture}
      &
      \def\p{2}
      \def\q{5}
  \begin{tikzpicture}[scale=3, alt={Slope -2/5 domain extended to twisted grid diagram}]
      \draw[white] (-.05,-.05) rectangle (1.05,1.05);
  \draw[thick] (0,0) rectangle (1,1);
  \begin{scope}
  \clip (0,0) rectangle (1,1);
  %
  \pgfmathsetmacro{\lastline}{\p+\q}
  \draw[darkgreen] (0,0) -- (1,1);
  \foreach \i in {0,...,\lastline} {
    \pgfmathsetmacro{\x}{-\q/\p + \i/\p}
    \pgfmathsetmacro{\z}{-\p/\q + \i/\q}
    \draw[red] ({1-\x}, 0) -- ({1-\x-\q}, \p);
    \draw[blue] ({1-\z}, 0) -- ({1-\z-\p}, \q);
    \draw[orange] ({1-\x-1/2/\p}, 0) -- ({1-\x-\q-1/2/\p}, \p);
    \draw[cyan] ({1-\z-1/2/\q}, 0) -- ({1-\z-\p-1/2/\q}, \q);
  }
  \fill[pattern=dots, pattern color=black!50] (5/14,5/14) -- (5/7,3/14) -- (4/7,4/7) -- (3/14,5/7) -- cycle;
  \end{scope}
  \draw[fill=white] (5/14,5/14) circle (0.025);
  \draw[fill=white] (4/7,4/7) circle (0.025);
  \fill (5/7,3/14) circle (0.025);
  \fill (3/14,5/7) circle (0.025);
  \draw[->] (0,0.46) -- (0,0.54);
  \draw[->] (1,0.46) -- (1,0.54);
  \draw[->] (0.40,0) -- (0.47,0);
  \draw[->] (0.53,0) -- (0.60,0);
  \draw[->] (0.40,1) -- (0.47,1);
  \draw[->] (0.53,1) -- (0.60,1);
      \end{tikzpicture}
      \\[1em]
      \def\p{1}
      \def\q{3}
  \begin{tikzpicture}[scale=3, alt={Equivariant decomposition of slope 1/3 domain}]
      \draw[white] (-.05,-.05) rectangle (1.05,1.05);
  \draw[thick] (0,0) rectangle (1,1);
  \begin{scope}
  \clip (0,0) rectangle (1,1);
  \fill[pattern=dots, pattern color=black!60]
    (0,0) -- (0,1/2) -- (9/16,11/16) -- (2/3,1) -- (1,1) -- (1,2/3) -- (11/16,9/16) -- (1/2,0) -- cycle;
  \fill[pattern=dots, pattern color=black!60]
    (1/2,1) -- (7/16,13/16) -- (0,2/3) -- (0,1) -- cycle;
  \fill[pattern=dots, pattern color=black!60]
    (2/3,0) -- (13/16,7/16) -- (1,1/2) -- (1,0) -- cycle;
  \fill[pattern=north west lines, pattern color=black!60]
    (0,2/3) -- (7/16,13/16) -- (1/2,1) -- (2/3,1) -- (9/16,11/16) -- (0,1/2) -- cycle;
  \fill[pattern=north west lines, pattern color=black!60]
    (13/16,7/16) -- (2/3,0) -- (1/2,0) -- (11/16,9/16) -- (1,2/3) -- (1,1/2) -- cycle;
  \pgfmathsetmacro{\lastline}{\p+\q}
  \draw[darkgreen] (0,0) -- (1,1);
  \foreach \i in {0,...,\lastline} {
    \pgfmathsetmacro{\x}{-\q/\p + \i/\p}
    \pgfmathsetmacro{\z}{-\p/\q + \i/\q}
    \draw[red] (\x, 0) -- (\x + \q, \p);
    \draw[blue] (\z, 0) -- (\z + \p, \q);
    \draw[orange] (\x+1/2/\p, 0) -- (\x + \q+1/2/\p, \p);
    \draw[cyan] (\z+1/2/\q, 0) -- (\z + \p+1/2/\q, \q);
  }
  \end{scope}
  \draw[red, line width=2pt] (7/16,13/16) -- (0,2/3);
  \draw[red, line width=2pt] (1,2/3) -- (11/16,9/16);
  \draw[cyan, line width=2pt] (7/16,13/16) -- (1/2,1);
  \draw[cyan, line width=2pt] (1/2,0) -- (11/16,9/16);
  \draw[blue, line width=2pt] (13/16,7/16) -- (2/3,0);
  \draw[blue, line width=2pt] (2/3,1) -- (9/16,11/16);
  \draw[orange, line width=2pt] (13/16,7/16) -- (1,1/2);
  \draw[orange, line width=2pt] (0,1/2) -- (9/16,11/16);
  \filldraw[fill=white, draw=black] (9/16-0.025,11/16-0.025) rectangle (9/16+0.025,11/16+0.025);
  \filldraw[fill=white, draw=black] (11/16-0.025,9/16-0.025) rectangle (11/16+0.025,9/16+0.025);
  \fill (7/16,13/16) circle (0.025);
  \fill (13/16,7/16) circle (0.025);
  \draw[->] (0,0.46) -- (0,0.54);
  \draw[->] (1,0.46) -- (1,0.54);
  \draw[->] (0.40,0) -- (0.47,0);
  \draw[->] (0.53,0) -- (0.60,0);
  \draw[->] (0.40,1) -- (0.47,1);
  \draw[->] (0.53,1) -- (0.60,1);
      \end{tikzpicture}
      &
      \def\p{2}
      \def\q{5}
  \begin{tikzpicture}[scale=3, alt={Equivariant decomposition of slope 2/5 domain}]
      \draw[white] (-.05,-.05) rectangle (1.05,1.05);
  \draw[thick] (0,0) rectangle (1,1);
  \begin{scope}
  \clip (0,0) rectangle (1,1);
  \fill[pattern=dots, pattern color=black!50]
    (0,0) -- (0,1/2) -- (10/21,29/42) -- (3/5,1) -- (1,1) -- (1,3/5) -- (29/42,10/21) -- (1/2,0) -- cycle;
  \fill[pattern=dots, pattern color=black!50]
    (1/2,1) -- (17/42,16/21) -- (0,3/5) -- (0,1) -- cycle;
  \fill[pattern=dots, pattern color=black!50]
    (1,0) -- (3/5,0) -- (16/21,17/42) -- (1,1/2) -- cycle;
  \fill[pattern=north west lines, pattern color=black!60]
    (0,3/5) -- (17/42,16/21) -- (1/2,1) -- (3/5,1) -- (10/21,29/42) -- (0,1/2) -- cycle;
  \fill[pattern=north west lines, pattern color=black!60]
    (1,1/2) -- (16/21,17/42) -- (3/5,0) -- (1/2,0) -- (29/42,10/21) -- (1,3/5) -- cycle;
  \pgfmathsetmacro{\lastline}{\p+\q}
  \draw[darkgreen] (0,0) -- (1,1);
  \foreach \i in {0,...,\lastline} {
    \pgfmathsetmacro{\x}{-\q/\p + \i/\p}
    \pgfmathsetmacro{\z}{-\p/\q + \i/\q}
    \draw[red] (\x, 0) -- (\x + \q, \p);
    \draw[blue] (\z, 0) -- (\z + \p, \q);
    \draw[orange] (\x+1/2/\p, 0) -- (\x + \q+1/2/\p, \p);
    \draw[cyan] (\z+1/2/\q, 0) -- (\z + \p+1/2/\q, \q);
  }
  \end{scope}
  \draw[orange, line width=2pt] (10/21,29/42) -- (0,1/2);
  \draw[orange, line width=2pt] (1,1/2) -- (16/21,17/42);
  \draw[blue, line width=2pt] (10/21,29/42) -- (3/5,1);
  \draw[blue, line width=2pt] (3/5,0) -- (16/21,17/42);
  \draw[red, line width=2pt] (29/42,10/21) -- (1,3/5);
  \draw[red, line width=2pt] (0,3/5) -- (17/42,16/21);
  \draw[cyan, line width=2pt] (29/42,10/21) -- (1/2,0);
  \draw[cyan, line width=2pt] (1/2,1) -- (17/42,16/21);
  \filldraw[fill=white, draw=black] (16/21-0.025,17/42-0.025) rectangle (16/21+0.025,17/42+0.025);
  \filldraw[fill=white, draw=black] (17/42-0.025,16/21-0.025) rectangle (17/42+0.025,16/21+0.025);
  \fill (10/21,29/42) circle (0.025);
  \fill (29/42,10/21) circle (0.025);
  \draw[->] (0,0.46) -- (0,0.54);
  \draw[->] (1,0.46) -- (1,0.54);
  \draw[->] (0.40,0) -- (0.47,0);
  \draw[->] (0.53,0) -- (0.60,0);
  \draw[->] (0.40,1) -- (0.47,1);
  \draw[->] (0.53,1) -- (0.60,1);
      \end{tikzpicture}
      &
      \def\p{1}
      \def\q{3}
  \begin{tikzpicture}[scale=3, alt={Equivariant decomposition of slope -1/3 domain}]
      \draw[white] (-.05,-.05) rectangle (1.05,1.05);
  \draw[thick] (0,0) rectangle (1,1);
  \begin{scope}
  \clip (0,0) rectangle (1,1);
  \fill[pattern=dots, pattern color=black!50]
    (0,0)--(1/3,0)--(1/4,1/4)--(0,1/3)--cycle;
  \fill[pattern=dots, pattern color=black!50]
    (1,0)--(1,1/3)--(1/2,1/2)--(1/3,1)--(0,1)--(0,1/2)--(3/8,3/8)--(1/2,0)--cycle;
  \fill[pattern=dots, pattern color=black!50]
    (1,1/2)--(1,1)--(1/2,1)--(5/8,5/8)--cycle;
  \fill[pattern=north east lines, pattern color=black!60]
    (1/3,0)--(1/2,0)--(3/8,3/8)--(0,1/2)--(0,1/3)--(1/4,1/4)--cycle;
  \fill[pattern=north east lines, pattern color=black!60]
    (1,1/3)--(1,1/2)--(5/8,5/8)--(1/2,1)--(1/3,1)--(1/2,1/2)--cycle;
  \pgfmathsetmacro{\lastline}{\p+\q}
  \draw[darkgreen] (0,0) -- (1,1);
  \foreach \i in {0,...,\lastline} {
    \pgfmathsetmacro{\x}{-\q/\p + \i/\p}
    \pgfmathsetmacro{\z}{-\p/\q + \i/\q}
    \draw[red] ({1-\x}, 0) -- ({1-\x-\q}, \p);
    \draw[blue] ({1-\z}, 0) -- ({1-\z-\p}, \q);
    \draw[orange] ({1-\x-1/2/\p}, 0) -- ({1-\x-\q-1/2/\p}, \p);
    \draw[cyan] ({1-\z-1/2/\q}, 0) -- ({1-\z-\p-1/2/\q}, \q);
  }
  \end{scope}
  \draw[red, line width=2pt] (1/4,1/4) -- (0,1/3);
  \draw[red, line width=2pt] (1,1/3) -- (1/2,1/2);
  \draw[blue, line width=2pt] (1/4,1/4) -- (1/3,0);
  \draw[blue, line width=2pt] (1/3,1) -- (1/2,1/2);
  \draw[cyan, line width=2pt] (5/8,5/8) -- (1/2,1);
  \draw[cyan, line width=2pt] (1/2,0) -- (3/8,3/8);
  \draw[orange, line width=2pt] (5/8,5/8) -- (1,1/2);
  \draw[orange, line width=2pt] (0,1/2) -- (3/8,3/8);
  \fill (1/2-0.025,1/2-0.025) rectangle (1/2+0.025,1/2+0.025);
  \fill (3/8-0.025,3/8-0.025) rectangle (3/8+0.025,3/8+0.025);
  \draw[fill=white] (1/4,1/4) circle (0.025);
  \draw[fill=white] (5/8,5/8) circle (0.025);
  \draw[->] (0,0.46) -- (0,0.54);
  \draw[->] (1,0.46) -- (1,0.54);
  \draw[->] (0.40,0) -- (0.47,0);
  \draw[->] (0.53,0) -- (0.60,0);
  \draw[->] (0.40,1) -- (0.47,1);
  \draw[->] (0.53,1) -- (0.60,1);
      \end{tikzpicture}
      &
      \def\p{2}
      \def\q{5}
  \begin{tikzpicture}[scale=3, alt={Equivariant decomposition of slope -2/5 domain}]
      \draw[white] (-.05,-.05) rectangle (1.05,1.05);
  \draw[thick] (0,0) rectangle (1,1);
  \begin{scope}
  \clip (0,0) rectangle (1,1);
  \fill[pattern=dots, pattern color=black!50]
    (0,0)--(2/5,0)--(2/7,2/7)--(0,2/5)--cycle;
  \fill[pattern=dots, pattern color=black!50]
    (1,0)--(1,2/5)--(4/7,4/7)--(2/5,1)--(0,1)--(0,1/2)--(5/14,5/14)--(1/2,0)--cycle;
  \fill[pattern=dots, pattern color=black!50]
    (1,1/2)--(1,1)--(1/2,1)--(9/14,9/14)--cycle;
  \fill[pattern=north east lines, pattern color=black!60]
    (2/5,0)--(1/2,0)--(5/14,5/14)--(0,1/2)--(0,2/5)--(2/7,2/7)--cycle;
  \fill[pattern=north east lines, pattern color=black!60]
    (1,2/5)--(1,1/2)--(9/14,9/14)--(1/2,1)--(2/5,1)--(4/7,4/7)--cycle;
  \pgfmathsetmacro{\lastline}{\p+\q}
  \draw[darkgreen] (0,0) -- (1,1);
  \foreach \i in {0,...,\lastline} {
    \pgfmathsetmacro{\x}{-\q/\p + \i/\p}
    \pgfmathsetmacro{\z}{-\p/\q + \i/\q}
    \draw[red] ({1-\x}, 0) -- ({1-\x-\q}, \p);
    \draw[blue] ({1-\z}, 0) -- ({1-\z-\p}, \q);
    \draw[orange] ({1-\x-1/2/\p}, 0) -- ({1-\x-\q-1/2/\p}, \p);
    \draw[cyan] ({1-\z-1/2/\q}, 0) -- ({1-\z-\p-1/2/\q}, \q);
  }
  \end{scope}
  \draw[orange, line width=2pt] (5/14,5/14) -- (0,1/2);
  \draw[orange, line width=2pt] (1,1/2) -- (9/14,9/14);
  \draw[cyan, line width=2pt] (5/14,5/14) -- (1/2,0);
  \draw[cyan, line width=2pt] (1/2,1) -- (9/14,9/14);
  \draw[blue, line width=2pt] (4/7,4/7) -- (2/5,1);
  \draw[blue, line width=2pt] (2/5,0) -- (2/7,2/7);
  \draw[red, line width=2pt] (4/7,4/7) -- (1,2/5);
  \draw[red, line width=2pt] (0,2/5) -- (2/7,2/7);
  \fill (9/14-0.025,9/14-0.025) rectangle (9/14+0.025,9/14+0.025);
  \fill (2/7-0.025,2/7-0.025) rectangle (2/7+0.025,2/7+0.025);
  \draw[fill=white] (5/14,5/14) circle (0.025);
  \draw[fill=white] (4/7,4/7) circle (0.025);
  \draw[->] (0,0.46) -- (0,0.54);
  \draw[->] (1,0.46) -- (1,0.54);
  \draw[->] (0.40,0) -- (0.47,0);
  \draw[->] (0.53,0) -- (0.60,0);
  \draw[->] (0.40,1) -- (0.47,1);
  \draw[->] (0.53,1) -- (0.60,1);
      \end{tikzpicture}
    \end{tabular}
    \caption{\textbf{Cutting up toroidal domains.} Top: toroidal domains corresponding to four different slopes $p/q$. Center: extending these domains to twisted grid diagrams. Bottom: the unique alternate equivariant decompositions of these domains.}
    \label{fig:tori-domain-types}
  \end{figure}
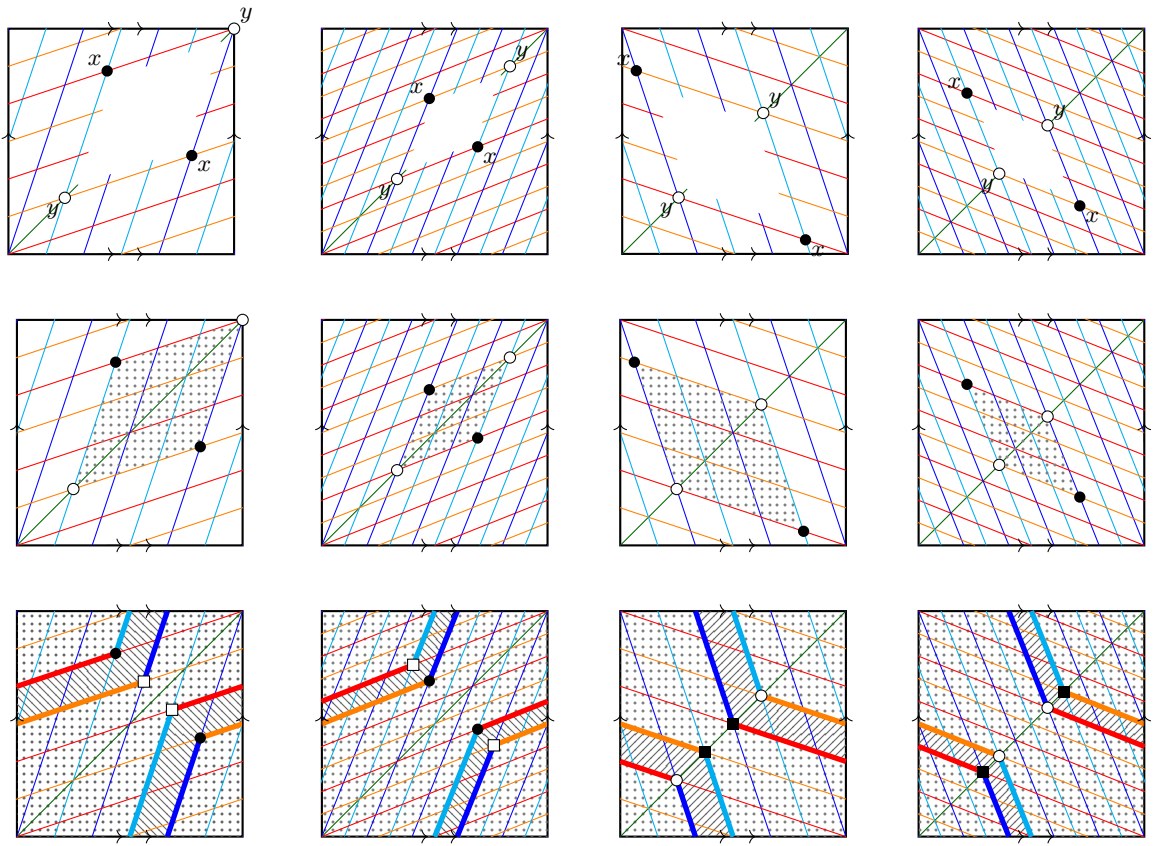

  By the same argument, there are four points on $(\alpha_1\cup\alpha_2)\cap b_2$, alternating between points on $\alpha_1$ and on $\alpha_2$; four points on $(\beta_1\cup\beta_2)\cap a_1$, alternating between points on $\beta_1$ and on $\beta_2$; and four points on $(\beta_1\cup\beta_2)\cap a_1$, alternating between points on $\beta_1$ and on $\beta_2$. It follows that we can glue a rectangle to the boundary of the domain and extend the $\alpha$- and $\beta$-circles to a twisted toroidal grid diagram $G$ for a lens space (as in~\cite{BGH08:lens-grid}).

  \begin{figure}
    \centering
    \includegraphics[alt={Extending curves in toroidal domains}]{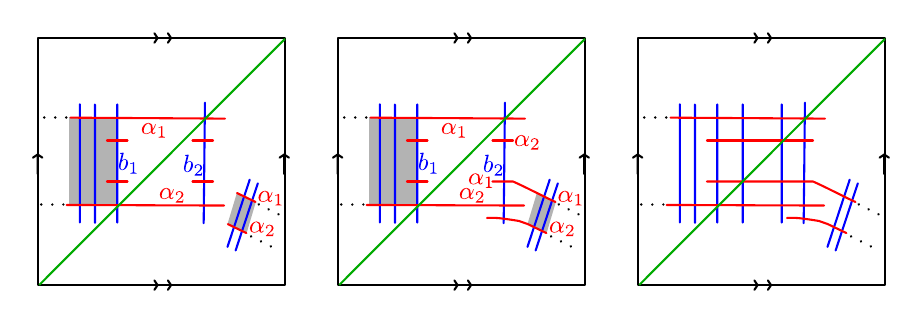}
    \caption{\textbf{Extending curves in toroidal domains.} Left: the curves $\alpha_1$ and $\alpha_2$ are parallel until one exits the domain, through the interior of $b_2$. Center: the existence of a rectangle $R$ (shaded) determines the order of the points where $\alpha_1$ and $\alpha_2$ exit through $b_2$: the points on $b_2$ alternate between $\alpha_1$ and $\alpha_2$. Right: extending the curves beyond the domain to obtain a grid diagram for a lens space.}
    \label{fig:extend-to-grid}
  \end{figure}

  The fixed set of the involution intersects the boundary of the domain at two opposite corners, and passes diagonally across small rectangles. So, we can extend the fixed set and the involution to the completed grid diagram $G$, so that the fixed set still passes diagonally across rectangles. Now, choose an identification of the torus with $\RR^2/\ZZ^2$ so that the fixed set maps to the line $y=x$. Then the $\alpha$-circles map to circles at some slope $p/q$, the $\beta$-circles map to circles at slope $q/p$, and the boundary maps to a $3\times 3$ sub-rectangle, as claimed. (See Figure~\ref{fig:tori-domain-types}.)

  View $G$ as a 2-pointed real Heegaard diagram (in the sense of~\cite[Section 3.3]{GM:real-HF}), by placing two basepoints on the fixed set, one in the middle of $D$ and the other in the middle of $G\setminus D$. We will use the fact that $\bdy^2=0$ on $\HFR^-(G)/(U_1^2,U_2^2)$ to prove that $D\subset G$ has a holomorphic representative. Let $D'=G\setminus D$ be the complement to $D$. The domain $D'\in \pi_2(\y,\x)$ is a rectangle of the first type. Observe that there is an extra symmetry $\sigma$ of the diagram, by rotation by $\pi$ around the basepoints; unlike $\tau$, the symmetry $\sigma$ is orientation-preserving and takes $\alpha$-circles to $\alpha$-circles (but exchanges the two $\alpha$-circles). The domains $D$ and $D'$ are preserved by $\sigma$. 
  
  We take two cases, depending on whether $p/q>0$ or $p/q<0$. For $p/q>0$, we consider the concatenation $D''=D*D'\in\pi_2(\x,\x)$; for $p/q<0$, we consider instead $D''=D'*D\in\pi_2(\y,\y)$. We will show that there is a unique second decomposition of $D''$ as a concatenation of index $1$ positive domains that are preserved by both $\sigma$ and $\tau$, and so that those domains are of types we have already shown have holomorphic representatives (i.e., are not tori). Perhaps there are also other decompositions of $D''$ into real index $1$ domains that are not $\sigma$-equivariant, but those decompositions come in pairs. Any two such decompositions $D_1*D_2,\sigma(D_1)*\sigma(D_2)$ contribute the same way to $\bdy^2$: if $D_1$ and $D_2$ are both not tori, then they necessarily contribute, while if they are tori we have not (yet) counted their contribution, but by independence of the count from the almost complex structure, the counts of representatives for $D_i$ and $\sigma(D_i)$ still agree. In particular, decompositions that are not $\sigma$-equivariant contribute an even number of terms to $\bdy^2$, and hence it suffices to consider only $\sigma$-equivariant decompositions. (Note that the holomorphic curves may not be $\sigma$-invariant, just the domains.)

  Consider the case that $p/q>0$. Any other $\sigma$- and $\tau$-invariant decomposition of $D''$ corresponds to making cuts of equal length in both directions from both $x$-corners (that is, four cuts along the same number of edges of the grid). These cuts are in the opposite direction from the ones decomposing $D''$ as $D*D'$. The cut starting along $\beta$ from one $x$-corner ends when it intersects the cut along $\alpha$ from the other $x$-corner, and vice versa. These cuts decompose the domain as a concatenation as a free annulus and a boundary-free annulus. See Figure~\ref{fig:tori-domain-types}.

  Finally, consider the case that $p/q<0$. In this case, the four cuts start at the $y$-corners, and again are of equal length and go in the opposite direction from the decomposition of $D''$ as $D'*D$. In this case, the cuts end when the $\beta$-cut from one $y$-corner intersects the $\alpha$-cut from the same corner. In this case, these cuts decompose the domain into a pair of type~\ref{item:A2} annuli.
  Again, see Figure~\ref{fig:tori-domain-types}.

  We have now shown that, in both cases, decompositions of $D''$ not involving $D$ contribute an odd number of terms to $\bdy^2=0$. Thus, $D$ has an odd number of holomorphic representatives, as desired.
\end{proof}

\begin{figure}
  \centering
  \includegraphics[alt={Another nice toroidal domain}]{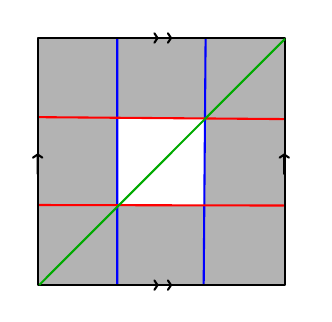}
  \caption{\textbf{Another nice toroidal domain.} Domains of this kind, which require two parallel $\alpha$-circles, appear in multi-basepointed real Heegaard Floer homology, but not in the single basepoint case.}
  \label{fig:another-torus}
\end{figure}

\begin{remark}
  For multi-basepointed real Heegaard Floer homology, there is an additional type of toroidal domain in nice diagrams, corresponding to removing a $1\times 1$ rather than $3\times 3$ sub-grid. See Figure~\ref{fig:another-torus}.
\end{remark}

\begin{remark}
  In the sutured, rather than bordered case, the results in this section have also been obtained by Binns-Guth-Xiao~\cite{BGX:real-sutured}. In particular, they pointed out that an earlier version of Proposition~\ref{prop:classify-domains} missed cases~(\ref{item:A-free}) and~(\ref{item:tori}).
\end{remark}

\subsection{Real Auroux-Zarev diagrams}\label{sec:AZ-diagrams}
Given a pointed matched circle $\PMC$, the \emph{Auroux-Zarev diagram} $\AZ(\PMC)$ is obtained from the triangle $T=\{(x,y)\in\RR^2\mid 0\leq x\leq x+y\leq 4k+1\}$ as follows. Identify the left boundary $T\cap\{x=0\}$ with $\PMC$, so the point $1\in\CircPts$ is identified with $(0,4k)$. This also identifies $T\cap\{x+y=4k+1\}$ with $\CircPts$; identify a neighborhood of $(i,4k+1-i)$ and $(M(i),4k+1-M(i))$ in $T\cap \{x+y=4k+1\}$, to obtain a surface of genus $k$ with one boundary component. Then identify neighborhoods of $(0,1/2)$ and $(0,4k+1/2)$, and of $(1/2,0)$ and $(4k+1/2,0)$ in $\bdy T$. The result is a surface with three boundary components; fill the boundary component intersecting $\{x+y=4k+1\}$ with a disk. The $\alpha$-arcs are the images of the segments $[0,4k+1-i]\times\{i\}$; each $\alpha$-arc consists of two line segments. Similarly, the $\beta$-arcs are the images of the segments $\{i\}\times[0,4k+1-i]$. The basepoint (or rather, base-arc) $\bbpt$ is an arc near $(0,0)$ connecting the two boundary components. See Figure~\ref{fig:AZ}. The diagram $\AZbar(\PMC)=-\AZ(-\PMC)$ is the result of reversing the orientation on $\AZ(-\PMC)$. See Figure~\ref{fig:AZ-bar}. We recall the bordered invariants of $\AZ(\PMC)$ and $\AZbar(\PMC)$ in Section~\ref{sec:AZ-modules-review}. (See~\cite{LOTHomPair} for more on these diagrams.) The diagrams $\AZ(\PMC)$ and $\AZbar(\PMC)$ represent $[0,1]\times F(\PMC)$, but with one side below the Heegaard surface ($\alpha$-bordered) and the other above ($\beta$-bordered). (The identifications of the boundaries on the two sides differ by a half boundary Dehn twist~\cite[Proposition 4.4]{LOTHomPair}, but because we are interested in computing invariants for closed 3-manifolds, the complications from considering strongly based mapping classes will not play a role in the present paper.)

\begin{definition}
Fix a pointed matched circle $\PMC=(Z,\CircPts,M,z)$. Let $w$ be a point
between the $2k\th$ and $(2k+1)\st$ points in $\CircPts$, and let
$\tau_\PMC\co Z\to Z$ denote reflection across $\{z,w\}$. We say that
$(\PMC,\tau_\PMC)$ is a \emph{real pointed matched circle} if $M$
respects $\tau_\PMC$, i.e., $M\circ\tau_{\PMC}=\tau_{\PMC}\circ M$. 
\end{definition}
See Figure~\ref{fig:real-PMC}. Given a real pointed matched circle $\PMC$, the involution $\tau_\PMC$ induces an isomorphism $\tau_{\PMC,*}\co \Alg(\PMC)\stackrel{\cong}{\lra} \Alg(\PMC)^\op$. 

For a real pointed matched circle $(\PMC,\tau_\PMC)$,
the diagrams $\AZ(\PMC)$ and $\AZbar(\PMC)$ are real bordered Heegaard diagrams, with the involution $\tau\co T\to T$, $\tau(x,y)=(4k+1-y,4k+1-x)$; see Figures~\ref{fig:AZ} and~\ref{fig:AZ-bar}. We have:
\begin{lemma}\label{lem:AZ-represents}
  Given a real pointed matched circle $\PMC$,
  the real Heegaard diagrams $(\AZ(\PMC),\tau)$ and $(\AZbar(\PMC),\tau)$ specify real thick surfaces. The central real surface has genus $k$ and one fixed circle, and its quotient is orientable if and only if the matching $M$ in $\PMC$ preserves $\{1,\dots,2k\}$ (and hence $\{2k+1,\dots,4k\}$).
\end{lemma}
\begin{proof}
  This is immediate from the construction of a real bordered manifold from a real bordered Heegaard diagram.
\end{proof}

In particular, if $\PMC$ is the antipodal pointed matched circle, then $F(\PMC)/\tau$ is always nonorientable. If $\PMC$ is an even-genus split pointed matched circle, then $F(\PMC)/\tau$ is orientable; for an odd-genus split pointed matched circle, $F(\PMC)/\tau$ is nonorientable (as it must be, by the classification of real involutions of surfaces).

\subsection{Classical bordered invariants of the Auroux-Zarev diagrams}\label{sec:AZ-modules-review}
It seems helpful to pause at this point and recall the bordered \AAm\ and \DD\ invariants of the Auroux-Zarev piece and its mirror. The computations in Section~\ref{sec:real-AZ} of the real invariants of these diagrams do not, in fact, depend on the answers in the unreal case, so the reader may skip this section; but it provides context for the computations below.

\begin{figure}
  \centering
  \begin{tikzpicture}[alt={Auroux-Zarev pieces}]
  \begin{scope}[yshift = 5cm, xshift=-2cm]
    \draw (0,5) arc (180:0:2.5);
    \draw (.25,5) arc (180:0:2.25);
    \draw (5,5) arc (90:-90:2.5);
    \draw (5,4.75) arc (90:-90:2.25);
    \draw (.8,4.2) arc (-225:-45:1.414*1.2);
    \draw (1.2,3.8) arc (-225:-45:1.414*.8);
    \draw (2.2,2.8) arc (-225:-45:1.414*.8);
    \draw (1.8,3.2) arc (-225:-45:1.414*1.2);
    \begin{scope}[even odd rule]
    \draw[fill=white] (1.8,3.2) arc (-225:-45:1.414*1.2) (2.2,2.8) arc (-225:-45:1.414*.8);
    \end{scope}
    \draw (0,5) to (5,5) to (5,0) to (0,5);
    \foreach \y in {1,2,3,4}
      \draw[red, thick] (5,\y) to (5-\y,\y); 
    \foreach \x in {1,2,3,4}
      \draw[blue,thick] (\x,5) to (\x,5-\x);
    \foreach \x / \y / \z in {4/1/$\iota_0$, 3/2/$\iota_1$, 2/3/$\iota_0$, 1/4/$\iota_1$}
      \node at (\x-.2,\y-.2) () {\z};
    \foreach \x / \y / \z in {4/2/$\rho_1$, 4/3/$\rho_{12}$, 4/4/$\rho_{123}$, 3/3/$\rho_2$, 3/4/$\rho_{23}$, 2/4/$\rho_3$}
      \node at (\x+.3,\y+.2) () {\z};
    \draw[darkgreen] (5,5) to (2.5,2.5);
    \end{scope}
  \begin{scope}
    \begin{scope}[rotate=180, shift = {(-5,-5)}]
    \draw (0,5) arc (180:0:2.5);
    \draw (.25,5) arc (180:0:2.25);
    \draw (5,5) arc (90:-90:2.5);
    \draw (5,4.75) arc (90:-90:2.25);
    \draw (.8,4.2) arc (-225:-45:1.414*1.2);
    \draw (1.2,3.8) arc (-225:-45:1.414*.8);
    \draw (2.2,2.8) arc (-225:-45:1.414*.8);
    \draw (1.8,3.2) arc (-225:-45:1.414*1.2);
    \begin{scope}[even odd rule]
    \draw[fill=white] (1.8,3.2) arc (-225:-45:1.414*1.2) (2.2,2.8) arc (-225:-45:1.414*.8);
    \end{scope}
    \end{scope}
    \draw[-] (0,5) to (0,0) to (5,0) to (0,5);
    \foreach \y in {1,2,3,4}
      \draw[red, thick] (0,\y) to (5-\y,\y); 
    \foreach \x in {1,2,3,4}
      \draw[blue,thick] (\x,0) to (\x,5-\x);
    \foreach \x / \y / \z in {4/1/$\iota_1$, 3/2/$\iota_0$, 2/3/$\iota_1$, 1/4/$\iota_0$}
      \node at (\x+.2,\y+.2) () {\z};
    \foreach \x / \y / \z in {1/3/$\rho_1$, 1/2/$\rho_{12}$, 1/1/$\rho_{123}$, 2/2/$\rho_2$, 2/1/$\rho_{23}$, 3/1/$\rho_3$}
      \node at (\x+.3,\y+.2) () {\z};
    \draw[darkgreen] (0,0) to (2.5,2.5);
  \end{scope}
  \begin{scope}[xshift=5.35cm, yshift=-2.5cm, scale=.9]
    \draw (.8,8.2) arc (-225:-45:1.414*1.2);
    \draw (1.2,7.8) arc (-225:-45:1.414*.8);
    \draw (2.2,6.8) arc (-225:-45:1.414*.8);
    \draw (1.8,7.2) arc (-225:-45:1.414*1.2);
    \begin{scope}[even odd rule]
    \draw[fill=white] (1.8,7.2) arc (-225:-45:1.414*1.2) (2.2,6.8) arc (-225:-45:1.414*.8);
    \end{scope}
    \draw (4.8,4.2) arc (-225:-45:1.414*1.2);
    \draw (5.2,3.8) arc (-225:-45:1.414*.8);
    \draw (6.2,2.8) arc (-225:-45:1.414*.8);
    \draw (5.8,3.2) arc (-225:-45:1.414*1.2);
    \begin{scope}[even odd rule]
    \draw[fill=white] (5.8,3.2) arc (-225:-45:1.414*1.2) (6.2,2.8) arc (-225:-45:1.414*.8);
    \end{scope}
    \draw (0,9) to (9,9) to (9,0) to (0,9);
    \foreach \y in {1,2,3,4,5,6,7,8}
      \draw[red, thick] (9,\y) to (9-\y,\y); 
    \foreach \x in {1,2,3,4,5,6,7,8}
      \draw[blue,thick] (\x,9) to (\x,9-\x);
    \draw[darkgreen] (9,9) to (4.5,4.5);
  \end{scope}
  \end{tikzpicture}
  \caption{\textbf{Auroux-Zarev pieces.} \textcolor{red}{Horizontal} line segments are \textcolor{red}{$\alpha$-curves}, and \textcolor{blue}{vertical} line segments are \textcolor{blue}{$\beta$-curves}. The left column shows the diagram for the genus $1$ pointed matched circle, drawn first so that the action of $\Alg(\PMC)$ on $\CFAAa(\AZ)$ is apparent, and then as one would draw the diagram to understand $\CFDDa(\AZ)$. The right side shows the case of the genus $2$ split pointed matched circle, drawn in the style of $\CFAAa(\AZ)$, with two half-handles omitted to save space. Also, a disk should be attached to the outer boundary of each diagram. The \textcolor{darkgreen}{diagonal line} (which becomes a circle after attaching the disk) is the fixed set $C$ of the real involution.}
  \label{fig:AZ}
\end{figure}
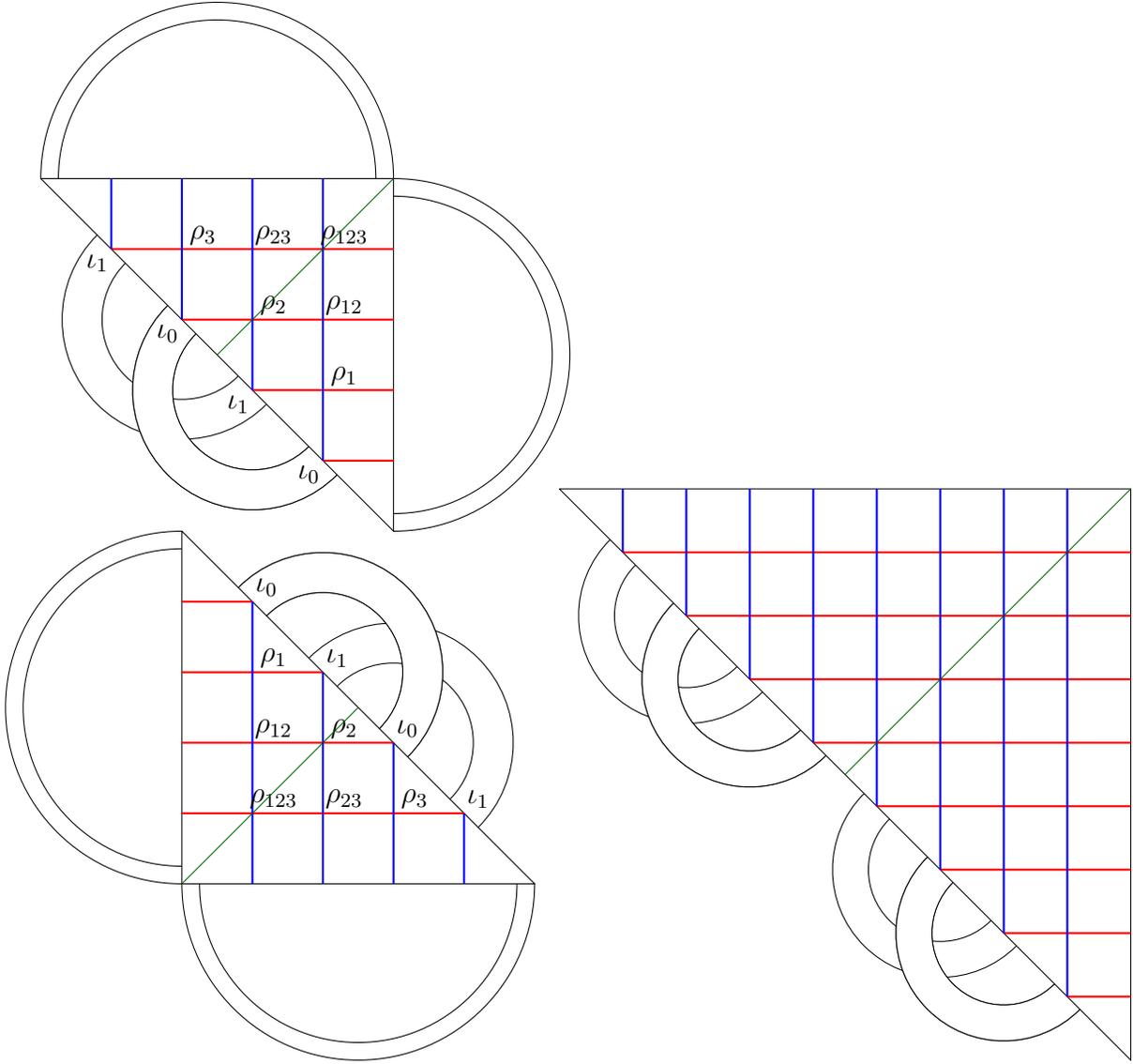

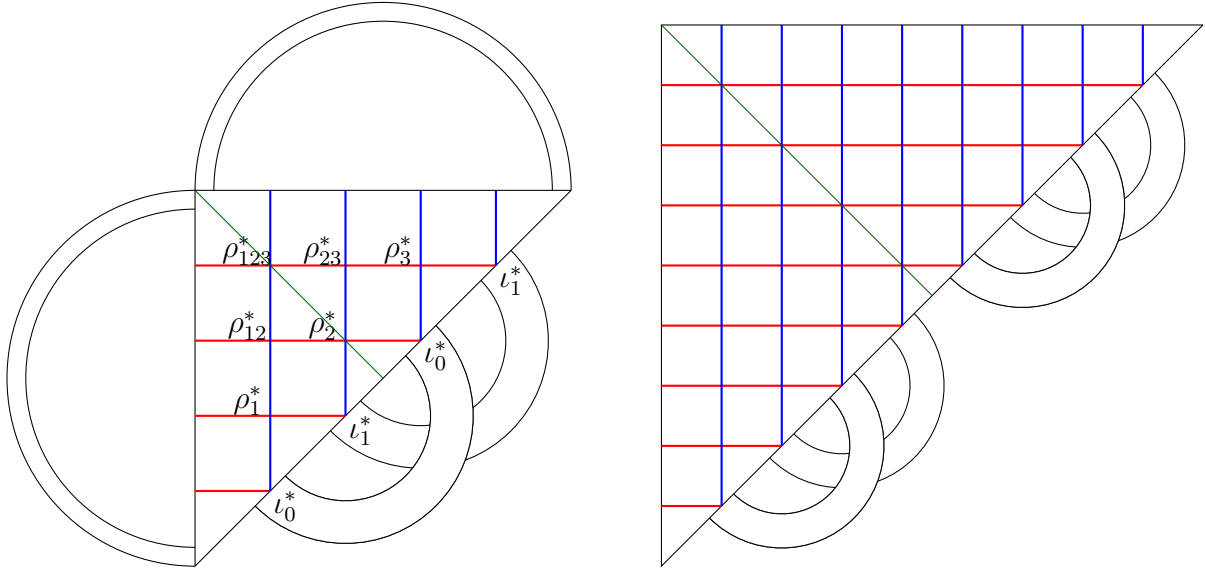
\begin{figure}
  \centering
    \begin{tikzpicture}[alt={Mirror Auroux-Zarev pieces}]
  \begin{scope}[xscale=-1, yshift=-2cm, xshift=-2cm]
    \draw (0,5) arc (180:0:2.5);
    \draw (.25,5) arc (180:0:2.25);
    \draw (5,5) arc (90:-90:2.5);
    \draw (5,4.75) arc (90:-90:2.25);
    \draw (.8,4.2) arc (-225:-45:1.414*1.2);
    \draw (1.2,3.8) arc (-225:-45:1.414*.8);
    \draw (2.2,2.8) arc (-225:-45:1.414*.8);
    \draw (1.8,3.2) arc (-225:-45:1.414*1.2);
    \begin{scope}[even odd rule]
    \draw[fill=white] (1.8,3.2) arc (-225:-45:1.414*1.2) (2.2,2.8) arc (-225:-45:1.414*.8);
    \end{scope}
    \draw (0,5) to (5,5) to (5,0) to (0,5);
    \foreach \y in {1,2,3,4}
      \draw[red, thick] (5,\y) to (5-\y,\y); 
    \foreach \x in {1,2,3,4}
      \draw[blue,thick] (\x,5) to (\x,5-\x);
    \foreach \x / \y / \z in {4/1/$\iota_0^*$, 3/2/$\iota_1^*$, 2/3/$\iota_0^*$, 1/4/$\iota_1^*$}
      \node at (\x-.2,\y-.2) () {\z};
    \foreach \x / \y / \z in {4/2/$\rho_1^*$, 4/3/$\rho_{12}^*$, 4/4/$\rho_{123}^*$, 3/3/$\rho_2^*$, 3/4/$\rho_{23}^*$, 2/4/$\rho_3^*$}
      \node at (\x+.3,\y+.2) () {\z};
    \draw[darkgreen] (5,5) to (2.5,2.5);
    \end{scope}
  \begin{scope}[xscale=-.8, yscale=.8, xshift=-13cm, yshift=-2.5cm]
    \draw (.8,8.2) arc (-225:-45:1.414*1.2);
    \draw (1.2,7.8) arc (-225:-45:1.414*.8);
    \draw (2.2,6.8) arc (-225:-45:1.414*.8);
    \draw (1.8,7.2) arc (-225:-45:1.414*1.2);
    \begin{scope}[even odd rule]
    \draw[fill=white] (1.8,7.2) arc (-225:-45:1.414*1.2) (2.2,6.8) arc (-225:-45:1.414*.8);
    \end{scope}
    \draw (4.8,4.2) arc (-225:-45:1.414*1.2);
    \draw (5.2,3.8) arc (-225:-45:1.414*.8);
    \draw (6.2,2.8) arc (-225:-45:1.414*.8);
    \draw (5.8,3.2) arc (-225:-45:1.414*1.2);
    \begin{scope}[even odd rule]
    \draw[fill=white] (5.8,3.2) arc (-225:-45:1.414*1.2) (6.2,2.8) arc (-225:-45:1.414*.8);
    \end{scope}
    \draw (0,9) to (9,9) to (9,0) to (0,9);
    \foreach \y in {1,2,3,4,5,6,7,8}
      \draw[red, thick] (9,\y) to (9-\y,\y); 
    \foreach \x in {1,2,3,4,5,6,7,8}
      \draw[blue,thick] (\x,9) to (\x,9-\x);
    \draw[darkgreen] (9,9) to (4.5,4.5);
  \end{scope}
  \end{tikzpicture}
  \caption{\textbf{Mirror Auroux-Zarev pieces.} In this case, we have drawn both pictures with the $\alpha$-arcs on the left.
  Otherwise, conventions are as in Figure~\ref{fig:AZ}.}
  \label{fig:AZ-bar}
\end{figure}

There are canonical isomorphisms 
\begin{align*}
\CFAAa(\AZ(\PMC))&=\Alg(\PMC) & &\text{\cite[Proposition 4.1]{LOTHomPair},}\\
\CFAAa(\AZbar(\PMC))&=\overline{\Alg(\PMC)} & &\text{\cite[Proposition 4.3]{LOTHomPair}, and}\\
\CFDDa(\AZbar(-\PMC))&=\lsup{\Alg(\PMC)}\mathit{bar}^{\Alg(\PMC)} & &\text{\cite[Proposition 5.11]{LOTHomPair}.}
\end{align*}
Here, $\overline{\Alg(\PMC)}$ is the dual type \AAm\ bimodule to $\Alg(\PMC)$, whose underlying $\FF_2$-vector space is $\Hom_{\FF_2}(\Alg(\PMC),\FF_2)$. The type \DD\ bimodule $\lsup{\Alg(\PMC)}\mathit{bar}^{\Alg(\PMC)}$ has underlying vector space $\overline{\Alg}=\Hom_{\FF_2}(\Alg(\PMC),\FF_2)$ and $\delta^1\co \overline{\Alg}\to\Alg(\PMC)\otimes\overline{\Alg}\otimes\Alg(\PMC)$ given by 
\[
\delta^1(\phi) = 1\otimes \overline{d}(\phi)\otimes 1 + \sum_{\substack{\xi\in\mathrm{Chord}(\PMC)\\(I,J)\text{ complementary}}}\Bigl( Ia(\xi)J\otimes (\phi\cdot Ia(\xi)J)\otimes 1 + 1\otimes (Ia(\xi)J\cdot \phi)\otimes Ia(\xi)J\Bigr).
\]
Here, $\overline{d}$ is the differential on $\overline{\Alg}$ induced by (i.e., the transpose of) the differential on $\Alg(\PMC)$, and $\mathrm{Chord}(\PMC)$ is the set of chords in $\PMC$. Note that, if we use the dual basis for $\overline{\Alg}$ and $\phi=a(\rhos)^*$, then $\overline{d}(\phi)$ is the sum of all $a(\rhos')^*$ so that $a(\rhos)$ is obtained from $a(\rhos')$ by resolving a crossing, and $\phi\cdot a(\xi)$ is the result of factoring $\xi$ from $\rhos$ on the right. Our convention for idempotents is that the left idempotent of $a^*$ is the complement to the right idempotent for $a$. As a concrete example, in the genus $1$ case, $\rho_{12}^*=\iota_1\rho_{12}^*\iota_1$ (because $\rho_{12}=\iota_0\rho_{12}\iota_0$) and
\begin{align*}
\delta^1(\rho_{12}^*) &= \iota_1\rho_2\iota_0\otimes (\rho_{12}^*\cdot \iota_1\rho_{2}\iota_0)\otimes 1+1\otimes(\iota_0\rho_1\iota_1\cdot \rho_{12}^*)\otimes \iota_0\rho_1\iota_1\\
&=\rho_2\otimes \rho_1^*\otimes 1+1\otimes\rho_2^*\otimes\rho_1.
\end{align*}
As another example, $\rho_1^*=\iota_1\rho_1^*\iota_0$ (because $\rho_1=\iota_0\rho_1\iota_1$) and
\begin{align*}
\delta^1(\rho_1^*)&=\iota_0\rho_1\iota_1\otimes(\rho_1^*\cdot\iota_0\rho_1\iota_1)\otimes 1+1\otimes(\iota_0\rho_1\iota_1\cdot \rho_1^*)\otimes\iota_0\rho_1\iota_1\\
&=\rho_1\otimes\iota_0^*\otimes 1+1\otimes\iota_0^*\otimes\rho_1.
\end{align*}

These results also give a description of $\CFDDa(\AZ(\PMC))$, by tensoring $\CFAAa(\AZ(\PMC))$ with the type \DD\ identity bimodule on both sides. The underlying vector space of $\CFDDa(\AZ(\PMC))$ is $\Alg(\PMC)$, and the differential is given by
\[
\delta^1(\wt{b}) = 1\otimes d(\wt{b})\otimes 1 + \sum_{\substack{\xi\in\mathrm{\Chord}(\PMC)\\(I,J)\text{ complementary}}}\Bigl( r(Ia(\xi)J)\otimes \wt{b\cdot Ia(\xi)J} \otimes 1 + 1\otimes \wt{Ia(\xi)J \cdot b}\otimes Ia(\xi)J\Bigr),
\]
where $r\co \PMC\to-\PMC$ is the orientation-reversing identity map. Here, we are writing the element of $\CFDDa(\AZ(\PMC))$ corresponding to $a\in\Alg(\PMC)$ as $\wt{a}$; the left idempotent of $\wt{a}$ is the image under $r$ of the complement of the right idempotent of $a$, and the right idempotent of $\wt{a}$ is the image under $r$ of the complement of the left idempotent of $a$. For example, $\wt{\rho_{12}}=\iota_0\wt{\rho_{12}}\iota_0$ and $\wt{\rho_2}=\iota_0\wt{\rho_2}\iota_1$, and
\begin{align*}
\delta^1(\wt{\rho_{12}})&=\delta^1(\iota_0\wt{\rho_{12}}\iota_0) = \iota_0\rho_1\iota_1\otimes\iota_1\wt{\rho_{12}\cdot\rho_3}\iota_0\otimes 1=\rho_1\otimes\wt{\rho_{123}}\otimes 1\\
\delta^1(\wt{\rho_{2}}) &= \delta^1(\iota_0\wt{\rho_{2}}\iota_1)=\iota_0\rho_1\iota_1\otimes\iota_1\wt{\rho_2\cdot\rho_3}\iota_1\otimes 1 + 1\otimes\iota_0\wt{\rho_1\cdot\rho_2}\iota_0\otimes \iota_0\rho_1\iota_1\\
&= \rho_1\otimes\wt{\rho_{23}}\otimes 1 + 1\otimes\wt{\rho_{12}}\otimes \rho_1
\end{align*}

\subsection{The real Auroux-Zarev modules}\label{sec:real-AZ}

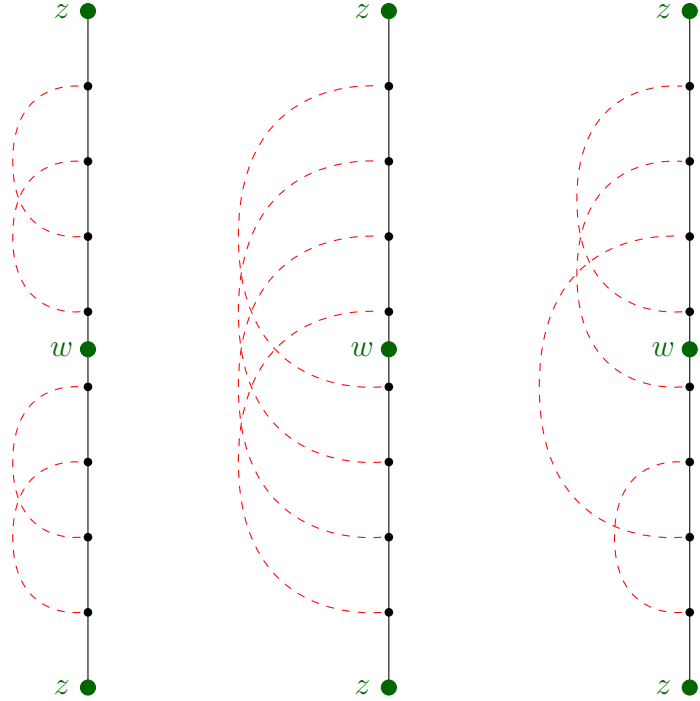
\begin{figure}
  \centering
  \begin{tikzpicture}[alt={Real pointed matched circles}]
    \begin{scope}
      \draw[-] (0,0) -- (0,9); 
      \foreach \y in {1,2,...,8}
        \filldraw (0,\y) circle (.05); 
      \foreach \y/\name in {0/z,4.5/w,9/z} {
        \filldraw[darkgreen] (0,\y) circle (.1); 
        \node at (-.35,\y) () {$\textcolor{darkgreen}{\name}$};
      }
      \foreach \y/\z in {1/3, 2/4, 5/7, 6/8}
        \draw[red, dashed] plot [smooth, tension=2] coordinates {(-.1,\y) (-\z/2+\y/2,\y/2+\z/2) (-.1,\z)};
    \end{scope}
    \begin{scope}[xshift=4cm]
      \draw[-] (0,0) -- (0,9); 
      \foreach \y in {1,2,...,8}
        \filldraw (0,\y) circle (.05); 
      \foreach \y/\name in {0/z,4.5/w,9/z} {
        \filldraw[darkgreen] (0,\y) circle (.1); 
        \node at (-.35,\y) () {$\textcolor{darkgreen}{\name}$};
      }
      \foreach \y/\z in {1/5, 2/6, 3/7, 4/8}
        \draw[red, dashed] plot [smooth, tension=2] coordinates {(-.1,\y) (-\z/2+\y/2,\y/2+\z/2) (-.1,\z)};
    \end{scope}
    \begin{scope}[xshift=8cm]
      \draw[-] (0,0) -- (0,9); 
      \foreach \y in {1,2,...,8}
        \filldraw (0,\y) circle (.05); 
      \foreach \y/\name in {0/z,4.5/w,9/z} {
        \filldraw[darkgreen] (0,\y) circle (.1); 
        \node at (-.35,\y) () {$\textcolor{darkgreen}{\name}$};
      }
      \foreach \y/\z in {1/3, 2/6, 4/7, 5/8}
        \draw[red, dashed] plot [smooth, tension=2] coordinates {(-.1,\y) (-\z/2+\y/2,\y/2+\z/2) (-.1,\z)};
    \end{scope}
  \end{tikzpicture}
  \caption{\textbf{Real pointed matched circles.} Left and center: genus $2$ real pointed matched circles (the split and antipodal pointed matched circles). Right: a genus $2$ pointed matched circle which is not real. The \textcolor{red}{dashed} semicircles indicate the matchings. As usual, to make drawing easier, we have cut the circles $Z$ open at the basepoints $z$.}
  \label{fig:real-PMC}
\end{figure}

Fix a real pointed matched circle $\PMC$. In the diagram $\AZ(\PMC)$, all points $x$ on $\alphas\cap C$ have $\sigma(x)=-1$; for $\AZbar(\PMC)$, all points $x$ on $\alphas\cap C$ have $\sigma(x)=+1$.
In particular, the diagram $\AZ(\PMC)$ (respectively $\AZbar(\PMC)$) is negative (respectively positive) real-nice. So, Proposition~\ref{prop:classify-domains} computes $\CFDRa(\AZ(\PMC))$ and $\CFDRa(\AZbar(\PMC))$. We devote the rest of this section to giving a more explicit description.

Fix a pointed matched circle $\PMC=(Z,\CircPts,M,z)$, where $|\CircPts|=4k$. Cutting $Z$ open at $z$ identifies $\CircPts$ with $\{1,\dots,4k\}$. Recall that a \emph{strands diagram} (generator of $\Alg(\PMC)$) consists of some number $n$ of intervals $[i,j]$ with $1\leq i<j\leq 4k$, the \emph{moving strands}, and $2k-2n$ additional points $\mathbf{h}\subset \{1,\dots,4k\}$, the \emph{horizontal strands}, so that $M(\mathbf{h})=\mathbf{h}$, $\mathbf{h}$ is disjoint from the initial and terminal points of the moving strands, and the initial (respectively terminal) points of the moving strands are themselves disjoint, and disjoint from their images under $M$. Equivalently, we can think of $\textbf{h}$ as a collection of trivial intervals $[i,i]$. We have restricted to the central $\SpinC$-structure on the surface, as usual.

Now, fix a real pointed matched circle $(\PMC,\tau_\PMC)$. The involution $\tau_\PMC$ acts on the set of strands diagrams. (A moving strand $[i,j]$ is sent to $[\tau_\PMC(j),\tau_\PMC(i)]$.) Call a strands diagram $a$ \emph{symmetric} if $\tau_\PMC(a)=a$, and asymmetric otherwise. Similarly, chords $[i,j]$ in $\PMC$ come in two types: \emph{fixed chords}, with $j=\tau(i)$, and \emph{non-fixed chords}.

\begin{figure}
  \centering
  \begin{tikzpicture}[scale=.666, alt={Numbering the segments in AZ and AZbar}]
    \begin{scope}[rotate=180]
    \draw (0,9) to (9,9) to (9,0) to (0,9);
    \foreach \y in {1,2,...,8}
      \draw[red, thick] (9,\y) to (9-\y,\y); 
    \foreach \x in {1,2,...,8}
      \draw[blue,thick] (\x,9) to (\x,9-\x);
    \draw[darkgreen] (9,9) to (4.5,4.5);
    \node at (5,10) (AZlab) {$\AZ$};
    \foreach \y in {1,2,...,8}
      {
      \node at (9.25,\y) (ylab\y) {$\y$};
      \node at (9-\y,9.3) (xlab\y) {$\y$};
      }
    \filldraw (7,5) circle (.1);
  \end{scope}
  \begin{scope}[xscale=-1, xshift=-12cm, yshift = -9cm]
    \draw (0,9) to (9,9) to (9,0) to (0,9);
    \foreach \y in {1,2,...,8}
      \draw[red, thick] (9,\y) to (9-\y,\y); 
    \foreach \x in {1,2,...,8}
      \draw[blue,thick] (\x,9) to (\x,9-\x);
    \draw[darkgreen] (9,9) to (4.5,4.5);
    \node at (5,-.95) (AZlab) {$\AZbar$};
    \foreach \y in {1,2,...,8}
      {
      \node at (9.25,\y) (ylab\y) {$\y$};
      \node at (9-\y,9.3) (xlab\y) {$\y$};
      }
    \filldraw (7,5) circle (.1);
  \end{scope}
  \end{tikzpicture}
  \caption{\textbf{Numbering the segments in $\AZ$ and $\AZbar$.} The diagram $\AZ$ is shown on the left and the diagram $\AZbar$ on the right, both drawn with the $\alpha$-boundary on the left. In both cases, the intersection point corresponding to the chord $[2,5]$ is marked. We have not drawn the matching or the corresponding handles, as they are irrelevant to the numbering scheme. Note that the top-to-bottom numbering of the points on the $\alpha$-boundary of $\AZ$ is consistent with type $A$, not type $D$ conventions: for the type $D$ structure, the chord $[1,2]$ goes from the point labeled $8$ to the point labeled $7$.}
  \label{fig:number-AZ}
\end{figure}
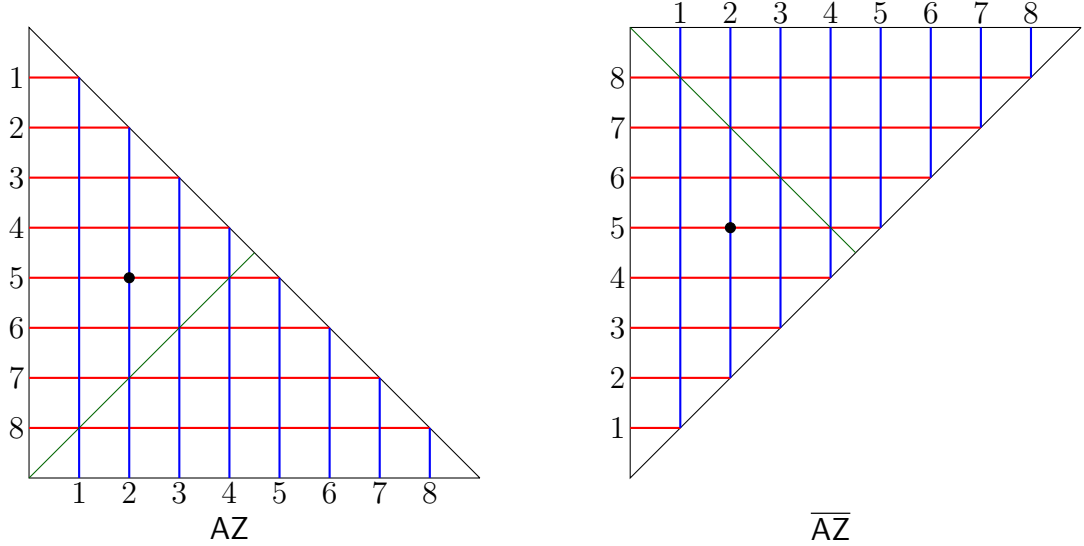

Starting with $\CFDRa(\AZbar(\PMC))$, we recall how points in $\AZbar(\PMC)$ correspond to chords. Each $\alpha$-arc in $\AZbar(\PMC)$ consists of two segments in the triangle (horizontal line segments in Figure~\ref{fig:AZ-bar}). By going along the boundary in the opposite of the boundary orientation, number these $\alpha$-segments $1,\dots,4k$. Similarly, each $\beta$-arc consists of two $\beta$-segments; number the $\beta$-segments, too, in the opposite of the boundary orientation. See Figure~\ref{fig:number-AZ}. Identify the intersection point between the $i\th$ beta-segment and the $j\th$-alpha segment with the strand $[i,j]$. (In the special case $i=j$, this is a horizontal strand rather than a moving strand.)

The module $\CFDRa(\AZbar(\PMC))$ is generated by the set of symmetric algebra elements. We will write a state corresponding to a symmetric algebra element $a$ as $a^*$, because the differential is reminiscent of the differential on $\CFDDa(\AZbar(\PMC))$. The left idempotent of $a^*$ is the complement of the right idempotent of $a$. (Compare Section~\ref{sec:AZ-modules-review}.)

Proposition~\ref{prop:classify-domains} specifies which domains contribute to the differential. From the form of the diagram $\AZbar(\PMC)$, there are no bigons, type 2 rectangles, annuli, tori, or bigon pairs. We describe the contributions of each of the remaining classes of domains explicitly:
\begin{enumerate}[label=(\roman*)]
\item\label{item:AZbar-1} Type 1 Rectangles: These contribute exactly as in the differential on $\CFDDa(\AZbar(\PMC))$, except that the ending chords must both be fixed. Explicitly, if $a$ contains a pair of non-fixed chords $[i,j]$ and $[\tau(j),\tau(i)]$ with $i<\tau(j)\leq 2k$ (and hence $2k+1\leq j<\tau(i)$), and there is no chord $[p,q]$ in $a$ with $i<p<\tau(j)<j<q<\tau(i)$, then $\delta^1(a^*)$ has a term $1\otimes b^*$, where $b$ is obtained from $a$ by replacing $\{[i,j],[\tau(j),\tau(i)]\}$ with $\{[i,\tau(i)],[j,\tau(j)]\}$. (The condition about $[p,q]$ corresponds to the rectangle being empty.) 
\item\label{item:AZbar-2} Rectangle Pairs: These correspond to a pair of terms in $\bdy^2$ on $\CFDDa(\AZbar(\PMC))$ corresponding to (anti-)resolving a pair of crossings exchanged by $\tau$, between four non-fixed strands. Explicitly, if $a$ contains chords $[i,j],[i',j'],[\tau(j),\tau(i)],[\tau(j'),\tau(i')]$ with $i<i'\leq j<j'$, and there is no chord $[p,q]$ in $a$ with $i<p<i'\leq j<q<j'$, then $\delta^1(a^*)$ has a term $1\otimes b^*$, where $b$ is obtained by replacing $\{[i,j],[i',j'],[\tau(j),\tau(i)],[\tau(j'),\tau(i')]\}$ in $a$ with $\{[i,j'],[j,i'],[\tau(j'),\tau(i)],[\tau(i'),\tau(j)]\}$. (Again, the condition about $[p,q]$ corresponds to the rectangles being empty.) 

In the case $j=i'$, $[j,i']$ and $[\tau(i'),\tau(j)]$ are horizontal, so we also add their matched points $[M(j),M(i')]$ and $[M(\tau(i')),M(\tau(j))]$ to the strand diagram $b$. This adding of matched points when some strands were horizontal also occurs in cases~\ref{item:AZbar-3}, \ref{item:AZbar-5}, \ref{item:AZbar-6}, \ref{item:AZbar-7}, and~\ref{item:AZbar-9}, but we will not mention it each time.
\item\label{item:AZbar-3} Compact Downward-Biting Hexagon: These are similar to the previous case, except that the two crossing are between a fixed strand and a pair of non-fixed strands. Explicitly, if $a$ contains non-fixed chords $[i,j]$ and $[\tau(j),\tau(i)]$ and a fixed chord $[\ell,\tau(\ell)]$ with $i<\ell\leq j<\tau(\ell)$, and no chord $[p,q]$ with $i<p<\ell\leq j<q<\tau(\ell)$, then $\delta^1(a^*)$ has a term $1\otimes b^*$, where $b$ is obtained from $a$ by replacing $\{[i,j],[\tau(j),\tau(i)],[\ell,\tau(\ell)]\}$ with $\{[i,\tau(i)],[\ell,j],[\tau(\ell),\tau(j)]\}$.
\item\label{item:AZbar-4} Compact Upward-Biting Hexagon: These are similar to the previous case. If $a$ contains non-fixed chords $[i,j]$ and $[\tau(j),\tau(i)]$ and a fixed chord $[\ell,\tau(\ell)]$ with $i<\ell<\tau(j)<j$, and no chord $[p,q]$ with $i<p<\tau(j)<j<q<\tau(\ell)$, then $\delta^1(a^*)$ has a term $1\otimes b^*$, where $b$ is obtained from $a$ by replacing $\{[i,j],[\tau(j),\tau(i)],[\ell,\tau(\ell)]\}$ with $\{[\tau(j),j],[\ell,\tau(i)],[i,\tau(\ell)]\}$.
\item\label{item:AZbar-5} Octagons: If $a$ contains non-fixed chords $[i,j]$ and $[\tau(j),\tau(i)]$ and fixed chords $[\ell,\tau(\ell)]$ and $[m,\tau(m)]$ with $i<\ell<m\leq j<\tau(m)$, and no chord $[p,q]$ with $i<p<m\leq j<q<\tau(\ell)$, then $\delta^1(a^*)$ has a term $1\otimes b^*$, where $b$ is obtained from $a$ by replacing $\{[i,j],[\tau(j),\tau(i)],[\ell,\tau(\ell)],[m,\tau(m)]\}$ with $\{[i,\tau(\ell)],[\ell,\tau(i)],[m,j],[\tau(j),\tau(m)]\}$.
\item\label{item:AZbar-6} Disjoint Half-Strip Pairs: Suppose $a$ contains non-fixed chords $[i,j]$ and $[\tau(j),\tau(i)]$; $\ell$ satisfies $i\leq \ell<j<\tau(i)$ and is not equal or matched to a terminal endpoint of any chord in $a$; and there is no chord $[p,q]$ in $a$ with $p<i\leq\ell<q<j$. Then $\delta^1(a^*)$ has a term $\rho\otimes b^*$ where $b^*$ is obtained from $a$ by replacing $\{[i,j],[\tau(j),\tau(i)]\}$ by $\{[i,\ell],[\tau(\ell),\tau(i)]\}$ and $\rho$ has a single moving strand $[\ell,j]$ and the same left idempotent as $a^*$.
\item\label{item:AZbar-7} Overlapping Half-Strip Pairs: Suppose $a$ contains two non-fixed chords $[i,j]$ and $[\tau(j),\tau(i)]$ exchanged by $\tau$; $\ell$ satisfies $i<\ell\leq j<\tau(i)$ and is not equal or matched to a terminal endpoint of any chord in $a$; and there is no chord $[p,q]$ in $a$ with $i<p<\ell\leq j<q$. Then $\delta^1(a^*)$ has a term $\rho\otimes b^*$ where $b^*$ is obtained from $a$ by replacing $\{[i,j],[\tau(j),\tau(i)]\}$ by $\{[\ell,j],[\tau(j),\tau(\ell)]\}$ and $\rho$ has a single moving strand $[\tau(\ell),\tau(j)]$ and the same left idempotent as $a^*$.
\item\label{item:AZbar-8} Noncompact Hexagon: Suppose $a$ contains a fixed chord $[i,\tau(i)]$; $j$ satisfies $i<j\leq 2k$ and is not equal or matched to the initial endpoint of any chord in $a$; and $a$ does not contain a chord $[p,q]$ where $i<p<j<\tau(j)<q$. Then $\delta^1(a^*)$ has a term $\rho\otimes b^*$ where $b^*$ is obtained from $a$ by replacing $[i,\tau(i)]$ by $[j,\tau(j)]$ and $\rho$ has a single moving strand $[\tau(j),\tau(i)]$ and the same left idempotent as $a^*$.
\item\label{item:AZbar-9} Noncompact Octagons: Suppose $a$ contains fixed chords $[i,\tau(i)]$ and $[j,\tau(j)]$; $\ell$ satisfies $j<\ell<\tau(j)<\tau(i)$ and is not equal or matched to the terminal endpoint of any chord in $a$; and no $[p,q]$ in $a$ satisfies $p<j$ and $\ell<q<\tau(i)$. Then $\delta^1(a^*)$ has a term $\rho\otimes b^*$ where $b^*$ is obtained from $a^*$ by replacing $\{[i,\tau(i)],[j,\tau(j)]\}$ by $\{[j,\ell], [\tau(\ell),\tau(j)]\}$ and $\rho$ has a single moving strand $[\ell,\tau(i)]$ and the same left idempotent as $a^*$.
\end{enumerate}
Examples (or schematics) of all of these kinds of differentials are shown in Figure~\ref{fig:AZbar-diff}; see also Figure~\ref{fig:AZ-diff-domains}.

\begin{figure}
  \centering
  \tikzset{strand/.style={out=0, in=180, looseness=.25}}
  \begin{tikzpicture}[scale=.5, alt={Differential on CFDR(AZbar(PMC))}]
  \begin{scope}[xshift=-1cm]
    \draw[-] (0,0) -- (0,9); 
    \foreach \y in {1,2,...,8} {
      \filldraw (0,\y) circle (.05); 
      \draw[ultra thin, dotted] (0,\y) -- (27,\y); 
      }
    \foreach \y/\name in {4.5/w} {
      \filldraw[darkgreen] (0,\y) circle (.1); 
      \node at (-.35,\y) () {$\textcolor{darkgreen}{\name}$};
      \draw[ultra thin, dashed, darkgreen] (0,\y) -- (27,\y);
    }
  \end{scope}
  \begin{scope}[xshift=0cm]
    \draw[-] (0,2) to[strand] (1,6);
    \draw[-] (0,3) to[strand] (1,7);
    \draw[->] (1.25,4.5) -- (1.75,4.5);
    \draw[-] (2,2) to[strand] (3,7);
    \draw[-] (2,3) to[strand] (3,6);
    \node at (1.5,-1) (ilab) {\ref{item:AZbar-1}};
  \end{scope}
  \begin{scope}[xshift=5cm]
    \draw[-] (0,1) to[strand] (1,3);
    \draw[-] (0,2) to[strand] (1,4);
    \draw[-] (0,9-4) to[strand] (1,9-2);
    \draw[-] (0,9-3) to[strand] (1,9-1);
    \draw[->] (1.25,4.5) -- (1.75,4.5);
    \draw[-] (2,1) to[strand] (3,4);
    \draw[-] (2,2) to[strand] (3,3);
    \draw[-] (2,9-4) to[strand] (3,9-1);
    \draw[-] (2,9-3) to[strand] (3,9-2);
    \node at (1.5,-1) (iilab) {\ref{item:AZbar-2}};
  \end{scope}
  \begin{scope}[xshift=10cm]
    \draw[-] (0,1) to[strand] (1,3);
    \draw[-] (0,2) to[strand] (1,7);
    \draw[-] (0,9-3) to[strand] (1,9-1);
    \draw[->] (1.25,4.5) -- (1.75,4.5);
    \draw[-] (2,1) to[strand] (3,8);
    \draw[-] (2,2) to[strand] (3,3);
    \draw[-] (2,9-3) to[strand] (3,9-2);
    \node at (1.5,-1) (iiilab) {\ref{item:AZbar-3}};
  \end{scope}
  \begin{scope}[xshift=15cm]
    \draw[-] (0,1) to[strand] (1,5);
    \draw[-] (0,3) to[strand] (1,6);
    \draw[-] (0,4) to[strand] (1,8);
    \draw[->] (1.25,4.5) -- (1.75,4.5);
    \draw[-] (2,1) to[strand] (3,6);
    \draw[-] (2,3) to[strand] (3,8);
    \draw[-] (2,4) to[strand] (3,5);
    \node at (1.5,-1) (iiilab) {\ref{item:AZbar-4}};
  \end{scope}
  \begin{scope}[xshift=20cm]
    \draw[-] (0,1) to[strand] (2,5);
    \draw[-] (0,2) to[strand] (2,7);
    \draw[-] (0,3) to[strand] (2,6);
    \draw[-] (0,4) to[strand] (2,8);
    \draw[->] (2.25,4.5) -- (2.75,4.5);
    \node at (2.5,-1) (vlab) {\ref{item:AZbar-5}};
    \draw[-] (3,1) to[strand] (5,7);
    \draw[-] (3,2) to[strand] (5,8);
    \draw[-] (3,3) to[strand] (5,5);
    \draw[-] (3,4) to[strand] (5,6);
  \end{scope}
  \begin{scope}[yshift = -11cm]
  \begin{scope}[xshift=-1cm]
    \draw[-] (0,0) -- (0,9); 
    \foreach \y in {1,2,...,8} {
      \filldraw (0,\y) circle (.05); 
      \draw[ultra thin, dotted] (0,\y) -- (31,\y); 
      }
    \foreach \y/\name in {4.5/w} {
      \filldraw[darkgreen] (0,\y) circle (.1); 
      \node at (-.35,\y) () {$\textcolor{darkgreen}{\name}$};
      \draw[ultra thin, dashed, darkgreen] (0,\y) -- (31,\y);
    }
  \end{scope}
  \begin{scope}[xshift=0cm]
    \draw[-] (0,1) to[strand] (1,4);
    \draw[-] (0,9-4) to[strand] (1,9-1);
    \draw[->] (1.25,4.5) -- (1.75,4.5);
    \node at (3.5,4.5) (tens) {$\otimes$};
    \draw[-, thick] (2,2) to[strand] (3,4);
    \draw[-] (4,1) to[strand] (5,2);
    \draw[-] (4,9-2) to[strand] (5,9-1);
    \node at (2.5,-1) (vilab) {\ref{item:AZbar-6}};
  \end{scope}
  \begin{scope}[xshift=8cm]
    \draw[-] (0,1) to[strand] (1,4);
    \draw[-] (0,9-4) to[strand] (1,9-1);
    \draw[->] (1.25,4.5) -- (1.75,4.5);
    \node at (3.5,4.5) (tens) {$\otimes$};
    \draw[-, thick] (2,6) to[strand] (3,8);
    \draw[-] (4,3) to[strand] (5,4);
    \draw[-] (4,9-4) to[strand] (5,9-3);
    \node at (2.5,-1) (vilab) {\ref{item:AZbar-7}};
  \end{scope}
  \begin{scope}[xshift=16cm]
    \draw[-] (0,2) to[strand] (1,9-2);
    \draw[->] (1.25,4.5) -- (1.75,4.5);
    \node at (3.5,4.5) (tens) {$\otimes$};
    \draw[-, thick] (2,9-3) to[strand] (3,9-2);
    \draw[-] (4,3) to[strand] (5,9-3);
    \node at (2.5,-1) (viilab) {\ref{item:AZbar-8}};
  \end{scope}
  \begin{scope}[xshift=24cm]
    \draw[-] (0,2) to[strand] (1,9-2);
    \draw[-] (0,3) to[strand] (1,9-3);
    \draw[->] (1.25,4.5) -- (1.75,4.5);
    \node at (3.5,4.5) (tens) {$\otimes$};
    \draw[-, thick] (2,9-4) to[strand] (3,9-2);
    \draw[-] (4,3) to[strand] (5,9-4);
    \draw[-] (4,4) to[strand] (5,9-3);
    \node at (2.5,-1) (viiilab) {\ref{item:AZbar-9}};
  \end{scope}
  \end{scope}
  \end{tikzpicture}  \caption{\textbf{Differential on $\CFDRa(\AZbar(\PMC))$.} In this schematic, we imagine that part of a real pointed matched circle $\PMC$ near the midpoint $w$ is shown, and all the points in $\CircPts$ indicated are matched to points outside this region. The first row shows terms $\delta^1(a^*)=1\otimes b^*$ coming from provincial domains, with only $a^*$ and $b^*$ drawn. In cases \ref{item:AZbar-2}, \ref{item:AZbar-3}, and~\ref{item:AZbar-5}, some non-invariant strands could also be horizontal. The second row shows examples of $c\otimes b^*$ where $c$ is nontrivial (and indicated to the left of the $\otimes$ symbol, and with a thicker strand). In cases~\ref{item:AZbar-6},~\ref{item:AZbar-7} and~\ref{item:AZbar-9}, some (thin) strands could be horizontal. In~\ref{item:AZbar-6}, there are cases where the algebra (thick) strand is above or crosses the $w$-line, in~\ref{item:AZbar-7}, there are cases where the algebra (thick) strand is below or crosses the $w$-line, and in~\ref{item:AZbar-9} there are cases where the algebra (thick) strand crosses the $w$-line.}
  \label{fig:AZbar-diff}
\end{figure}
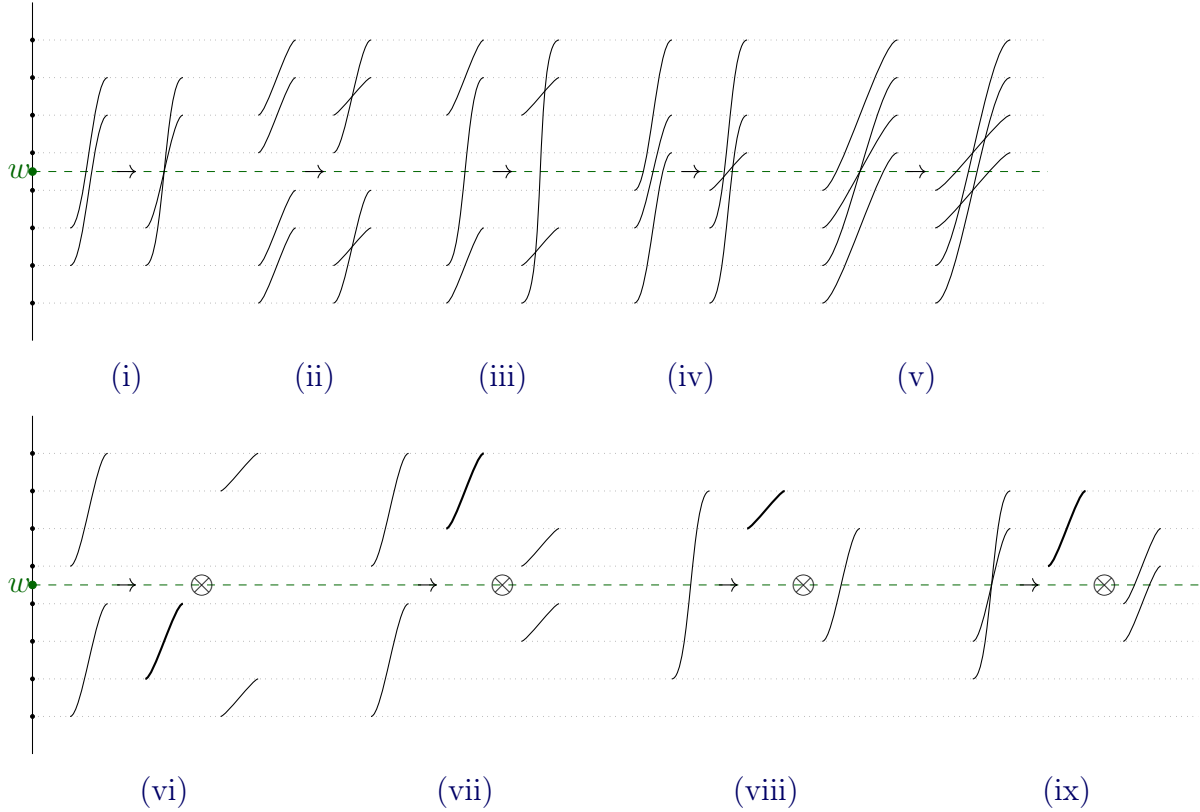

Next, we describe $\CFDRa(\AZ(\PMC))$, which we view as a left type $D$ structure over $\Alg(-\PMC)$ (though we could use $\tau_Z$ to view it as a left type $D$ structure over $\Alg(\PMC)$). The underlying $\FF_2$-vector space of $\CFDRa(\AZ(\PMC))$ is $\Alg(\PMC)$. We will denote the basis element of $\CFDRa(\AZ(\PMC))$ corresponding to a strands diagram $a\in\Alg(\PMC)$ by $\wt{a}$. Explicitly, if we draw $\AZ(\PMC)$ as in Section~\ref{sec:AZ-diagrams},
 then the chord $[i,j]$ corresponds to the point $(i,4k+1-j)$. (See Figure~\ref{fig:number-AZ}.)

The left idempotent of $\wt{a}$ is the complementary idempotent to the left idempotent of $a$. (Equivalently, because $a=r(a)$, this is the image under $r$ of the complementary idempotent to the right idempotent of $a$; compare Section~\ref{sec:AZ-modules-review}.)

The diagram is again nice, and the differential has the following terms. In cases~\ref{item:AZ-6}--\ref{item:AZ-9}, the term only appears if the idempotents permit it, i.e., if it does not result in two strands starting on the same matched pair, or ending on the same matched pair. (In the $\AZbar$ case above, we included this idempotent condition in the individual cases.)
\begin{enumerate}[label=(\roman*)]
\item\label{item:AZ-1} Type (1) Rectangles: These contribute exactly as the differential on $\Alg(\PMC)$: given strands $[i,\tau(i)]$ and $[j,\tau(j)]$ in $a$ with $i<j<\tau(j)<\tau(i)$, and no strand $[p,q]$ with $i<p<j<\tau(j)<q<\tau(i)$, $\delta^1(\wt{a})$ has a term $1\otimes \wt{b}$ where ${b}$ is obtained from ${a}$ by replacing $\{[i,\tau(i)],[j,\tau(j)]\}$ with $\{[i,\tau(j)],[j,\tau(i)]\}$.
\item\label{item:AZ-2} Rectangle Pairs: These contribute as pairs of differentials on $\Alg(\PMC)$ between quadruples of non-fixed strands exchanged by $\tau$. That is, given four strands $[i,j]$, $[\ell,m]$, $[\tau(j),\tau(i)]$, and $[\tau(m),\tau(\ell)]$ in $a$ with $i<\ell\leq m<j<\tau(\ell)$, and no strand $[p,q]$ with $i<p<\ell\leq m<q<j$, $\delta^1(\wt{a})$ has a term $1\otimes\wt{b}$ where $b$ is obtained from $a$ by replacing 
$\{[i,j], [\ell,m], [\tau(j),\tau(i)], [\tau(m),\tau(\ell)]\}$ with $\{[i,m], [\ell,j], [\tau(j),\tau(\ell)], [\tau(m),\tau(i)]\}$.

If $\ell=m$, we drop the horizontal strands $[M(\ell),M(m)]$ and $[M(\tau(m)),M(\tau(\ell))]$ matched to $[\ell,m]$ and $[\tau(m),\tau(\ell)]$ from $b$. This also occurs in Cases~\ref{item:AZ-4}, \ref{item:AZ-5}, \ref{item:AZ-6}, \ref{item:AZ-7}, and \ref{item:AZ-9} if there are horizontal strands, but we do not mention it again.
\item\label{item:AZ-3} Compact Downward-Biting Hexagon: These are similar to the previous case, but where we have three strands, with one fixed by $\tau$ and the other two exchanged by $\tau$. That is, given strands $[i,j]$, $[\tau(j),\tau(i)]$, and $[m,\tau(m)]$ in $a$ with $i<\tau(j)<m<\tau(m)<j<\tau(i)$, if there is no strand $[p,q]$ with $i<p<m<\tau(m)<q<j$, then $\delta^1(\wt{a})$ has a term $1\otimes\wt{b}$ where $b$ is obtained by replacing $\{[i,j],[\tau(j),\tau(i)],[m,\tau(m)]\}$ in $a$ with $\{[i,\tau(m)],[\tau(j),j],[m,\tau(i)]\}$. 
\item\label{item:AZ-4} Compact Upward-Biting Hexagon: These are similar to the previous case. Given strands $[i,j]$, $[\tau(j),\tau(i)]$, and $[m,\tau(m)]$ in $a$ with $m<i\leq j<\tau(m)$, if there is no strand $[p,q]$ with $m<p<i\leq j<q<\tau(m)$, then $\delta^1(\wt{a})$ has a term $1\otimes\wt{b}$ where $b$ is obtained from $a$ by replacing $\{[i,j],[\tau(j),\tau(i)],[m,\tau(m)]\}$ with $\{[m,j], [\tau(j),\tau(m)],[i,\tau(i)]\}$. 
\item\label{item:AZ-5} Octagons: Given non-fixed strands $[i,j]$, $[\tau(j),\tau(i)]$, $[\ell,m]$, $[\tau(m),\tau(\ell)]$ in $a$ with $i<\tau(j)<\ell<\tau(m)$, and no strand $[p,q]$ with $i<p<\ell\leq m<q<j$, $\delta^1(\wt{a})$ has a term $1\otimes\wt{b}$ where $b$ is obtained  by replacing $\{[i,j],[\tau(j),\tau(i)],[\ell,m],[\tau(m),\tau(\ell)]\}$ in $a$ with $\{[i,m],[\tau(j),j],[\ell,\tau(\ell)],[\tau(m),\tau(i)]\}$.
\item\label{item:AZ-6} Disjoint Half-Strip Pairs: Given non-fixed strands $[i,j]$ and $[\tau(j),\tau(i)]$ in $a$ and an integer $\ell$ with $i\leq j<\ell<\tau(i)$, and no strand $[p,q]$ with $p<i<j<q<\ell$, $\delta^1(\wt{a})$ has a term $\rho\otimes \wt{b}$ where $b$ is obtained from $a$ by replacing $\{[i,j],[\tau(j),\tau(i)]\}$ with $\{[i,\ell],[\tau(\ell),\tau(i)]\}$ and $\rho$ has a single moving strand $[\tau(\ell),\tau(j)]$ and the same left idempotent as $\wt{a}$.
\item\label{item:AZ-7} Overlapping Half-Strip Pairs. Given non-fixed strands $[i,j]$ and $[\tau(j),\tau(i)]$ in $a$ and an integer $\ell$ with $\ell<i\leq j<\tau(i)$, and no strand $[p,q]$ with $\ell<p<i\leq j<q$, $\delta^1(\wt{a})$ has a term $\rho\otimes \wt{b}$ where $b$ is obtained from $a$ by replacing $\{[i,j],[\tau(j),\tau(i)]\}$ with $\{[\ell,j],[\tau(j),\tau(\ell)]\}$, and $\rho$ has a single moving strand $[\ell,i]$.
\item\label{item:AZ-8} Noncompact Hexagon: Given a fixed strand $[i,\tau(i)]$ in $a$, an integer $\ell>\tau(i)$, and no strand $[p,q]$ in $a$ with $p<i<\tau(i)<q<\ell$, $\delta^1(\wt{a})$ has a term $\rho\otimes\wt{b}$ where $b$ is obtained from $a$ by replacing $[i,\tau(i)]$ with $[\tau(\ell),\ell]$ and $\rho$ has a single moving strand $[\tau(\ell),i]$ and the same left idempotent as $\wt{a}$.
\item\label{item:AZ-9} Noncompact Octagons: Given non-fixed strands $[i,j]$ and $[\tau(j),\tau(i)]$ in $a$ and an integer $\ell$ with $i\leq j<\tau(i)<\ell$, and no strand $[p,q]$ in $a$ with $p<i\leq j<q<\ell$, $\delta^1(\wt{a})$ has a term $\rho\otimes \wt{b}$ where $b$ is obtained by replacing $\{[i,j],[\tau(j),\tau(i)]\}$ in $a$ with $\{[i,\tau(i)],[\tau(\ell),\ell]\}$ and $\rho$ has a single moving strand $[\tau(\ell),\tau(j)]$ and the same left idempotent as $\wt{a}$.
\end{enumerate}
See Figures~\ref{fig:AZ-diff} and~\ref{fig:AZ-diff-domains}.

\begin{figure}
  \centering
  \tikzset{strand/.style={out=0, in=180, looseness=.25}}
  \begin{tikzpicture}[scale=.5, alt={Differential on CFDR(AZ(PMC))}]
  \begin{scope}[xshift=-1cm]
    \draw[-] (0,0) -- (0,9); 
    \foreach \y in {1,2,...,8} {
      \filldraw (0,\y) circle (.05); 
      \draw[ultra thin, dotted] (0,\y) -- (27,\y); 
      }
    \foreach \y/\name in {4.5/w} {
      \filldraw[darkgreen] (0,\y) circle (.1); 
      \node at (-.35,\y) () {$\textcolor{darkgreen}{\name}$};
      \draw[ultra thin, dashed, darkgreen] (0,\y) -- (27,\y);
    }
  \end{scope}
  \begin{scope}[xshift=0cm]
    \draw[-] (0,2) to[strand] (1,7);
    \draw[-] (0,3) to[strand] (1,6);
    \draw[->] (1.25,4.5) -- (1.75,4.5);
    \draw[-] (2,2) to[strand] (3,6);
    \draw[-] (2,3) to[strand] (3,7);
    \node at (1.5,-1) (ilab) {\ref{item:AZ-1}};
  \end{scope}
  \begin{scope}[xshift=5cm]
    \draw[-] (0,1) to[strand] (1,4);
    \draw[-] (0,2) to[strand] (1,3);
    \draw[-] (0,9-4) to[strand] (1,9-1);
    \draw[-] (0,9-3) to[strand] (1,9-2);
    \draw[->] (1.25,4.5) -- (1.75,4.5);
    \draw[-] (2,1) to[strand] (3,3);
    \draw[-] (2,2) to[strand] (3,4);
    \draw[-] (2,9-4) to[strand] (3,9-2);
    \draw[-] (2,9-3) to[strand] (3,9-1);
    \node at (1.5,-1) (iilab) {\ref{item:AZ-2}};
  \end{scope}
  \begin{scope}[xshift=10cm]
    \draw[-] (0,1) to[strand] (1,6);
    \draw[-] (0,3) to[strand] (1,8);
    \draw[-] (0,4) to[strand] (1,5);
    \draw[->] (1.25,4.5) -- (1.75,4.5);
    \draw[-] (2,1) to[strand] (3,5);
    \draw[-] (2,3) to[strand] (3,6);
    \draw[-] (2,4) to[strand] (3,8);
    \node at (1.5,-1) (iiilab) {\ref{item:AZ-3}};
  \end{scope}
  \begin{scope}[xshift=15cm]
    \draw[-] (0,1) to[strand] (1,8);
    \draw[-] (0,2) to[strand] (1,3);
    \draw[-] (0,9-3) to[strand] (1,9-2);
    \draw[->] (1.25,4.5) -- (1.75,4.5);
    \draw[-] (2,1) to[strand] (3,3);
    \draw[-] (2,2) to[strand] (3,7);
    \draw[-] (2,9-3) to[strand] (3,9-1);
    \node at (1.5,-1) (ivlab) {\ref{item:AZ-4}};
  \end{scope}
  \begin{scope}[xshift=20cm]
    \draw[-] (0,1) to[strand] (2,7);
    \draw[-] (0,2) to[strand] (2,8);
    \draw[-] (0,3) to[strand] (2,5);
    \draw[-] (0,4) to[strand] (2,6);
    \draw[->] (2.25,4.5) -- (2.75,4.5);
    \node at (2.5,-1) (vlab) {\ref{item:AZ-5}};
    \draw[-] (3,1) to[strand] (5,5);
    \draw[-] (3,2) to[strand] (5,7);
    \draw[-] (3,3) to[strand] (5,6);
    \draw[-] (3,4) to[strand] (5,8);
  \end{scope}
  \begin{scope}[yshift = -11cm]
  \begin{scope}[xshift=-1cm]
    \draw[-] (0,0) -- (0,9); 
    \foreach \y in {1,2,...,8} {
      \filldraw (0,\y) circle (.05); 
      \draw[ultra thin, dotted] (0,\y) -- (31,\y); 
      }
    \foreach \y/\name in {4.5/w} {
      \filldraw[darkgreen] (0,\y) circle (.1); 
      \node at (-.35,\y) () {$\textcolor{darkgreen}{\name}$};
      \draw[ultra thin, dashed, darkgreen] (0,\y) -- (31,\y);
    }
  \end{scope}
  \begin{scope}[xshift=0cm]
    \draw[-] (0,1) to[strand] (1,2);
    \draw[-] (0,9-2) to[strand] (1,9-1);
    \draw[->] (1.25,4.5) -- (1.75,4.5);
    \node at (3.5,4.5) (tens) {$\otimes$};
    \draw[-, thick] (2,5) to[strand] (3,7);
    \node at (2.5,-1) (vilab) {\ref{item:AZ-6}};
    \draw[-] (4,1) to[strand] (5,4);
    \draw[-] (4,9-4) to[strand] (5,9-1);
  \end{scope}
  \begin{scope}[xshift=8cm]
    \draw[-] (0,3) to[strand] (1,4);
    \draw[-] (0,9-4) to[strand] (1,9-3);
    \draw[->] (1.25,4.5) -- (1.75,4.5);
    \node at (3.5,4.5) (tens) {$\otimes$};
    \draw[-, thick] (2,1) to[strand] (3,3);
    \node at (2.5,-1) (vilab) {\ref{item:AZ-7}};
    \draw[-] (4,1) to[strand] (5,4);
    \draw[-] (4,9-4) to[strand] (5,9-1);
  \end{scope}
  \begin{scope}[xshift=16cm]
    \draw[-] (0,3) to[strand] (1,9-3);
    \draw[->] (1.25,4.5) -- (1.75,4.5);
    \node at (3.5,4.5) (tens) {$\otimes$};
    \draw[-, thick] (2,2) to[strand] (3,3);
    \draw[-] (4,2) to[strand] (5,9-2);
    \node at (2.5,-1) (viilab) {\ref{item:AZ-8}};
  \end{scope}
  \begin{scope}[xshift=24cm]
    \draw[-] (0,3) to[strand] (1,9-4);
    \draw[-] (0,4) to[strand] (1,9-3);
    \draw[->] (1.25,4.5) -- (1.75,4.5);
    \node at (3.5,4.5) (tens) {$\otimes$};
    \draw[-, thick] (2,2) to[strand] (3,4);
    \draw[-] (4,2) to[strand] (5,9-2);
    \draw[-] (4,3) to[strand] (5,9-3);
    \node at (2.5,-1) (viiilab) {\ref{item:AZ-9}};
  \end{scope}
  \end{scope}
  \end{tikzpicture}  \caption{\textbf{Differential on $\CFDRa(\AZ(\PMC))$.} Conventions are as in Figure~\ref{fig:AZbar-diff}.}
  \label{fig:AZ-diff}
\end{figure}
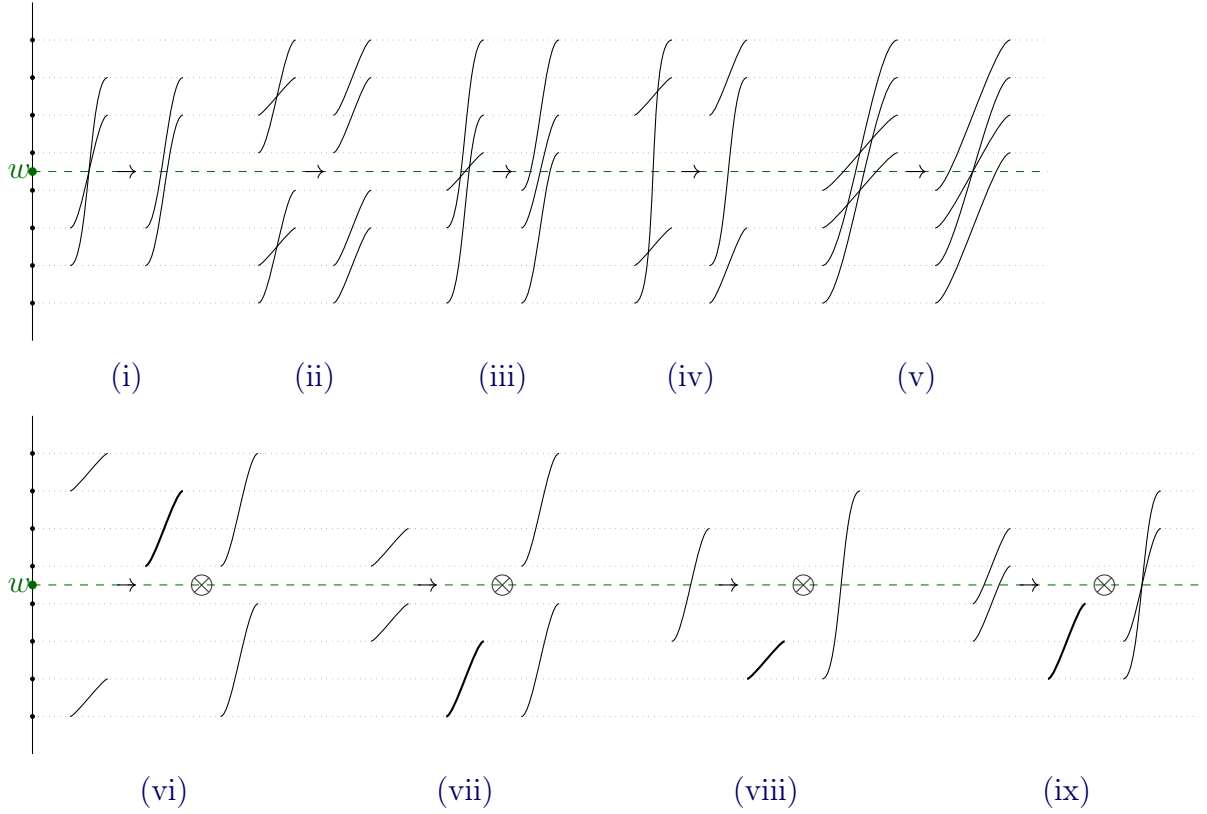

\begin{figure}
  \centering
  \begin{tikzpicture}[scale=.5, alt={Domains for terms in the differential on CFDR(AZ(PMC))}]
  \begin{scope}
    \draw (9,0) to (0,0) to (0,9) to (9,0);
    \foreach \y in {1,2,...,8}
      \draw[red, thick] (0,9-\y) to (\y,9-\y); 
    \foreach \x in {1,2,...,8}
      \draw[blue,thick] (9-\x,0) to (9-\x,\x);
    \draw[darkgreen] (0,0) to (4.5,4.5);
    \foreach \y in {1,2,...,8}
      {
      \node at (-.35,9-\y) (ylab\y) {$\y$};
      \node at (\y,-.4) (xlab\y) {$\y$};
      }
    \fill[semitransparent] (3,3) rectangle (4,4);
    \fill[pattern=crosshatch] (1,4) rectangle (2,6);
    \fill[pattern=crosshatch] (4,1) rectangle (6,2);
    \fill[pattern=dots] (1,2) rectangle (3,3);
    \fill[pattern=dots] (2,1) rectangle (3,2);
  \end{scope}
  \begin{scope}[xshift=9.5cm]
    \draw (9,0) to (0,0) to (0,9) to (9,0);
    \foreach \y in {1,2,...,8}
      \draw[red, thick] (0,9-\y) to (\y,9-\y); 
    \foreach \x in {1,2,...,8}
      \draw[blue,thick] (9-\x,0) to (9-\x,\x);
    \draw[darkgreen] (0,0) to (4.5,4.5);
    \foreach \y in {1,2,...,8}
      {
      \node at (-.35,9-\y) (ylab\y) {$\y$};
      \node at (\y,-.4) (xlab\y) {$\y$};
      }
    \fill[semitransparent] (3,3) rectangle (4,5);
    \fill[semitransparent] (4,3) rectangle (5,4);
    \fill[pattern=crosshatch] (1,2) rectangle (2,4);
    \fill[pattern=crosshatch] (2,1) rectangle (3,4);
    \fill[pattern=crosshatch] (3,1) rectangle (4,3);
    \fill[pattern=dots] (0,5) rectangle (1,7);
    \fill[pattern=dots] (5,0) rectangle (7,1);
  \end{scope}
  \begin{scope}[xshift=19cm]
    \draw (9,0) to (0,0) to (0,9) to (9,0);
    \foreach \y in {1,2,...,8}
      \draw[red, thick] (0,9-\y) to (\y,9-\y); 
    \foreach \x in {1,2,...,8}
      \draw[blue,thick] (9-\x,0) to (9-\x,\x);
    \draw[darkgreen] (0,0) to (4.5,4.5);
    \foreach \y in {1,2,...,8}
      {
      \node at (-.35,9-\y) (ylab\y) {$\y$};
      \node at (\y,-.4) (xlab\y) {$\y$};
      }
    \fill[semitransparent] (3,0) rectangle (4,5);
    \fill[semitransparent] (0,3) rectangle (5,4);
    \fill[pattern=crosshatch] (0,1) rectangle (2,2);
    \fill[pattern=crosshatch] (1,0) rectangle (2,1);
    \fill[pattern=dots] (3,0) rectangle (5,4);
    \fill[pattern=dots] (0,3) rectangle (4,5);
  \end{scope}
  \end{tikzpicture}
  \caption{\textbf{Domains for terms in the differential on $\CFDRa(\AZ(\PMC))$.} Left: domains of Type~\ref{item:AZ-1} (solid gray), \ref{item:AZ-2} (crosshatched), and \ref{item:AZ-3} (dots). Center: Type~\ref{item:AZ-4} (solid gray), \ref{item:AZ-5} (crosshatched), and \ref{item:AZ-6} (dots). Right: Type~\ref{item:AZ-7} (solid gray), \ref{item:AZ-8} (crosshatched), and \ref{item:AZ-9} (dots). To conserve space, these are mostly not the same examples as in Figure~\ref{fig:AZ-diff}.}
  \label{fig:AZ-diff-domains}
\end{figure}
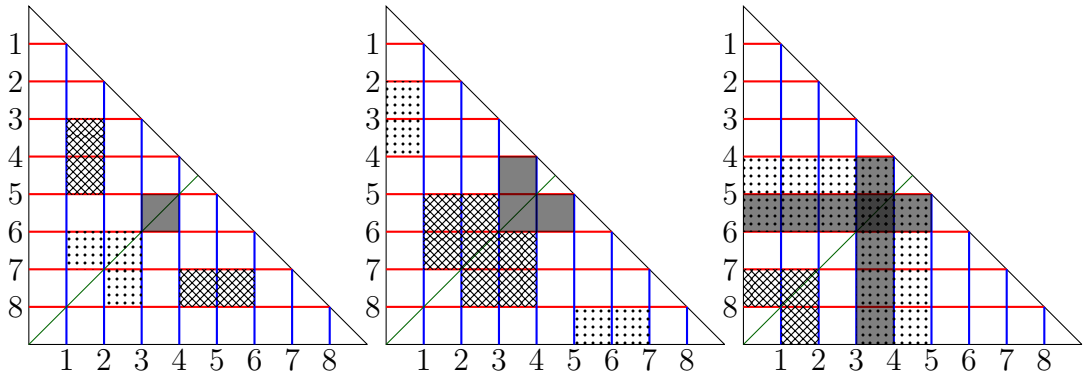

\begin{example}\label{eg:genus-1-AZ-CFDR}
  We spell out $\CFDRa(\AZbar(\PMC))$ and $\CFDRa(\AZ(\PMC))$ if $\PMC$ is the genus $1$ pointed matched circle. The module $\CFDRa(\AZbar(\PMC))$ is generated by two elements, $[2,3]^*$ and $[1,4]^*$, and the differential $\delta^1([1,4]^*)=[3,4]\otimes [2,3]^*$.
  That is, in our usual notation for the torus algebra, $\CFDRa(\AZbar)$ is generated by $\rho_2^*$ and $\rho_{123}^*$. The right idempotent of $\rho_2$ is $\iota_0$ and the right idempotent of $\rho_{123}$ is $\iota_1$ (i.e., $\rho_2\iota_0=\rho_2$ and $\rho_{123}\iota_1=\rho_{123}$), so the left idempotent of $\rho_2^*$ is $\iota_1$ and the left idempotent of $\rho_{123}^*$ is $\iota_0$.
  The differential is $\delta^1(\rho_{123}^*)=\rho_3\otimes\rho_{2}^*$ (which is consistent with the idempotents). The differential comes from a noncompact hexagon.

  Similarly, the module $\CFDRa(\AZ(\PMC))$ is generated by elements $\wt{[2,3]}$ and $\wt{[1,4]}$. The differential is given by $\delta^1(\wt{[2,3]})=[1,2]\otimes\wt{[1,4]}$. In the usual notation for the torus algebra, this is $\delta^1(\wt{\rho_{2}})=\rho_1\otimes \wt{\rho_{123}}$. (The differential again comes from a noncompact hexagon.) The left idempotent of $\rho_2$ is $\iota_1$, so the left idempotent of $\wt{\rho_2}$ is $\iota_0$. Similarly, the left idempotent of $\wt{\rho_{123}}$ is $\iota_1$ (which is consistent with $\delta^1$).
\end{example}

\begin{example}\label{eg:split-genus-2}
  The complex $\CFDRa(\AZ(\PMC))$ for $\PMC$ the split, genus $2$ pointed matched circle is shown in Figure~\ref{fig:CFDR-AZ-2-split}.
  Because of the low genus, only terms of types~\ref{item:AZ-1},~\ref{item:AZ-6},~\ref{item:AZ-7}, and~\ref{item:AZ-9} contribute to the differential. (It is immediate that there are no terms of types~\ref{item:AZ-2}--\ref{item:AZ-5}; to see that there are no terms of type~\ref{item:AZ-8} involves considering the idempotents.) As an example of the idempotents, the idempotent of the state $\wt{\{[2,2],[4,4],[5,5],[7,7]\}}$ (which is denoted $\begin{bsmallmatrix}2 & 5\\\cdot & \cdot\end{bsmallmatrix}$ in Figure~\ref{fig:CFDR-AZ-2-split}) has the matched pairs $\{1,3\}$ and $\{6,8\}$ occupied (i.e., is $\{[1,1],[3,3],[6,6],[8,8]\}$). The idempotent of the state $\wt{\{[3,6],[4,5]\}}$ (which is denoted $\begin{bsmallmatrix} 3 & 4\\ 6 & 5
  \end{bsmallmatrix}$ in Figure~\ref{fig:CFDR-AZ-2-split}) has the matched pairs $\{5,7\}$ and $\{6,8\}$ occupied.  
  
  As an example of the differential,
  \begin{align*}
  \delta^1&(\wt{\{[2,2],[4,4],[5,5],[7,7]\}})\\ 
  &= 
  \{[6,7],[1,1],[3,3]\}\otimes \wt{\{[2,3],[6,7]\}}
  + \{[3,7],[6,6],[8,8]\}\otimes \wt{\{[2,6],[3,7]\}}\\
  &\qquad+ \{[1,2],[6,6],[8,8]\}\otimes \wt{\{[1,2],[7,8]\}}
  + \{[1,4],[6,6],[8,8]\}\otimes \wt{\{[1,4],[5,8]\}}\\
  &\qquad+ \{[3,4],[6,6],[8,8]\}\otimes \wt{\{[3,4],[5,6]\}}
  + \{[1,7],[6,6],[8,8]\}\otimes \wt{\{[1,8],[2,7]\}}\\
  &\qquad+ \{[3,5],[6,6],[8,8]\}\otimes \wt{\{[3,6],[4,5]\}}
  + \{[1,5],[6,6],[8,8]\}\otimes \wt{\{[1,8],[4,5]\}}.
  \end{align*}
  The first two terms are of type~\ref{item:AZ-6}, the next three of type~\ref{item:AZ-7}, and the last three are of type~\ref{item:AZ-9}. Note that there is, for example, no type~\ref{item:AZ-6} term $\{[6,7],[1,1],[3,3]\}\otimes \wt{\{[2,4],[5,7]\}}$ here because of the idempotents. As another example,
  \[
  \delta^1(\wt{\{[3,6],[4,5]\}})=1\otimes \wt{\{[3,5],[4,6]\}},
  \]
  coming from a term of type~\ref{item:AZ-1}. We could equivalently write $1$ as the idempotent corresponding to this state, which is $\{[5,5],[6,6],[7,7],[8,8]\}$.

  \begin{figure}
    \centering
    \includegraphics[alt={The complex CFDR(AZ(PMC)) for the split genus 2 pointed matched circle}]{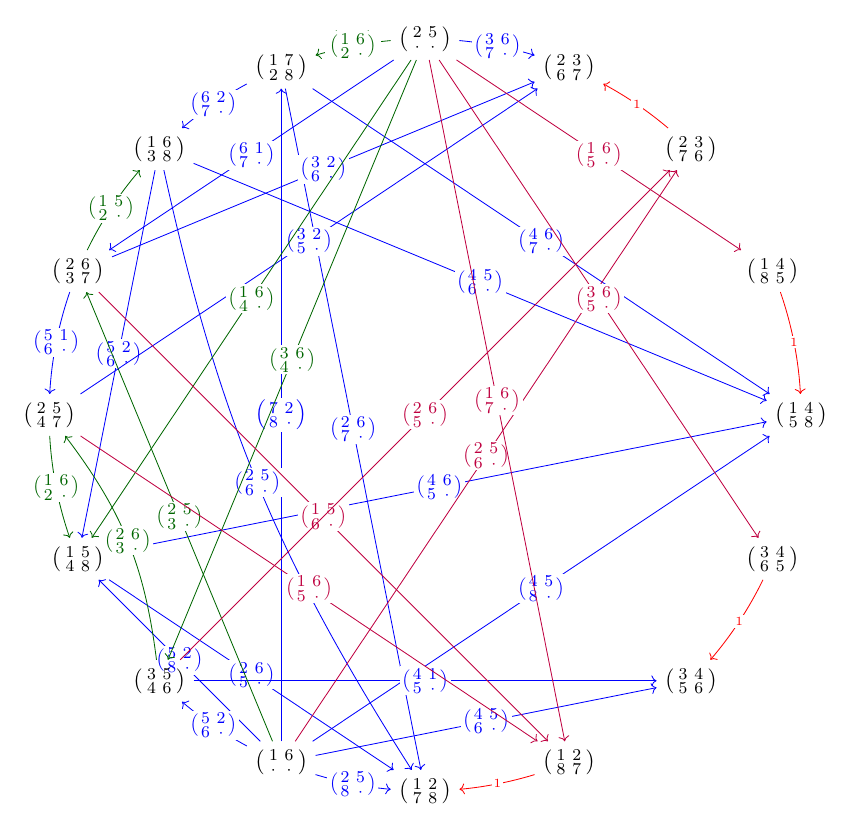}
    \caption{\textbf{The complex $\CFDRa(\AZ(\PMC))$ for the split, genus $2$ pointed matched circle.} For readability, we are using slightly different notation for generators of the algebra: $\begin{psmallmatrix}1 & 2\\ 8 & 7\end{psmallmatrix}$, for instance, means the pair of strands $[1,8],[2,7]$, and $\begin{psmallmatrix}1 & 6\\ 2 & \cdot\end{psmallmatrix}$ means the single strand $[1,2]$ with left idempotent $\{1,3,6,8\}$. All the algebra coefficients in the differential are either $1$ or have a single moving strand (hence have one dot inside the matrix). Type~\textcolor{red}{(i)},~\textcolor{blue}{(vi)},~\textcolor{darkgreen}{(vii)}, and~\textcolor{purple}{(ix)} terms contribute to the differential.}
    \label{fig:CFDR-AZ-2-split}
  \end{figure}
\end{example}

Recall that the strands diagrams in $\Alg(\PMC)$ where two moving strands overlap---the elements whose supports have multiplicity $>1$ somewhere---span an acyclic ideal in $\Alg(\PMC)$~\cite[Proposition 4.2]{LOT2}. So, the quotient $\Alg'(\PMC)$ by this ideal is quasi-isomorphic to $\Alg(\PMC)$, and when performing computations in bordered Floer homology it suffices to work over the smaller algebra $\Alg'(\PMC)$. We show that a similar phenomenon occurs for $\CFDRa(\AZ(\PMC))$. We then observe that, in the case of real pointed matched circles giving surfaces with orientable quotients, this leads to a particularly simple description of $\CFDRa(\AZ(\PMC))$.

\begin{lemma}\label{lem:mult-2-contractible}
  For any real pointed matched circle $\PMC$, the strands diagrams which have multiplicity greater than $1$ somewhere span a contractible sub-type-$D$-structure of $\CFDRa(\AZ(\PMC))$. Thus, $\CFDRa(\AZ(\PMC))$ is homotopy equivalent to its quotient by this sub-type-$D$-structure.
\end{lemma}
The quotient is, of course, generated by the strands diagrams with multiplicity $0$ or $1$ everywhere.
\begin{proof}
  It is immediate from the definition of the differential that these strands diagrams span a sub-type-$D$-structure.
  The proof that this substructure is contractible is an adaptation of the corresponding proof for the bordered Floer algebras~\cite[Theorem 9(2)]{LOT2}. Let $P$ denote this sub-structure.

  Let $\Alg_+(\PMC)$ be the augmentation ideal, generated by all strands diagrams with at least one moving strand, so $\Alg(\PMC)/\Alg_+(\PMC)=\Idem(\PMC)$ is a direct sum of copies of $\FF_2$, one for each idempotent strands diagram. Then $P$ is contractible over $\Alg(\PMC)$ if and only if the induced type $D$ structure $\Idem(\PMC)\DT P$ is contractible. That is, when proving that $P$ is contractible, it suffices to consider only the provincial differential---the terms $1\otimes \wt{b}$ in the differential---and we will do so for the rest of the argument. We can (and do) regard $P$ with this provincial differential as an ordinary chain complex, by identifying $1\otimes\wt{b}$ with $\wt{b}$.

  There is a filtration on $P$ by the number of horizontal strands; the differential preserves or decreases this filtration. We will show that the homology of the associated graded complex is trivial. Additionally, $P$ decomposes as direct sums over the left and right idempotents, and over the supports of the algebra elements (in $H_1(Z,\CircPts)$). (This uses the fact that we are only considering the provincial differential.) So, we fix a choice of left and right idempotents and a support, and denote by $\wt{P}$ the summand of the associated graded complex with this pair of idempotents and support, and $j$ moving strands. (From the definition of $P$, $j\geq 2$.)
 
  In $\wt{P}$, the differential does not resolve crossings on horizontal strands. It follows that the chain complex $\wt{P}$ is independent of the choice of matching.

  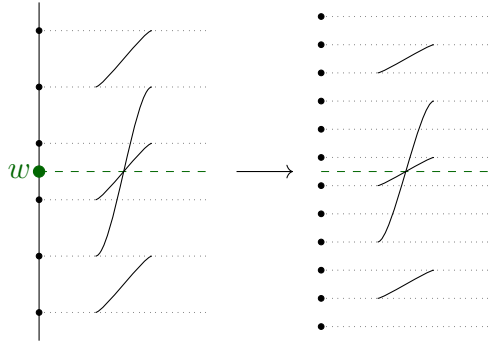
\begin{figure}
  \centering
    \tikzset{strand/.style={out=0, in=180, looseness=.25}}
  \begin{tikzpicture}[scale=.75, alt={Perturbing the endpoints}]
  \begin{scope}[xshift=-1cm]
    \draw[-] (0,.5) -- (0,6.5); 
    \foreach \y in {1,2,...,6} {
      \filldraw (0,\y) circle (.05); 
      \draw[ultra thin, dotted] (0,\y) -- (3,\y); 
      \draw[ultra thin, dotted] (0,\y) -- (3,\y); 
      }
    \foreach \y/\name in {3.5/w} {
      \filldraw[darkgreen] (0,\y) circle (.1); 
      \node at (-.35,\y) () {$\textcolor{darkgreen}{\name}$};
      \draw[ultra thin, dashed, darkgreen] (0,\y) -- (3,\y);
      \draw[ultra thin, dashed, darkgreen] (5,\y) -- (8,\y);
    }
  \end{scope}
  \begin{scope}[xshift=0cm]
    \draw[-] (0,1) to[strand] (1,2);
    \draw[-] (0,2) to[strand] (1,5);
    \draw[-] (0,3) to[strand] (1,4);
    \draw[-] (0,5) to[strand] (1,6);
  \end{scope}
    \draw[->] (2.5,3.5) -- (3.5,3.5);
  \begin{scope}[xshift=4cm,yshift=.25cm,scale=.5]
    \foreach \y in {1,2,...,12} {
      \filldraw (0,\y) circle (.1); 
      \draw[ultra thin, dotted] (0,\y) -- (6,\y); 
      \draw[ultra thin, dotted] (0,\y) -- (6,\y); 
      }
    \draw[-] (2,2) to[strand] (4,3);
    \draw[-] (2,4) to[strand] (4,9);
    \draw[-] (2,6) to[strand] (4,7);
    \draw[-] (2,10) to[strand] (4,11);
  \end{scope}
  \end{tikzpicture}
    \caption{\textbf{Perturbing the endpoints.} The element of $\Alg(\PMC)$ on the left maps to the element of $\Alg(8k)$ on the right, by sending initial endpoints $i$ to $2i$ and terminal endpoints $j$ to $2j-1$.}
    \label{fig:perturb-endpoints}
  \end{figure}

  Consider a chain complex $Q$ defined like $\CFDRa(\AZ(\PMC))$, except:
  \begin{itemize}
  \item If $\CircPts$ for $\PMC$ consists of $4k$ points $\{1,\dots,4k\}$, then $Q$ is based on a pointed matched circle with $8k$ points $\{1,\dots,8k\}$, and 
  \item there is no matching in the definition of $Q$.
  \end{itemize}
  As in the algebra case~\cite[Proof of Theorem 9(2)]{LOT2}, we can embed $\wt{P}$ in $Q$. Fix a state $\wt{a}\in\wt{P}$. For each moving (non-horizontal) strand $\xi=[i,j]$ of $a$, consider a new strand $[2i,2j-1]$.
  Let $a'\in C_*$ consist of all the strands $\xi'$, for $\xi$ a strand of $a$.
  See Figure~\ref{fig:perturb-endpoints}. Note that the symmetry of $a$ exchanging $[i,j]$ and $[4k+1-i,4k+1-j]$ induces a symmetry of $a'$ exchanging $[\ell,m]$ and $[8k+1-m,8k+1-\ell]$. In particular, $a'\in C_*$. Moreover, because the differential on $\wt{P}$ does not affect the horizontal strands, this induces a chain map $\wt{P}\to Q$. Of course, all elements in the image have the same support; let $\wt{Q}$ denote the summand with this support. Then $\wt{P}\cong\wt{Q}$, and we will show that $\wt{Q}$ is acyclic.

  \begin{figure}
  \centering
  \tikzset{strand/.style={out=0, in=180, looseness=.25}}
  \begin{tikzpicture}[scale=.5, alt={Perturbing the endpoints}]
  \begin{scope}[xshift=-1cm]
    \draw[-] (0,0) -- (0,9); 
    \foreach \y in {1,2,...,8} {
      \filldraw (0,\y) circle (.05); 
      \draw[ultra thin, dotted] (0,\y) -- (25,\y); 
      }
    \foreach \y/\name in {4.5/w} {
      \filldraw[darkgreen] (0,\y) circle (.1); 
      \node at (-.35,\y) () {$\textcolor{darkgreen}{\name}$};
      \draw[ultra thin, dashed, darkgreen] (0,\y) -- (25,\y);
    }
  \end{scope}
  \begin{scope}[xshift=0cm]
    \draw[-] (0,2) to[strand] (1,7);
    \draw[-] (0,3) to[strand] (1,6);
    \draw[->] (1.25,4.5) -- node[above]{\lab{H}} (1.75,4.5);
    \node at (2,4.5) (zero1) {$0$}; 
    \node at (1.25,-1) (ilab) {\ref{item:H-cross}};
    \node at (-.5,3) (ipos) {\lab{i}};
  \end{scope}
  \begin{scope}[xshift=3cm]
    \draw[-] (0,1) to[strand] (1,8);
    \draw[-] (0,2) to[strand] (1,3);
    \draw[-] (0,6) to[strand] (1,7);
    \draw[->] (1.25,4.5) -- node[above]{\lab{H}} (1.75,4.5);
    \node at (2,4.5) (zero1) {$0$}; 
    \node at (1.25,-1) (ilab) {\ref{item:H-cross}};
    \node at (-.5,2) (ipos) {\lab{i}};
  \end{scope}
  \begin{scope}[xshift=6cm]
    \draw[-] (0,1) to[strand] (1,4);
    \draw[-] (0,2) to[strand] (1,3);
    \draw[-] (0,5) to[strand] (1,8);
    \draw[-] (0,6) to[strand] (1,7);
    \draw[->] (1.25,4.5) -- node[above]{\lab{H}} (1.75,4.5);
    \node at (2,4.5) (zero1) {$0$}; 
    \node at (1.25,-1) (ilab) {\ref{item:H-cross}};
    \node at (-.5,2) (ipos) {\lab{i}};
  \end{scope}
  \begin{scope}[xshift=10cm]
    \draw[-] (2,2) to[strand] (3,7);
    \draw[-] (2,3) to[strand] (3,6);
    \draw[->] (1.25,4.5) -- node[above]{\lab{H}} (1.75,4.5);
    \draw[-] (0,2) to[strand] (1,6);
    \draw[-] (0,3) to[strand] (1,7);
    \node at (1.5,-1) (ilab) {\ref{item:H-exchanged}};
    \node at (-.5,2) (jpos) {\lab{\tau(j)}};
    \node at (-.5,3) (ipos) {\lab{i}};
  \end{scope}
  \begin{scope}[xshift=15cm]
    \draw[-] (2,1) to[strand] (3,8);
    \draw[-] (2,2) to[strand] (3,3);
    \draw[-] (2,9-3) to[strand] (3,9-2);
    \draw[->] (1.25,4.5) -- node[above]{\lab{H}} (1.75,4.5);
    \draw[-] (0,1) to[strand] (1,3);
    \draw[-] (0,2) to[strand] (1,7);
    \draw[-] (0,9-3) to[strand] (1,9-1);
    \node at (1.5,-1) (ivlab) {\ref{item:H-one-fixed}};
    \node at (-.5,1) (jpos) {\lab{j}};
    \node at (-.5,2) (ipos) {\lab{i}};
    \node at (-.5,3) (lpos) {\lab{\ell}};
  \end{scope}
  \begin{scope}[xshift=20cm]
    \draw[-] (2,1) to[strand] (3,4);
    \draw[-] (2,2) to[strand] (3,3);
    \draw[-] (2,9-4) to[strand] (3,9-1);
    \draw[-] (2,9-3) to[strand] (3,9-2);
    \draw[->] (1.25,4.5) -- node[above]{\lab{H}} (1.75,4.5);
    \draw[-] (0,1) to[strand] (1,3);
    \draw[-] (0,2) to[strand] (1,4);
    \draw[-] (0,9-4) to[strand] (1,9-2);
    \draw[-] (0,9-3) to[strand] (1,9-1);
    \node at (1.5,-1) (iilab) {\ref{item:H-none-fixed}};
    \node at (-.5,2) (ipos) {\lab{i}};
    \node at (-.5,4) (jpos) {\lab{j}};
    \node at (-.5,1) (lpos) {\lab{\ell}};
    \node at (-.5,3) (mpos) {\lab{m}};
  \end{scope}
  \end{tikzpicture}  \caption{\textbf{The homotopy $H$.} Conventions are as in Figures~\ref{fig:AZbar-diff} and~\ref{fig:AZ-diff}.}
  \label{fig:H}
\end{figure}
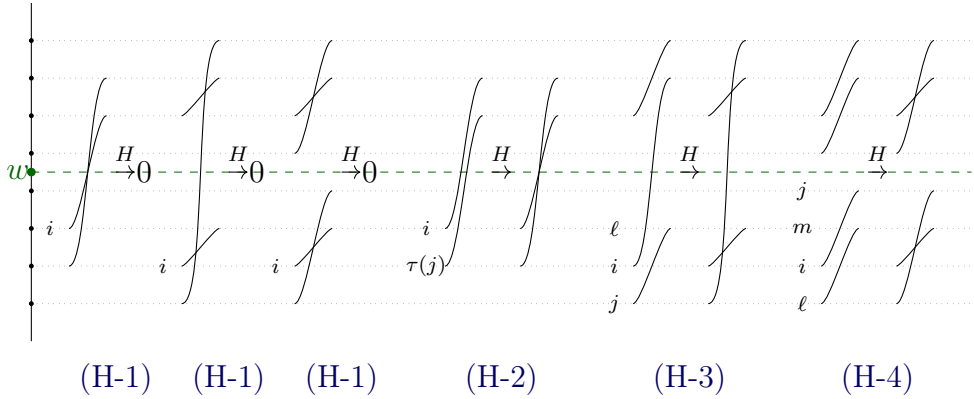

  We define a map $H\co \wt{Q}\to\wt{Q}$ so that 
  \begin{equation}\label{eq:H-is-htpy}
    \bdy\circ H+H\circ\bdy = \Id.
  \end{equation}
  Consider the first point $i\in\{1,\dots,8k\}$ so that elements of $\wt{Q}$ have multiplicity $2$ above $i$ and $1$ below $i$. Given a state $\wt{a}\in\wt{Q}$, let $\xi$ and $\xi'$ be the two strands that cross the point $i+1/2$, so that $\xi$ starts at $i$ and $\xi'$ starts at a point below $i$ and ends at a point above $i$. Observe that no strand starts or ends between the initial points of $\xi$ and $\xi'$. Define $H(\wt{a})$ in cases as follows:
  \begin{enumerate}[label=(H-\arabic*)]
    \item\label{item:H-cross} If $\xi$ and $\xi'$ cross, define $H(\wt{a})=0$.
    \item\label{item:H-exchanged} If $\xi=[i,j]$ and $\xi'=[8k+1-j,8k+1-i]=[\tau(j),\tau(i)]$ (so $\xi$ and $\xi'$ do not cross, but are exchanged by $\tau$), let $H(\wt{a})$ be the result of replacing $\{\xi,\xi'\}$ with $\{[i,8k+1-i],[8k+1-j,j]\}$ (i.e., making them cross). (Note that $8k+1-j<i$, from how we chose $\xi$ and $\xi'$.)
    \item\label{item:H-one-fixed} If $\xi=[i,8k+1-i]=[i,\tau(i)]$ and $\xi'=[j,\ell]$, since $\xi$ and $\xi'$ do not cross we have $\ell<8k+1-i<8k+1-j$. Thus, $a$ also contains $\xi''=[\tau(\ell),\tau(j)]$. Let $H(\wt{a})$ be the result of replacing $\{\xi,\xi',\xi''\}$ by $\{[j,\tau(j)],[i,\ell],[\tau(\ell),\tau(i)]\}$. That is, we introduce an equivariant pair of crossings.
    \item\label{item:H-none-fixed} Otherwise, $\xi=[i,j]$, $\xi'=[\ell,m]$, and $a$ also contains the chords $[\tau(j),\tau(i)]$ and $[\tau(m),\tau(\ell)]$. Let $H(\wt{a})$ be the result of replacing these four chords by the chords $\{[i,m],[\ell,j],[\tau(j),\tau(\ell)],[\tau(m),\tau(i)]\}$.
  \end{enumerate}
  See Figure~\ref{fig:H}.

  We verify Equation~\eqref{eq:H-is-htpy} by a case analysis. First, suppose that $H(\wt{a})=0$, as in Case~\ref{item:H-cross} of the definition of $H$. There is exactly one term in $d(a)$ of type~\ref{item:AZ-1}--\ref{item:AZ-4} that results in $\xi$ and $\xi'$ not crossing. Indeed, none of these terms creates a double crossing, because no strands start between the initial points of $\xi$ and $\xi'$; and type~\ref{item:AZ-5} does not contribute, because in that case the bottom two strands do not cross. So, there is exactly one term that contributes to $H\circ\bdy$, and $(H\circ\bdy)(\wt{a})=\wt{a}$.

  Next, suppose that $\xi$ and $\xi'$ are as in Case~\ref{item:H-exchanged} of the definition of $H$. Terms in $\bdy\circ H$ and $H\circ\bdy$ not involving these strands at all cancel in pairs. (This uses the fact that no strands start between the initial points of $\xi$ and $\xi'$, to deduce that none of these terms involve double crossings with $\xi$ or $\xi'$ or their images under $H$.) Applying $H$ and then resolving the new crossing gives the identity map. The remaining terms where we apply this type of $H$ and then a differential are of the following sorts:
  \begin{itemize}
  \item Performing $H$ and then a type~\ref{item:AZ-1} differential (not at the new crossing) cancels with performing a type~\ref{item:AZ-3} differential and then a case~\ref{item:H-one-fixed} $H$. (Note that the type~\ref{item:AZ-1} differential leaves the outermost strand unchanged, by the double crossing rule.)
  \item Performing $H$ and then a type~\ref{item:AZ-4} differential (not involving the outermost strand) cancel with performing a type~\ref{item:AZ-5} differential and then a case~\ref{item:H-one-fixed} $H$. (Type~\ref{item:AZ-4} differentials involving the outermost strand do not occur.)
  \end{itemize}
  Other than the identity, there are no nonzero terms coming from applying a differential and then a type~\ref{item:H-exchanged} $H$.

  Next, suppose that $\xi$ and $\xi'$ are as in Case~\ref{item:H-one-fixed} of the definition of $H$. Applying $H$ and then a type~\ref{item:AZ-4} differential gives the identity map. Other terms in $\bdy\circ H$ (for this case of $H$) are:
  \begin{itemize}
  \item A case~\ref{item:H-one-fixed} $H$ followed by a type~\ref{item:AZ-1} differential. These cancel with applying a type~\ref{item:AZ-1} differential and then a case~\ref{item:H-none-fixed} $H$.
  \item A case~\ref{item:H-one-fixed} $H$ followed by a type~\ref{item:AZ-2} differential. These cancel with applying a type~\ref{item:AZ-2} differential and then a case~\ref{item:H-one-fixed} $H$.
  \item A case~\ref{item:H-one-fixed} $H$ followed by a type~\ref{item:AZ-4} differential. These cancel with applying a type~\ref{item:AZ-4} differential and then a case~\ref{item:H-one-fixed} $H$.
  \item A case~\ref{item:H-one-fixed} $H$ followed by a type~\ref{item:AZ-5} differential. These cancel with applying a type~\ref{item:AZ-5} differential and then a case~\ref{item:H-one-fixed} $H$, where the~\ref{item:H-one-fixed} does not use either of the symmetric strands produced by the type~\ref{item:AZ-5} differential.
  \end{itemize}
  This also accounts for all the terms in $H\circ\bdy$ for this case of $H$, except for type~\ref{item:AZ-5} differentials and then case~\ref{item:H-one-fixed} using the longer of the two symmetric strands. Those terms canceled terms from case~\ref{item:H-exchanged}, above. (There is also the type~\ref{item:AZ-4} differential followed by $H$ on the resulting three strands giving the identity map, from case~\ref{item:H-cross}.)
  
  Finally, suppose that $\xi$ and $\xi'$ are as in Case~\ref{item:H-none-fixed} of the definition of $H$. Applying $H$ and then a type~\ref{item:AZ-2} differential to the resulting four strands gives the identity map term. Other terms in $\bdy\circ H$ (for this case of $H$) are:
  \begin{itemize}
  \item A case~\ref{item:H-none-fixed} $H$ followed by a type~\ref{item:AZ-2} differential using two of these four strands. These cancel with type~\ref{item:AZ-2} differentials followed by case~\ref{item:H-none-fixed} $H$.
  \item A case~\ref{item:H-none-fixed} $H$ followed by a type~\ref{item:AZ-3} differential. These cancel with type~\ref{item:AZ-3} differentials followed by case~\ref{item:H-none-fixed} $H$.
  \item A case~\ref{item:H-none-fixed} $H$ followed by a type~\ref{item:AZ-5} differential. These cancel with type~\ref{item:AZ-5} differentials followed by case~\ref{item:H-none-fixed} $H$.
  \end{itemize}
  The other terms in this case of $H\circ\bdy$ are the identity from a type~\ref{item:AZ-2} differential followed by $H$; a type~\ref{item:AZ-1} differential followed by a type~\ref{item:H-none-fixed} $H$, which canceled with terms in a type~\ref{item:H-one-fixed} $H$ followed by a type~\ref{item:AZ-1} differential; and a type~\ref{item:AZ-4} differential followed by a type~\ref{item:H-none-fixed} $H$, which canceled with terms in a type~\ref{item:H-one-fixed} $H$ followed by a type~\ref{item:AZ-4} differential.

  This completes the verification of Equation~\eqref{eq:H-is-htpy}, and hence of the proof of the lemma.
\end{proof}

\begin{corollary}\label{cor:AZ-simplest-model}
  Let $\PMC$ be a real pointed matched circle so that the real surface
  $(F(\PMC),\tau)$ has $F(\PMC)/\tau$ orientable. In particular, $\PMC$ is the connected sum $\PMC'\#(-\PMC')$ of a pointed matched circle with its mirror, and there is an embedding $\Alg(\PMC')\otimes\Alg(-\PMC')\hookrightarrow \Alg(\PMC)$. Let $r\co \Alg(\PMC')\to\Alg(-\PMC')$ be the anti-isomorphism induced by reflection.
  Then
  $\CFDRa(\AZ(\PMC))$ is homotopy equivalent to a type $D$ structure
  with generators in bijection with multiplicity-1 strands diagrams in $\Alg(\PMC')$. Moreover, writing $[a\otimes ra]$ for the generator corresponding to $a$, the left idempotent of $[a\otimes ra]$ is $\iota'\otimes r\iota'$, where $\iota'$ is the complementary idempotent to the left idempotent of $a$; and the differential is given by 
  \[
  \delta^1([a\otimes ra])=\sum_{b\in\bdy(a)} 1\otimes[b\otimes rb] + \!\!\!\!\sum_{\text{chords }\rho\in\Alg(\PMC')}\!\!\!\! (1\otimes r\rho)\otimes [a\rho\otimes (r\rho)(ra)]+(\rho\otimes 1)\otimes [\rho a\otimes (ra)(r\rho)].
  \]
\end{corollary}
\begin{proof}
  Consider the set of generators $\wt{a}$ of $\CFDRa(\AZ(\PMC))$ so that some strand in $a$ crosses $w$ (or, equivalently, there is a strand $[i,j]$ in $a$ with $i\leq 2k<j$). We claim that there are an even number of strands in $a$ crossing $w$. Indeed, suppose the left idempotent of $a$ has $2m$ horizontal strands below $w$ (i.e., $m$ pairs of such strands), and $a$ has $n$ moving strands which cross $w$. Then the left idempotent of $a$ has $4k-2m$ horizontal strands above $w$, so the right idempotent of $a$ has $4k-2m+2n$ horizontal strands above $w$. Since $a$ is fixed by $\tau$, $4k-2m+2n=2m$, so $n$ is even, as claimed. Thus, if we quotient by the sub-type-$D$-structure from Lemma~\ref{lem:mult-2-contractible}, we are left with a type $D$ structure generated by symmetric strands diagrams which do not cross $w$, that is, the set of generators from the statement of the corollary. The differential on the quotient only includes terms which start and end at such generators, which again are exactly the ones in the lemma statement (from rectangle pairs and the two types of half-strip pairs).
\end{proof}

\begin{example}
  Figure~\ref{fig:CFDR-AZ-2-split-small} shows the model from Corollary~\ref{cor:AZ-simplest-model} for $\CFDRa(\AZ(\PMC))$ for $\PMC$ the split genus $2$ pointed matched circle (cf.\ Example~\ref{eg:split-genus-2}). Figure~\ref{fig:CFDR-AZ-2-antipodal-small} shows the model from Lemma~\ref{lem:mult-2-contractible} for $\CFDRa(\AZ(\PMC))$ where $\PMC$ is the antipodal genus $2$ pointed matched circle; note that Corollary~\ref{cor:AZ-simplest-model} does not apply in this case, but the answer is still attractively simple.
  \begin{figure}
  \centering
  \includegraphics[alt={Small model for CFDR(AZ(PMC)) for the split genus 2 pointed matched circle}]{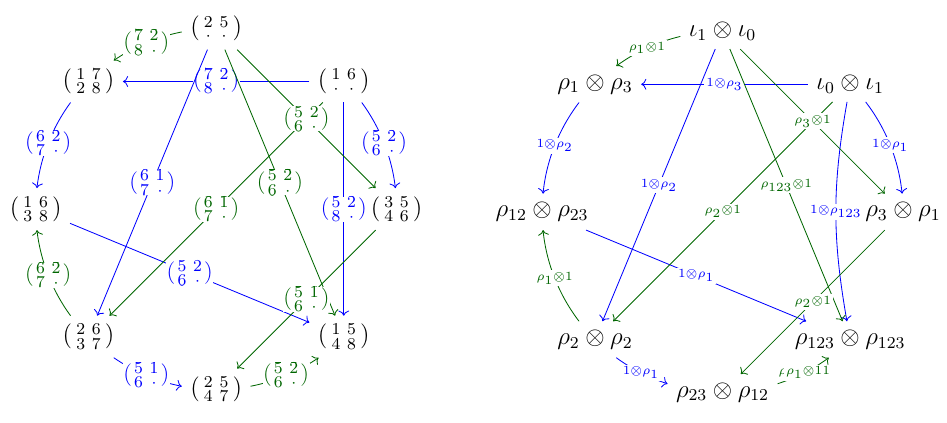}
  \caption{\textbf{The small model for the complex $\CFDRa(\AZ(\PMC))$ for the split, genus $2$ pointed matched circle.} Left: Conventions are as in Figure~\ref{fig:CFDR-AZ-2-split}. Right: the same model, but in terms of the elements $a\otimes ra$ and $\rho\otimes 1$ or $1\otimes r\rho$ of the statement of Corollary~\ref{cor:AZ-simplest-model}.}
  \label{fig:CFDR-AZ-2-split-small}
  \end{figure}

  \begin{figure}
  \centering
  \includegraphics[alt={Small model for CFDR(AZ(PMC)) for the antipodal genus 2 pointed matched circle}]{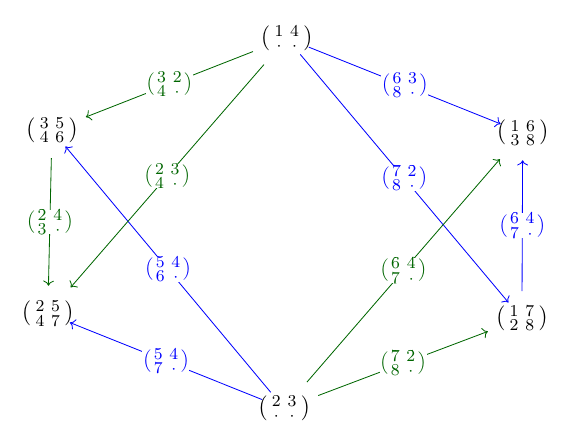}
  \caption{\textbf{The small model for the complex $\CFDRa(\AZ(\PMC))$ for the antipodal, genus $2$ pointed matched circle.} Conventions are as in Figure~\ref{fig:CFDR-AZ-2-split}.}
  \label{fig:CFDR-AZ-2-antipodal-small}
  \end{figure}
\end{example}

\section{Pairing theorem}\label{sec:pairing}
In this section, we give two proofs of the pairing theorem for the tensor product $\CFAa(\HD')\DT\CFDRa(\HD,\tau)$. First, we give a self-contained proof via nice diagrams, assuming $\HD'$ and $\HD$ are nice. Then we explain how the time dilation proof of the pairing theorem extends to the real case; to avoid repeating a large amount of verbiage, that argument is meant to be read alongside the proof in the unreal case.

\subsection{The pairing theorem via nice diagrams}\label{sec:nice-pairing}
\begin{theorem}\label{thm:nice-pairing} 
  Let $(\HD,\tau)$ be a real nice bordered Heegaard diagram with
  boundary $(-\PMC)\amalg \PMC^\beta$, and $\HD'$ a nice arced bordered
  Heegaard diagram with boundary $(-\PMC')\amalg \PMC$. Then there is a
  homotopy equivalence of type $D$ structures
  \[
  \lsup{\Alg(\PMC')}\CFDRa(\HD'\cup\HD\cup(-\HD')^\beta,\wt{\tau})\simeq 
  \lsup{\Alg(\PMC')}\CFDAa(\HD')_{\Alg(\PMC)}\DT\lsup{\Alg(\PMC)}\CFDRa(\HD,\tau).
  \]
\end{theorem}
\begin{proof}
  For definiteness, assume that $\HD$ is negative real nice. The identification of generators is clear. We first analyze rigid holomorphic curves in $\HD'\cup\HD\cup(-\HD')^\beta$ and show that they appear in the differential on the tensor product $\CFDAa(\HD')\DT\CFDRa(\HD,\tau)$; then we show that no other terms appear in that differential.

  There are three kinds of rigid holomorphic curves in $\HD'\cup\HD\cup(-\HD')^\beta$:
  \begin{enumerate}[label=(\arabic*)]
    \item\label{item:pairing-proof-case-1} Holomorphic curves as in Proposition~\ref{prop:classify-domains} where the domain lies entirely in $\HD'\amalg (-\HD')^\beta$ (and the components of the curve over $\HD$ are trivial strips).  
    \item\label{item:pairing-proof-case-2} Holomorphic curves as in Proposition~\ref{prop:classify-domains} (or Figure~\ref{fig:real-nice-curves}) where the domain lies entirely in $\HD$ (and the components of the curve over $\HD'\amalg (-\HD')^\beta$ are trivial strips).  
    \item\label{item:pairing-proof-case-3} Holomorphic curves as in Proposition~\ref{prop:classify-domains} where the domain intersects all three of $\HD$, $\HD'$, and $(-\HD')^\beta$ nontrivially.
  \end{enumerate}

  In Case~\ref{item:pairing-proof-case-1}, the curve is a bigon pair, rectangle pair, or half-strip pair (since those are the only rigid domains with two connected components). So, the component in $\HD'$ contributes to the operation $\delta^1_1$ on $\CFDAa(\HD')$, and thus contributes to the differential on the tensor product via the operation
  \begin{equation}
  \label{eq:pair-proof-tens-diag-1}
  \mathcenter{\begin{tikzpicture}[alt={Diagram showing a delta-one-one operation}]
    \node at (0,-1/2) (deltanode) {$\delta^1_1$};
    \draw[->] (0,0) to (deltanode);
    \draw[->] (deltanode) to (0,-1); 
    \draw[->] (.5,0) to (.5,-1); 
    \draw[->] (deltanode) to (-1/2,-1);
  \end{tikzpicture}.}
  \end{equation}

  In Case~\ref{item:pairing-proof-case-2}, by definition, the domain is one that contributes to the differential on $\CFDRa(\HD)$ with coefficient $1$. So, this contributes to the differential on the tensor product via the operation
  \begin{equation}
  \label{eq:pair-proof-tens-diag-2}
  \mathcenter{\begin{tikzpicture}[alt={Diagram showing delta-one-two composed with delta-one}]
    \node at (0,-1.3) (deltanode) {$\delta^1_2$};
    \node at (1,-.6) (dn2) {$\delta^1$};
    \draw[->] (0,0) to (deltanode);
    \draw[->] (deltanode) to (0,-2); 
    \draw[->] (1,0) to (dn2);
    \draw[->] (deltanode) to node[above]{\lab{1}} (-1,-2);
    \draw[->] (dn2) to node[above]{\lab{1}} (deltanode);
    \draw[->] (dn2) to (1,-2);
  \end{tikzpicture}.}
  \end{equation}

  In Case~\ref{item:pairing-proof-case-3}, we have a holomorphic curve $u\co S\to (\Sigma'\cup\Sigma\cup(-\Sigma'))\times[0,1]\times\RR$ of one of the types in Proposition~\ref{prop:classify-domains} (Figure~\ref{fig:real-nice-curves}). Write $\bdy\Sigma=Z\amalg Z^\beta$. Then (perhaps after perturbing the almost complex structure to be generic) the preimages $A = (\pi_\Sigma\circ u)^{-1}(Z)$ and $B = (\pi_\Sigma\circ u)^{-1}(Z^\beta)$ satisfy:
  \begin{itemize}
  \item $A$ and $B$ are 1-dimensional submanifolds of $S$ (with boundary),
  \item $A\cap B=\emptyset$,
  \item $\bdy A$ is contained in the part of $\bdy S$ mapped to the $\alpha$-curves, and $\bdy B$ in the part of $\bdy S$ mapped to the $\beta$-curves, 
  \item $A$ and $B$ are disjoint from the preimage of the fixed set in $\Sigma$, and
  \item $A\cup B$ separates $S$ into $S_A\cup S'\cup S_B$, each of which contains at least one $x$-corner and one $y$-corner. (This follows from the fact that the $\alpha$-curves in a bordered Heegaard diagram are homologically linearly independent.)
  \end{itemize}
  In most cases, no such 1-manifolds $A$ and $B$ exist. (For example,
  for a bigon, any arc $A$ with boundary on the $\alpha$-boundary of $S$
  splits off a component with no $x$ or $y$-punctures.) Indeed, the only
  cases in which such $A$ and $B$ exist are compact hexagons, octagons, and
  rectangle pairs. In each case, there is essentially one way to
  decompose the curve: as a half-strip in each of $\HD'$ and
  $(-\HD')^\beta$, and a noncompact hexagon, noncompact octagon, or
  half-strip pair in $\HD$. In particular, the part of the curve in
  $\HD'$ contributes an operation of the form $\delta^1_2(\x',\rho)=\y'$
  to $\CFDAa(\HD')$, and the part in $\HD$ contributes an operation
  $\delta^1(\x)=\rho\otimes \y$. Taken together, the domain contributes
  a term $1\otimes(\y'\otimes \y)$ in $\delta^1(\x'\otimes \x)$ on the
  glued diagram, as
  \begin{equation}
  \label{eq:pair-proof-tens-diag-3}
  \mathcenter{\begin{tikzpicture}[alt={Diagram showing paired operations in the pairing theorem proof}]
    \node at (0,-1.3) (deltanode) {$\delta^1_2$};
    \node at (1,-.6) (dn2) {$\delta^1$};
    \draw[->] (0,0) to node[left]{\lab{\x'}} (deltanode);
    \draw[->] (deltanode) to node[left]{\lab{\y'}} (0,-2); 
    \draw[->] (1,0) to node[left]{\lab{\x}} (dn2);
    \draw[->] (deltanode) to node[above]{\lab{1}} (-1,-2);
    \draw[->] (dn2) to node[above]{\lab{\rho}} (deltanode);
    \draw[->] (dn2) to node[left]{\lab{\y}} (1,-2);
  \end{tikzpicture}.}
  \end{equation}

  So, we have proved that every term in the differential on $\CFDRa(\HD'\cup\HD\cup(-\HD')^\beta,\wt{\tau})$ is also a term in the differential on the tensor product. Conversely, since $\HD'$ is nice, the only nontrivial operations on $\CFDAa(\HD')$ are $\delta^1_1$ and $\delta^1_2$, and these come from half-strips with boundary on $\bdy_L\HD'$ and provincial bigons and rectangles, for $\delta^1_1$, and the map $\delta^1_2(\x,1)=1\otimes \x$ from the definition of $\CFDAa(\HD')$ as strictly unital and half-strips with boundary on $\bdy_R\HD'$ for $\delta^1_2$. In the latter cases, the algebra input must come from a rigid curve in $\CFDRa(\HD)$. So, the only terms in the differential on the tensor product are the ones accounted for by Formulas~\eqref{eq:pair-proof-tens-diag-1}--\eqref{eq:pair-proof-tens-diag-3}, above.
\end{proof}

\subsection{General pairing theorem via time dilation}
In this section, we observe that the time dilation proof of the pairing theorem also carries over to the real case. In particular, this proves a version of Theorem~\ref{thm:nice-pairing} without the requirement that the diagrams be nice:

\begin{theorem}\label{thm:gen-pairing}
  Let $(\HD,\tau)$ be a real bordered Heegaard diagram with
  boundary $(-\PMC)\amalg \PMC^\beta$, and $\HD'$ an arced bordered
  Heegaard diagram with boundary $(-\PMC')\amalg \PMC$. Then there is a
  homotopy equivalence of type $D$ structures
  \[
  \lsup{\Alg(\PMC')}\CFDRa(\HD'\cup\HD\cup(-\HD')^\beta,\wt{\tau})\simeq 
  \lsup{\Alg(\PMC')}\CFDAa(\HD')_{\Alg(\PMC)}\DT\lsup{\Alg(\PMC)}\CFDRa(\HD,\tau).
  \]
\end{theorem}

\begin{corollary}\label{cor:pairing-to-closed}
  Suppose $(Y'\cup_{F(\PMC)}Y\cup_{F(\PMC)}Y',\wt{\tau})$ is a closed real 3-manifold decomposed into a real bordered $3$-manifold $(Y,\tau)$ and a pair of bordered $3$-manifold $Y'$ exchanged by $\wt{\tau}$. 
  Then
  \[
  \CFRa(Y'\cup_{F(\PMC)}Y\cup_{F(\PMC)}Y',\wt{\tau})\simeq \CFAa(Y')_{\Alg(\PMC)}\DT\lsup{\Alg(\PMC)}\CFDRa(Y,\tau).
  \]
\end{corollary}

\begin{proof}
  We recall the pairing proof in the unreal case, and note which arguments change in the real case. To minimize confusion, we will prove Corollary~\ref{cor:pairing-to-closed} (where $\bdy Y'$ is connected), rather than the more general Theorem~\ref{thm:gen-pairing}.

  The proof starts by stretching the neck along $\bdy \HD$~\cite[Section 9.1]{LOT1}; now we do so while respecting the involution $\tau$ (i.e., using the same conformal parameters on the two components of $\bdy\HD$). Holomorphic curves degenerate to matched holomorphic curves (cf.~\cite[Definition 9.2]{LOT1}). Since we are now splitting the surface into three components, not two, we obtain matched triples, instead of pairs: holomorphic curves $(u_1,u_2,u_3)$ with $u_1$ in $\Sigma'\times[0,1]\times\RR$, $u_2$ a real curve in $\Sigma\times[0,1]\times\RR$, and $u_3=\tau(u_1)$ in $(-\Sigma')^\beta\times[0,1]\times\RR$. The matched expected dimension formula~\cite[Equation~(9.3)]{LOT1} is replaced by
  \begin{align*}
    \ind^R&(B_1,S_1;B_2,S_2)= \ind(B_1,S_1,P_1)+\ind^R(B_2,S_2,P_2)-m\\
    &= g_1+\OneHalf g_2+2e(B_1)+e(B_2)-\chi(S_1)-\OneHalf\chi(S_2)+m+\OneQuarter\bigl(\sigma(\alphas,\y)-\sigma(\alphas,\x)\bigr). 
  \end{align*}
  (Really, there are three surfaces $S_1,S_2,S_3=\eta(S_1)$ and three domains $B_1,B_2,B_3=-\tau_*(B_1)$, and the formula looks a little more symmetric in these terms. The partition $P_2$ is essentially two copies of $P$, as in Proposition~\ref{prop:real-index}, and $m=|P_1|=\OneHalf|P_2|$.)
  The moduli space of real matched triples is transversely cut out by a simple extension of the unreal case~\cite[Proposition 9.4]{LOT1}. The curves $u_i$ are strongly boundary monotone~\cite[Lemma 9.5]{LOT1}. The count of the matched moduli space agrees with the counts of holomorphic curves in $\HD'\cup\HD\cup(-\HD')$ for appropriate almost complex structures by an analogous argument to the unreal case~\cite[Proposition 9.6]{LOT1}, but gluing respecting the real involution. As in the unreal case, when stretching the neck, embedded curves converge to embedded curves, and every pair of Reeb chords at the same height are nested or disjoint~\cite[Lemma 9.8]{LOT1}. (That proof uses the fact that the homology class determines the Euler characteristic for embedded curves which, as noted in the proof of Proposition~\ref{prop:real-index}, holds without change in the real case.)

  Next, we deform the matching condition $\ev(u_1)=\ev(u_2)=\ev(u_3)$ to $\ev(u_1)=T\ev(u_2)=\ev(u_3)$, to obtain the moduli space of $T$-matched pairs. Transversality (for generic $T$, cf.~\cite[Lemma 4.14]{LOT1}) follows as in the previous paragraph. The next result is a classification of ends of 1-dimensional moduli spaces~\cite[Proposition 9.17]{LOT1}, via the source-dependent index formula. The claim, that these correspond to 2-story curves or degenerations of a single split or join component, still holds (except that components at $e\infty$ come in pairs, exchanged by $\tau$). Running briefly through that part of the argument, by Formula~\eqref{eq:index-source} we have
  \[
  2 = g_1+\OneHalf g_2-\chi(S_1)-\OneHalf\chi(S_2)+e(B)+m_0,
  \]
  where $m_0$ is the number of east punctures of $S_1$ (which is half the number of east punctures of $S_2$).
  Moreover, if $T$ denotes the components at $e\infty$ between $u_1$ and $u_2$, then
  \[
  \chi(S_1)+\OneHalf\chi(S_2)-m_0=\chi(S_1')+\OneHalf\chi(S_2')+\chi(T)-m_1-\OneHalf m_2.
  \]
  (Here $T$ refers only to the components between $S_1$ and $S_2$, not those between $S_2$ and $S_3=\eta(S_1)$, but we are using $m_2$ to denote the count of $e$ punctures on both sides of $S_2'$.)
  Thus, with $k$ the number of components of $T$, an upper bound for the dimension of the space of (real) limit curves for generic $J$ is
  \begin{align*}
  \ind(B_1,S_1',P_1)+&\ind^R(B_2,S_2',P_2)-k-1\\
  &=g_1+\OneHalf g_2-\chi(S_1')-\OneHalf\chi(S_2')+e(B)+\OneHalf k-1+\OneQuarter\bigl(\sigma(\alphas,\y)-\sigma(\alphas,\x)\bigr)\\
  &=\chi(S_1)-\chi(S_1')+\OneHalf\bigl(\chi(S_2)-\chi(S'_2)\bigr)-m_0+\OneHalf k+1\\
  &=(\chi(T)-k)+(k-m_1)+\OneHalf(k-m_2)+1.
  \end{align*}
  This restricts the cases as in the unreal proof. (Note that, because of the symmetry, $m_2$ is even and the case $m_2=k+1$ is replaced by $m_2=k+2$, corresponding to one new puncture on each side of $u_2$.)

  The proof that combs in the ends of these moduli spaces with curves at $e\infty$ cancel in pairs is the same as in the unreal case~\cite[Proposition 9.18]{LOT1}, but gluing respecting the real involution (as usual). This implies that the $T$-matched moduli spaces induce a chain complex, as in the unreal case~\cite[Proposition 9.19]{LOT1}.

  The next step is to consider moduli spaces interpolating between different values of $T$, the so-called $\psi$-matched moduli spaces. There is again a classification of ends of the 1-dimensional moduli spaces~\cite[Lemma 9.21]{LOT1} (this time, with proof omitted); the (omitted) proof adapts to the real case the same way as above (with the same index computations). In particular, different values of $T$ give homotopy equivalent chain complexes, as in the unreal case~\cite[Proposition 9.22]{LOT1}.

  The next step is to send $T\to\infty$. In this limit, $T$-matched curves converge to (the evident real analogue of) ideal matched holomorphic combs; since the real moduli spaces are subspaces of the unreal ones, and the statement did not use transversality or the index, the unreal case~\cite[Proposition 9.26]{LOT1} implies the real case. We then introduce simple ideal-matched curves~\cite[Definition 9.28]{LOT1}; in the real case, these are triples $(U_1,U_2,U_3)$ where $U_1$ has at most one story, $U_3=\tau(U_1)$, $U_2$ is a real comb with perhaps many stories; and either:
  \begin{itemize}
  \item[(1)] $U_1$ and $U_3$ are trivial while $U_2$ is a provincial holomorphic curve (no $e\infty$ punctures) with one story and real index $1$, or
  \item[(1$'$)] $U_2$ is trivial, and $U_1$ is a provincial holomorphic curve with one story and index $1$ (and $U_3=\tau(U_1)$), or
  \item[(2)] each story of $U_2$ has at least one $e$ puncture and real index $1$ (with the discrete partition), and $U_1$ (and hence $U_3$) has index $1$.
  \end{itemize}
  Simple ideal-matched curves have index $1$ as real matched curves, by a simple modification of the unreal case~\cite[Lemma 9.29]{LOT1}. Conversely, we adapt the index computation in the unreal case~\cite[Lemma 9.30]{LOT1} to show that they are the only ideal matched combs with index $1$. With the obvious extension of the notation in that lemma, an upper bound for the dimension of the ideal-matched moduli space is now
  \begin{align*}
  \ind&(B_1,S_1',P_1)+\ind^R(B_2,S_2',P_2')-2\\
    &=g_1+\OneHalf g_2+2e(B_1)+e(B_2)-\chi(S_1')-\OneHalf\chi(S_2')+\OneQuarter\bigl(\sigma(\alphas,\y)-\sigma(\alphas,\x)\bigr)+k_2-1\\
    &=g_1+\OneHalf g_2+2e(B_1)+e(B_2)-\bigl(\chi(S_1)-\chi(T_1)+m_1\bigr)-\OneHalf\bigl(\chi(S_2)-2\chi(T_2)+2m_2\bigr)\\
    &\qquad\qquad\qquad\qquad+\OneQuarter\bigl(\sigma(\alphas,\y)-\sigma(\alphas,\x)\bigr)+k_2-1\\
    &=\bigl(\chi(T_1)-m_1\bigr)+\bigl(\chi(T_2)-m_0\bigr)+(k_2-m_2)+\bigl(\ind^R(B_1,S_1;B_2,S_2)-1\bigr).
  \end{align*}
  If $\ind^R(B_1,S_1;B_2,S_2)<1$, the sum is negative, so there are no ideal matched combs, and when $\ind^R(B_1,S_1;B_2,S_2)=1$ the three other terms must vanish, so $(U_1,U_2)$ is a simple ideal-matched curve.

  Finally, we introduce trimmed simple ideal-matched curves~\cite[Definition 9.31]{LOT1}, which again has an evident extension to the real case (with three combs, $w_1,w_2,w_3=\tau(w_1)$ instead of two). The technical lemma relating boundary monotonicity and the algebra in this setting~\cite[Lemma 9.32]{LOT1} does not use transversality or the expected dimensions, so carries over without change. The spine of a simple ideal-matched curve is a trimmed simple ideal-matched curve: the proof in the unreal case uses the relation between the expected dimensions and the grading in the algebra, in a similar way to ruling out non-composable collisions in the proof of Theorem~\ref{thm:CFDR-defined}. Again, this carries over to the real case because the terms $|\vec{\rho}|$ and $\iota(\vec{\rho})$ appear the same way in the unreal and real embedded index formulas (because in the real case, each chord appears twice, once on each boundary component of $\Sigma$).

  Trimmed curves can again be untrimmed in a combinatorially unique way~\cite[Lemmas 9.37 and 9.38]{LOT1} (no changes here). As in the unreal case, it is easy to see that there are finitely many trimmed simple ideal-matched curves in a given homology class~\cite[Lemma 9.39]{LOT1} (again, no changes). Finally, the proof that the counts of the $T$-matched moduli spaces for $T$ sufficiently large and of the trimmed simple ideal-matched moduli spaces agree is the same as in the unreal case, but where we perform the gluings respecting the involution (as usual). Indeed, the argument is exactly the same as the argument there, focusing on the fiber product of moduli spaces containing $(w_1,v)$ and $w_2$, with the components $(w_3=\tau(w_1),\tau(v))$ coming along for the ride.

  With these ingredients in place, the proof of the pairing theorem is formal, and identical to the unreal case~\cite[Section 9.1]{LOT1}.
\end{proof}

\section{Gradings}\label{sec:gradings}
\subsection{Review of the unreal case}
Recall that the bordered algebra $\Alg(\PMC)$ is graded by a noncommutative group. In fact, there are two versions: the \emph{big grading group} $\bigGroup(k)$ and the \emph{small grading group} $G(\PMC)$, which are central extensions
\begin{align*}
  \ZZ\to G(\PMC)\to H_1(F(\PMC))\\
  \ZZ\to \bigGroup(k)\to \ZZ^{4k-1}.
\end{align*}
Even though $\bigGroup$ depends only on $k$, it will be clearer to write it as $\bigGroup(\PMC)$, thinking of $\ZZ^{4k-1}$ as $H_1(Z\setminus\{z\},\CircPts)$. We can realize elements of $\bigGroup(\PMC)$ as certain pairs $(m;s)$ where $s\in H_1(Z\setminus\{z\},\CircPts)$ and $m\in\OneHalf\ZZ$~\cite[Section 3.3]{LOT1}. The grading of an algebra element $a$ with support $[a]$ is given by
\begin{equation}\label{eq:grb-on-alg}
  \grb(a)=(\iota(a);[a]).
\end{equation}
The image of the grading in $H_1(F(\PMC))$ or $H_1(Z\setminus\{z\},\CircPts)$ is the \emph{$\SpinC$-component} of the grading; the element of $\OneHalf\ZZ$ is the \emph{Maslov component}.

The bordered bimodule $\CFDAa(\HD)$, where $\bdy \HD=\PMC_L\amalg\PMC_R$, is graded by a set $S'(\HD)$ with a commuting left action by $\bigGroup(-\PMC_L)$ and right action by $\bigGroup(\PMC_R)$, restricting to the same action by $\ZZ$. Equivalently, $S'(\HD)$ has a right action by $\bigGroup(-\PMC_L)^\op\times_\ZZ\bigGroup(\PMC_R)$. The orbits of $S'(\HD)$ correspond to $\SpinC$-structures on $Y$. Fix a $\SpinC$-structure on $Y$ represented by a state $\x_0\in\Gen(\HD)$. Given a domain $B\in\pi_2(\x,\y)$, let 
\[
  g'(B)=(-e(B)-n_\x(B)-n_\y(B);\bdy_LB;\bdy_RB)\in \bigGroup(-\PMC_L)^\op\times_\ZZ\bigGroup(\PMC_R).
\]
The orbit of $S'(\HD)$ corresponding to $\x_0$ is given by 
\[
  \langle g'(P)\mid P\in\pi_2(\x_0,\x_0)\rangle\backslash\bigl(\bigGroup(-\PMC_L)^\op\times_\ZZ\bigGroup(\PMC_R)\bigr),
\]
a right $\bigGroup(-\PMC_L)^\op\times_\ZZ\bigGroup(\PMC_R)$-set,
and the grading of $\x$ is $\grb(\x)=g'(B)$ for any $B\in\pi_2(\x_0,\x)$~\cite[Section 6.5]{LOT2}.

To obtain a grading by $G(\PMC)$, we choose \emph{grading refinement data}, consisting of a basic idempotent $\iota_0$ and for every other basic idempotent $\iota$ and element $\psi(\iota)\in G'(\PMC)$ whose boundary is, in a suitable sense, $\iota-\iota_0\in H_0(\CircPts/M)$~\cite[Section 3.3.2]{LOT1}. Then, the refined grading of an algebra element $a=\iota_1 a\iota_2$ is given by 
\begin{equation}\label{eq:gr-on-alg}
  \gr(a)=\psi(\iota_1)\grb(a)\psi(\iota_2)^{-1}.
\end{equation}

With respect to the smaller group, the orbit of the grading set $S(\HD)$ for $\CFDAa(\HD)$ corresponding to the $\SpinC$ structure of $\x_0$ is given by
\begin{equation}\label{eq:gr-set-refined-unreal}
  \langle g'(P)\mid P\in\pi_2(\x_0,\x_0)\rangle\backslash \bigl(G(-\PMC_L)^\op\times_\ZZ G(\PMC_R)\bigr),
\end{equation}
a right $G(-\PMC_L)^\op\times_\ZZ G(\PMC_R)$-set.
Here, we choose the base idempotents $\iota_0$ for $\PMC_L$ and $\PMC_R$ to be the idempotents for the base state $\x_0$, and observe that each $g'(P)$ in fact lies in $G(\PMC)\subset\bigGroup(\PMC)$~\cite[Lemma 6.29]{LOT2}. The grading of a state $\x$ with left and right idempotents $\iota_L$ and $\iota_R$ is
\begin{equation}\label{eq:refined-gr}
  \gr(\x)=\psi_L(\iota_L)\grb(\x)\psi_R(\iota_R)^{-1}.
\end{equation}

\subsection{Grading real bordered modules}
Now, let $(\HD,\tau)$ be a real bordered Heegaard diagram, with $\bdy_L\HD=-\PMC$, so $\CFDRa(\HD,\tau)$ is a left module over $\Alg(\PMC)$. Given a real domain $B\in\pi_2^R(\x,\y)$, define
\[
g'_R(B)=\Bigl(-\OneHalf\bigl(e(B)+n_\x(B)+n_\y(B)\bigr)-\OneQuarter\bigl(\sigma(\alphas,\y)-\sigma(\alphas,\x)\bigr);\bdy_L B\Bigr).
\]
Fix a real state $\x_0\in\Gen_R(\HD)$; we define a grading set for the real states that are real $\SpinC$-equivalent to $\x_0$. Specifically, with respect to the big and small grading groups, this grading set is
\begin{align}
S'_R(\HD,\x_0)&= \bigGroup(\PMC)/\langle g'_R(P)\mid P\in\pi_2^R(\x_0,\x_0)\rangle \\
S_R(\HD,\x_0)&= G(\PMC)/\langle g'_R(P)\mid P\in\pi_2^R(\x_0,\x_0)\rangle,\label{eq:gr-set-refined-real}
\end{align}
which have left actions by $G'(\PMC)$ and $G(\PMC)$, respectively. Given a state $\x\in\Gen_R(\HD,\tau)$ which is real $\SpinC$ equivalent to $\x_0$, with left idempotent $\iota$, choose a real domain $B\in\pi_2^R(\x_0,\x)$ and define
\begin{align}
  \grb(\x)&=g'_R(B)\\
  \gr(\x)&=\psi(\iota)\grb(\x)=\psi(\iota)g'_R(B)\label{eq:refined-gr-real}
\end{align}
where $\psi$ is some chosen grading refinement data.

There are several properties to check:
\begin{proposition}\label{prop:real-gradings}
  These definitions have the following properties:
  \begin{enumerate}[label=(\arabic*)]
    \item\label{item:gpR-in-group} For any real domain $B\in\pi_2^R(\x,\y)$, the tuple $g'_R(B)$ lies in $\bigGroup(\PMC)$. Thus, $S'_R(\HD,\x_0)$ is well-defined and $\grb(\x)$ is an element of $S'_R(\HD,\x_0)$.
    \item\label{item:gpR-in-subgroup} The elements $g'_R(P)$ in fact lie in $G(\PMC)\subset \bigGroup(\PMC)$, so $S_R(\HD,\x_0)$ is well-defined. 
    \item\label{item:gp-additive} If $B_1\in\pi_2^R(\w,\x)$ and $B_2\in\pi_2^R(\x,\y)$, and $B_1*B_2\in\pi_2(\w,\y)$ is their concatenation, then 
    \[
    g'_R(B_1*B_2)=g'_R(B_1)g'_R(B_2).
    \]
    \item\label{item:grb-indept} The element $\grb(\x)$ is independent of the choice of real domain $B$.
    \item\label{item:gr-in-SR} The element $\gr(\x)$ lies in $S_R(\HD,\x_0)$.
    \item\label{item:gr-is-grading} The maps $\grb$ and $\gr$ define gradings on $\CFDRa(\HD,\tau)$.
  \end{enumerate}
\end{proposition}
\begin{proof}
  Recall that $\bigGroup(\PMC)$ consists of pairs $(m;h)$ where 
  \[
  m\equiv\OneQuarter\#\bigl(\text{parity changes in }h\bigr) \pmod{1}
  \]
  \cite[Definition 3.33]{LOT1}. To verify that $g'_R(B)$ lies in $\bigGroup(\PMC)$, we must check that it satisfies this congruence. We imitate the proof that $g'(B)$ has this property~\cite[Lemma 10.3]{LOT1}, and for brevity will assume the reader is following that proof.
  
  \begin{figure}
    \centering
    \includegraphics[alt={A local modification to make sigma positive}]{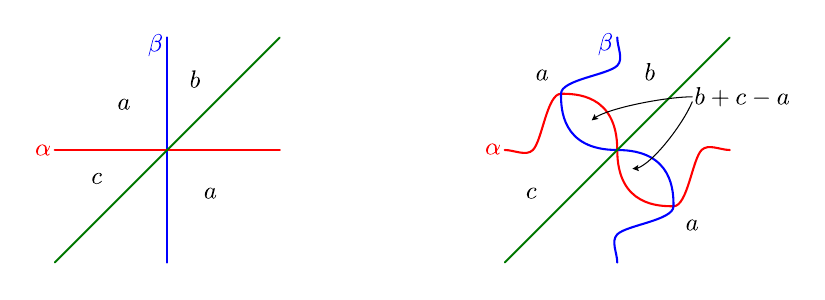}
    \caption{\textbf{A local modification to make $\sigma$ positive.} The diagram near a fixed intersection point before and after the modification are shown. The labels $a$, $b$, and $c$ in the left figure indicate the coefficients of a domain before the modification, and the labels in the right figure indicate the corresponding coefficients after the modification.}
    \label{fig:make-sigma-positive}
  \end{figure}

  We start with some reductions. First, we may assume that all the points in $\alphas\cap\betas$ in the fixed set have $\sigma=1$, as follows. For each fixed point with $\sigma=-1$, let $\HD'$ be the result of performing a small isotopy as in Figure~\ref{fig:make-sigma-positive}, introducing a new pair of bigons and two new non-fixed intersection points, and changing the fixed intersection point to have $\sigma=1$. Any domain in $\HD$ gives a domain in $\HD'$ with the same corners, as indicated in the figure. The boundary of the domain (and hence $\epsilon$) is unchanged. If the point under consideration is in both $\x$ and $\y$ or in neither $\x$ nor $\y$, then the $\sigma$-contribution to the grading is unchanged, and $n_\x(B)+n_\y(B)$ changes by an even integer. If the point is in one of $\x$ or $\y$, then the $\sigma$ contribution changes by $1/2$ and, in the notation from the figure, the new Euler measure and point measures are 
  \begin{align*}
    e(B')&=e(B)+(b+c)/2-a\\
    n_\x(B')+n_\y(B')&=n_\x(B)+n_\y(B)-a+(b+c)/2.
  \end{align*}
  So, the change to the grading is an integer plus $1/2+(b+c)/2$. Since this corner is in exactly one of $\x$ or $\y$, $b+c$ is odd, so the overall grading change is an integer.

  Second, we reduce to the case that an even number of the points in $\x$ (and $\y$) are fixed by $\tau$. For each component of the fixed set which intersects $\x$ (and hence $\y$) in an odd number of points, perform a fixed point stabilization as shown in Figure~\ref{fig:stabilize-near-fixed-set} (see also~\cite[Figure 3.5]{GM:real-HF}), with feet adjacent to this component. If the local multiplicity of $B$ at the feet is $a$, then this decreases the Euler measure by $2a$, increases $n_\x+n_\y$ by $2a$, and does not change any of the other terms in the grading formula. Of course, it also does not change $\bdy_L B$. So, it suffices to prove the result for the stabilized diagram. Thus, we may assume that $\x$ and $\y$ have an even number of points on each component of the fixed set.
  
  \begin{figure}
    \centering
    \includegraphics[alt={Stabilizing near the fixed set}]{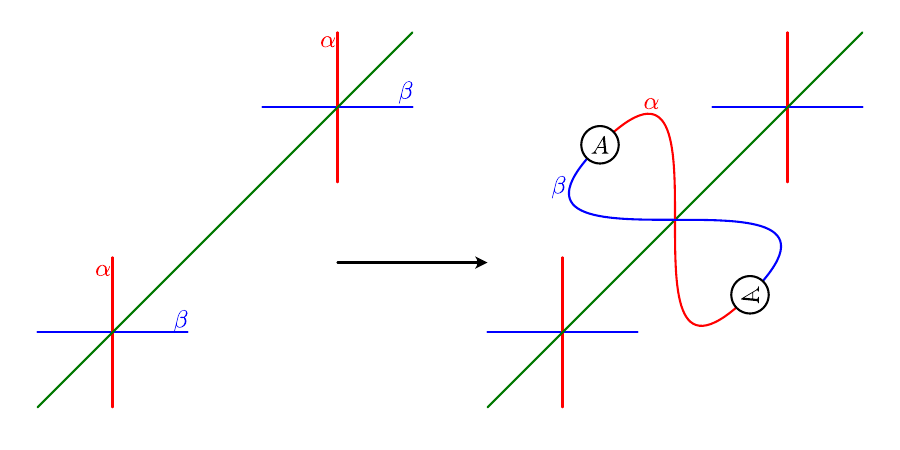}
    \caption{\textbf{Stabilizing near the fixed set.} A small neighborhood of part of a component of the fixed set is shown, before and after the stabilization. In Guth-Manolescu's terminology, this is a fixed point stabilization.}
    \label{fig:stabilize-near-fixed-set}
  \end{figure}

  Third, we reduce to the case that none of the points in $\x$ or $\y$ is fixed by $\tau$. If $\x$ contains some points on the fixed set, choose a pair of adjacent points in $\x$ on a component of the fixed set and perform an isotopy as in Figure~\ref{fig:make-gens-free}. In the new diagram, $\x$ is connected to another state $\x'$ by a rectangle, so that $\x'$ has two fewer points on the fixed set. In particular, the real domain in $B\in \pi_2(\x,\y)$ gives a real domain $B'\in\pi_2(\x',\y)$ with $\bdy_LB=\bdy_LB'$ and grading differing by $1$. Thus, it suffices to prove the result for $B'$ instead of $B$. Repeat this procedure until the state $\x$ has no fixed points, and repeat an analogous construction for $\y$ until $\y$ has no fixed points.

  \begin{figure}
    \centering
    \includegraphics[alt={Third modification near the fixed set}]{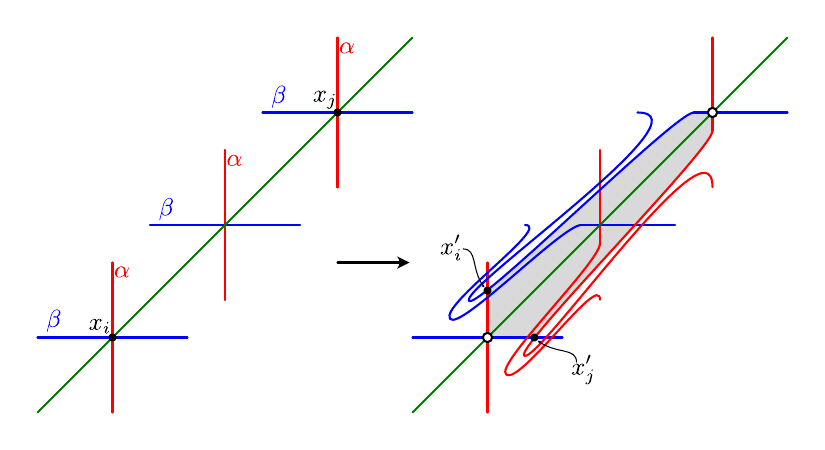}
    \caption{\textbf{Third modification near the fixed set.} A small neighborhood of a component of the fixed set is shown, with two fixed components of a state marked. After the isotopy, there is a real rectangle connecting this state to one with two fewer points on the fixed set.}
    \label{fig:make-gens-free}
  \end{figure}

  So, it suffices to prove the statement when $\tau\co \x\to \x$ and $\tau\co \y\to\y$ are free. In particular, the $\sigma$ terms in the definition of $g'_R$ vanish, so we will drop them from the discussion.

  Now, begin the proof as in the unreal case~\cite[Lemma 10.3]{LOT1}. Because our bordered Heegaard diagrams have two boundary components, $e([\Sigma])+n_\x([\Sigma])+n_\y([\Sigma])$ is even. Thus, 
  \[
    \OneHalf\bigl(e(B)+n_\x(B)+n_\y(B)\bigr)\equiv \OneHalf\bigl(e(\wt{B})+n_\x(\wt{B})+n_\y(\wt{B})\bigr)\pmod{1},  
  \]
  and it suffices to prove the result for $\wt{B}$ instead of $B$.

  Build the surface $F$ as in the unreal case. The surface $F$ inherits an involution $\tau_F$ which sends $R_i^{(j)}$ to $\tau(R_i)^{(m(R_i)-j)}$. That is, $\tau_F$ applies $\tau$ from the Heegaard surface and also reverses the indexing of the $R_i$. It is straightforward to check that $\tau_F$ respects the gluing relations.

  We have
  \[
  e(\wt{B})=e(F)=2e(F/\tau_F)
  \pmod{2}.
  \]
  Here, we view $F/\tau$ as a surface-with-corners as follows. Because $\tau$ does not fix any points in $\x\cup\y$, the fixed set of $\tau_F$ consists entirely of circles. So, the boundary of $F/\tau$ is identified with the boundary of $F$ disjoint union some circles; corners are the images of corners in $F$, with the same local models.
  
  Define $n_\pm$ as in the unreal case. Because of the symmetry $\tau$, $n_+$ and $n_-$ are both even; to remember that, write $n_\pm = 2n_{\pm}^\tau$ (so $n_\pm^\tau$ are counts of corners of $F/\tau_F$). When defining the $\delta_i$, consider only the $\alpha$-boundary of the Heegaard diagram. Then
  \[
  e(F/\tau)\equiv \frac{1}{4}\left(-n_+^\tau+n_-^\tau+\sum_{i=1}^{4k}|\delta_i|\right) \pmod{1}.
  \]
  Also,
  \begin{align*}
    n_\x(\wt{B})+n_\y(\wt{B}) &\equiv \sum_{x\in \x+\y}\bigl(\OneQuarter C(x)+\OneHalf h(x)\bigr) \pmod{2}\\
    &\equiv \OneQuarter(2n_+^\tau-2n_-^\tau)+\OneHalf\hspace{-2em}\sum_{x\in [(\x+\y)\cap(\alphas^a\cup\betas^a)]}\hspace{-2em}h(x)\pmod{2}.
  \end{align*}
  (For the second congruence, note that in the real case, for each $\alpha$-circle there are four points that contribute the same way to $h$: the two points in $\x$ and $\y$ on that $\alpha$-circle and the points in $\x$ and $\y$ on the corresponding $\beta$-circle.)
  Thus,
  \begin{align}
  \OneHalf\bigl(e(\wt{B})&+n_\x(\wt{B})+n_\y(\wt{B})\bigr)\nonumber\\
  &\equiv \frac{1}{4}\left(-n_+^\tau+n_-^\tau+\sum_{i=1}^{4k}|\delta_i|\right)+
  \OneQuarter(n_+^\tau-n_-^\tau)+\OneQuarter\hspace{-2em}\sum_{x\in [(\x+\y)\cap(\alphas^a\cup\betas^a)]}\hspace{-2em}h(x)
  \pmod{1}\nonumber\\
  &=\OneQuarter\sum_{i=1}^{4k}|\delta_i|
  +\OneHalf\hspace{-1em}\sum_{x\in (\x+\y)\cap\alphas^a}\hspace{-1em}h(x).\label{eq:gr-lem-penultimate}
  \end{align}
  In the last line, the $h$ sum only counts half-disks with boundary along the $\alpha$-arcs; the factor of $1/2$ instead of $1/4$ is that there are corresponding half-disks with boundary along the $\beta$-arcs.

  This last formula is now identical to the formula in the unreal case~\cite[p.~192]{LOT1}. So, the same case analysis implies that its fractional part agrees with $\epsilon(\bdy_LB)$, giving Point~\ref{item:gpR-in-group}.

  Point~\ref{item:gpR-in-subgroup} is easier. The subgroup $G(\PMC)\subset G'(\PMC)$ consists of pairs $(m;h)$ so that $\bdy h = -M_*\bdy(h)$, and this condition clearly holds for the boundaries of periodic domains (real or not).

  For Point~\ref{item:gp-additive}, recall that $g'(B)$ is additive in the sense that
  $g'(B_1*B_2)=g'(B_1)g'(B_2)$
  \cite[Lemma 10.4]{LOT1}. Equivalently,
  \begin{multline*}
  e(B_1*B_2)+n_\w(B_1*B_2)+n_\y(B_1*B_2)\\
  =e(B_1)+n_\w(B_1)+n_\x(B_1)+e(B_2)+n_\x(B_2)+n_\y(B_2)-L(\bdy_LB_1,\bdy_LB_2)-L(\bdy_RB_1,\bdy_RB_2),
  \end{multline*}
  where $L$ is the linking number~\cite[Formula 3.32]{LOT1}. Since the linking number satisfies $L(\bdy_LB_1,\bdy_LB_2)=L(\bdy_RB_1,\bdy_RB_2)$, this gives
  \begin{multline*}
  \OneHalf\bigl(e(B_1*B_2)+n_\w(B_1*B_2)+n_\y(B_1*B_2)\bigr)\\
  =\OneHalf\bigl(e(B_1)+n_\w(B_1)+n_\x(B_1)\bigr)+\OneHalf\bigl(e(B_2)+n_\x(B_2)+n_\y(B_2)\bigr)-L(\bdy_LB_1,\bdy_LB_2).
  \end{multline*}
  The $\sigma$ terms in the definition of $g'_R(B)$ are clearly additive, so $g'_R$ is also additive.

  Point~\ref{item:grb-indept} follows from Point~\ref{item:gp-additive}, since any other real domain $B'\in\pi_2^R(\x_0,\x)$ can be written as $B'=P*B$ where $P\in\pi_2^R(\x_0,\x_0)$. 

  Point~\ref{item:gr-in-SR} is straightforward, and is the same as in the unreal case. The projection of $\psi(\iota)g_R'(B)$ has boundary $\iota-\iota_0-\iota+\iota_0=0$ (because $\bdy_L B$ goes from $\iota_0$ to $\iota$ in $-Z_L$, hence from $\iota$ to $\iota_0$ in $Z_L$).

  Point~\ref{item:gr-is-grading} is again essentially the same as in the unreal case, but using the real index, Formula~\eqref{eq:index-embedded}, in place of the ordinary index. Suppose a pair $(B,\vec{\rho})$ contributes $a(\vec{\rho})\otimes \y$ to $\delta^1(\x)$. Then $\ind^R(B,\vec{\rho})=1$, so 
  \[
  g'_R(B)=\bigl(-1+|\vec\rho|+\iota(\vec{\rho});\bdy_LB\bigr)=(-1;0)\grb(\vec{\rho}).
  \]
  (For the second equality, see~\cite[Formula (10.21)]{LOT1}.) 
  It now follows from Point~\ref{item:gp-additive} that $\grb$ defines a grading. To see that $\gr$ also defines a grading, we have
  \begin{align*}
  \gr(a(\vec{\rho})\otimes \y)&=\gr(a(\vec{\rho}))\gr(\y)\\
  &=\psi(\iota_\x)\grb(a)\psi(\iota_\y)^{-1}\psi(\iota_\y)\grb(\y)\\
  &=\psi(\iota_\x)\grb(a)\grb(\y)\\
  &=\psi(\iota_\x)(-1;0)\grb(\x)\\
  &=(-1;0)\gr(\x).
  \end{align*}
  This concludes the proof.
\end{proof}

\subsection{Comparison with the unreal gradings}
Of course, $\Gen_R(\HD,\tau)\subset\Gen(\HD)$, so the set of real states inherits a grading from the (unreal) grading on $\Gen(\HD)$. The Maslov component of this grading is not well-behaved, because the real and unreal indices of holomorphic curves differ (arbitrarily), but the $\SpinC$ components of the grading on $\Gen(\HD)$ do behave well.

Write the boundary of $\HD$ as $(-\PMC)\cup\PMC^\beta$, and let $Y(\HD)$ be the 3-manifold specified by $\HD$. Considering only the $\SpinC$-components of the gradings, $\CFDRa(\HD,\tau)$ is graded by 
\[
  H_1(F(\PMC))/\bdy_L H_2(Y,\bdy Y)^{-\tau_*}
\]
where $\bdy_L$ is the composition $H_2(Y,\bdy Y)^{-\tau_*}\stackrel{\bdy}{\to} H_1(\bdy Y)=H_1(F(\PMC))\amalg H_1(F(\PMC))\to H_1(F(\PMC))$ where the second map is projection to the first $H_1(F(\PMC))$ summand. On the other hand, $\CFDDa(\HD)$ is graded by 
\[
\bigl[H_1(F(\PMC))\oplus H_1(F(\PMC))\bigr]/\bdy H_2(Y,\bdy Y).
\]

We reformulate the grading on $\CFDRa(\HD,\tau)$ to be closer to the grading on $\CFDDa(\HD)$. Inside $H_1(F(\PMC))\oplus H_1(F(\PMC))$ is an anti-diagonal $\nabla=\{(x,-x)\mid x\in H_1(F(\PMC))\}$, and 
\[
 H_1(F(\PMC))/\bdy_L H_2(Y,\bdy Y)^{-\tau_*} \stackrel{\cong}{\lra} \nabla / \bdy H_2(Y,\bdy Y)^{-\tau_*}.
\]
Observe that $\bdy H_2(Y,\bdy Y)^{-\tau_*} \subset [\bdy H_2(Y,\bdy Y)]\cap \nabla$, so we obtain a map
\begin{equation}\label{eq:real-gr-set-to-unreal}
  H_1(F(\PMC))/\bdy_L H_2(Y,\bdy Y)^{-\tau_*}\to \bigl[H_1(F(\PMC))\oplus H_1(F(\PMC))\bigr]/\bdy H_2(Y,\bdy Y)
\end{equation}
of grading sets.

To define a grading on $\Gen_R(\HD,\tau)$, we choose base states for each real $\SpinC$-equivalence class. Fixing a real $\SpinC$-equivalence class and a base state $\x_0$ for it, if $\x\in \Gen_R(\HD,\tau)$ is another state in the equivalence class, and we also use $\x_0$ to define the grading for $\Gen(\HD)$, then Formula~\eqref{eq:real-gr-set-to-unreal} respects the gradings.

Now, suppose $\x_0$ and $\y$ are $\SpinC$-equivalent, but not real $\SpinC$-equivalent. Then using the ordinary grading (and continuing to take only the $\SpinC$-components), $\gr(\y)\in \bigl[H_1(F(\PMC))\oplus H_1(F(\PMC))\bigr]/\bdy H_2(Y,\bdy Y)$. If $\Delta=\{(x,x)\mid x\in H_1(F(\PMC))\}$ is the diagonal, then 
\[
\gr(\y)+\tau(\gr(\y))\in \Delta / \bdy [(1+\tau_*)H_2(Y,\bdy Y)].
\]
By construction, this is exactly $\bdy\zeta(\x,\y)$.
In particular, the $\SpinC$-components of the ordinary (unreal) grading determines the image of the difference cycle $\zeta$.

Finally, a few words about how this grading relates to Guth-Manolescu's in the closed case. Each real $\SpinC$-equivalence class of states corresponds to a real $\SpinC$-structure~\cite[Section 3.7]{GM:real-HF} $\spinc$. The grading we have defined in a given $\SpinC$-equivalence class is by the $\ZZ$-set
\[
\ZZ/\Bigl\langle -\OneHalf\bigl(e(P)+2n_\x(P)\bigr)\mid P\in\pi_2(\x,\x)\Bigr\rangle.
\]
This is half the divisibility of $c_1(\spinc)$ (where we view $\spinc$ as an ordinary $\SpinC$-structure). Equivalently, this is a relative grading by $\ZZ/n\ZZ$ where $n$ is half the divisibility of $c_1(\spinc)$, exactly as Guth-Manolescu describe~\cite[Section 4.6]{GM:real-HF}. Tracing through the definitions, our relative grading of states agrees with theirs.

\subsection{A graded pairing theorem}\label{sec:graded-pairing}
As in Theorem~\ref{thm:nice-pairing}, let $(\HD,\tau)$ be a real nice
bordered Heegaard diagram with boundary $(-\PMC)\amalg \PMC^\beta$, and
$\HD'$ a nice arced bordered Heegaard diagram with boundary
$(-\PMC')\amalg \PMC$. 
Let $(Y,\tau)$ and $Y'$ be the 3-manifolds represented by $(\HD,\tau)$ and $\HD'$, so $\bdy Y=-F(\PMC)\amalg -F(\PMC)$ and $\bdy Y'=-F(\PMC')\amalg F(\PMC)$.
Fix a $\SpinC$-equivalence class of states
for $\HD'\cup\HD\cup(-\HD')^\beta$, a base state
$\x'_0\cup\x_0\cup\x'_0$ representing it, 
and grading refinement data $\psi_{\PMC'}$ for $\Alg(\PMC')$ and
$\psi_{\PMC}$ for $\Alg(\PMC)$. To keep notation shorter, we will write $\x'_0\cup\x_0\cup\x'_0$ as $\x'_0\x_0\x'_0$, and $\HD'\cup\HD\cup(-\HD')^\beta$ as $\HD'\HD\HD'$. These data determine:
\begin{itemize}
\item Gradings on $\Alg(\PMC)$ and $\Alg(\PMC')$ by $G(\PMC)$ and $G(\PMC')$, via Formulas~\eqref{eq:grb-on-alg} and~\eqref{eq:gr-on-alg}. 
\item Grading sets for the summands of $\CFDRa(\HD,\tau)$ and $\CFDAa(\PMC')$ corresponding to the $\SpinC$-equivalence class of $\x_0$ and $\x_0'$, as in Formulas~\eqref{eq:gr-set-refined-real} and~\eqref{eq:gr-set-refined-unreal}.
\item A grading set for the summand of $\CFDRa(\HD'\HD\HD',\wt{\tau})$ in the $\SpinC$-equivalence class of $\x'_0\x_0\x'_0$, using the base state $\x'_0\x_0\x'_0$, again as in Formula~\eqref{eq:gr-set-refined-real}.
\item Gradings of these summands of $\CFDRa(\HD,\tau)$, $\CFDAa(\PMC')$, and $\CFDRa(\HD'\HD\HD',\wt{\tau})$ by these sets, as in Formulas~\eqref{eq:refined-gr-real} and~\eqref{eq:refined-gr}.
\end{itemize}
The last point is a little subtle. To compute the grading of a state $\x\in\Gen_R(\HD,\tau)$, say, we choose a domain $B\in\pi_2^R(\x_0,\x)$, but (by Proposition~\ref{prop:real-gradings} part~\ref{item:grb-indept}) the grading of $\x$ is independent of that choice.

As in the pairing theorem for bimodules~\cite[Section 7.1.1]{LOT2}, given states $\x,\y\in\Gen_R(\HD)$ and $\x',\y'\in\Gen(\HD')$ so that $\epsilon(\x,\y)=\epsilon(\x',\y')=0$ and $\zeta(\x,\y)=0$, we can compute $\epsilon(\x'\x\x',\y'\y\y')$ from the grading data, as follows. It suffices to compute $\epsilon(\x'_0\x_0\x'_0,\x'\x\x')$. The $\SpinC$-components of the gradings of $\x'$ and $\x$ are elements of $H_1(F(\PMC))$; denote them $[\gr(\x')]$ and $[\gr(\x)]$. It follows from the definitions that $\epsilon(\x'_0\x_0\x'_0,\x'\x\x')$ is the image of $[\gr(\x')]+[\gr(\x)]$ under the composition
\begin{align*}
  H_1\bigl(F(\PMC)\bigr)/\bigl(\bdy_L H_2(Y,\bdy Y)+\bdy_R H_2(Y',\bdy Y') \bigr)&\to H_1(Y'\cup Y\cup Y')/H_1\bigl(F(\PMC')\amalg F(\PMC')\bigr)\\
  &\stackrel{1-\tau_*}{\lra}H_1(Y'\cup Y\cup Y')^{-\tau_*}/H_1\bigl(F(\PMC')\amalg F(\PMC')\bigr)^{-\tau_*}\!\!\!\!,
\end{align*}
where the first map is induced by inclusion of the left boundary of $Y$.

We digress briefly to discuss how to extract $\zeta$ from the grading data.
So, suppose that $\epsilon(\x'_0\x_0\x'_0,\x'\x\x')=0$; for convenience, assume also that $\x_0$ and $\x$ occupy the same $\alpha$-arcs. Then, for $B\in\pi_2^R(\x_0,\x)$ and $B'\in\pi_2(\x'_0,\x')$, $(\bdy_RB'+\bdy_LB,\bdy_LB+\bdy_RB')\in H_1(F(\PMC)\amalg F(\PMC))$ is in the kernel of the map 
\[
H_1(F(\PMC)\amalg F(\PMC))\to H_1(F(\PMC)\amalg F(\PMC))/K \hookrightarrow H_1(Y'YY',\bdy Y'YY')
\]
where $K$ is the image of the boundary map
\begin{align*}
  H_2(Y',\bdy Y')\oplus H_2(Y,\bdy Y)\oplus H_2(Y',\bdy Y')&\to H_1(\bdy Y',-F(\PMC'))\oplus H_1(\bdy Y)\oplus H_1(\bdy Y',-F(\PMC'))\\
  &\to H_1(F(\PMC))\oplus H_1(F(\PMC))
\end{align*}
(as we can change $B'$, $B$, and $\tau(B')$ by periodic domains to get a domain in $\pi_2(\x_0'\x_0\x'_0,\x'\x\x')$).
Choose a preimage $(C,D,E)\in H_2(Y',\bdy Y')\oplus H_2(Y,\bdy Y)\oplus H_2(Y',\bdy Y')$ of $(\bdy_RB'+\bdy_LB,\bdy_LB+\bdy_RB')$. Then $\zeta(\x'_0\x_0\x'_0,\x'\x\x')$ is represented by 
\[
  (C+\tau_*(E),D+\tau_*(D),E+\tau_*(C))\in H_2(Y'YY',\bdy Y'YY')^{\tau_*}/\Image(1+\tau_*).
\]
It is straightforward to see that this element depends only on the grading data in the tensor product grading set.

Returning from the digression, if
$\zeta(\x'_0\x_0\x'_0,\x'\x\x')=0$, then we can choose a real domain connecting $\x'_0\x_0\x'_0$ and $\x'\x\x'$. Intersecting this domain with $\HD'$ and $\HD$ gives domains that can be used to define the gradings of $\x'$ and $\x$, and hence the grading of $\x'\otimes\x$ in the tensor product grading set. The result is clearly in the same orbit as the grading of $\x'_0\otimes\x_0$.
On the other hand, if $(\gr(\x'_0),\gr(\x_0))$ and $(\gr(\x'),\gr(\x))$ are in the same orbit of the tensor product grading set, then we can choose $\bdy_RB'=\bdy_LB$, so $\x'_0\x_0\x'_0$ and $\x'\x\x'$ are also connected by a real domain.

In conclusion, $\x'_0\otimes\x_0$ and $\x'\otimes\x$ are in the same orbit of the tensor product grading set if and only if $\x'_0\x_0\x'_0$ and $\x'\x\x'$ are in the same orbit of the grading set on $\HD'\cup\HD\cup(-\HD')^\beta$.

Now, consider only the orbit corresponding to $\x'_0\x_0\x'_0$. The grading set on $\HD'\cup\HD\cup(-\HD'^\beta)$ is 
\[
G(\PMC')/\left\langle g'_R(\wt{P})\mid \wt{P}\in\pi_2^R(\x'_0\x_0\x'_0,\x'_0\x_0\x'_0)\right\rangle 
\]
while the grading set on the tensor product is the $G(\PMC')$-orbit of the identity in
\[
[G(\PMC')\times_\ZZ G(\PMC)]/\left\langle g'(P')\cdot (1,g'_R(P))\mid P\in\pi_2^R(\x_0,\x_0),\ P'\in\pi_2(\x'_0,\x'_0)\right\rangle.
\]
Define a map of grading sets sending $h$ to $(h,1)$. Then $h\cdot g'_R(\wt{P})$ maps to $(h\cdot g'_R(\wt{P}),1)$, but writing $\wt{P}=P'PP'$ with $P'\in\pi_2(\x'_0,\x'_0)$ and $P\in\pi_2^R(\x_0,\x_0)$, we have
\[
(h\cdot g'_R(\wt{P}),1)=(h,1)\cdot g'(P')\cdot(1,g'_R(P))=(h,1),
\]
where the first equality uses Lemma~\ref{lem:grading-additivity-pairing-1} below. In particular, we have a well-defined map of grading sets.

\begin{lemma}\label{lem:grading-additivity-pairing-1}
  With notation as above, $(g'_R(\wt{P}),1)=g'(P')\cdot (1,g'_R(P))$.
\end{lemma}
\begin{proof}
  We have
  \begin{align*}
  (g'_R(\wt{P}),1)&=\Bigl(-\OneHalf\bigl(e(\wt{P})+2n_{\x'_0\x_0\x'_0}(\wt{P})\bigr)-\OneQuarter\bigl(\sigma(\alphas,\x_0)-\sigma(\alphas,\x_0)\bigr); \bdy_L\wt{P};0\Bigr)\\
  &=\Bigl(-\OneHalf\bigl(e(P)+2e(P')+4n_{\x'_0}(P')+2n_{\x_0}(P)\bigr); \bdy_LP';0\Bigr),
  \end{align*}
  while
  \begin{align*}
  g'(P')\cdot (1,g'_R(P))&=\Bigl(-e(P')-2n_{\x'_0}(P');\bdy_L P';\bdy_R P'\Bigr)\cdot\Bigl(-\OneHalf\bigl(e(P)+2n_{\x_0}(P)\bigr);0,\bdy_LP\Bigr)\\
  &=\Bigl(-\OneHalf\bigl(e(P)+2e(P')+4n_{\x'_0}(P')+2n_{\x_0}(P)\bigr); \bdy_LP';0\Bigr),
  \end{align*}
  where the second equality uses the fact that if the $\SpinC$-components agree, then the Maslov components simply add, and $\bdy_RP'=-\bdy_LP$.
\end{proof}

\begin{theorem}\label{thm:graded-pairing}
  With respect to this isomorphism of grading sets, the homotopy
  equivalence in Theorem~\ref{thm:nice-pairing} respects the gradings.
\end{theorem}
\begin{proof}
  The proof is essentially the same as the proof of Lemma~\ref{lem:grading-additivity-pairing-1}. We have already explained in what sense the isomorphism respects the decomposition of the grading group into $G(\PMC')$-orbits. So, consider the orbit containing $\gr(\x'_0\x_0\x'_0)$ for the glued diagram, and the corresponding orbit containing $\x'_0\otimes \x_0$ for the tensor product. Given another state $\x'\x\x'$ and corresponding state $\x'\otimes\x$ in this orbit, there are domains $B\in\pi_2^R(\x_0,\x)$ and $B'\in\pi_2(\x'_0,\x')$ so that $\bdy_LB=-\bdy_RB'$. Then
  \begin{align*}
  \gr(\x'\x\x')&=\psi'(\iota')g'_R(B'BB') & \gr(\x'\otimes \x)=\psi'(\iota')g'(B')\psi(\iota)\psi(\iota)^{-1}g'_R(B).
  \end{align*}
  We have
  \begin{align*}
    g'_R(B'BB')&=\Bigl(-\OneHalf\bigl(e(B'BB')+n_{\x'_0\x_0\x'_0}(B'BB')+n_{\x'\x\x'}(B'BB')\bigr)\\
    &\qquad\qquad\qquad\qquad-\OneQuarter\bigl(\sigma(\alphas,\x)-\sigma(\alphas,\x_0)\bigr);\bdy_LB'\Bigr)\\
    g'(B')&=(-e(B')-n_{\x'_0}(B')-n_{\x'}(B');\bdy_LB';\bdy_RB')\\
    g'_R(B)&=\Bigl(-\OneHalf\bigl(e(B)+n_{\x_0}(B)+n_\x(B)\bigr)-\OneQuarter\bigl(\sigma(\alphas,\x)-\sigma(\alphas,\x_0)\bigr);\bdy_L B\Bigr).
  \end{align*}
  So,
  \begin{align*}
  g'_R(B'BB')&=\Bigl(-e(B')-n_{\x'_0}(B')-n_{\x'}(B')-\OneHalf\bigl(e(B)+n_{\x_0}(B)+n_{\x}(B)\bigr)\\
  &\qquad\qquad\qquad-\OneQuarter\bigl(\sigma(\alphas,\x)-\sigma(\alphas,\x_0)\bigr);\bdy_LB'\Bigr) \\
  &=g'(B')g'_R(B),
  \end{align*}
  as desired.
\end{proof}

\section{Computing \texorpdfstring{$\HFRa$}{real HF-hat} by factoring mapping classes}\label{sec:compute-HFRa}
We now have all the tools to compute $\HFRa(Y,\tau)$ as long as $\tau$ has connected fixed set. Nagase showed that $Y$ admits a real Heegaard splitting $Y=H\cup_\Sigma H'$, where $\tau(\Sigma)=\Sigma$ and $\tau(H)=H'$~\cite{Nagase79:real-Heegaard} (see also~\cite[Section 3.1]{GM:real-HF}). Fix a real pointed matched circle $\PMC$ and a diffeomorphism $\phi\co F(\PMC)\to \Sigma$ intertwining the involutions. (If $\Sigma/\tau$ is nonorientable, we can take $\PMC$ to be the antipodal pointed matched circle; if $\Sigma/\tau$ is orientable, and hence the genus of $\Sigma$ is even, we can take $\PMC$ to be the split pointed matched circle.) From the description of $Y(\AZ(\PMC))$ (Section~\ref{sec:AZ-diagrams}), the diffeomorphism $\phi$ extends to an identification of $Y(\AZ(\PMC))$ and a neighborhood of $\Sigma$.

\begin{figure}
  \centering
  \includegraphics[alt={The 0-framed handlebody}]{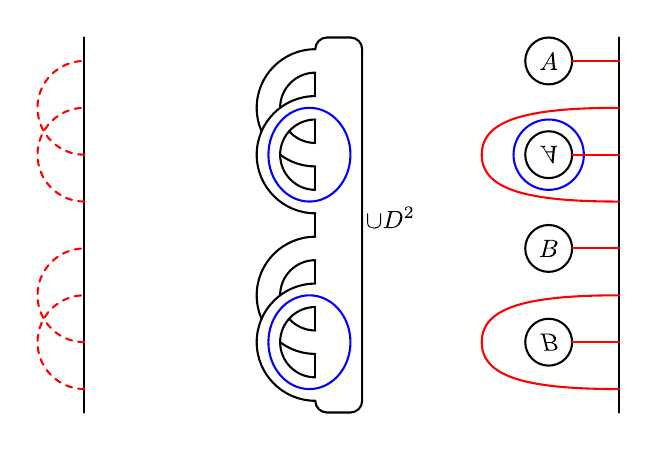}
  \caption{\textbf{The 0-framed handlebody}. Left: the split pointed matched circle $\PMC'$ of genus $2$. Center: the associated surface $F(\PMC)$, and two standard circles on it. Right: a bordered Heegaard diagram representing the $0$-framed handlebody, in which these two circles bound disks.}
  \label{fig:0-fr-handlebody}
\end{figure}

Choose $g$ circles $a_1,\dots,a_g\subset F(\PMC)$ so that the handlebody $H$ is determined by compressing $\phi(a_1),\dots,\phi(a_g)$. So, $H'$ is determined by compressing $\tau(\phi(a_i))=\phi(\tau(a_i))$, $i=1,\dots,g$.
Let $\PMC'$ be the split pointed matched circle with the same genus as $\PMC$. Choose a diffeomorphism $\psi\co F(\PMC)\to F(\PMC')$ so that $\psi(a_1),\dots,\psi(a_g)$ are the $g$ standard circles on $F(\PMC')$ specifying the $0$-framed handlebody $H_0=(H_0,\phi_0\co F(\PMC')\to\bdy H_0)$ of genus $g$ (see Figure~\ref{fig:0-fr-handlebody}).

So,
\begin{equation}\label{eq:decompose-Y}
Y = H_0\cup_{F(\PMC')} Y_\psi \cup_{F(\PMC)}Y(\AZ(\PMC))\cup_{F(\PMC)}Y_\psi\cup_{F(\PMC')}H_0,
\end{equation}
where $Y_\psi$ is the mapping cylinder of $\psi$,
and the involution $\tau$ is the standard involution on $\AZ(\PMC)$ and exchanges the two copies of $Y_\psi$ (respectively $H_0$).

By construction, $\AZ(\PMC)$ is a real nice bordered Heegaard diagram. Thus, by Corollary~\ref{cor:pairing-to-closed} and the usual pairing theorem for bordered Heegaard Floer homology~\cite[Theorem 2]{LOT2}, 
\begin{align*}
\CFRa(Y,\tau)&\simeq \CFAa(H_0)\DT\CFDAa(Y_\psi)\DT\CFDRa(\AZ(\PMC))\\
&\simeq \Mor(\CFDa(-H_0),\CFDAa(Y_\psi)\DT\CFDRa(\AZ(\PMC))).
\end{align*}
In a previous paper, we gave an algorithm for computing $\CFDAa(Y_\psi)$, by factoring $\psi$ into arcslides, as well as an explicit description of $\CFDa(-H_0)$ (the latter is easy). Thus, we obtain an algorithm for computing the dimension of $\HFRa(Y,\tau)$.

We described in Section~\ref{sec:graded-pairing} how to compute the obstruction class $\epsilon$ from the grading data for $\CFAa(H_0\cup Y_\psi)$ and $\CFDRa(\AZ(\PMC))$, as well as the obstruction class $\zeta$ when $\epsilon$ vanishes. Together, $\epsilon$ and $\zeta$ give an affine version of the decomposition of $\HFRa(Y,\tau)$ into real $\SpinC$-structures. Finally, in each $\SpinC$-structure, Theorem~\ref{thm:graded-pairing} gives the relative Maslov grading.

\subsection{Computer implementation}
We have implemented Corollary~\ref{cor:AZ-simplest-model}'s small model for $\CFDRa(\AZ(\PMC))$, for $\PMC$ a real pointed matched circle representing a surface with orientable quotient, in Bohua Zhan's \verb|bfh_python| package~\cite{Zhan:code}. This allows one to compute, for instance, $\HFRa$ of branched double covers of knots in $S^3$. At the time of writing, we have not implemented the gradings, so the code only computes the total dimension of $\HFRa$.

As an example interaction, to compute the real Floer homology of the branched double cover of a trefoil, one could run:
\begin{verbatim}
from real import realAZ, selfGluingHB
from pmc import splitPMC
from arcslide import Arcslide
from arcslideDA import ArcslideDA
hb = selfGluingHB(splitPMC(1))
twists = [ArcslideDA(Arcslide(hb.algebra.pmc,1,0)),
   ArcslideDA(Arcslide(hb.algebra.pmc,2,1))]
for twist in twists:
  hb = twist.tensorD(hb)
  hb.simplify()
cx = realAZ(splitPMC(1)).morToD(hb)
cx.simplify()
print("Dimension of HFR is: ",len(cx))
\end{verbatim}
\noindent
This prints that the dimension is $3$. Indeed, this is a special case of~\cite[Example 6.4]{GM:real-HF}, which in turn is a special case of~\cite[Corollary 1.3]{Hendricks:HFR-loc} (which was conjectured in~\cite[Remark 6.6]{GM:real-HF}).

Of course, to use the program, one needs a description of the bordered handlebodies being attached on the two sides. For fibered knots, like the trefoil, these have a particularly convenient description. Let $K$ be a fibered knot with fiber $F$ and monodromy $\phi$. Then the branched double cover of $K$ is given by
\[
\sgHB\cup_{F\cup(-F)} Y_{\phi\cup\Id}\cup_{F\cup (-F)}[0,1]\times (F\cup (-F))\cup_{F\cup (-F)}Y_{\phi\cup\Id}\cup_{F\cup(-F)}\sgHB,
\]
where $\sgHB$ is the \emph{self-gluing handlebody}~\cite[Section 9.5]{LOT4}, that is, $[0,1]\times F$. For this to be a decomposition into bordered $3$-manifolds, we need parametrizations of the boundaries where the pieces are glued. If we parameterize $F$ by the \emph{linear pointed matched circle}~\cite[Section 5.1]{LOT:DCov1}, then the Dehn twists coming from branched double covers of braid generators have simple descriptions in terms of arcslides, and in fact \verb|bfh_python| already knows this description. So, implementing $\CFDa(\sgHB)$ (which was computed in our earlier paper~\cite[Theorem 9.3]{LOT4}) makes it easy to compute $\HFRa(\Sigma(K))$ from the monodromy for $K$.

From Hendricks's result~\cite[Corollary 1.3]{Hendricks:HFR-loc}, if $\Sigma(K)$ is an $L$-space, then we have $\HFRa(\Sigma(K),\tau)\cong \HFa(\Sigma(K))\cong \FF_2^{\det(K)}$. KnotInfo~\cite{knotinfo} lists monodromies for knots through 12 crossings. Of these, there are three fibered, genus $2$ knots $K$ for which $\Sigma(K)$ is not an $L$-space, and for these we computed $\HFRa(\Sigma(K),\tau)$: 
\begin{center}
\begin{tabular}{cccc}
  \toprule
  $K$ & $\det(K)$ & $\dim \HFa(\Sigma(K))$ & $\dim \HFRa(\Sigma(K),\tau)$\\
  \midrule
  $10_{145}$ & 3 & 5 & 5\\
  $12n121$ & $1$ & 3 & 3\\
  $12n642$ & $27$ & $29$ & $29$\\
  \bottomrule
\end{tabular}
\end{center}

These computations took a few minutes each on a modern laptop (MacBook Air M1, 16GB RAM). The most time-consuming step is computing and simplifying the morphism complex between type $D$ structures at the end. Indeed, for $12n642$, the computer produces a complex with 569,433 generators at this stage. 

If we had defined a type $A$ real invariant, tensoring with it instead of taking the morphism complex at this stage would result in a much smaller complex, and hence faster computations. In fact, for the case of $\AZ(\PMC)$ from Corollary~\ref{cor:AZ-simplest-model}, this takes little extra work:
\begin{proposition}\label{prop:CFAR-AZ}
  Let $\PMC=\PMC'\#(-\PMC')$ be a real pointed matched circle so that $F(\PMC)/\tau$ is orientable, as in Corollary~\ref{cor:AZ-simplest-model}. Define a differential right module $\CFARa(\AZ(\PMC),\tau)$ over the multiplicity $1$ algebra $\Alg'(-\PMC)$ as follows. As a chain complex, $\CFARa(\AZ(\PMC),\tau)$ is $\Alg'(-\PMC')$. Define a right action of $\Alg'(-\PMC)$ on $\Alg'(-\PMC')$ by
  \[
  a\cdot b = \begin{cases}
  0 & \text{if $b$ has a strand that crosses from $\PMC'$ to $-\PMC'$}\\
  r(b'')ab' &\text{if $a=b'b''$ where $b'\in\Alg'(-\PMC')$ and $b''\in \Alg'(\PMC')$}.
  \end{cases}
  \]
  Here, $r\co \Alg'(\PMC')\to\Alg'(-\PMC')$ denotes vertical mirroring (which is an anti-homomorphism).

  Then for any bordered $3$-manifold $Y$ with boundary $F(\PMC)\amalg F(\PMC)$, we have the following pairing theorem:
  \[
  \CFRa(Y\cup Y(\AZ(\PMC))\cup Y,\wt{\tau})\simeq \CFARa(\AZ(\PMC),\tau)\DT\CFDa(Y).
  \]
\end{proposition}
\begin{proof}
  First, we show that $\CFARa(\AZ(\PMC),\tau)\DT\CFDDa(\Id_{\PMC})\cong\CFDRa(\AZ(\PMC))$. Recall that $\CFDDa(\Id_{\PMC})$ is a (left-left) type \DD\ bimodule over $\Alg'(\PMC)$ and $\Alg'(-\PMC)$, with one generator for each pair of complementary idempotents, and 
  \[
  \delta^1(i\otimes i')=\sum_{j\otimes j'}\sum_{\text{chords }\xi}(i\otimes i')\bigl(\xi\otimes \ol{r}(\xi)\bigr)\otimes(j\otimes j')
  \]
  \cite[Theorem 1]{LOT4}. Here, $\ol{r}\co \Alg'(\PMC)\to\Alg'(-\PMC)$ is vertical reflection.
    
  Write a generator $i\otimes i'$ simply as $i$. Tensoring this with $\CFARa(\AZ(\PMC),\tau)$ over $\Alg'(-\PMC)$, there is one generator $a\otimes i$ for each generator $a\in\Alg(\PMC')$, where $i$ is the right idempotent of $a$, and the differential is
  \begin{align*}
  \delta^1(a\otimes i) &= (\bdy(a)\otimes i)+ \sum_{\text{chords $\xi$ in bottom $-\PMC'$}} (\xi\otimes 1)\otimes \xi a + \sum_{\text{chords $\xi$ in top $\PMC'$}} (1\otimes \xi)\otimes a r(\xi)\\
  &= (\bdy(a)\otimes i)+ \sum_{\text{chords $\xi$ in bottom $-\PMC'$}} (\xi\otimes 1)\otimes \xi a + \sum_{\text{chords $\xi$ in bottom $-\PMC'$}} (1\otimes r(\xi))\otimes a \xi.
  \end{align*}
  This agrees with the differential from Corollary~\ref{cor:AZ-simplest-model}.

  Now, from the pairing theorem, Corollary~\ref{cor:pairing-to-closed},
  \[
  \CFRa(Y\cup Y(\AZ(\PMC))\cup Y,\wt{\tau})\simeq \CFDRa(\AZ(\PMC),\tau)\DT\CFAa(Y).
  \]
  But we have
  \begin{align*}
  \CFDRa(\AZ(\PMC),\tau)\DT\CFAa(Y)&\simeq \CFARa(\AZ(\PMC),\tau)\DT\Bigl(\CFAa(Y)\DT\CFDDa(\Id_\PMC)\Bigr)\\
  &\simeq \CFARa(\AZ(\PMC),\tau)\DT\CFDa(Y),
  \end{align*}
  as desired.
\end{proof}

Tensoring with this module instead of taking the morphisms to $\CFDRa(\AZ(\PMC))$, for $12n642$ the final chain complex has 521 generators (a thousand-fold savings).

\section{Other Applications}\label{sec:applications}

\subsection{The first differential in Hendricks's spectral sequence}
Real Heegaard Floer homology is a version of $\ZZ/2$-equivariant Floer theory, and so it is natural to expect that it would relate to nonequivariant Floer homology via a localization theorem.
Hendricks showed that this expectation is correct, constructing a spectral sequence
\begin{equation}\label{eq:Hendricks-ss}
\HFa(Y,\spinc)\otimes\FF_2[\theta,\theta^{-1}]\Rightarrow \bigoplus_{\spinc_r\in\spinc}\HFRa(Y,\spinc_r,\tau)\otimes\FF_2[\theta,\theta^{-1}]
\end{equation}
\cite[Theorem 1.2]{Hendricks:HFR-loc}. She proves that if $Y$ admits a symmetric almost complex structure achieving transversality for the nonequivariant moduli spaces, then the $d_1$-differential is given by $(1+\iota_*\tau_*)\theta$. Thus:
\begin{theorem}\label{thm:Hendricks-differential}
  Suppose that $(Y,\tau)$ is a real 3-manifold so that the fixed set of $\tau$ is a single circle. Then the $d_1$-differential in Hendricks's spectral sequence~\eqref{eq:Hendricks-ss} is $(1+\iota_*\tau_*)\theta$.
\end{theorem}
\begin{proof}
  Decompose $Y$ as in Formula~\eqref{eq:decompose-Y} and choose nice diagrams for $H_0$ and $Y_\psi$. Gluing these to $\AZ(\PMC)$ gives a real nice diagram for $(Y,\tau)$. Since the diagram is, in particular, nice the moduli spaces of holomorphic curves of Maslov index $\leq 1$ are transversely cut out for any choice of almost complex structure on the Heegaard diagram---in particular, for symmetric ones. (There are no positive domains of Maslov index $\leq 0$, and the only ones of Maslov index $1$ are empty bigons and rectangles.) This is enough to guarantee that we may compute the Heegaard Floer homology using such an almost complex structure, and hence for Hendricks's argument~\cite[Lemma 2.8]{Hendricks:HFR-loc} to go through.
\end{proof}

\begin{remark}\label{rem:BGX-2}
  Binns-Guth-Xiao~\cite{BGX:real-sutured} have constructed nice diagrams for arbitrary real Heegaard diagrams, giving a proof of Theorem~\ref{thm:Hendricks-differential} without the hypothesis on the fixed set. While our construction of real nice diagrams is independent of theirs, we were aware that they had such a construction before beginning this project.
\end{remark}

\subsection{A surgery exact triangle}
\begin{theorem}\label{thm:exact-triangle}
  Let $(Y,\tau)$ be a real 3-manifold,
   and $\Sigma$ a connected, separating surface in $Y$ so that $\tau(\Sigma)=\Sigma$. Let $K$ be a framed knot in $Y\setminus\Sigma$ and $K'=\tau(K)$. Let $(Y_{00},\tau_{00})$ (respectively $(Y_{11},\tau_{11})$) be the result of performing $0$-surgery (respectively $1$-surgery) on both $K$ and $K'$ and $\tau_{00}$ (respectively $\tau_{11}$) the involution induced by $\tau$. Then there is an exact triangle
   \[
  \begin{tikzcd}[ampersand replacement=\&, alt={Exact triangle in real Heegaard Floer homology}]
    \HFRa(Y_{00},\tau_{00})\arrow[rr] \& \& \HFRa(Y_{11},\tau_{11})\arrow[dl]\\
    \& \HFRa(Y,\tau)\arrow[ul] \& .
  \end{tikzcd}
  \]
\end{theorem}
\begin{proof}
  Let $T_0$, $T_{1}$, and $T_\infty$ be the $0$-framed, $1$-framed, and $\infty$-framed solid tori. There is a short exact sequence
  \begin{equation}\label{eq:CFA-surg-seq}
    0\to \CFAa(T_0)\to \CFAa(T_1)\to \CFAa(T_\infty)\to 0,
  \end{equation}
  as follows. In our first paper on bordered Floer homology, we constructed a short exact sequence
  \[
  0\to \CFDa(T_\infty)\to\CFDa(T_{-1})\to\CFDa(T_0)\to 0.
  \]
  Tensor this
  with the type \AAm\ bimodule of the identity, to obtain the desired sequence. (Recall that tensoring the type $D$ invariant of the $0$-framed solid torus with $\CFAAa(\Id)$ gives the type $A$ invariant of the $\infty$-framed solid torus, as tensoring $\CFAAa(\Id)$ with the type $D$ invariant of the $0$-framed solid torus on both sides gives $S^3$.)

  Now, decompose 
  \[
  Y = T_\infty\cup Y'\cup [0,1]\times\Sigma \cup Y'\cup T_\infty,
  \]
  so that $\tau$ is the evident involution. Choosing  
  diagrams for $Y'$, $[0,1]\times\Sigma$, and $T_i$ ($i\in\{0,1,\infty\}$), Theorem~\ref{thm:gen-pairing} gives
  \[
  \CFRa(Y_{ii},\tau_{ii})\simeq \CFAa(T_i)\DT \CFDAa(Y')\DT\CFDRa([0,1]\times\Sigma,\tau)
  \]
  Combining this with the exact sequence~\eqref{eq:CFA-surg-seq} gives the result.
\end{proof}

\subsection{Branched double covers of certain satellites}
Real bordered Floer homology gives satellite formulas for the real Floer homology of branched double covers, for even winding number patterns. Indeed, if $P\subset S^1\times D^2$ is a pattern with even winding number, then the branched double cover $\Sigma(S^1\times D^2,P)$ has boundary $T^2\amalg T^2$. Given a knot $K\subset S^3$, if $Y=S^3\setminus \nbd(K)$ denotes the exterior of $K$ and $P(K)$ denotes the satellite of $K$ with pattern $P$, then
\[
\Sigma(P(K))=Y\cup_{T^2}\Sigma(S^1\times D^2,P)\cup_{T^2}Y.
\]
In particular, by the pairing theorems,
\begin{align*}
\CFa(\Sigma(P(K)))&\simeq \CFAa(Y)\DT\CFDDa(\Sigma(S^1\times D^2,P))\DT\CFAa(Y)\\
\CFRa(\Sigma(P(K)),\tau)&\simeq \CFAa(Y)\DT\CFDRa(\Sigma(S^1\times D^2,P),\tau).
\end{align*}
On the other hand, $\CFDa(Y)$ is determined by $\CFKm(K)$~\cite[Theorem A.11]{LOT1}, and $\CFAa(Y)$ is determined by $\CFDa(Y)$, either by tensoring with $\CFAAa(\Id_{T^2})$ or Hedden-Levine's explicit formula~\cite[Theorem 2.2]{HeddenLevine16:splicing}.
Thus, the pairing theorem implies:
\begin{proposition}
  Let $P\subset S^1\times D^2$ be a pattern with even winding number, and let $K,K'\subset S^3$ be knots. If $\CFKm(K)\simeq\CFKm(K')$, then $\HFa(\Sigma(P(K)))\cong \HFa(\Sigma(P(K')))$ and $\HFRa(\Sigma(P(K)),\tau)\cong\HFRa(\Sigma(P(K')),\tau)$.
\end{proposition}

As a first example, we consider the case of Whitehead doubling.
Figure~\ref{fig:Whitehead} is a real Heegaard diagram for the
Whitehead doubling pattern $P\subset S^1\times D^2$; we will discuss
its framing presently. The diagram has $9$ regions. The involution
preserves $D_1$, $D_2$ (the region with the basepoint), $D_3$, $D_4$,
and $D_5$, and exchanges $D_6$ with $D_9$ and $D_7$ with $D_8$. There are
$7$ real states: $ab$, $p_i x_j$, with $i,j\in \{1,2\}$, and $p_i y$,
with $i\in\{1,2\}$. (The Heegaard state $\x=st$ is
$\tau$-invariant, but both $\alpha$-arcs contain a
point in $\x$, so $st$ is not a real state.) These are all real
$\SpinC$-equivalent, as we shall see.

\begin{figure}
  \centering
  \includegraphics[alt={Branched double cover of the Whitehead doubling pattern}]{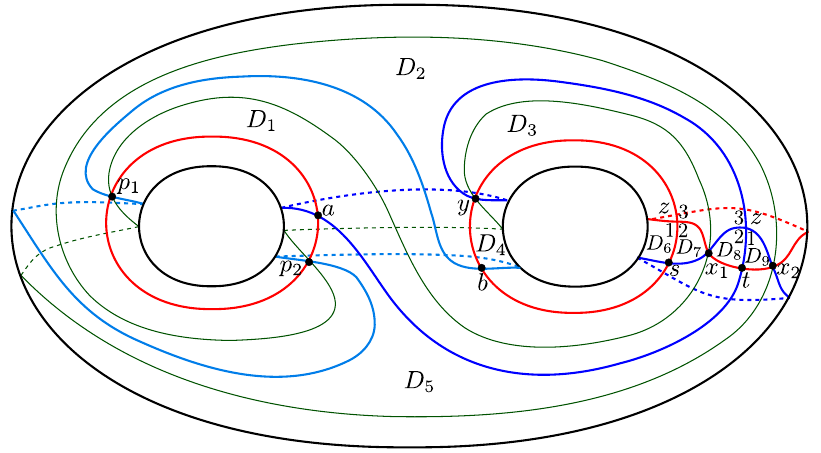}
  \caption{\textbf{Branched double cover of the Whitehead doubling pattern.} The \textcolor{darkgreen}{thin} curve is the fixed set. The boundary is drawn as two intersection points, one between the two $\alpha$-arcs and the other between the two $\beta$-arcs. The numbers $1,2,3$ indicate the chords $\rho_1,\rho_2,\rho_3$ around the boundary components.}
  \label{fig:Whitehead}
\end{figure}

The space of (unreal) periodic domains is generated by
\begin{align*}
\mathcal{P}_1&=D_6+D_7-D_8-D_9 &
\mathcal{P}_2&=D_1 + 2 D_3 - D_5 -2 D_4 - D_6 +D_7 + 2 D_8.
\end{align*}
The space of real periodic domains is generated by
\[ 
\mathcal{P}_1+2 \mathcal{P}_2 = 2D_1 + 4 D_3 - 2 D_5 - 4 D_4  -D_6 + 3 D_7 + 3 D_8 - D_9.
\]
The differential is as follows, where we list first the domain and then its contribution:
\begin{align*}
  D_4\co p_2y&\to 1\otimes ab &
  D_3 \co p_1y&\to \rho_3\otimes p_1x_1 \\
  D_3 \co p_2y &\to \rho_3\otimes  p_2x_1 &
  D_1+D_7+D_8 \co p_1x_1 &\to \rho_2\otimes ab \\
  D_4+D_5\co p_1y &\to \rho_1\otimes  p_2x_2. 
\end{align*}
Cancelling the differential coming from $D_4$, we find that the type $D$ structure
has $5$ generators and the following arrows:
\[
\begin{tikzpicture}[alt={Type D structure for the branched double cover of the Whitehead doubling pattern}]
  \node at (0,0) (YP1) {$\iota_0p_1y$};
  \node at (2,0) (X1P1) {$\iota_1p_1x_1$};
  \node at (4,0) (X1P2) {$\iota_1p_2x_1$};
  \node at (0,-2) (X2P2) {$\iota_1p_2x_2$};
  \node at (-2,-2) (X2P1) {$\iota_1p_1x_2$};
  \draw[->] (YP1) to node[above]{$\rho_3$} (X1P1);
  \draw[->] (X1P1) to node[above]{$\rho_{23}$} (X1P2);
  \draw[->] (YP1) to node[left]{$\rho_{1}$} (X2P2);
\end{tikzpicture}
\]
(We have also indicated the idempotents.)

Implicitly, in drawing Figure~\ref{fig:Whitehead}, we have chosen framings for the boundary components.
It is more convenient to frame the boundary as follows: identifying $\Sigma(P)$ with the left-handed $(2,4)$ torus link, choose the $0
$-framing of both boundary components. With this choice, gluing in a $0$-framed knot complement to each boundary component then gives the branched double cover of the untwisted Whitehead double of the knot. For this framing, if we fill both boundary components at slope $p/q$, then a presentation matrix for the first homology of the result is $\begin{bsmallmatrix}p&-2q\\-2q&p\end{bsmallmatrix}$, so the order of the first homology of this filling is $|p^2-4q^2|$.

By contrast, for the diagram in Figure~\ref{fig:Whitehead}, if we identify $H_1(F(\PMC))=H_1(T^2)=\ZZ^2$ by sending $[\rho_{23}]$ to $[1,0]^T$ and $[\rho_{12}]$ to $[0,1]^T$, then
\begin{align*}
(\bdy_L \mathcal{P}_1,\bdy_R\mathcal{P}_1)&=\left(\begin{bmatrix}0\\1\end{bmatrix},\begin{bmatrix}0\\-1\end{bmatrix}\right) & 
(\bdy_L \mathcal{P}_2,\bdy_R\mathcal{P}_2)&=\left(\begin{bmatrix}2\\-1\end{bmatrix},\begin{bmatrix}2\\0\end{bmatrix}\right).
\end{align*}
If we glue a bordered Heegaard diagram for the Dehn twist $\tau_\mu$ twice to both sides of the diagram, we obtain a diagram with new periodic domains $\mathcal{P}'_1,\mathcal{P}'_2$ with boundaries
\begin{align*}
(\bdy_L \mathcal{P}'_1,\bdy_R\mathcal{P}'_1)&=\left(\begin{bmatrix}2\\1\end{bmatrix},\begin{bmatrix}-2\\-1\end{bmatrix}\right) & 
(\bdy_L \mathcal{P}'_2,\bdy_R\mathcal{P}'_2)&=\left(\begin{bmatrix}0\\-1\end{bmatrix},\begin{bmatrix}2\\0\end{bmatrix}\right).
\end{align*}
In particular, if we fill both boundary components of this new diagram with $p/q$-framed solid tori, a presentation matrix for the first homology of the result is
\[
\begin{bmatrix}
p & 0 & 2 & 0\\
q & 0 & 1 & -1\\
0 & p & -2 & 2\\
0 & q & -1 & 0
\end{bmatrix},
\]
so the order of the first homology is again $|p^2-4q^2|$. Thus, we have found the desired framing.

So, to obtain $\CFDRa(\Sigma(P),\tau)$ for the untwisted Whitehead doubling pattern $P$, we should tensor the type $D$ structure above with $\CFDAa(\tau_\mu)$ twice. The invariant $\CFDAa(\tau_\mu)$ was computed in an earlier paper~\cite[Section 10.2]{LOT2}, and an easy computation of the tensor product gives that $\CFDRa(\Sigma(P))$ is the following, delightfully simple, type $D$ structure:
\[
  \begin{tikzpicture}[alt={Simplified type D structure after tensoring with tau-mu}]
  \node at (0,0) (r) {$\iota_1 r$};
  \node at (1.5,.5) (s) {$\iota_0s$};
  \node at (1.5,-.5) (t) {$\iota_1t$};
  \node at (2.1,-.6) (period) {.};
  \draw[->] (s) to node[right]{$\rho_1$} (t);
  \end{tikzpicture}
\]

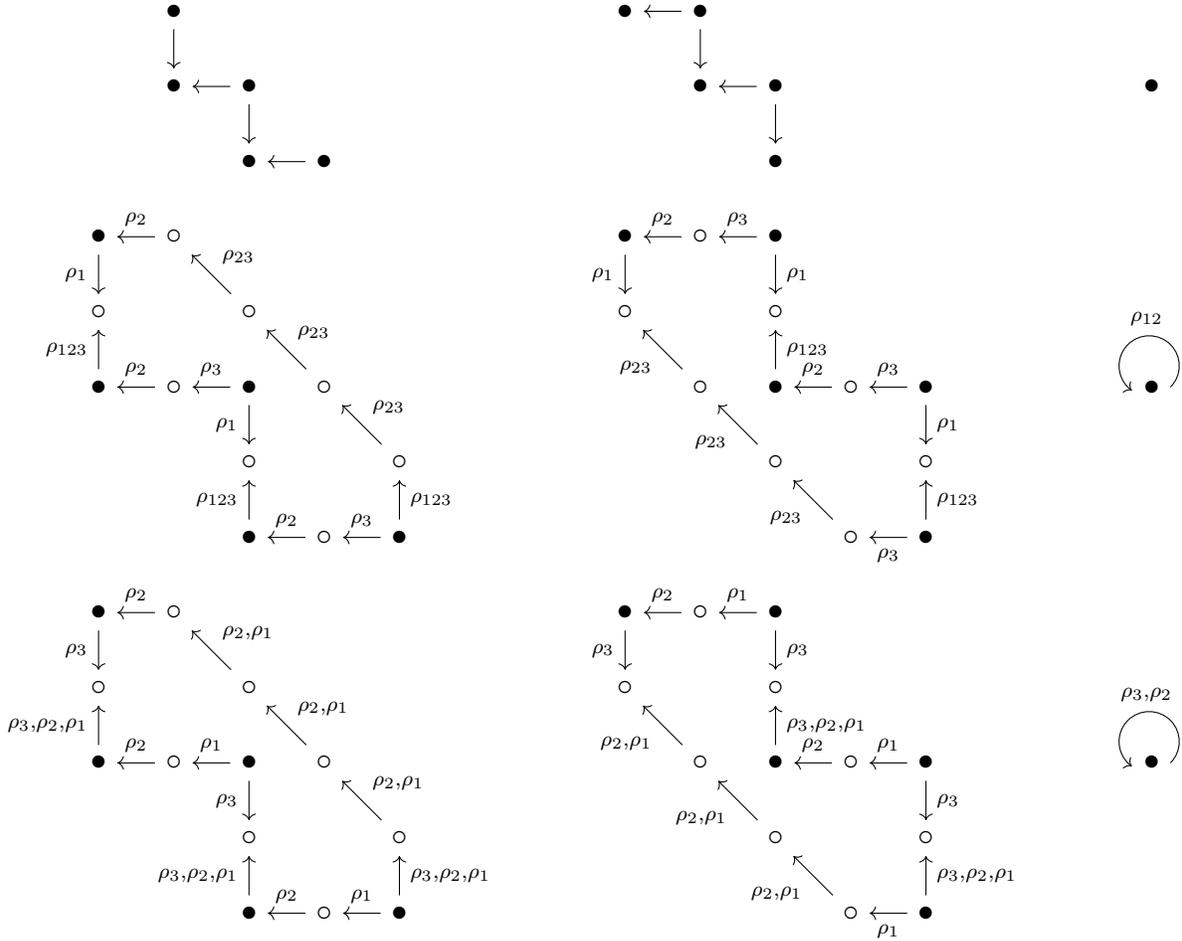
\begin{figure}
  \centering
  \begin{tikzpicture}[alt={Staircases}]
    \begin{scope}[xshift=-6cm]
        \foreach \x/\y in {0/0,0/-1,1/-1,1/-2,2/-2} {
          \node at (\x,\y) (n\x\y) {$\bullet$};
        }
        \foreach \x/\y/\z/\w in {0/0/0/-1,1/-1/0/-1,1/-1/1/-2,2/-2/1/-2} {
          \draw[->] (n\x\y) to (n\z\w);
        }
    \end{scope}
    \begin{scope}
        \foreach \x/\y in {0/0,1/0,1/-1,2/-1,2/-2} {
          \node at (\x,\y) (n\x\y) {$\bullet$};
        }
        \foreach \x/\y/\z/\w in {1/0/0/0,1/0/1/-1,2/-1/1/-1,2/-1/2/-2} {
          \draw[->] (n\x\y) to (n\z\w);
        }    
    \end{scope}
    \begin{scope}[xshift=7cm,yshift=-1cm]
      \node at (0,0) (imhere) {$\bullet$};
    \end{scope}
    \begin{scope}[xshift=-7cm,yshift=-3cm]
      \foreach\x/\y in {0/0,0/-2,2/-2,2/-4,4/-4} {
        \node at (\x,\y) (b\x\y) {$\bullet$};
      }
      \foreach\x/\y in {0/-1,1/-2,2/-3,3/-4,1/0,2/-1,3/-2,4/-3} {
        \node at (\x,\y) (c\x\y) {$\circ$};
      }
      \foreach\x/\y/\z/\w/\r in {0/0/0/-1/\rho_{1},0/-2/0/-1/\rho_{123},2/-2/2/-3/\rho_1,2/-4/2/-3/\rho_{123}} {
        \draw[->] (b\x\y) to node[left]{\lab{\r}} (c\z\w);
      }
      \foreach\x/\y/\z/\w/\r in {2/-2/1/-2/\rho_{3},4/-4/3/-4/\rho_{3}} {
        \draw[->] (b\x\y) to node[above]{\lab{\r}} (c\z\w);
      }
      \foreach\x/\y/\z/\w/\r in {1/-2/0/-2/\rho_{2},3/-4/2/-4/\rho_{2}} {
        \draw[->] (c\x\y) to node[above]{\lab{\r}} (b\z\w);
      }
      \foreach\x/\y/\z/\w in {4/-3/3/-2,3/-2/2/-1,2/-1/1/0} {
        \draw[->] (c\x\y) to node[above right]{\lab{\rho_{23}}} (c\z\w);
      }
      \draw[->] (c10) to node[above]{\lab{\rho_2}} (b00);
      \draw[->] (b4-4) to node[right]{\lab{\rho_{123}}} (c4-3);
    \end{scope}
    \begin{scope}[yshift=-3cm]
      \foreach\x/\y in {0/0,2/0,2/-2,4/-2,4/-4} {
        \node at (\x,\y) (b\x\y) {$\bullet$};
      }
      \foreach\x/\y in {1/0,2/-1,3/-2,4/-3,0/-1,1/-2,2/-3,3/-4} {
        \node at (\x,\y) (c\x\y) {$\circ$};
      }
      \foreach\x/\y/\z/\w/\r in {2/0/2/-1/\rho_{1},2/-2/2/-1/\rho_{123},4/-2/4/-3/\rho_1,4/-4/4/-3/\rho_{123}} {
        \draw[->] (b\x\y) to node[right]{\lab{\r}} (c\z\w);
      }
      \foreach\x/\y/\z/\w/\r in {2/0/1/0/\rho_{3},4/-2/3/-2/\rho_{3}} {
        \draw[->] (b\x\y) to node[above]{\lab{\r}} (c\z\w);
      }
      \foreach\x/\y/\z/\w/\r in {1/0/0/0/\rho_{2},3/-2/2/-2/\rho_{2}} {
        \draw[->] (c\x\y) to node[above]{\lab{\r}} (b\z\w);
      }
      \foreach\x/\y/\z/\w in {3/-4/2/-3,2/-3/1/-2,1/-2/0/-1} {
        \draw[->] (c\x\y) to node[below left]{\lab{\rho_{23}}} (c\z\w);
      }
      \draw[->] (b00) to node[left]{\lab{\rho_1}} (c0-1);
      \draw[->] (b4-4) to node[below]{\lab{\rho_{3}}} (c3-4);
    \end{scope}
    \begin{scope}[xshift=7cm,yshift=-5cm]
      \node at (0,0) (imhere) {$\bullet$};
      \draw [->] (imhere.0) arc (-45:235:4mm) node[pos=0.5,above] {\lab{\rho_{12}}} (imhere);
    \end{scope}
    \begin{scope}[xshift=-7cm,yshift=-8cm]
      \foreach\x/\y in {0/0,0/-2,2/-2,2/-4,4/-4} {
        \node at (\x,\y) (b\x\y) {$\bullet$};
      }
      \foreach\x/\y in {0/-1,1/-2,2/-3,3/-4,1/0,2/-1,3/-2,4/-3} {
        \node at (\x,\y) (c\x\y) {$\circ$};
      }
      \foreach\x/\y/\z/\w/\r in {0/0/0/-1/\rho_{3},0/-2/0/-1/{\rho_3,\rho_2,\rho_1},2/-2/2/-3/\rho_3,2/-4/2/-3/{\rho_3,\rho_2,\rho_1}} {
        \draw[->] (b\x\y) to node[left]{\lab{\r}} (c\z\w);
      }
      \foreach\x/\y/\z/\w/\r in {2/-2/1/-2/\rho_{1},4/-4/3/-4/\rho_{1}} {
        \draw[->] (b\x\y) to node[above]{\lab{\r}} (c\z\w);
      }
      \foreach\x/\y/\z/\w/\r in {1/-2/0/-2/\rho_{2},3/-4/2/-4/\rho_{2}} {
        \draw[->] (c\x\y) to node[above]{\lab{\r}} (b\z\w);
      }
      \foreach\x/\y/\z/\w in {4/-3/3/-2,3/-2/2/-1,2/-1/1/0} {
        \draw[->] (c\x\y) to node[above right]{\lab{\rho_{2},\rho_1}} (c\z\w);
      }
      \draw[->] (c10) to node[above]{\lab{\rho_2}} (b00);
      \draw[->] (b4-4) to node[right]{\lab{\rho_3,\rho_2,\rho_1}} (c4-3);
    \end{scope}
    \begin{scope}[yshift=-8cm]
      \foreach\x/\y in {0/0,2/0,2/-2,4/-2,4/-4} {
        \node at (\x,\y) (b\x\y) {$\bullet$};
      }
      \foreach\x/\y in {1/0,2/-1,3/-2,4/-3,0/-1,1/-2,2/-3,3/-4} {
        \node at (\x,\y) (c\x\y) {$\circ$};
      }
      \foreach\x/\y/\z/\w/\r in {2/0/2/-1/\rho_{3},2/-2/2/-1/{\rho_3,\rho_2,\rho_1},4/-2/4/-3/\rho_3,4/-4/4/-3/{\rho_3,\rho_2,\rho_1}} {
        \draw[->] (b\x\y) to node[right]{\lab{\r}} (c\z\w);
      }
      \foreach\x/\y/\z/\w/\r in {2/0/1/0/\rho_{1},4/-2/3/-2/\rho_{1}} {
        \draw[->] (b\x\y) to node[above]{\lab{\r}} (c\z\w);
      }
      \foreach\x/\y/\z/\w/\r in {1/0/0/0/\rho_{2},3/-2/2/-2/\rho_{2}} {
        \draw[->] (c\x\y) to node[above]{\lab{\r}} (b\z\w);
      }
      \foreach\x/\y/\z/\w in {3/-4/2/-3,2/-3/1/-2,1/-2/0/-1} {
        \draw[->] (c\x\y) to node[below left]{\lab{\rho_2,\rho_1}} (c\z\w);
      }
      \draw[->] (b00) to node[left]{\lab{\rho_3}} (c0-1);
      \draw[->] (b4-4) to node[below]{\lab{\rho_1}} (c3-4);
    \end{scope}
    \begin{scope}[xshift=7cm,yshift=-10cm]
      \node at (0,0) (imhere) {$\bullet$};
      \draw [->] (imhere.0) arc (-45:235:4mm) node[pos=0.5,above] {\lab{\rho_3,\rho_2}} (imhere);
    \end{scope}
  \end{tikzpicture}
  \caption{\textbf{Staircases.} Top row: staircases with $\tau=-2$, $\tau=2$, and $\tau=0$. Middle row: the corresponding type $D$ structures. Bottom row: the corresponding type $A$ modules. For the type $A$ structures, we have not shown a full set of operations, but just operations that generate the type $A$ structure, under an evident composition operation. Solid dots have idempotent $\iota_0$, and empty dots have idempotent $\iota_1$.}
  \label{fig:staircases}
\end{figure}

\begin{figure}
  \centering
  \begin{tikzpicture}[alt={Bordered invariants for a box}]
    \begin{scope}
      \foreach\x/\y in {0/0,1/0,0/1,1/1} {
        \node at (\x,\y) (n\x\y) { $\bullet$ };
        }
      \foreach \x/\y/\z/\w in {1/1/0/1,1/1/1/0,1/0/0/0,0/1/0/0} {
        \draw[->] (n\x\y) to (n\z\w);
      }
    \end{scope}
    \begin{scope}[xshift=4cm]
      \foreach\x/\y in {0/0,2/0,0/2,2/2} {
        \node at (\x,\y) (b\x\y) { $\bullet$ };
        }
      \foreach\x/\y in {1/0,0/1,2/1,1/2} {
        \node at (\x,\y) (c\x\y) {$\circ$};
      }
      \draw[->] (b22) to node[above]{\lab{\rho_3}} (c12);
      \draw[->] (b20) to node[above]{\lab{\rho_3}} (c10);
      \draw[->] (b22) to node[right]{\lab{\rho_1}} (c21);
      \draw[->] (b02) to node[right]{\lab{\rho_1}} (c01);
      \draw[->] (c12) to node[above]{\lab{\rho_2}} (b02);
      \draw[->] (c10) to node[above]{\lab{\rho_2}} (b00);
      \draw[->] (b00) to node[right]{\lab{\rho_{123}}} (c01);
      \draw[->] (b20) to node[right]{\lab{\rho_{123}}} (c21);
    \end{scope}
    \begin{scope}[xshift=9cm]
      \foreach\x/\y in {0/0,2/0,0/2,2/2} {
        \node at (\x,\y) (b\x\y) { $\bullet$ };
        }
      \foreach\x/\y in {1/0,0/1,2/1,1/2} {
        \node at (\x,\y) (c\x\y) {$\circ$};
      }
      \draw[->] (b22) to node[above]{\lab{\rho_1}} (c12);
      \draw[->] (b20) to node[above]{\lab{\rho_1}} (c10);
      \draw[->] (b22) to node[right]{\lab{\rho_3}} (c21);
      \draw[->] (b02) to node[right]{\lab{\rho_3}} (c01);
      \draw[->] (c12) to node[above]{\lab{\rho_2}} (b02);
      \draw[->] (c10) to node[above]{\lab{\rho_2}} (b00);
      \draw[->] (b00) to node[right]{\lab{\rho_3,\rho_2,\rho_1}} (c01);
      \draw[->] (b20) to node[right]{\lab{\rho_3,\rho_2,\rho_1}} (c21);
    \end{scope}
  \end{tikzpicture}
  \caption{\textbf{Bordered invariants for a box.} Left: a $1\times 1$ box, as appearing in the knot Floer complex. Center: the corresponding type $D$ structure. Right: the corresponding type $A$ module. Again, we have not drawn all the operations, just a generating set.}
  \label{fig:box}
\end{figure}
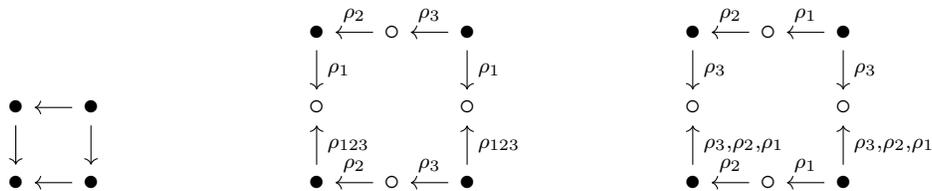

To compare the real and unreal invariants of the branched covers, the following lemma will be convenient.
To pin down conventions, we use the conventions on the signature $\sigma(K)$ for which the right-handed trefoil knot has $\sigma=-2$ and  $\tau=1$; in general, for alternating knots, $\sigma(K)=-2\tau(K)$.
We let $\det(K)=|\Delta(-1)|$. (We apologize that $\tau$ is being used both for the involution and the concordance invariant~\cite{OS03:4BallGenus}, and hope this will not cause confusion.)
\begin{lemma}\label{lem:surg-comp}
  Let $K$ be an alternating knot in $S^3$.
  Then the dimension of $\HFa(S^3_1(K))$ is given by the formula
  \begin{equation}
    \label{eq:OneSurgery}
    \dim\HFa(S^3_1(K))=\frac{1}{2}\left(
    \det(K)
    +6|\tau(K)|+
      \begin{cases}
        -7 & \tau(K)>0 \\ 
        1 &\tau(K)=0 \\ 
        -3 &\tau(K)<0
      \end{cases}\right).
  \end{equation}
  The dimension of $\HFa(S^3_{1/2}(K))$ is given by
  \begin{equation}
    \label{eq:HalfSurgery}
    \dim\HFa(S^3_{1/2}(K))=
    \det(K)
    +6|\tau(K)|
      +\begin{cases}
        -6 &\tau(K)>0\\
        0 &\tau(K)=0 \\ 
        -4 & \tau(K)<0. 
      \end{cases}
  \end{equation}
\end{lemma}
\begin{proof}
  This can be deduced from the knot Floer homology for alternating
  knots~\cite{AltKnots}, combined with surgery formulas for Heegaard
  Floer homology~\cite{IntSurg,OS11:RatSurg}.

  Specifically, the knot
  Floer homology $\HFKa(K)$ of an alternating knot $K$ is determined
  by its Alexander polynomial $\Delta_K(t)=\sum_{i} a_i T^i$, and its
  signature $\sigma$, by the formula
  \begin{equation}
    \label{eq:AltKnots}
    \HFKa_d(K,s)=\begin{cases}
      \Field^{|a_s|} & {\text{if $d=s+\frac{\sigma}{2}$}} \\
      0 &{\text{otherwise.}}
    \end{cases}
  \end{equation}
  The knot Floer complex splits as a direct sum of $1\times 1$ boxes (Figure~\ref{fig:box}) and a staircase (Figure~\ref{fig:staircases})~\cite{AltKnots,Petkova13:cableofthin}.

  A staircase is determined by its dimension (which is odd) and its sign.
  The dimension of the staircase is given by $2|\tau(K)|+1$, and the sign is the sign of $\tau(K)$.

  The surgery formula~\cite{IntSurg} expresses $\HFa(S^3_{+1}(K))$ as the homology of a mapping cone of a chain map
  $\mathbb{A}\to \mathbb{B}$, which can be described in terms of the knot Floer complex.
  
  Each box gives rise to a 2-dimensional subspace in $H_*({\mathbb A})$ which maps trivially to $H_*({\mathbb B})$.
  By Equation~\eqref{eq:AltKnots}, the number of boxes is
  \[ \frac{\det(K)-(2|\tau(K)|+1)}{4}.\]
  A straightforward computation shows that when $\tau=0$, the trivial staircase contributes $1$;
  when $\tau>0$, the staircase contributes $4\tau-3$; and when $\tau<0$,
  the staircase contributes $-4\tau-1$. Equation~\eqref{eq:OneSurgery} follows.

  For $1/2$ surgery, we rely on the rational surgery
  formula~\cite{OS11:RatSurg}. In that case, each box in $\HFKa$ contributes $4$ to the rank.
  When $\tau=0$, the trivial staircase contributes $\Field$;
  when $\tau>0$, the staircase contributes $\Field^{-5+8\tau}$;
  and when $\tau<0$, the staircase contributes $\Field^{-8\tau-3}$.
\end{proof}

\begin{proof}[Proof of Theorem~\ref{thm:Whitehead-nontriv}]
  As in the previous proof, since $K$ is alternating, its knot Floer complex is the direct sum of a staircase and a collection of $1\times 1$ boxes.
  The type $D$ structures for staircases are shown in Figure~\ref{fig:staircases}. The corresponding type $A$ modules are also shown in Figure~\ref{fig:staircases}, following the algorithm from Hedden-Levine~\cite[Theorem 2.2]{HeddenLevine16:splicing}.
  The type $D$ structure for a $1\times 1$ box and the corresponding type $A$ module are shown in Figure~\ref{fig:box}. Taking the box product with $\CFDRa(\Sigma(P))$:
  \begin{itemize}
  \item A $\tau>0$ staircase contributes $4\tau+2\tau+1+4\tau-2\tau-2=8\tau-1$ to $\HFRa(\Sigma(P(K)),\tau)$.
  \item A $\tau<0$ staircase contributes $4|\tau|+2|\tau|+1+4|\tau|-2|\tau|=8|\tau|+1$ to $\HFRa(\Sigma(P(K)),\tau)$.
  \item A $\tau=0$ staircase (a single dot) contributes $1$ to $\HFRa(\Sigma(P(K)),\tau)$.
  \item A $1\times 1$ box contributes $4+8-4=8$ to $\HFRa(\Sigma(P(K)),\tau)$.
  \end{itemize}
  Combining these,
  \[
  \dim\HFRa(\Sigma(P(K)),\tau)=\begin{cases}
  2\det(K)+4\tau-3 & \tau>0\\
  2\det(K)+4|\tau|-1 & \tau\leq0.
  \end{cases}
  \]

  Next, we compute unreal $\HFa$ of the branched double cover.
  By Montesinos's trick, the branched double cover of the Whitehead double of any knot $K$ is $S^3_{1/2}(K\# r(K))$, where $r$ denotes the reverse. Since $K\# K^r$ is an 
  alternating knot, Lemma~\ref{lem:surg-comp} implies that
  \[
    \dim\HFa(\Sigma(P(K)))=
    \det(K)^2
    +12|\tau(K)|
      +\begin{cases}
        -6 &\tau(K)>0\\
        0 &\tau(K)=0 \\ 
        -4 & \tau(K)<0. 
      \end{cases}
  \]
  Finally, using $\det(K)\geq 2|\tau|+1$, it is easy to verify the inequality $\dim\HFRa(\Sigma(P(K)),\tau)\leq \dim\HFa(\Sigma(P(K)))$, with equality holding if and only if 
  $(\det(K),\tau(K))=(1,0)$.
\end{proof}

\begin{figure}
  \centering
  \includegraphics[alt={A diagram for the (2,1)-cable on a torus}]{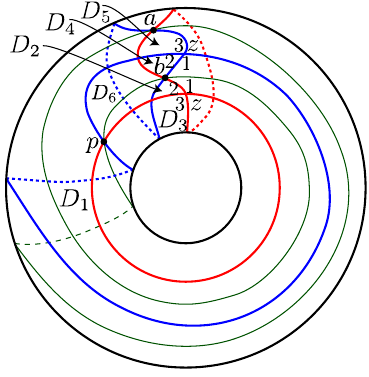}\qquad\qquad
  \includegraphics[alt={A diagram for the (2,1)-cable on a rectangle with opposite edges identified}]{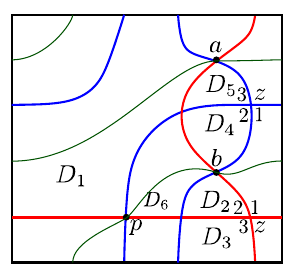}
  \caption{\textbf{A diagram for the $(2,1)$-cable.} We have drawn the diagram on a torus and on the rectangle with opposite edges identified.}
  \label{fig:cable}
\end{figure}

As a second example, we consider $(2,1)$-cables. A real bordered Heegaard diagram for the branched double cover of the cabling pattern is shown in Figure~\ref{fig:cable}; this is, of course, $[0,1]\times T^2$, with boundary parametrizations differing by a Dehn twist, and involution coming from the involution on $T^2$ with quotient a M\"obius band. Since this is a genus $1$ diagram, it is easy to compute its invariants; its $\CFDRa$ is given by:
\[
\begin{tikzcd}[ampersand replacement=\&, alt={Type D structure for the (2,1)-cable pattern}]
\iota_1p\arrow[r,"\rho_2"] \& \iota_0b\arrow[r,"\rho_{12}"] \& \iota_0a.
\end{tikzcd}
\]
To assist with analyzing the framing later, observe that the space of periodic domains is spanned by
\begin{align*}
\mathcal{P}_1&=D_2+D_3-D_4-D_5 & \mathcal{P}_2=D_1+D_2+2D_4+D_5+2D_6.
\end{align*}
The elements of $H_1(F(\PMC_L))$ and $H_1(F(\PMC_R))$ given by the boundaries of these domains are
\begin{align*}
(\bdy_L \mathcal{P}_1,\bdy_R\mathcal{P}_1)&=\left(\begin{bmatrix}1\\0\end{bmatrix},\begin{bmatrix}-1\\0\end{bmatrix}\right) & 
(\bdy_L \mathcal{P}_2,\bdy_R\mathcal{P}_2)&=\left(\begin{bmatrix}0\\1\end{bmatrix},\begin{bmatrix}1\\1\end{bmatrix}\right)
\end{align*}
(listing multiplicities at $\rho_3,\rho_1$, as we did for the Whitehead double).

The branched double cover of the $(2,1)$-cable pattern is the complement of the $(2,2)$ torus link (that is, the Hopf link) in $S^3$, with each component $0$-framed. The order of the first homology of the result of doing $p/q$ surgery to both components of the Hopf link is $|p^2-q^2|$. We obtain this framing by tensoring the module above with $\tau_{\lambda}$, which transforms the boundaries of the periodic domains into 
\begin{align*}
(\bdy_L \mathcal{P}'_1,\bdy_R\mathcal{P}'_1)&=\left(\begin{bmatrix}1\\-1\end{bmatrix},\begin{bmatrix}-1\\1\end{bmatrix}\right) & 
(\bdy_L \mathcal{P}'_2,\bdy_R\mathcal{P}'_2)&=\left(\begin{bmatrix}0\\1\end{bmatrix},\begin{bmatrix}1\\0\end{bmatrix}\right).
\end{align*}
The corresponding module is
\[
\begin{tikzcd}[ampersand replacement=\&, alt={Type D structure for the framed (2,1)-cable}]
\iota_1x\arrow[r,"\rho_2"] \& \iota_0y.
\end{tikzcd}
\]

\begin{theorem}
  Given a knot $K$, let $P(K)$ be the $(2,1)$-cable of $K$.
  Then, for any nontrivial alternating knot $K$, $\dim\HFa(\Sigma(P(K)))>\dim\HFRa(\Sigma(P(K)),\tau)$.
\end{theorem}
\begin{proof}
  This is similar to the proof of Theorem~\ref{thm:Whitehead-nontriv}. Now, a staircase contributes to the dimension of $\HFRa$ as follows:
  \begin{itemize}
  \item $2\tau(K)+1+4\tau(K)-2\tau(K)=4\tau(K)+1$ if $\tau(K)>0$
  \item $1$ if $\tau(K)=0$, and
  \item $2|\tau(K)|+1+4|\tau(K)|-2|\tau(K)|-2=4|\tau(K)|-1$ if $\tau(K)<0$.
  \end{itemize}
  A box contributes $4$ to $\HFRa$. So,
  \[
  \dim\HFRa(\Sigma(P(K)),\tau)=\begin{cases}
  \det(K)+2|\tau(K)|-2 & \tau<0\\
  \det(K)+2\tau(K) & \tau\geq0.
  \end{cases}
  \]

  For the knot Floer homology computation, recall that $\Sigma(P(K))$ is $1$-surgery on $K\#K^r$. So, the dimension of $\HFa(\Sigma(P(K)))$ is given by Lemma~\ref{lem:surg-comp}, but with $\tau(K)$ replaced by $2\tau(K)$ and $\det(K)$ replaced by $\det(K)^2$. From these formulas, the result is straightforward.
\end{proof}

\subsection{Genus-1 splittings}
The next two propositions leverage computations for real thick tori to show that, in very special cases, $\HFRa$ can be identified with more familiar Heegaard Floer homology groups.

\begin{figure}
  \centering
  \includegraphics[alt={A Heegaard diagram for a solid torus}]{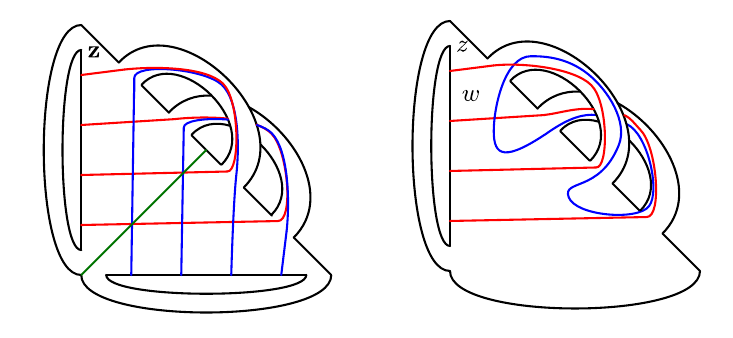}
  \caption{\textbf{A Heegaard diagram for a solid torus.} Left: the real Auroux-Zarev diagram for the genus 1 pointed matched circle. Right: a doubly-pointed bordered Heegaard diagram whose invariant is the same as the real bordered invariant of the Auroux-Zarev diagram. The fixed set in the first diagram and the $\beta$-circle in the second have the same intersection pattern with the $\alpha$-arcs, hence represent isotopic circles in the torus. The knot lies on the torus and intersects the meridional disk once, hence is a longitude.}
  \label{fig:g1-AZ-to-solid-torus}
\end{figure}

\begin{proposition}\label{prop:g1-one-fixed-circs}
  Suppose that $(Y,\tau)$ is a (closed) real $3$-manifold so that the fixed set $K$ of $\tau$ is connected, and there is a separating, genus $1$, $\tau$-invariant surface $\Sigma$ containing $K$. Write $Y=Y_1\cup_\Sigma Y_2$, so $Y_2=\tau(Y_1)$. Let $Y'$ be the result of Dehn filling $Y_1$ with slope $K$, and let $K'$ be the core (longitude) of the new solid torus, with framing induced by $\Sigma$. Then
  \[
  \HFRa(Y,\tau)\cong \HFKa(Y',K')=\SFH(Y_1,\Gamma)
  \]
  where the sutures $\Gamma$ consist of two parallel copies of $K$.
\end{proposition}
\begin{proof}
  This follows from the computation of $\CFDRa(\AZ(Z))$ where $Z$ is the genus $1$ pointed matched circle, from Example~\ref{eg:genus-1-AZ-CFDR}. Figure~\ref{fig:g1-AZ-to-solid-torus} shows a doubly-pointed bordered Heegaard diagram $\HD$ so that $\CFDa(\HD)\cong\CFDRa(\AZ(Z))$. It is clear that $\HD$ represents some knot in some bordered solid torus; the figure sketches a proof that the solid torus is the one filling $K$, and that the knot is the longitude.
\end{proof}

\begin{proposition}\label{prop:g1-two-fixed-circs}
  Suppose that $(Y,\tau)$ is a (closed) real $3$-manifold so that the fixed set $L$ of $\tau$ consists of two circles, and there is a separating, genus $1$, $\tau$-invariant surface $\Sigma$ containing $L$. Write $Y=Y_1\cup_\Sigma Y_2$. Let $Y'$ be the result of Dehn filling $Y_1$ along the slope specified by $L$, and let $K'$ be the core (longitude) of the new solid torus (with framing induced by $\Sigma$). Then
  \[
  \HFRa(Y,\tau)\cong \HFa(Y')\oplus\HFKa(Y',K').
  \]
\end{proposition}
\begin{proof}
  This follows from the computation of $\CFDRa(\HD)$ for the real thick torus $\HD$ in Section~\ref{sec:CFDR-example}. 
  Observe that $\CFDRa(\HD)\cong \CFDa(\HD_1)\oplus\CFDa(\HD_2)$ where $\HD_1$ is a bordered Heegaard diagram for the solid torus in which the fixed set of $\tau$ bounds a disk, and $\HD_2$ is the same, but with an extra basepoint; see Figure~\ref{fig:solid-torus-hds}. It is clear that this extra basepoint specifies the longitude of the solid torus.

  \begin{figure}
    \centering
    \includegraphics[alt={Heegaard diagrams for the solid torus and a knot}]{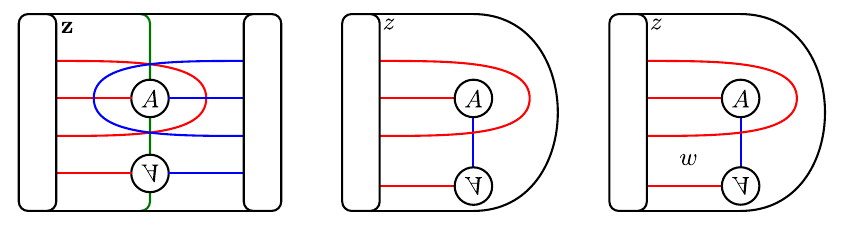}
    \caption{\textbf{Heegaard diagram for the solid torus and a knot in it.} Left: the diagram $\HD$ again, with the two components of the fixed set indicated. Center: the diagram $\HD_1$ from the proof of Proposition~\ref{prop:g1-two-fixed-circs}. Right: the diagram $\HD_2$ from that proof.}
    \label{fig:solid-torus-hds}
  \end{figure}

  It follows from the pairing theorem, Corollary~\ref{cor:pairing-to-closed}, that 
  \begin{align*}
  \CFRa(Y,\tau)&\simeq \CFAa(Y_1)\DT \CFDRa(\HD)\\
  &\cong \bigl(\CFAa(Y_1)\DT\CFDa(\HD_1)\bigr)\oplus \bigl(\CFAa(Y_1)\DT\CFDa(\HD_2)\bigr)\\
  &\simeq \CFa(Y')\oplus \CFKa(Y',K').
  \end{align*}
  This proves the result.
\end{proof}

\subsection{Splittings with connected fixed sets and orientable quotients}

\begin{proposition}
  Let $(Y,\tau)$ be a (closed) real $3$-manifold.
  Suppose that the fixed set $K$ of $\tau$ is connected and $(Y,\tau)$ admits a $\tau$-invariant separating surface $\Sigma$ containing the fixed set, so that $\Sigma/\tau$ is orientable. Let $K'$ and $K''$ be pushoffs of $K$ in $\Sigma$, disjoint from $K$ and so that $\tau(K')=K''$, and choose a $\tau$-invariant arc $A$ from $K'$ to $K''$. There is an induced real sutured manifold $(Y\setminus\nbd(K'\cup K''\cup A),\Gamma,\tau)$, where $\Gamma$ consists of three sutures: a meridian of $K'$, a meridian of $K''$, and a suture near $A$ such that $\Gamma$ is separating. Then 
  \[
  \HFRa(Y,\tau)\cong \mathit{SFHR}(Y\setminus\nbd (K'\cup K''\cup A),\Gamma,\tau).
  \]
\end{proposition}
\begin{proof}
  By Corollary~\ref{cor:AZ-simplest-model} (or its proof), placing an extra basepoint in $\AZ(\PMC)$ adjacent to the chord $[2k+1,2k+2]$ (the middle of the pointed matched circle) in $\bdy_L\AZ(\PMC)$, and an extra basepoint adjacent to the corresponding chord in $\bdy_R\AZ(\PMC)$, does not change the homotopy type of $\CFDRa$. Deleting a neighborhood of these basepoints and a point on $\bbpt$ gives a real bordered sutured diagram representing $(Y\setminus\nbd (K'\cup K''\cup A),\Gamma,\tau)$. See Figure~\ref{fig:AZ-add-basepoints}.
\end{proof}

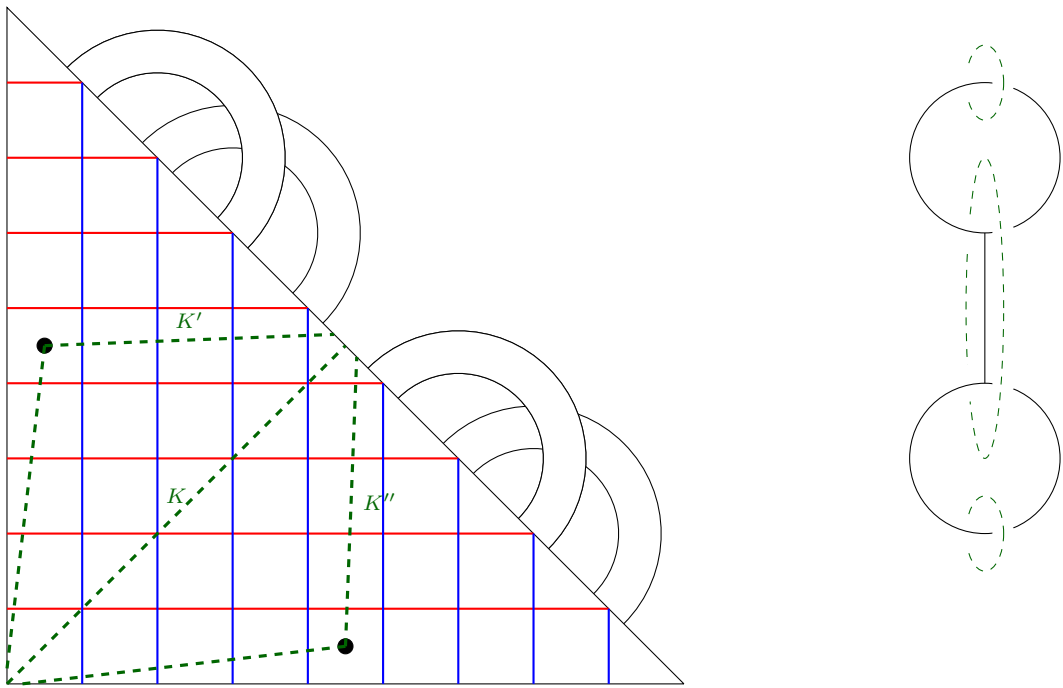
\begin{figure}
  \centering
  \begin{tikzpicture}[alt={Adding basepoints to AZ(PMC)}]
  \begin{scope}[rotate=180]
    \draw (.8,8.2) arc (-225:-45:1.414*1.2);
    \draw (1.2,7.8) arc (-225:-45:1.414*.8);
    \draw (2.2,6.8) arc (-225:-45:1.414*.8);
    \draw (1.8,7.2) arc (-225:-45:1.414*1.2);
    \begin{scope}[even odd rule]
    \draw[fill=white] (1.8,7.2) arc (-225:-45:1.414*1.2) (2.2,6.8) arc (-225:-45:1.414*.8);
    \end{scope}
    \draw (4.8,4.2) arc (-225:-45:1.414*1.2);
    \draw (5.2,3.8) arc (-225:-45:1.414*.8);
    \draw (6.2,2.8) arc (-225:-45:1.414*.8);
    \draw (5.8,3.2) arc (-225:-45:1.414*1.2);
    \begin{scope}[even odd rule]
    \draw[fill=white] (5.8,3.2) arc (-225:-45:1.414*1.2) (6.2,2.8) arc (-225:-45:1.414*.8);
    \end{scope}
    \draw (0,9) to (9,9) to (9,0) to (0,9);
    \foreach \y in {1,2,3,4,5,6,7,8}
      \draw[red, thick] (9,\y) to (9-\y,\y); 
    \foreach \x in {1,2,3,4,5,6,7,8}
      \draw[blue,thick] (\x,9) to (\x,9-\x);
    \draw[darkgreen,dashed, very thick] (9,9) to node[above]{\lab{K}} (4.5,4.5);
    \filldraw (4.5,8.5) circle (.1);
    \filldraw (8.5,4.5) circle (.1);
    \draw[darkgreen, dashed, very thick] (4.5,8.5) -- node[right]{\lab{K''}} (4.35,4.65);
    \draw[darkgreen, dashed, very thick] (4.5,8.5) -- (8.8,9);
    \draw[darkgreen, dashed, very thick] (8.5,4.5) -- node[above]{\lab{K'}} (4.65,4.35);
    \draw[darkgreen, dashed, very thick] (8.5,4.5) -- (9,8.8);
  \end{scope}
  \begin{scope}[xshift = 4cm, yshift=-2cm]
    \draw[darkgreen, dashed] (0,1) ellipse (.25 cm and .5 cm);
    \fill[white] (-.25,1) circle (.25);
    \draw[darkgreen, dashed] (0,-5) ellipse (.25 cm and .5 cm);
    \fill[white] (-.25,-5) circle (.25);
    \draw[darkgreen, dashed] (0,-2) ellipse (.25 cm and 2 cm);
    \fill[white] (-.25,-3) circle (.25);
    \fill[white] (-.25,-1) circle (.25);
    \draw (0,0) circle (1);
    \draw (0,-4) circle (1);
    \fill[white] (.25,1) circle (.15);
    \fill[white] (.25,-5) circle (.15);
    \fill[white] (.25,-3) circle (.15);
    \fill[white] (.25,-1) circle (.15);
    \draw [-] (0,-1) -- (0,-3);
    \begin{scope}
      \clip (.25,1) circle (.15);
      \draw[darkgreen, dashed] (0,1) ellipse (.25 cm and .5 cm);
    \end{scope}
    \begin{scope}
      \clip (.25,-5) circle (.15);
      \draw[darkgreen, dashed] (0,-5) ellipse (.25 cm and .5 cm);
    \end{scope}
    \begin{scope}
      \clip (.25,-3) circle (.15);
      \draw[darkgreen, dashed] (0,-2) ellipse (.25 cm and 2 cm);
    \end{scope}
    \begin{scope}
      \clip (.25,-1) circle (.15);
      \draw[darkgreen, dashed] (0,-2) ellipse (.25 cm and 2 cm);
    \end{scope}
  \end{scope}
  \end{tikzpicture}

  \caption{\textbf{Adding basepoints to $\AZ(\PMC)$.} Left: the diagram with the extra basepoints, and the knots $K$, $K'$, and $K''$ indicated (\textcolor{darkgreen}{dashed}). Recall that there is a disk glued to the outer boundary of the diagram; $K$, $K'$, and $K''$ close up in that disk. Right: a schematic of the sutures; in this picture, $K'\cup K''\cup A$ is solid and the sutures are \textcolor{darkgreen}{dashed}.}
  \label{fig:AZ-add-basepoints}
\end{figure}

\section{Detours}\label{sec:not-here}

\begin{center}
  \textit{Yet all experience is an arch wherethro'\\
  Gleams that untravell'd world whose margin fades\\
  For ever and for ever when I move.} ---Tennyson, ``Ulysses'' 
\end{center}

Our goal in this paper has been to give a usable algorithm to compute $\HFRa(Y)$ for branched double covers $Y$ of knots in (closed) 3-manifolds, and we have developed the general theory of real bordered Heegaard Floer homology only as needed for that computation. We conclude by noting results we have not covered, which a researcher interested in further developing real bordered Heegaard Floer homology might (need to) explore. A topic being on this list does not imply there is some unexpected difficulty with it: we have avoided addressing some presumably easy problems to streamline this paper.

\begin{enumerate}
\item Prove that every real bordered 3-manifold admits a real bordered Heegaard diagram. Equivalently, prove that for any real bordered 3-manifold $(Y,\tau)$, there is a connected, separating,  $\tau$-invariant surface $F\subset Y$ (containing the fixed set). (To deduce the first statement from the second, start with a real bordered Heegaard diagram for a real thick copy of $F$, and then glue on bordered Heegaard diagrams for the rest of $Y$.)
\item Define $\CFARa(Y,\tau)$ for a real bordered 3-manifold $(Y,\tau)$, satisfying a pairing theorem with $\CFDAa$. (Proposition~\ref{prop:CFAR-AZ} will presumably be a special case; in fact, by gluing on a bordered manifold, Proposition~\ref{prop:CFAR-AZ} induces a definition of $\CFARa(Y,\tau)$ whenever the fixed set is connected and the manifold admits a connected, separating, $\tau$-invariant surface.)
\item Determine a complete set of real bordered Heegaard moves, and prove that $\CFDRa$ and $\CFARa$ are invariant under them.
\item Compute the invariants of enough real thick surfaces to extend the algorithm from Section~\ref{sec:compute-HFRa} to compute the real bordered Floer homology of all closed real 3-manifolds, not just ones with connected fixed set.
\end{enumerate}

The paper also suggests some natural questions:
\begin{enumerate}[resume]
\item Propositions~\ref{prop:g1-one-fixed-circs} and~\ref{prop:g1-two-fixed-circs} gave interpretations of $\HFRa$ in very special cases in terms of ordinary sutured Floer homology groups. Are there similar interpretations in other cases, or in general?
\item Is there a version of real bordered Floer homology where the fixed sets contain arcs with endpoints on the bordered boundary, instead of just closed 1-manifolds in the interior?
\item What is the natural algebraic or Fukaya-categorical framework for real bordered Heegaard Floer homology? That is, is there a purely algebraic setting in which analogues of the real bordered invariants of thickened surfaces are defined?
\end{enumerate}

\begin{center}
  \textit{How dull it is to pause, to make an end}
\end{center}

\bibliographystyle{hamsalpha}\bibliography{heegaardfloer}
\end{document}